\documentclass{article} % For LaTeX2e
\usepackage{iclr2025_conference,times}

\usepackage{hyperref}
\usepackage{url}

\newcommand{\inner}[2]{\left\langle #1, #2 \right\rangle}

\newcommand{\bbS}{\mathbb{S}}

\renewcommand{\d}[1]{\ensuremath{\operatorname{d}\!{#1}}}

\newcommand{\marina}{\maroon{\text{MARINA}}}
\newcommand{\gd}{\maroon{\text{GD}}}
\newcommand{\sgd}{\maroon{\text{SGD}}}
\newcommand{\cgd}{\maroon{\text{CGD}}}
\newcommand{\detmarina}{\maroon{\text{det-MARINA}}}
\newcommand{\detdasha}{\maroon{\text{det-DASHA}}}
\newcommand{\dasha}{\maroon{\text{DASHA}}}
\newcommand{\diana}{\maroon{\text{DIANA}}}
\newcommand{\dcgd}{\maroon{\text{DCGD}}}
\newcommand{\detcgdone}{\maroon{\text{det-CGD1}}}
\newcommand{\detcgdtwo}{\maroon{\text{det-CGD2}}}
\newcommand{\detcgd}{\maroon{\text{det-CGD}}}
\newcommand{\newtonstar}{\maroon{\text{NS}}}
\newcommand{\newton}{\maroon{\text{NM}}}
\newcommand{\sgdstar}{\maroon{$\sgd_{\star}$}}
\newcommand{\svrg}{\maroon{\text{SVRG}}}
\newcommand{\sag}{\maroon{\text{SAG}}}
\newcommand{\sdca}{\maroon{\text{SDCA}}}
\newcommand{\saga}{\maroon{\text{SAGA}}}
\newcommand{\miso}{\maroon{\text{MISO}}}
\newcommand{\katu}{\maroon{\text{Katyusha}}}
\newcommand{\lsvrg}{\maroon{\text{L-SVRG}}}
\newcommand{\lkatu}{\maroon{\text{L-Katyusha}}}
\newcommand{\storm}{\maroon{\text{STORM}}}

\definecolor{maroon}{cmyk}{0,0.87,0.68,0.32}
\definecolor{azure}{rgb}{0.2, 0.5, 1.0}
\definecolor{dkgreen}{rgb}{0.0, 0.40, 0.10}
\definecolor{airforceblue}{rgb}{0.36, 0.54, 0.66}
\newcommand{\maroon}[1]{{\color{maroon} #1}}

\usepackage{multirow}
\usepackage{calligra}
\usepackage{layout}

\usepackage{amsmath}
\usepackage{amsthm}
\usepackage{amssymb}
\usepackage{amsfonts}
\usepackage{mathtools}

\usepackage{url}
\usepackage{color}
\usepackage{graphicx}
\usepackage{algorithmic}
\usepackage{algorithm}
\usepackage{verbatim}
\usepackage{thmtools} 
\usepackage{thm-restate}
\usepackage{xcolor,colortbl}
\definecolor{midnightblue}{HTML}{0059b3}
\definecolor{darkmidnightblue}{HTML}{154c84}
\definecolor{noonblue}{HTML}{e5eef7}
\definecolor{chromered}{HTML}{f14233}
\definecolor{darkgreen}{HTML}{0e6029}

\usepackage{nicefrac}
\usepackage{adjustbox}

\usepackage{cleveref}

\newcommand{\norm}[1]{\left\| #1 \right\|}

\newcommand{\cO}{\mathcal{O}}

\newcommand{\cS}{\mathcal{S}}

\newcommand{\del}[1]{}

\newcommand{\R}{\mathbb{R}} % reals
\newcommand{\eqdef}{:=}

\newcommand{\Exp}[1]{{\mathbb E}\left[#1\right]}
\newcommand{\ExpCond}[2]{{\rm E}\left[\left.#1\right\vert#2\right]}
\newcommand{\ExpSub}[2]{{\rm E}_#1\left[#2\right]}
\newcommand{\ExpS}[1]{{\mathbb E}_{\mathcal{S}}\left[#1\right]}

\DeclareMathOperator{\diag}{diag}       % diag(D) = the diagonal vector of matrix D
\newtheorem{assumption}{Assumption}
\newtheorem{lemma}{Lemma}

\newtheorem{theorem}{Theorem}
\newtheorem{proposition}{Proposition}

\newtheorem{corollary}{Corollary}
\newtheorem{definition}{Definition}
\newtheorem{fact}{Fact}
\newtheorem{remark}{Remark}

\theoremstyle{definition}

\newcommand{\brr}[1]{\left( #1 \right)}   % brackets round
\newcommand{\brs}[1]{\left[ #1 \right]}  % brackets square
\newcommand{\brc}[1]{\left\{ #1 \right\}} % brackets curly

\usepackage[colorinlistoftodos,bordercolor=orange,backgroundcolor=orange!20,linecolor=orange,textsize=tiny]{todonotes}

\usepackage{amsmath,amsfonts,bm}

\def\1{\bm{1}}

\def\mA{{\bm{A}}}
\def\mB{{\bm{B}}}

\def\mD{{\bm{D}}}

\def\mF{{\bm{F}}}

\def\mI{{\bm{I}}}

\def\mL{{\bm{L}}}
\def\mM{{\bm{M}}}
\def\mN{{\bm{N}}}
\def\mO{{\bm{O}}}

\def\mQ{{\bm{Q}}}

\def\mS{{\bm{S}}}
\def\mT{{\bm{T}}}

\def\mW{{\bm{W}}}
\def\mX{{\bm{X}}}

\DeclareMathAlphabet{\mathsfit}{\encodingdefault}{\sfdefault}{m}{sl}
\SetMathAlphabet{\mathsfit}{bold}{\encodingdefault}{\sfdefault}{bx}{n}

\newcommand{\E}{\mathbb{E}}

\usepackage{makecell}
\usepackage{booktabs}
\usepackage{enumitem}
\usepackage{subfigure}
\usepackage{threeparttable}

\title{Variance Reduced Distributed Nonconvex Optimization Using Matrix Stepsizes}

\author{Hanmin Li, Avetik Karagulyan \& Peter Richt\'{a}rik \\
The AI Initiative, King Abdullah University of Science and Technology, Saudi Arabia \\
\texttt{\{hanmin.li, avetik.karagulyan, peter.richtarik\}@kaust.edu.sa} \\
}

\iclrfinalcopy % Uncomment for camera-ready version, but NOT for submission.
\begin{document}

\maketitle

\begin{abstract}
   Matrix-stepsized gradient descent algorithms have been shown to have superior performance in non-convex optimization problems compared to their scalar counterparts. 
   The {\detcgd} algorithm, as introduced by \citet{li2023det}, leverages matrix stepsizes to perform compressed gradient descent for non-convex objectives and matrix-smooth problems in a federated manner. 
   The authors establish the algorithm's convergence to a neighborhood of a weighted stationarity point under a convex condition for the symmetric and positive-definite matrix stepsize. 
   In this paper, we propose two variance-reduced versions of the {\detcgd} algorithm, incorporating {\marina} and {\dasha} methods. 
   Notably, we establish theoretically and empirically, that {\detmarina} and {\detdasha} outperform {\marina}, {\dasha} and the distributed {{\detcgd}} algorithms in terms of iteration and communication complexities. 
\end{abstract}

\addtocontents{toc}{\protect\setcounter{tocdepth}{0}}

\section{Introduction}

We focus on optimizing the finite sum non-convex objective
\begin{equation}\label{eq:objective}
\min_{x\in \R^d} \left\{ f(x) \eqdef \frac{1}{n}\sum_{i=1}^{n} f_i(x) \right\}.
\end{equation}

In this context, each function $f_i : \mathbb{R}^d \rightarrow \mathbb{R}$ is differentiable and bounded from below. 
This type of objective function finds extensive application in various practical machine learning algorithms, which increase not only in terms of the data size but also in the model size and overall complexity as well.
As a result, most neural network architectures result in highly non-convex empirical losses, which need to be minimized. 
In addition, it becomes computationally infeasible to train these models on one device, often excessively large, and one needs to redistribute them amongst different devices/clients.
This redistribution results in a high communication overhead, which often becomes the bottleneck in this framework.

In other words, we have the following setting. 
The data is partitioned into $n$ {clients}, where the $i$-th client has access to the component function $f_i$ and its derivatives. 
The clients are connected to each other through a central device, called {the server}. 
In this work, we are going to study iterative gradient descent-based algorithms that operate as follows. The clients compute the local gradients in parallel. Then they  compress these gradients to reduce the communication cost and send them to the server in parallel. 
The server then aggregates these vectors and broadcasts the iterate update back to the clients. This meta-algorithm is called federated learning. 
We refer the readers to \citet{Konecny2016federated,mcmahan2017communication,kairouz2021advances} for a more thorough introduction to federated learning.

\subsection{Contributions} 
In this paper, we introduce two novel federated learning algorithms named {\detmarina} and {\detdasha}. 
These algorithms extend a recent method called {\detcgd} \citep{li2023det}, which aims to solve problem \eqref{eq:objective} using matrix stepsized gradient descent. 
Under the matrix smoothness assumption proposed by \citet{safaryan2021smoothness}, the authors demonstrate that the matrix stepsized version of the Distributed Compressed Gradient 
Descent \citep{khirirat2018distributed} algorithm enhances communication complexity compared to its scalar counterpart. 
However, in their analysis, \citet{li2023det} show stationarity only within a certain neighborhood due to stochastic compressors.
Our algorithm addresses this issue by incorporating previously known variance reduction schemes, namely,  
{\marina} \citep{gorbunov2021marina} and {\dasha} \citep{tyurin2024dasha}.
We establish theoretically and empirically, that both algorithms outperform their scalar alternatives, as well as the distributed {\detcgd} algorithms.
In addition, we describe specific matrix stepsize choices, for which our algorithms beat {\marina}, {\dasha} and distributed {\detcgd} both in theory and in practice.

\section{Background and motivation}\label{sec:background}

For a given $\varepsilon > 0$, finding an approximately global optimum, that is $x_{\varepsilon}$ such that $f(x_{\varepsilon}) - \min_x f(x) < \varepsilon$, is known to be NP-hard  \citep{jain2017non, danilova2022recent}. 
However, gradient descent based methods are still useful in this case.
When these methods are applied to non-convex objectives, they treat the function $f$ as locally convex and aim to converge to a local minimum. 
Despite this simplification, such methods have gained popularity in practice due to their superior performance compared to other approaches for non-convex optimization, such as convex relaxation-based methods \citep{tibshirani1996regression,cai2010singular}.        

\subsection{Stochastic gradient descent} 

Arguably, one of the most prominent meta-methods for tackling non-convex optimization problems is stochastic gradient descent ({\sgd}). The formulation of {\sgd} is presented as the following iterative algorithm:
$x^{k+1} = x^k - \gamma g^k$.
Here, $g^k \in \mathbb{R}^d$ serves as a stochastic estimator of the gradient $\nabla f(x^k)$. 
{\sgd} essentially mimics the classical gradient descent algorithm, and recovers it when $g^k = \nabla f(x^k)$. 
In this scenario, the method approximates the objective function $f$ using a linear function and takes a step of size $\gamma$ in the direction that maximally reduces this approximation. 
When the stepsize is sufficiently small, and the function $f$ is suitably smooth, it can be demonstrated that the function value decreases, as discussed in \citep{bubeck2015convex,gower2019sgd}.

However, computing the full gradient can often be computationally expensive. 
In such cases, stochastic approximations of the gradient come into play. 
Stochastic estimators of the gradient can be employed for various purposes, leading to the development of different methods. 
These include stochastic batch gradient descent \citep{nemirovski2009robust,johnson2013accelerating,defazio2014saga}, randomized coordinate descent \citep{nesterov2012efficiency,wright2015coordinate}, and compressed gradient descent 
\citep{alistarh2017qsgd,khirirat2018distributed,mishchenko2019distributed}.
The latter, compressed gradient descent, holds particular relevance to this paper, and we will delve into a more detailed discussion of it in subsequent sections.

\subsection{Second order methods} 

The stochastic gradient descent is considered as a first-order method as it uses only the first order derivative information. 
Although being immensely popular, the first order methods are not always optimal. 
Not surprisingly, using higher order derivatives in deciding update direction can yield to faster algorithms. 
A simple instance of such algorithms is the Newton Star algorithm \citep{islamov2021distributed}:
\begin{equation} 
  x^{k+1} = x^k - \brr{\nabla^2 f(x^{\star})}^{-1} \nabla f(x^k), \tag{\newtonstar}
\end{equation} 
where $x^{\star}$ is the minimum point of the objective function. 
The authors establish that under specific conditions, the algorithm's convergence to the unique solution $x^{\star}$ in the convex scenario occurs at a local quadratic rate. 
Nonetheless, its practicality is limited since we do not have prior knowledge of the Hessian matrix at the optimal point. 
Despite being proposed recently, the Newton-Star algorithm gives a deeper insight on the generic Newton method \citep{gragg1974optimal, miel1980majorizing, yamamoto1987convergence}:
\begin{equation}\label{eq:newton}
  x^{k+1} = x^k - \gamma \brr{\nabla^2 f(x^k)}^{-1} \nabla f(x^k).  \tag{\newton}
\end{equation} 
Here, the unknown Hessian of the Newton-Star algorithm, is estimated progressively along the iterations. 
The latter causes elevated computational costs, as the inverting a large square matrix is expensive. 
As an alternative, quasi-Newton methods replace the inverse of the Hessian at the iterate with a computationally cheaper estimate \citep{broyden1965class, dennis1977quasi, al2007overview, al2014broyden}. 

\subsection{Fixed matrix stepsizes} 

The {\detcgd} algorithm falls into this framework of the second order methods as well.  Proposed by \citet{li2023det}\footnote{In the original paper, the algorithm is referred to as {\detcgd}, as there is a variant of the same algorithm named {\detcgdtwo}. Since we are going to use only the first one and our framework is applicable to both, we will remove the number in the end for the sake of brevity.}, the algorithm suggests using a uniform ``upper bound'' on the inverse Hessian matrix. 
Assuming matrix smoothness of the objective \citep{safaryan2021smoothness}, they replace the scalar stepsize with a positive definite matrix $\mD$. 
The algorithm is given as follows:
\begin{equation}\label{eq:detCGD1}
   x^{k+1} = x^k - \mD \mS^k \nabla f(x^k). 
   \tag{{\detcgd}}
\end{equation}

\paragraph{Matrix $\mD$.} Here, $\mD$ plays the role of the stepsize. 
   Essentially, it uniformly lower bounds the inverse Hessian.
   The standard {\sgd} is a particular case of this method, as the scalar stepsize $\gamma$ can be seen as a matrix 
   $\gamma \mI_d$, where $\mI_d$ is the $d$-dimensional identity matrix. 
   An advantage of using a matrix stepsize is more evident if we take the perspective of the second order methods. 
   Indeed, the scalar stepsize $\gamma \mI_d$ uniformly estimates the largest eigenvalue of the Hessian matrix, while $\mD$ can capture the Hessian more accurately. 
   The authors show both theoretical and empirical improvement that comes with matrix stepsizes. 

\paragraph{Matrix  $\mS^k$.}  
   We assume that $\mS^k$ is a positive semi-definite, stochastic sketch matrix. 
   Furthermore, it is  unbiased: $\E[\mS^k] = \mI_d$.
   We notice that  \ref{eq:detCGD1} can be seen as a matrix stepsize instance of {\sgd}, with $g^k = \mS^k \nabla f(x^k)$. 
   The sketch matrix can be seen as a linear compressing operator, hence the name of the algorithm: Compressed Gradient Descent ({\cgd}) \citep{alistarh2017qsgd,khirirat2018distributed}.
   A commonly used example of such a compressor is the Rand-$\tau$ compressor. 
   This compressor randomly selects $\tau$ entries from its input and scales them using a scalar multiplier to ensure an unbiased estimation. 
   By adopting this approach, instead of using all $d$ coordinates of the gradient, only a subset of size $\tau$ is communicated. 
   Formally, Rand-$\tau$ is defined as follows:
   \begin{equation}
      \mS = \frac{d}{\tau} \sum_{j = 1}^{\tau} e_{i_j}e_{i_j}^{\top}.
   \end{equation}
   Here, $e_{i_j}$ denotes the $i_j$-th standard basis vector in $\mathbb{R}^d$. 
   For a more comprehensive understanding of compression techniques, we refer to \citet{safaryan2022uncertainty}.

\subsection{The neighborhood of the distributed det-CGD1}
The distributed version of {{\detcgd}} follows the standard federated learning paradigm \citep{mcmahan2017communication}. 
The pseudocode of the method, as well as the convergence result of \citet{li2023det}, can be found in \Cref{sec:dist-detcgd}. 
Informally, their convergence result can be written as
\begin{equation*}
    \min_{k=1,\ldots,K} \E \left[\norm{\nabla f(x^k)}_{\mD}^2 \right] \leq \cO\brr{\frac{(1 + \alpha)^K}{K}} + \cO\brr{\alpha},
\end{equation*}
where $\alpha > 0$ is a constant that can be controlled. 
The crucial insight from this result is that the error bound does not diminish as the number of iterations increases. 
In fact, by controlling $\alpha$ and considering a large $K$, it is impossible to make the second term smaller than $\varepsilon$. This implies that the algorithm converges to a certain neighborhood surrounding the (local) optimum. 
This phenomenon is a common occurrence in {\sgd} and is primarily attributable to the {variance} associated with the stochastic gradient estimator. In the case of {\detcgd} the stochasticity comes from the sketch $\mS^k$.

\subsection{Variance reduction}

To eliminate this neighborhood, various techniques for reducing variance are employed. One of the simplest techniques applicable to {\cgd} is gradient shifting.
By replacing $\mS^k \nabla f(x^k)$ with $\mS^k (\nabla f(x^k) - \nabla f(x^{\star})) + \nabla f(x^{\star}),$ the neighborhood effect is removed from the general {\cgd}. 
This algorithm is an instance of a more commonly known method called {\sgdstar} \citep{gower2020variance}. 
However, since the exact optimum $x^{\star}$ is typically unknown, this technique encounters similar challenges as the Newton-Star algorithm mentioned earlier. 
Fortunately, akin to quasi-Newton methods, one can employ methods that iteratively learn the optimal shift \citep{shulgin2022shifted}. 
A line of research focuses on variance reduction for {\cgd} based algorithms on this insight. 

To eliminate the neighborhood in the distributed version of {\cgd}, denoted as {\detcgdone}, we apply a technique called {\marina} \citep{gorbunov2021marina}. 
{\marina} cleverly combines the general shifting \citep{shulgin2022shifted} technique with loopless variance reduction techniques \citep{qian2021svrg}. 
This approach introduces an alternative gradient estimator specifically designed for the federated learning setting. 
Thanks to its structure, it allows to establish an upper bound on the stationarity error that diminishes significantly with a large number of iterations. 
In this paper, we construct the analog of the this algorithm called {\detmarina}, using matrix stepsizes and sketch gradient compressors. 
For this new method, we prove a convergence guarantee similar to the results of \citet{li2023det} without a neighborhood term. 

Furthermore, we also propose {\detdasha}, which is the extension of {\dasha} in the matrix stepsize setting. 
The latter was proposed by \citet{tyurin2024dasha} and it combines {\marina} with momentum variance reduction techniques \citep{cutkosky2019momentum}. {\dasha} offers better practicality compared to {\marina}, as it always sends compressed gradients and does not need to synchronize among all the nodes. 

\subsection{Organization of the paper}

The rest of the paper is organized as follows. \Cref{sec:framework} discusses the general mathematical framework.
\Cref{sec:marina} and \Cref{sec:dasha} present the {\detmarina} and {\detdasha} algorithms, respectively.
We show the superior theoretical performance of our algorithms compared to the relevant existing algorithms, that is {\marina}, {\dasha} and {\detcgd} in \Cref{sec:complexities}.
The experimental results validating our theoretical findings are presented in \Cref{main:experiments}, with additional details and setups available in the Appendix. 
We conclude the paper by outlining several directions of future work in \Cref{sec:conclusion}.

\section{Mathematical framework}\label{sec:framework}

In this section we present the assumptions that we further require in the analysis.
\begin{assumption} {\rm (Lower Bound)}
   \label{assmp:1}
   There exists $f^{\star} \in \R$ such that, $f(x) \geq f^{\star}$ for all $x \in \R^d$.
\end{assumption}
This is a standard assumption in optimization, as otherwise the problem of minimizing the objective would not be correct mathematically. 
We then need a matrix version of Lipschitz continuity for the gradient.
\begin{definition}{\rm ($\mL$-Lipschitz Gradient)}
   \label{assmp:3}
   Assume that $f: \R^d \rightarrow \R$ is a continuously differentiable function and matrix $\mL \in \bbS^d_{++}$. We say the gradient of $f$ is $\mL$-Lipschitz if for all $ x, y \in \R^d$
   \begin{equation}
      \label{eq:assmp:3}
      \norm{\nabla f(x) - \nabla f(y)}_{\mL^{-1}} \leq \norm{x - y}_{\mL}.
   \end{equation}
\end{definition}
In the following, we will assume that \eqref{eq:assmp:3} is satisfied for component functions 
$f_i$. 
\begin{assumption}
    \label{assmp:4}
    Each function $f_i$ is $\mL_i$-gradient Lipschitz, while $f$ is $\mL$-gradient Lipschitz.
\end{assumption}
In fact, the second half of the assumption is a consequence of the first one. 
Below, we formalize this claim.
\begin{proposition}
    \label{ppst:4}
    If $f_i$ is $\mL_i$-gradient Lipschitz for every $i = 1,\ldots,n$, then function $f$ has $\mL$-Lipschitz gradient with $\mL \in \bbS^d_{++}$ satisfying
    \begin{equation*}
        \label{eq:ppst:4}
        \frac{1}{n}\sum_{i=1}^{n}\lambda_{\max}\left(\mL^{-1}\right)\cdot\lambda_{\max}\left(\mL_i\right)\cdot\lambda_{\max}\left(\mL_i\mL^{-1}\right) = 1.
    \end{equation*}  
\end{proposition}
\begin{remark}
    In the scalar case, where $\mL = L\mI_d$, $\mL_i = L_i\mI_d$, the relation becomes 
        $L^2 = \frac{1}{n}\sum_{i=1}^{n} L_i^2.$
    This corresponds to the statement in Assumption 1.2 in \citep{gorbunov2021marina}.
\end{remark}
Nevertheless, the matrix $\mL$ found according to \Cref{ppst:4} is only an estimate.  
In principle, there might exist a better $\mL_f \preceq \mL$ such that $f$ has 
$\mL_f$-Lipschitz gradient. 

More generally, this condition can be interpreted as follows. 
The gradient of $f$ naturally belongs to the dual space of $\R^d$, 
as it is defined as a linear functional on $\R^d$. 
In the scalar case,  $\ell_2$-norm is self-dual, thus \eqref{eq:assmp:3} reduces to the standard Lipschitz continuity of the gradient. However, with the matrix smoothness assumption, we are using the $\mL$-norm for the iterates,  which naturally induces the $\mL^{-1}$-matrix norm for the gradients in the dual space. 
This insight, which is originally presented by \citet{nemirovskij1983problem}, plays a key role in our analysis.

See \Cref{app:smoothness} for a more thorough discussion on the properties of \Cref{assmp:4}, as well as its connection to matrix smoothness \citep{safaryan2021smoothness}.

\section{\texorpdfstring{MARINA-based variance reduction}{MARINA-based variance reduction}}
\label{sec:marina}
In this section, we present our algorithm {\detmarina} with its convergence result.
We construct a sequence of vectors $g^k$ which are stochastic estimators of $\nabla f(x^k)$. 
At each iteration, the server samples a Bernoulli random variable (coin flip) $c_k$ and broadcasts it in parallel to the clients, along with the current gradient estimate $g^k$.
Each client, then, does a {\detcgd}-type update with the stepsize $\mD$ and a gradient estimate $g^k$. 
The next gradient estimate $g^{k+1}$ is then computed. With a low probability, that is when $c_k = 1$, we take the $g^{k+1}$ to be the full gradient $\nabla f(x^{k+1})$. Otherwise, we update it using the compressed gradient differences at each client. 
See Algorithm~\ref{alg:detCGD-MARINA} for the pseudocode of \detmarina.

\begin{algorithm}
   \caption{\detmarina}
   \label{alg:detCGD-MARINA}
   \begin{algorithmic}[1]
   \STATE {\bf Input:} starting point $x^0$, stepsize matrix $\mD$, probability $p \in (0, 1]$, number of iterations $K$
   \STATE Initialize $g^0 = \nabla f(x^0)$
   \FOR {$k=0,1,\ldots,K-1$}
      \STATE Sample $c_k \sim \text{Be}(p)$
      \STATE Broadcast $g^k$ to all workers
      \FOR {$i=1,2,\ldots$ in parallel} 
         \STATE $x^{k+1} = x^k - \mD \cdot g^k$
         \IF {$c_k = 1$} 
         \STATE $g_i^{k+1} = \nabla f_i(x^{k+1})$
         \ELSE 
         \STATE $g_i^{k+1} = g^k + \mS_i^k\left(\nabla f_i(x^{k+1}) - \nabla f_i(x^k)\right)$
         \ENDIF
      \ENDFOR
      \STATE $g^{k+1} = \frac{1}{n}\sum_{i=1}^{n} g_i^{k+1}$
   \ENDFOR
   \STATE {\bf Return:} $\tilde{x}^K$ chosen uniformly at random from $\{x^k\}_{k=0}^{K-1}$
   \end{algorithmic}
\end{algorithm}

\subsection{Convergence guarantees}
In the following theorem, we formulate one of  the main results of this paper, which guarantees the convergence of \Cref{alg:detCGD-MARINA} under the above-mentioned assumptions.
\begin{theorem}
   \label{thm:1:detCGDVR1}
   Assume that Assumptions \ref{assmp:1} and \ref{assmp:4} hold, and the following condition on stepsize matrix $\mD \in \bbS^d_{++}$ holds, 
   \begin{equation}
      \label{eq:stepsize-cond-thm1}
      \mD^{-1} \succeq \left(\frac{(1-p)\cdot R(\mD, \cS)}{np} + 1\right)\mL,
   \end{equation}
   where 
   $
      R(\mD, \cS) \eqdef \frac{1}{n}\sum_{i=1}^{n} \lambda_{\max}\left(\mL_i\right)\lambda_{\max}\left(\mL^{-\frac{1}{2}}\mL_i\mL^{-\frac{1}{2}}\right)$  $\times\lambda_{\max}\left(\Exp{\mS_i^k\mD\mS_i^k} - \mD\right).$
   Then, after $K$ iterations of {\detmarina}, we have 
   \begin{equation}
      \label{eq:thm1_result}
      \Exp{\norm{\nabla f(\tilde{x}^K)}^2_{\frac{\mD}{\det(\mD)^{1/d}}}} \leq \frac{2\left(f(x^0) - f^{\star}\right)}{\det(\mD)^{1/d}\cdot K}.
   \end{equation}
   Here, $\tilde{x}^K$ is chosen uniformly randomly  from the first $K$ iterates of the algorithm.
\end{theorem}
The criterion $\norm{\cdot}^2_{\mD / \det(\mD)^{1/d}}$ is the same as that used in \citet{li2023det}, known as determinant normalization.
The weight matrix of the matrix norm has determinant $1$ after normalization, which makes it comparable to the standard Euclidean norm.
\begin{remark}
   We notice that the right-hand side of the algorithm vanishes with the number of iterations, thus solving the neighborhood issue of the distributed {\detcgd}. 
   Therefore, {\detmarina} is indeed the variance reduced version of {\detcgd} in the distributed setting and has better convergence guarantees.
\end{remark}
\begin{remark}
   \Cref{thm:1:detCGDVR1} implies the following iteration complexity for the algorithm. In order to get an $\varepsilon^2$ stationarity error\footnote{We say a (possibly random) vector $x \in \R^d$ is an $\varepsilon$-stationary point of a possibly non-convex function $f:\R^d \mapsto \R$, if $\Exp{\norm{\nabla f(x)}^2} \leq \varepsilon^2$. The expectation is over the randomness of the algorithm}, the algorithm requires $K$ iterations, with
   \begin{equation*}
      K \geq \frac{2(f(x^0) - f^{\star})}{\det(\mD)^{1/d}\cdot \varepsilon^2}.
   \end{equation*}
\end{remark}

\begin{remark}
   \label{remark:GD}
   In the case where no compression is applied, that is we have $\mS_i^k = \mI_d$, condition \eqref{eq:stepsize-cond-thm1} reduces to $\mD \preceq \mL^{-1}$.        
   The latter is due to $\Exp{\mS^k_i\mD\mS^k_i} = \mD$, which results in $R(\mD, \cS) = 0$. This is expected, since in the deterministic case {\detmarina} reduces to {\gd} with matrix stepsize.
\end{remark}
The convergence condition and rate of matrix stepsize {\gd} can be found in \citep{li2023det}. 
Below we do a sanity check to verify that the convergence condition for scalar {\marina} can be obtained.
\begin{remark} \label{rem:scalar-marina}
   Let us consider the scalar case. 
   That is $\mL_i = L_i\mI_d, \mL = L\mI_d, \mD = \gamma\mI_d$
      and $\omega = \lambda_{\max}\left(\Exp{\left(\mS_i^k\right)^{\top}\mS_i^k}\right) - 1.$
   Then, the condition \eqref{eq:stepsize-cond-thm1}  reduces to 
   % \begin{equation}
   %     \label{eq:rmk-scalar-cond}
   %     \frac{\gamma(1-p)\omega L^2}{np} - \frac{1}{\gamma} + L \leq 0.
   % \end{equation}
   % One can check that the  bound 
   \begin{equation*}
      \gamma \leq \brs{L\left(1 + \sqrt{\frac{(1-p)\omega}{pn}}\right)}^{-1}.
   \end{equation*}
   % implies \eqref{eq:rmk-scalar-cond}.
\end{remark}
The latter coincides with the stepsize condition of the convergence result of scalar {\marina}.

\subsection{Optimizing the matrix stepsize}
Now let us look at the right-hand side of \eqref{eq:thm1_result}. 
We notice that it decreases in terms of the determinant of the stepsize matrix. 
Therefore, one needs to solve the following optimization problem to find the optimal stepsize:
\begin{eqnarray*}
    \text{minimize } && \log\det(\mD^{-1}) \\
    \text{subject to } && \mD \text{  satisfying  } \eqref{eq:stepsize-cond-thm1}. \\
\end{eqnarray*}
The solution of this constrained minimization problem on $\bbS^d_{++}$ is not explicit. 
In theory, one may show that the constraint \eqref{eq:stepsize-cond-thm1} is convex and attempt to solve the problem numerically. 
However, as stressed by \citet{li2023det}, the similar stepsize condition for {\detcgd} is not easily computed using solvers like {CVXPY} \citep{diamond2016cvxpy}. 
Instead, we may relax the problem to certain linear subspaces of $\bbS^d_{++}$. 
In particular, we fix a matrix $\mW \in \bbS^d_{++}$, and define $\mD \eqdef \gamma\mW$. 
Then, the condition on the matrix $\mD$ becomes a condition for the scalar $\gamma$, which is given in the following corollary.
\begin{corollary}
   \label{ppst:optimal-D-var}
   Let $\mW \in \bbS^d_{++}$, defining $\mD \eqdef \gamma\cdot\mW$, where $\gamma \in \R_{+}$. then the condition in \eqref{eq:stepsize-cond-thm1} reduces to the following condition on $\gamma$
   \begin{equation}
      \label{eq:opt-cond-var}
      \gamma \leq \frac{2\lambda_{\mW}}{1 + \sqrt{1 + 4\alpha\beta\cdot\Lambda_{\mW, \cS}\lambda_{\mW}}},
   \end{equation}
   where $ \Lambda_{\mW, \cS} = \lambda_{\max}\left(\Exp{\mS_i^k\mW\mS_i^k} - \mW\right)$, 
   $ \lambda_{\mW} = \lambda_{\max}^{-1}\big(\mW^\frac12\mL\mW^\frac{1}{2}\big)$, $\alpha = \frac{1 - p}{np}$ and $\beta = \frac{1}{n}\sum_{i=1}^{n}\lambda_{\max}
   \left(\mL_i\right)\cdot\lambda_{\max}\left(\mL^{-1}\mL_i\right)$.
   % \begin{equation*}
   %     \alpha = \frac{1 - p}{np}; \quad \beta = \frac{1}{n}\sum_{i=1}^{n}\lambda_{\max}
   %     \left(\mL_i\right)\cdot\lambda_{\max}\left(\mL^{-1}\mL_i\right).
    % \end{equation*}
\end{corollary}
This means that for every fixed $\mW$, we can find the optimal scaling coefficient $\gamma$. 
In section \Cref{sec:complexities}, we will use this corollary to prove that a suboptimal matrix step size, determined in this efficient way, is already better than the optimal scalar step size.

\textbf{Extension to  {\detcgdtwo}. }
A variant of {\detcgd}, called {\detcgdtwo}, was also proposed by \citet{li2023det}. 
This algorithm, has the same structure as {\detcgd} with the sketch and stepsize interchanged. 
It was shown, that this algorithm has explicit stepsize condition in the single node setting.  
In \Cref{app:detcgd2}, we propose the variance reduced extension of the distributed {\detcgdtwo} following the {\marina} scheme.

\section{\texorpdfstring{DASHA-based variance reduction}{DASHA-based variance reduction}}
\label{sec:dasha}
In this section, we present our second algorithm based on {\dasha}. 
The latter utilizes a different type of variance reduction based on momentum (MVR).
Compared to {\marina}, {dasha} makes simpler optimization steps and does not require periodic synchronization with all the nodes. 
Notice that one may further simplify the notations here used in the algorithm. 
However, we keep it this way as it is consistent with \citep{tyurin2024dasha}.

\begin{algorithm}
  \caption{\detdasha}
  \label{alg:detCGD-DASHA}
  \begin{algorithmic}[1]
  \STATE {\bf Input:} starting point $x^0 \in \R^d$, stepsize matrix $\mD \in \bbS^d_{++}$, momentum $a \in (0, 1]$, number of iterations $K$
  \STATE Initialize $g_i^0, h_i^0 \in \R^d$ on the nodes and $g^0 = \frac{1}{n}\sum_{i=1}^{n} g_i^0$ on the server
  \FOR {$k=0,1,\ldots,K-1$}
      \STATE $x^{k+1} = x^k - \mD\cdot g^k$
      \STATE Broadcast $x^{k+1}$ to all nodes
      \FOR {$i=1,2,\ldots n$ in parallel} 
          \STATE $h_i^{k+1} = \nabla f_i(x^{k+1})$
          \STATE $m_i^{k+1} = \mS_i^k\left(h_i^{k+1} - h_i^k - a\left(g_i^k - h_i^k\right)\right)$
          \STATE $g_i^{k+1} = g_i^k + m_i^{k+1}$
          \STATE Send $m_i^{k+1}$ to the server.
      \ENDFOR
      \STATE $g^{k + 1} = g^k + \frac{1}{n}\sum_{i=1}^{n}m_i^{k+1}$
  \ENDFOR
  \STATE {\bf Return:} $\tilde{x}^K$ chosen uniformly at random from $\{x^k\}_{k=0}^{K-1}$
  \end{algorithmic}
\end{algorithm}

\subsection{Theoretical guarantees}
\begin{theorem}
   \label{dasha:thm:main}
   Suppose that Assumptions \ref{assmp:1} and \ref{assmp:4} hold. Let us initialize $g_i^0 = h_i^0 = \nabla f_i(x^0)$ for all $i \in [n]$ in \Cref{alg:detCGD-DASHA}, and define 
   \begin{align*}
     \label{dasha:eq:def-two-sign}
     \Lambda_{\mD, \cS} = \lambda_{\max}\left(\Exp{\mS_i^k\mD\mS_i^k} - \mD\right), \quad \omega_{\mD} = \lambda_{\max}\left(\mD^{-1}\right)\cdot\Lambda_{\mD, \cS}.
   \end{align*}
   If $a = \frac{1}{2\omega_{\mD} + 1}$, and  the following condition on stepsize $\mD \in \bbS^d_{++}$ is satisfied
   \begin{equation*}
     \label{dasha:eq:cond-D-main-thm}
     \mD^{-1} \succeq \mL - \frac{4\lambda_{\max}\left(\mD\right)\omega_{\mD}\left(4\omega_{\mD} + 1\right)}{n^2}\sum_{i=1}^{n}\lambda_{\max}\left(\mL_i\right)\mL_i,
   \end{equation*}
   then the following inequality holds for the iterates of \Cref{alg:detCGD-DASHA}
   \begin{align*}
     \Exp{\norm{\nabla f(\tilde{x}^K)}^2_{\mD/\left(\det(\mD)\right)^{1/d}}} \leq \frac{2(f(x^0) - f^{\star})}{\det(\mD)^{1/d}\cdot K}.
   \end{align*}
   Here  $\tilde{x}^K$ is chosen uniformly randomly from the first K iterates of the algorithm.
\end{theorem}
\begin{remark}
   The term $\Lambda_{\mD, \cS}$ can be viewed as the matrix version of $\gamma \cdot \omega$, where $\omega$ is associated with the sketch, and $\gamma$ is the scalar stepsize. On the other hand, the $\omega_{\mD}$ is the extension of $\omega$ in matrix norm. Similar to \Cref{rem:scalar-marina}, plugging in scalar arguments in the algorithm, we recover the result from \cite{tyurin2024dasha}.
\end{remark}
Following the same scheme as in \Cref{sec:marina}, we choose $\mD = \gamma_{\mW} \cdot \mW$, where $\mW \in \bbS^d_{++}$. 
Thus, for a fixed $\mW$, we relax the problem of finding the optimal stepsize to the problem of finding the optimal scaling factor $\gamma_{\mW} > 0$.
\begin{corollary}
   \label{dasha:col:scaling}
   For a fixed $\mW \in \bbS^d_{++}$, the optimal scaling factor $\gamma_{\mW} \in \R_{+}$ is given by 
   \begin{equation*}
     \gamma_{\mW} = \frac{2\lambda_{\mW}}{1 + \sqrt{1 + 16C_{\mW}\lambda_{\min}\left(\mL\right)\cdot\lambda_{\mW}}},
   \end{equation*}
   where $C_{\mW} := {\lambda_{\max}\left(\mW\right)\cdot\omega_{\mW}\left(4\omega_{\mW} + 1\right)}/n$,
   and $\lambda_{\mW} := \lambda_{\max}^{-1}\left(\mL^\frac{1}{2}\mW\mL^\frac{1}{2}\right)$.
\end{corollary} 
We observe that the  structure of the optimal scaling factor for obtained above  is similar to the one obtained in \Cref{ppst:optimal-D-var}.

\paragraph{The availability of $\mL$:}
For both {\detmarina} and {\detdasha}, in order to determine the matrix stepsize, the knowledge of $\mL$ is needed, if $\mL$ is known, better complexities are guaranteed.
When $\mL$ is unknown, a closed-form solution can be obtained for generalized linear models.
In more general cases, $\mL_i$ can be treated as hyperparameters and estimated using first-order information via a gradient-based method \citep{wang2022theoretically}.
One can think of this as some type of preprocessing step, after which the matrices are learnt.

\section{Complexities of the algorithms}\label{sec:complexities}

\subsection{det-MARINA}

The following corollary formulates the iteration complexity for {\detmarina} for $\mW = \mL^{-1}$.
\begin{corollary}
   \label{col:iteration-comp-L-inv}
   If we take $\mW = \mL^{-1}$, then the condition \eqref{eq:opt-cond-var} on $\gamma$ is given by
   \begin{equation}
      \label{eq:max-gamma-L-inv}
      \gamma \leq {2}\brr{1 + \sqrt{1 + 4\alpha\beta\cdot\Lambda_{\mL^{-1},\cS}}}^{-1}.
   \end{equation}
   In order to satisfy $\varepsilon$-stationarity, that is $\Exp{\norm{\nabla f(\tilde{x}^K)}^2_{\frac{\mD }{ \det(\mD)^{1/d}}}} \leq \varepsilon^2$, we require
   \begin{equation*}
      K \geq \cO\left(\frac{\Delta_0\cdot\det(\mL)^\frac{1}{d}}{\varepsilon^2} \cdot {\left(1 + \sqrt{1 + 4\alpha\beta\cdot\Lambda_{\mL^{-1},\cS}}\right)}\right),
   \end{equation*}
   where $\Delta_0 \eqdef f(x^0) - f(x^{\star})$.  
   Moreover, this iteration complexity is always better than the one of {\marina}.
\end{corollary}
The proof can be found in the Appendix.
In fact, we can show that in cases where we fix $\mW = \mI_d$ and $\mW = \diag^{-1}\left(\mL\right)$, the same conclusion also holds, relevant details can be found in \Cref{sec:cmp-step}. 
This essentially means that {\detmarina} always has a ``larger" stepsize compared to {\marina}, even if the stepsize is suboptimal for the sake of efficiency, which leads to a better iteration complexity. 
In addition, because we are using the same compressor for those two algorithms, the communication complexity of {\detmarina} is also provably better than that of {\marina}.

In order to compute the communication complexity, we borrow the concept of expected 
density from \citet{gorbunov2021marina}.
\begin{definition}
   For a given sketch matrix $\mS \in \bbS^d_{+}$, the expected density is defined as
   \begin{equation*}
      \label{eq:def-exp-density}
      \zeta_{\mS} = \sup_{x \in \R^d}\Exp{\norm{\mS x}_0},
   \end{equation*}
   where $\norm{x}_0$ denotes the number of non-zero components of $x \in \R^d$.
\end{definition}
In particular, we have $\zeta_{\texttt{Rand-$\tau$}} = \tau$. 
Below, we state the communication complexity of {\detmarina} with $\mW = \mL^{-1}$ and the Rand-$\tau$ compressor.

\begin{corollary}
   \label{col:communication-comp}
   Assume that we are using sketch $\mS \sim \cS$ with expected density $\zeta_{\cS}$. 
   Suppose also we are running {\detmarina} with probability $p$ and we use the optimal stepsize matrix with respect to $\mW = \mL^{-1}$. 
   Then the overall communication complexity of the algorithm is given by $\cO\big((Kp+1)d + (1-p)K\zeta_{\cS}\big)$. 
   Specifically, if we pick $p = \zeta_{\cS}/d$,
   then the communication complexity is given by 
   \begin{equation*}
      \cO\left(d + \frac{\Delta_0\det(\mL)^\frac{1}{d}}{\varepsilon^2}\left(\zeta_{\cS} + \sqrt{\frac{\beta}{n}\Lambda_{\mL^{-1}, \cS}\zeta_{\cS}(d - \zeta_{\cS})}\right)\right).
   \end{equation*}
\end{corollary}
Notice that in case where no compression is applied, the communication complexity reduces to $\cO( \nicefrac{d\Delta_0\cdot\det(\mL)^\frac{1}{d}}{\varepsilon^2})$. The latter coincides with the rate of matrix stepsize GD (see \citep{li2023det}).
Therefore, the dependence on $\varepsilon$ is not possible to improve further since {\gd} is optimal among first order methods \citep{carmon2020lower}.

{\begin{figure*}[h]
   \centering
     \subfigure{
       \begin{minipage}[t]{0.96\textwidth}
         \includegraphics[width=0.32\textwidth]{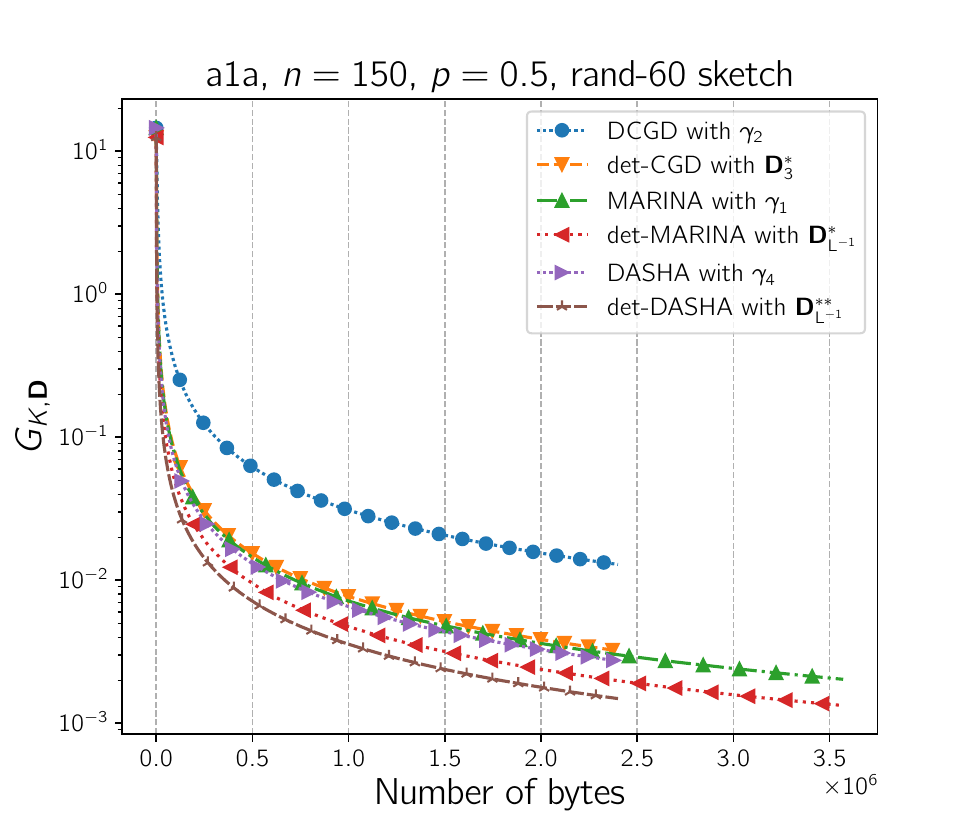}
         \includegraphics[width=0.32\textwidth]{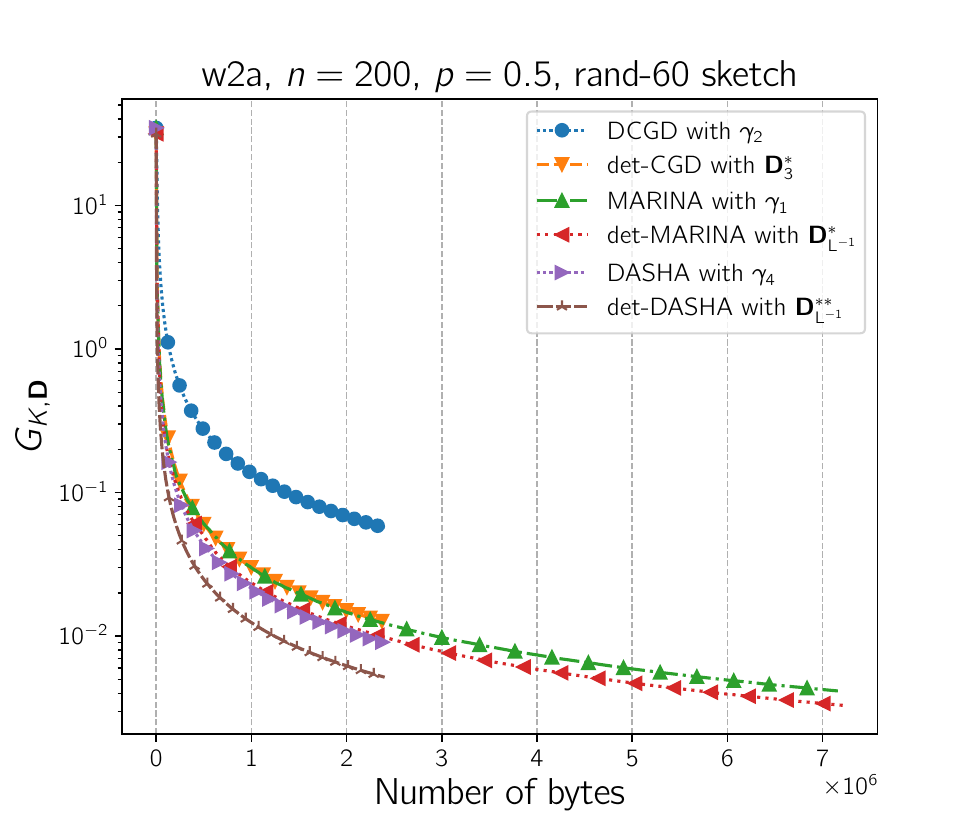}
         \includegraphics[width=0.322\textwidth]{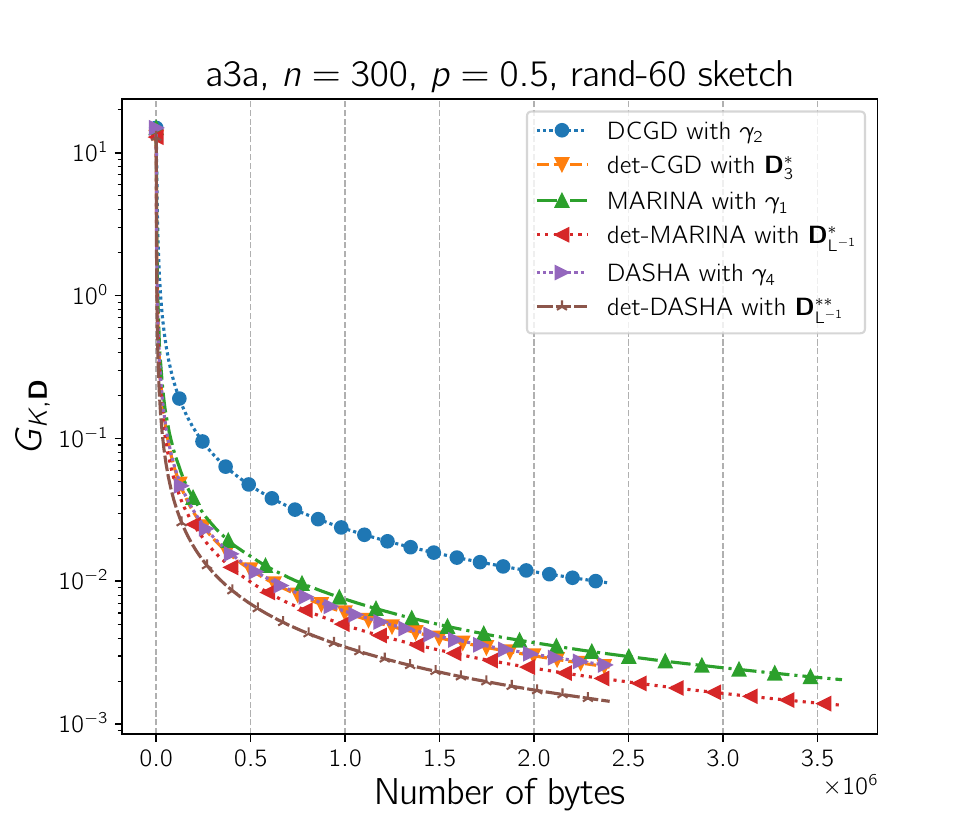}
       \end{minipage}
     }
     \caption{Comparison of {\dcgd} with optimal scalar stepsize, {\detcgd} with matrix stepsize $\mD^*_3$, {\marina} with optimal scalar stepsize, {\dasha} with optimal scalar stepsize, {\detmarina} with optimal stepsize $\mD^{*}_{\mL^{-1}}$ and {\detdasha} with optimal stepsize $\mD^{**}_{\mL^{-1}}$. Throughout the experiment, we are using Rand-$\tau$ sketch with $\tau = 60$, and each algorithm is run for a fixed number of iterations $K=10000$. 
     The $G_{K, \mD}$ in the y-axis is defined in \eqref{eq:def-G--KD}, which is the average squared matrix norm of the gradients.
     }
     \label{fig:experiment-7-dasha}
\end{figure*}}

\subsection{Det-DASHA} 
The difference of compression mechanisms, does not allow to have a direct comparison of the complexities of these algorithms.  
In particular, {\detmarina} compresses the gradient difference with some probability $p$, while {\detdasha} compresses the gradient difference with momentum in each iteration.
\begin{corollary}
  \label{dasha:col:complexity}
  If we pick $\mD = \gamma_{\mL^{-1}} \cdot \mL^{-1}$, then in order to reach an $\varepsilon^2$ stationary point, {\detdasha} needs $K$ iterations with
  \begin{equation*}
   K \geq \frac{f(x^0) - f^{\star}}{\det(\mL)^{-\frac{1}{d}}\varepsilon^2}\left(1 + \sqrt{1 + 16C_{\mL^{-1}}\lambda_{\min}\left(\mL\right)}\right).
  \end{equation*}
\end{corollary}
The following corollary compares the complexities of {\dasha} and {\detdasha}.
For the sake of brevity, we defer the complexities and other details to the proof of this corollary.
\begin{corollary}
   \label{col:dasha:1}
   Suppose that the conditions in \Cref{dasha:thm:main} hold, then compared to {\dasha}, {\detdasha} with $\mW = \mL^{-1}$ always has a {\bf \textit{better}} iteration complexity, therefore, communication complexity as well.
\end{corollary}

The following corollary suggests that the communication complexity of {\detdasha} is better than that of {\detmarina}, 
\begin{corollary}
   \label{col:dasha:2}
   The iteration complexity of {\detmarina} with $p = \nicefrac{1}{(\omega_{\mL^{-1}} + 1)}$ and {\detdasha} with momentum $\nicefrac{1}{(2\omega_{\mL^{-1}} + 1)}$ is the same, therefore the communication complexity of {\detdasha} is {\bf \textit{better than}} the communication complexity of {\detmarina}.
\end{corollary}
This is expected since the same relation occurs between {\marina} and {\dasha} as it is described by \citet[Table 1]{tyurin2024dasha}. We refer the readers to \Cref{sec:dasha:corollaries}.

\section{Experiments}\label{main:experiments}
This section contains several plots which confirm our theoretical improvements on the existing methods.
\Cref{fig:experiment-7-dasha} shows that the performance in terms of communication complexity of {\detdasha} and {\detmarina} is better than their scalar counterpart {\dasha} and {\marina} respectively. 
This validates the efficiency of using a matrix stepsize over a scalar stepsize. 
Further, we notice that {\detdasha} and {\detmarina} have better communication complexity in this case, compared to {\detcgd}. 
This demonstrates the effectiveness of applying variance reduction. 
% Notice that the optimal stepsizes of {\detcgd} and {\dcgd} require information of function value differences at 
% $x^{\star}$. 
% Furthermore the stepsizes are also constrained by the number of iterations $K$ and the error $\varepsilon^2$. 
% Meanwhile, for the variance reduced methods, we do not require such considerations, which is much more practical in general. 
Finally, as expected, {\detdasha} has better communication complexity than {\detmarina}.
We refer the readers to the appendix for more technical details of the experiments.

\section{Future work}\label{sec:conclusion}

% \subsection{}
   i) In this paper, we have only considered (linear) sketches as the compression operator. However, there exists a variety of compressors which are useful in practice that do not fall into this category. Extending {\detcgd} and {\detmarina} for general unbiased compressors is a promising future work direction. 
   ii) Additionally, given recent successes with adaptive stepsizes (e.g., \citep{loizou2021stochastic,orvieto2022dynamics,schaipp2023stochastic}), designing an adaptive matrix stepsize tailored to our case could be viable.
   iii) Finally, recent advances suggest that server step sizes play a key role in accelerating federated learning algorithms \citep{jhunjhunwala2023fedexp,li2024power,li2024convergence}. Designing a matrix version of the server step size could also be interesting.

% \subsection{Limitations}
% \label{sec:limitation}
% For matrix stepsized method such as {\detcgd} and our proposed variance reduced version of it, acceleration hinges on the additional information at our disposal, such as the smoothness matrix.
% In the absence of such information or an accurate estimation, our proposed methods will reduce to the standard scalar method.

\bibliography{iclr2025_conference}
\bibliographystyle{iclr2025_conference}

\newpage
\appendix

\addtocontents{toc}{\protect\setcounter{tocdepth}{3}}

\tableofcontents 

\section{Additional details}\label{sec:add-details}

\subsection{Notations} \label{subsec:notation}

The standard Euclidean norm on $\R^d$ is defined as $\norm{\cdot}$. We use  $\bbS^d_{++}$ (resp. $\bbS^d_{+}$) to denote the positive definite (resp. semi-definite) cone of dimension $d$. $\bbS^d$ is used to denote all symmetric matrices of dimension $d$. We use the notation $\mI_d$ to denote the identity matrix of size $d \times d$, and $\mO_d$ to denote the zero matrix of size $d \times d$. Given $\mQ \in \bbS^d_{++}$ and $x \in \R^d$, $$\norm{x}_{\mQ} \eqdef \sqrt{x^{\top}\mQ x} = \sqrt{\inner{x}{\mQ x}},$$ where $\inner{\cdot}{\cdot}$ is the standard Euclidean inner product on $\R^d$. For a matrix $\mA \in \bbS^d$, we use $\lambda_{\max}\left(\mA\right)$ (resp. $\lambda_{\min}\left(\mA\right)$) to denote the largest (resp. smallest) eigenvalue of the matrix $\mA$. For a function $f:\R^d \mapsto \R$, its gradient and its Hessian at a point $x \in \R^d$ are respectively denoted as $\nabla f(x)$ and $\nabla^2 f(x)$. For the sketch matrices $\mS^k_i$ used in the algorithm, we use the superscript $k$ to denote the iteration and subscript $i$ to denote the client, the matrix $\mS^k_i$ is thus sampled for client $i$ in the $k$-th iteration from the same distribution $\cS$. For any matrix $\mA \in \bbS^d$, we use the notation $\diag\left(\mA\right) \in \bbS^d$ to denote the diagonal of matrix $\mA$.

\subsection{Additional prior work}
\label{subsec:additional-prior}
Numerous effective convex optimization techniques have been adapted for application in non-convex scenarios. Here's a selection of these techniques, although it's not an exhaustive list: adaptivity \citep{dvinskikh2019adaptive, zhang2020adaptive}, variance reduction \citep{jreddi2016proximal, li2021page}, and acceleration \citep{guminov2019accelerated}. Of particular relevance to our work is the paper by \citet{khaled2022better}, which introduces a unified approach for analyzing stochastic gradient descent for non-convex objectives. A comprehensive overview of non-convex optimization can be found in \citep{jain2017non, danilova2022recent}.

An illustrative example of a matrix stepsized method is Newton's method, which has been a long-standing favorite in the optimization community \citep{gragg1974optimal, miel1980majorizing, yamamoto1987convergence}. However, the computational complexity involved in computing the stepsize as the inverse of the Hessian of the current iteration is substantial. Instead, quasi-Newton methods employ a readily computable estimator to replace the inverse Hessian \citep{broyden1965class, dennis1977quasi, al2007overview, al2014broyden}. 
An important direction of research that is relevant to our work, studies distributed second order methods. Here is a non-exhaustive list of papers in this area: \citep{wang2018giant,crane2019dingo,zhang2020achieving,islamov2021distributed,alimisis2021communication,safaryan2022fednl}.

The Distributed Compressed Gradient Descent (\dcgd) algorithm, initially proposed by \citet{khirirat2018distributed}, has seen improvements in various aspects, as documented in works such as \citep{li2020acceleration, horvath2022natural}. 
Its variance reduced version with gradients shifts was studied by \citet{shulgin2022shifted} {in the (strongly) convex setting}.
Additionally, there exists a substantial body of literature on other federated learning algorithms employing unbiased compressors \citep{alistarh2017qsgd, mishchenko2019distributed, gorbunov2021marina, pmlr-v162-mishchenko22b, maranjyan2022gradskip, horvath2023stochastic}. 

Variance reduction techniques have gained significant attention in the context of stochastic batch gradient descent that is prevalent in machine learning. 
Numerous algorithms have been developed in this regard, including well-known ones like {\svrg} \citep{johnson2013accelerating}, {\sag} \citep{schmidt2017minimizing}, {\sdca}\citep{richtarik2014iteration}, {\saga} \citep{defazio2014saga}, {\miso} \citep{mairal2015incremental}, and {\katu} \citep{allen2017katyusha}. 
An overview of more advanced methods can be found in \citep{gower2020variance}. 
Notably, {\svrg} and {\katu} have been extended with loopless variants, namely {\lsvrg} and {\lkatu} \citep{kovalev2020don,qian2021svrg}. These loopless versions streamline the algorithms by eliminating the outer loop and introducing a biased coin-flip mechanism at each step. 
This simplification eases both the algorithms' structure  and their analyses, while preserving their worst-case complexity bounds. 
{\lsvrg}, in particular, offers the advantage of setting the exit probability from the outer loop independently of the condition number, thus, enhancing both robustness and practical efficiency.

This technique of coin flipping allows to obtain variance reduction for the {\cgd} algorithm. 
A relevant example is the {\diana} algorithm proposed by \citet{mishchenko2019distributed}. Its convergence was proved both in the convex and non-convex cases.
Later, {\marina} \citep{gorbunov2021marina} obtained the optimal convergence rate, improving in communication complexity compared to all previous first order methods. 
Finally, there is a line of work developing variance reduction in the federated setting using other methods and techniques \citep{chraibi2019distributed,hanzely2020federated,dinh2020federated,peng2022byzantine}.

Another method to obtain variance reduction is based on momentum. 
It was initially studied by \citet{cutkosky2019momentum}, where they propose the {\storm} algorithm, which is a stochastic gradient descent algorithm with a momentum term for non-convex objectives. 
They obtain stationarity guarantees using adaptive stepsizes with optimal convergence rates. 
However, they require the variance of the stochastic gradient to be bounded by a constant, which is impractical. 
Using momentum for variance reduction has since been widely studied \citep{liu2020optimal,khanduri2020distributed,tran2022hybrid,li2022local}.

\section{Basic facts}\label{subsec:fact}

In this section, we present some basic facts along with their proofs that will be used later in the analysis. 
\begin{fact}
   \label{fact:1}
   For two matrices $\mA, \mB \in \bbS_{+}^d$, denote the $i$-th largest eigenvalues of $\mA, \mB$ as $\lambda_i(\mA), \lambda_i(\mB)$, if $\mA \succeq \mB$, then the following holds
   \begin{equation}
      \label{eq:fact:1}
      \lambda_i(\mA) \geq \lambda_i(\mB).
   \end{equation}
\end{fact}
\begin{proof}
   According to the Courant-Fischer theorem, we write 
   \begin{eqnarray*}
      \lambda_i(\mB) &=& \max_{S: \dim S = i} \min_{x \in S\backslash\{0\}} \frac{x^{\top}\mB x}{x^{\top}x}. \\
   \end{eqnarray*}
   Let $S^i_{\max}$ be a subspace of dimension $i$ where the maximum is attained, we then have 
   \begin{eqnarray*}
      \lambda_i(\mB) &=& \min_{x\in S^i_{\max} \backslash \{0\}} \frac{x^{\top}\mB x}{x^{\top}x} \\
      &\leq& \min_{x\in S^i_{\max} \backslash \{0\}} \frac{x^{\top}\mA x}{x^{\top}x} \leq \max_{S: \dim S = i} \min_{x \in S\backslash \{0\}} \frac{x^{\top}\mA x}{x^{\top}x} = \lambda_i(\mA).
   \end{eqnarray*} 
\end{proof}

The following is a generalization of the bias-variance decomposition for the matrix norm.
\begin{fact}{\rm (Variance Decomposition)}
   \label{lemma:2:var-decomp}
   Given a matrix $\mM \in \bbS^d_{++}$,  any vector $c \in \R^d$, and a random vector $x \in \R^d$ such that $\Exp{\norm{x}}\leq+\infty$, the following bound holds
   \begin{equation}
      \label{eq:lemma:2}
      \Exp{\norm{x - \Exp{x}}^2_{\mM}} = \Exp{\norm{x - c}^2_{\mM}} - \norm{\Exp{x} - c}^2_{\mM}.
   \end{equation}
\end{fact}
\begin{proof}
   We have
   \begin{align*}
      \Exp{\norm{x-c}^2_{\mM}} &- \norm{\Exp{x} - c}^2_{\mM} \\
      &= \Exp{x^{\top}\mM x} - 2\Exp{x}^{\top}\mM c + c^{\top}\mM c  - \Exp{x}^{\top}\mM\Exp{x} + 2\Exp{x}^{\top}\mM c - c^{\top}\mM c \\
      &= \Exp{x^{\top}\mM x} - \Exp{x}^{\top}\mM\Exp{x} \\
      &= \Exp{x^{\top}\mM x} - 2\cdot \Exp{x}^{\top}\mM\Exp{x} + \Exp{x}^{\top}\mM\Exp{x} \\
      &= \Exp{\norm{x - \Exp{x}}^2_{\mM}}.
   \end{align*}
   This completes the proof.
\end{proof}
\begin{fact}
   \label{fact:2}
   The map $\left(\mA, \mB, \mX\right) \mapsto \mA - \mX\mB^{-1}\mX$ is jointly concave on $\bbS^d_{+} \times \bbS^d_{++} \times \bbS^d$. It is also monotone increasing in variables $\mA$ and $\mB$.
\end{fact}
We refer the reader to Corollary 1.5.3 of \citet{bhatia2009positive} for the details and the proof.
The following is a result of \Cref{fact:1} and \Cref{fact:2}.
\begin{fact}
   \label{fact:3}
   Suppose $\mL_i \in \bbS^d_{++}$, for $i = 1,\ldots,n$. 
   Then, for every matrix $\mX \in \bbS^d_{++}$, we define the following mapping
   \begin{equation*}
      f(\mX, \mL_1, \hdots, \mL_n) = \frac{1}{n}\sum_{i=1}^{n}\lambda_{\max}(\mL_i)\cdot\lambda_{\max}\left(\mL_i\mX^{-1}\right)\cdot\lambda_{\max}\left(\mX^{-1}\right).
   \end{equation*}
   Then the above mapping is monotone decreasing in $\mX$.
\end{fact}
\begin{proof}
   First we notice that from \Cref{fact:2} the mapping $\mX \mapsto \mX^{-1}$ is monotone decreasing. 
   The latter means that if we have any $\mX_1, \mX_2 \in \bbS^d_{++}$ such that $\mX_1 \succeq \mX_2$, we have 
   \begin{eqnarray*}
      \mX_1^{-1} \preceq \mX_2^{-1}.
   \end{eqnarray*}
   Then it immediately follows, due to \Cref{fact:1}, that 
   \begin{eqnarray*}
      0 < \lambda_{\max}(\mX_1^{-1}) \leq \lambda_{\max}(\mX_2^{-1}).
   \end{eqnarray*}
   We also notice that the relation $\lambda_{\max}\left(\mL_i\mX^{-1}\right) = \lambda_{\max}\left(\mL_i^\frac{1}{2}\mX^{-1}\mL_i^\frac{1}{2}\right) = \lambda_{\max}\left(\mX^{-1}\mL_i\right)$, and that the mapping $\mX \mapsto \mL_i^\frac{1}{2}\mX^{-1}\mL_i^\frac{1}{2}$ is also monotone decreasing for every $i \in [n]$, so we have 
   \begin{equation*}
      0 < \lambda_{\max}\left(\mL_i\mX_1^{-1}\right) \leq \lambda_{\max}\left(\mL_i\mX_2^{-1}\right).
   \end{equation*}
   Since we have the coefficient $\lambda_{\max}\left(\mL_i\right) > 0$, it follows that, 
   \begin{eqnarray*}
      f(\mX_1, \mL_1, \hdots, \mL_n) \leq f(\mX_2, \mL_1, \hdots, \mL_n).
   \end{eqnarray*}
   This means that $f(\mX)$ is monotone decreasing in $\mX$.
\end{proof}

\begin{fact}
   \label{fact:4}
   For any two matrices $\mA, \mB \in \bbS^d_{++}$, the following relation regarding their largest eigenvalue holds
   \begin{equation}
      \label{eq:fact:4}
      \lambda_{\max}\left(\mA\mB\right) \leq \lambda_{\max}\left(\mA\right)\cdot\lambda_{\max}\left(\mB\right).
   \end{equation}
\end{fact}
\begin{proof}
   Using the Courant-Fischer theorem, we can write 
   \begin{eqnarray*}
      \lambda_{\max}\left(\mA\mB\right) &=& \min_{S: \dim S = d}\max_{x\in S\backslash\{0\}}\frac{x^{\top}\mA\mB x}{x^{\top}x} \\
      &=& \max_{x \in \R^d\backslash\{0\}} \frac{x^{\top}\mA\mB x}{x^{\top}x} \\
      &\leq& \max_{x \in \R^d\backslash\{0\}} \frac{x^{\top}\mA x}{x^{\top}x} \cdot \max_{x \in \R^d\backslash\{0\}} \frac{x^{\top}\mB x}{x^{\top}x} \\
      &=& \lambda_{\max}\left(\mA\right)\cdot\lambda_{\max}\left(\mB\right).
   \end{eqnarray*}
\end{proof}

\begin{fact}
   \label{fact:5}
   Given matrix $\mQ \in \bbS^d_{++}$ and its matrix norm $\norm{\cdot}_{\mQ}$, its associated dual norm is $\norm{\cdot}_{\mQ^{-1}}$.
\end{fact}

\begin{proof}
   Let us first recall the definition of the dual norm $\norm{\cdot}_*$. 
   For any vector $z\in \R^d$, it is defined as 
   \begin{equation*}
      \norm{z}_* := \sup\{z^{\top}x: \norm{x}_{\mQ} \leq 1\}.
   \end{equation*}
   Solving this optimization problem is equivalent to solving $\sup\{z^{\top}x: \norm{x}^2_{\mQ} = 1\}$. 
   The Lagrange function is given as
   \begin{equation*}
      f(x, \lambda) = z^{\top}x - \lambda\left(\norm{x}^2_{\mQ} - 1\right) = z^{\top}x - \lambda\left(x^{\top}\mQ x - 1\right).
   \end{equation*}
   Computing the derivatives we deduce that
   \begin{equation*}
      \frac{\partial f(x, \lambda)}{\partial x} = z - 2\lambda\cdot\mQ x = 0, \qquad \frac{\partial f(x, \lambda)}{\partial \lambda} = \norm{x}^2_{\mQ} - 1 = 0.
   \end{equation*}
   This leads to
   \begin{equation*}
      \lambda = \frac{\norm{z}_{\mQ^{-1}}}{2}, \qquad x = \frac{\mQ^{-1}z}{\norm{z}_{\mQ^{-1}}}.
   \end{equation*}
   As a result, we have 
   \begin{eqnarray*}
      \sup\{z^{\top}x: \norm{x}_{\mQ} \leq 1\} &=& \sup\{z^{\top}x: \norm{x}^2_{\mQ} = 1\} \\
      &=& z^{\top}z = \frac{z^{\top}\mQ^{-1}z}{\norm{z}_{\mQ^{-1}}} = \norm{z}_{\mQ^{-1}}.
   \end{eqnarray*}
\end{proof}

\section{Properties of matrix smoothness}\label{app:smoothness}

\subsection{The matrix Lipschitz-continuous gradient}
In this section we describe some properties of matrix smoothness, matrix gradient Lipschitzness and their relations.  The following proposition describes a sufficient condition for the matrix Lipschitz-continuity of the gradient.
\begin{proposition}
   \label{ppst:bounded:Hessian}
   Given twice continuously differentiable function $f: \R^d \mapsto \R$ with bounded Hessian,
   \begin{equation}
      \label{eq:ppst:bounded:Hessian}
      \nabla^2 f(x) \preceq \mL, 
   \end{equation}
   where $\mL \in \bbS^d_{++}$ and the generalized inequality holds for any $x \in \R^d$. Then $f$ satisfies \eqref{eq:assmp:3} with the matrix $\mL$.
\end{proposition}
The below proposition is a variant of \Cref{ppst:4} and it characterizes the smoothness matrix of the objective function $f$, given the smoothness matrices of the component functions $f_i$.
\begin{proposition}
   \label{dasha:lemma:global-smooth}
   Assume that $f_i$ has $\mL_i$-Lipschitz continuous gradient for every $i \in [n]$, then function $f$ has $\mL$-Lipschitz gradient with $\mL \in \bbS^d_{++}$ satisfying 
   \begin{equation}
   \label{dasha:eq:global-smooth}
      \mL \cdot \lambda_{\min}\left(\mL\right) = \frac{1}{n}\sum_{i=1}^{n} \lambda_{\max}\left(\mL_i\right)\cdot\mL_i.
  \end{equation}
\end{proposition}

\subsubsection{Quadratics}
Given a matrix $\mA \in \bbS^d_{++}$ and a vector $b \in \R^d$, consider the function $f(x) = \frac{1}{2}x^{\top}\mA x + b^{\top}x + c$. 
Then its gradient is computed as $\nabla f(x) = \mA x + b$ and $\nabla^2 f(x) = \mA$. 
Inserting gradients formula into  \eqref{eq:assmp:3} we deduce
\begin{equation*}
   \sqrt{(x - y)^{\top}\mA\mL^{-1}\mA (x-y)} \leq \sqrt{(x-y)^{\top}\mL (x-y)},
\end{equation*}
for any $x, y\in \R^d$. This reduces to
\begin{equation}
   \label{eq:cond-for-quad}
   \mA\mL^{-1}\mA \preceq \mL.
\end{equation}
Since $\mA \in \bbS^d_{++}$, we can also rewrite \eqref{eq:cond-for-quad} as 
\begin{equation*}
   \mA^{\frac{1}{2}}\mL^{-1}\mA^{\frac{1}{2}} \preceq \mA^{-\frac{1}{2}}\mL\mA^{-\frac{1}{2}},
\end{equation*}
which is equivalent to 
\begin{equation}
   \label{eq:equiv-write}
   \mA \preceq \mL.
\end{equation}
Therefore, the ``best'' $\mL \in \bbS^d_{++}$ that satisfies \eqref{eq:assmp:3} is $\mL = \mA = \nabla^2 f(x)$, for every $x \in \R^d$. 
Now, let us  look at a more general setting. 
Consider $f$ given as follows,
\begin{equation*}
   f(x) = \sum_{i=1}^{s}\phi_i(\mM_i x),
\end{equation*}
where $\mM_i \in \R^{q_i \times d}$. Here $f:\R^d \mapsto \R$ is the sum of functions $\phi_i:\R^{q_i} \mapsto \R$. We assume that each function $\phi_i$ has matrix $\mL_i$ Lipschitz gradient. We have the following lemma regarding the matrix gradient Lipschitzness of $f$.
\begin{proposition}
   \label{ppst:ext-quad}
   Assume that functions $f$ and $\{\phi_i\}_{i=1}^s$ are described above. 
   Then function $f$ has $\mL$-Lipschitz gradient, if the following condition is satisfied:
   \begin{equation}
      \label{eq:ppst:ext-quad-cond}
      \sum_{i=1}^{s}\lambda_{\max}\left(\mL_i^\frac{1}{2}\mM_i\mL^{-1}\mM_i^{\top}\mL_i^\frac{1}{2}\right) = 1.
   \end{equation}
\end{proposition}
Note that \Cref{ppst:ext-quad} is a generalization of the previous case of quadratics, if we pick $s=1$, $\mM_i = \mA^\frac{1}{2}$ and $\phi_1(x) = x^{\top}\mI_d x$, the condition becomes $\mL = \mA$, which is exactly the solution given by \eqref{eq:equiv-write}. 
Thus we recover the result for quadratics. The linear term $bx + c$ is ignored in this case. In \Cref{ppst:ext-quad}, we only intend to give a way of finding a matrix $\mL \in \bbS^d_{++}$, so that $f$ has $\mL$-Lipschitz gradient. This does not mean, however, the $\mL$ here is optimal. 
The proof is deferred to \Cref{app:proof-prop}.

\subsection{Comparison of the different smoothness conditions}

Let us recall the definition of matrix smoothness. 

\begin{definition}{\rm ($\mL$-smoothness)}
   Assume that $f: \R^d \rightarrow \R$ is a continuously differentiable function and matrix $\mL \in \bbS^d_{++}$. We say that $f$ is $\mL$-smooth if for all $ x, y \in \R^d$
   \begin{equation}
      \label{eq:assmp:3-sup}
      f(y) \leq f(x) + \inner{\nabla f(x)}{x - y} + \frac{1}{2}\norm{x - y}_{\mL}^2.
   \end{equation}
\end{definition}

We provide a proposition here which describes an equivalent form of stating $\mL$-matrix smoothness of a function $f$. This proposition is used to illustrate the relation between matrix smoothness and matrix Lipschitz gradient.

\begin{proposition}
   \label{ppst:1}
   Let function $f: \R^d \rightarrow \R$ be continuously differentiable. Then the following statements are equivalent.
   \begin{itemize}
      \item[{\rm (i)}] {\rm $f$ is $\mL$-matrix smooth.}
      \item[{\rm (ii)}] {\rm $\inner{\nabla f(x) - \nabla f(y)}{x - y} \leq \norm{x-y}^2_{\mL}$ for all $x, y \in \R^d$.}
   \end{itemize}
\end{proposition}

The two propositions, \Cref{ppst:3} and \Cref{ppst:converse}, formulated below illustrate the relation between matrix smoothness of $f$ and matrix gradient Lipschitzness of $f$.

\begin{proposition}
   \label{ppst:3}
   Assume $f: \R^d \mapsto \R$ is a continuously differentiable function, and its gradient is $\mL$-Lipschitz continuous with $\mL \in \bbS_{++}^d$. Then function $f$ is $\mL$-matrix smooth.
\end{proposition}

\begin{proposition}
   \label{ppst:converse}
   Assume $f: \R^d \rightarrow \R$ is a continuously differentiable function. Assume also that $f$ is convex and $\mL$-matrix smooth. Then $\nabla f$ is $\mL$-Lipschitz continuous.
\end{proposition}

The next proposition shows that standard Lipschitzness of the gradient of a function is an immediate consequence of matrix Lipschitzness.
\begin{proposition}
   \label{ppst:mat-scl}
   Assume that the gradient of $f$ is $\mL$-Lipschitz continuous. Then $\nabla f$ is also $L$-Lipschitz with $L = \lambda_{\max}\left(\mL\right)$.
\end{proposition}

\subsection{Proofs of the propositions regarding smoothness} \label{app:proof-prop}

\subsubsection{\texorpdfstring{Proof of \Cref{ppst:4}}{Proof of Proposition~\ref{ppst:4}}}
   We start with the definition of $\mL$-Lipschitz gradient of function $f$, and pick 
   two arbitrary points $x, y \in \R^d$,
   \begin{eqnarray*}
      \norm{\nabla f(x) - \nabla f(y)}^2_{\mL^{-1}} &=& \norm{\frac{1}{n}\sum_{i=1}^{n}
      \left(\nabla f_i(x) - \nabla f_i(y)\right)}^2_{\mL^{-1}}. 
   \end{eqnarray*}
   Applying the convexity of $\norm{\cdot}^2_{\mL^{-1}}$, we have 
   \begin{eqnarray*}
      \norm{\nabla f(x) - \nabla f(y)}^2_{\mL^{-1}} &\leq& \frac{1}{n}\sum_{i=1}^{n}\norm{\nabla f_i(x) - \nabla f_i(y)}^2_{\mL^{-1}}. \\
   \end{eqnarray*}
   For each term within the summation, we use the definition of matrix norms and replace the matrix $\mL^{-1}$ with $\mL_i^{-1/2}\mL_i^{1/2}\mL^{-1}\mL_i^{1/2}\mL_i^{-1/2}$, for every $i = 1,\ldots,n$:
   \begin{eqnarray*}
      \norm{\nabla f(x) - \nabla f(y)}^2_{\mL^{-1}}&=&\frac{1}{n}\sum_{i=1}^{n}\left(\mL_i^{-\frac{1}{2}}(\nabla f_i(x) - \nabla f_i(y))\right)^{\top}\mL_i^\frac{1}{2}\mL^{-1}\mL_i^\frac{1}{2}\left(\mL_i^{-\frac{1}{2}}(\nabla f_i(x) - \nabla f_i(y))\right) \\
      &\leq& \frac{1}{n}\sum_{i=1}^{n}\lambda_{\max}\left(\mL_i^\frac{1}{2}\mL^{-1}\mL_i^\frac{1}{2}\right)\norm{\mL_i^{-\frac{1}{2}}(\nabla f_i(x) - \nabla f_i(y))}^2 \\
      &=& \frac{1}{n}\sum_{i=1}^{n}\lambda_{\max}\left(\mL_i^\frac{1}{2}\mL^{-1}\mL_i^\frac{1}{2}\right)\norm{\nabla f_i(x) - \nabla f_i(y)}^2_{\mL_i^{-1}}.
   \end{eqnarray*}
   Using the assumption that  the gradient of each function $f_i$ is $\mL_i$-Lipschitz, we obtain,
   \begin{eqnarray*}
      \norm{\nabla f(x) - \nabla f(y)}^2_{\mL^{-1}}&{\leq}& \frac{1}{n}\sum_{i=1}^{n}\lambda_{\max}\left(\mL_i^\frac{1}{2}\mL^{-1}\mL_i^\frac{1}{2}\right)\norm{x-y}^2_{\mL_i}.
   \end{eqnarray*}
   Replacing $\mL_i^{-1}$ with $\mL^{-1/2}\mL^{1/2}\mL_i^{-1}\mL^{1/2}\mL^{-1/2}$
   \begin{eqnarray*}
      \norm{\nabla f(x) - \nabla f(y)}^2_{\mL^{-1}}&=& \frac{1}{n}\sum_{i=1}^{n}\lambda_{\max}\left(\mL_i^\frac{1}{2}\mL^{-1}\mL_i^\frac{1}{2}\right)\cdot\left[(\mL^{\frac{1}{2}}(x-y))^{\top}\mL^{-\frac{1}{2}}\mL_i\mL^{-\frac{1}{2}}(\mL^{\frac{1}{2}}(x-y))\right] \\
      &\leq&\frac{1}{n}\sum_{i=1}^{n}\lambda_{\max}\left(\mL_i^\frac{1}{2}\mL^{-1}\mL_i^\frac{1}{2}\right)\cdot\lambda_{\max}\left(\mL^{-\frac{1}{2}}\mL_i\mL^{-\frac{1}{2}}\right)\norm{\mL^\frac{1}{2}(x-y)}^2.
   \end{eqnarray*}
   Using \Cref{fact:4}, we are deduce the following bound, 
   \begin{eqnarray*}
      \norm{\nabla f(x) - \nabla f(y)}^2_{\mL^{-1}} &\leq& \left(\frac{1}{n}\sum_{i=1}^{n}\lambda_{\max}\left(\mL^{-1}\right)\cdot\lambda_{\max}\left(\mL_i\right)\cdot\lambda_{\max}\left(\mL_i\mL^{-1}\right)\right)\cdot\norm{x-y}^2_{\mL}\\
      &=& \norm{x - y}^2_{\mL}.
   \end{eqnarray*}

\subsubsection{\texorpdfstring{Proof of \Cref{ppst:bounded:Hessian}}{Proof of Proposition~\ref{ppst:bounded:Hessian}}}

   We start with picking any two vector $x, y \in \R^d$. We have 
   \begin{align*}
      & \norm{\nabla f(x) - \nabla f(y)}^2_{\mL^{-1}} \\
      &\quad = \norm{\int_{0}^{1}\nabla^2 f(\theta x + (1 - \theta)y)(x - y) \d \theta}^2_{\mL^{-1}} \\
      &\quad = (x-y)^{\top}\left(\int_{0}^{1}\nabla^2 f(\theta x + (1-\theta)y)\d \theta\right)^{\top}\mL^{-1}\left(\int_{0}^{1}\nabla^2f(\theta x + (1-\theta)y)\d \theta\right)(x-y).
   \end{align*}
   Denote $\mF \eqdef \int_{0}^{1}\nabla^2 f(\theta x + (1-\theta)y)\d \theta$, notice that $\mF$ is a symmetric matrix. 
   Then, the previous identity becomes
   \begin{align*}
      \norm{\nabla f(x) - \nabla f(y)}^2_{\mL^{-1}} 
      &= (x-y)^{\top}\mF^{\top}\mL^{-1}\mF(x-y).
   \end{align*}
   From the definition of $\mF$ and the bounded Hessian assumption, we have $\mF \preceq \mL$. Let us prove that $\mF\mL^{-1}\mF \preceq \mL$: 
   \begin{eqnarray*}
      \mF\mL^{-1}\mF \preceq \mL 
      &\iff& \mL^{-\frac{1}{2}}\mF\mL\mF\mL^{-\frac{1}{2}} \preceq \mI_d \\
      &\iff& \mL^{-\frac{1}{2}}\mF\mL^{-\frac{1}{2}}\cdot\mL^{-\frac{1}{2}}\mF\mL^{-\frac{1}{2}} \preceq \mI_d \\
      &\iff& \mL^{-\frac{1}{2}}\mF\mL^{-\frac{1}{2}} \preceq \mI_d \\
      &\iff& \mF \preceq \mL.
   \end{eqnarray*}
   This means that
   \begin{eqnarray*}
      \norm{\nabla f(x) - \nabla f(y)}^2_{\mL^{-1}} 
      &\leq& (x-y)^{\top}\mL(x - y) = \norm{x-y}^2_{\mL},
   \end{eqnarray*}
   which completes the proof.

\subsubsection{\texorpdfstring{Proof of \Cref{dasha:lemma:global-smooth}}{Proof of Proposition~\ref{dasha:lemma:global-smooth}}}

   Suppose $\mL$ is a symmetric positive definite matrix satisfying \eqref{dasha:eq:global-smooth}. Let us now show that the function $\nabla f$ is $\mL$-Lipschitz continuous.
   We start with picking any two points $x, y \in \R^d$, and notice that 
   \begin{align*}
      \norm{\nabla f(x) - \nabla f(y)}^2_{\mL^{-1}} = \norm{\frac{1}{n}\sum_{i=1}^{n}\left(\nabla f_i(x) - \nabla f_i(y)\right)}^2_{\mL^{-1}}.
   \end{align*}
   Applying Jensen's inequality, we obtain 
   \begin{align*}
      \norm{\nabla f(x) - \nabla f(y)}^2_{\mL^{-1}} \leq \frac{1}{n}\sum_{i=1}^{n}\norm{\nabla f_i(x) - \nabla f_i(y)}^2_{\mL^{-1}}.
   \end{align*}
   We then re-weight the norm appears in the summation individually, 
   \begin{align*}
      \norm{\nabla f(x) - \nabla f(y)}^2_{\mL^{-1}} &\leq \frac{1}{n}\sum_{i=1}^{n}\left(\nabla f_i(x) - \nabla f_i(y)\right)^{\top}\mL_i^{-\frac{1}{2}}\mL_i^\frac{1}{2}\mL^{-1}\mL_i^\frac{1}{2}\mL_i^{-\frac{1}{2}}\left(\nabla f_i(x) - \nabla f_i(y)\right) \\
      &\leq \frac{1}{n}\sum_{i=1}^{n}\lambda_{\max}\left(\mL_i\right)\cdot\lambda_{\max}\left(\mL^{-1}\right)\cdot\norm{\nabla f_i(x) - \nabla f_i(y)}^2_{\mL_i^{-1}}.
   \end{align*}
   Utilizing the assumption that each $f_i$ has $\mL_i$ Lipschitz gradient, we obtain 
   \begin{align*}
      \norm{\nabla f(x) - \nabla f(y)}^2_{\mL^{-1}} &\leq \frac{1}{n}\sum_{i=1}^{n}\lambda_{\max}\left(\mL_i\right)\cdot\lambda_{\max}\left(\mL^{-1}\right)\cdot\norm{x - y}^2_{\mL_i} \\
      &= \norm{x - y}^2_{\lambda_{\max}\left(\mL^{-1}\right)\cdot\frac{1}{n}\sum_{i=1}^{n}\lambda_{\max}\left(\mL_i\right)\cdot\mL_i} \overset{\eqref{dasha:eq:global-smooth}}{=} \norm{x - y}^2_{\mL}.
   \end{align*}

\subsubsection{\texorpdfstring{Proof of \Cref{ppst:ext-quad}}{Proof of Proposition~\ref{ppst:ext-quad}}}

   For any $x$ and $y$ from $\R^d$, we have 
   \begin{align*}
      & \norm{\nabla f(x) - \nabla f(y)}_{\mL^{-1}} \\ 
      &\quad = \norm{\sum_{i=1}^{s}\mM_i^{\top}\nabla\phi_i(\mM_i x) - \sum_{i=1}^{s}\mM_i^{\top}\nabla\phi_i(\mM_i y)}_{\mL^{-1}} \\
      &\quad = s\cdot\norm{\frac{1}{s}\sum_{i=1}^{s}\mM_i^{\top}\left(\nabla\phi_i(\mM_i x) - \nabla\phi_i\left(\mM_i y\right)\right)}_{\mL^{-1}}.
   \end{align*}
   Applying the convexity of the norm $\norm{\cdot}_{\mL^{-1}}$, 
   \begin{eqnarray*}
      \norm{\nabla f(x) - \nabla f(y)}_{\mL^{-1}}
      &\leq& s\cdot\frac{1}{s}\sum_{i=1}^{s}\norm{\mM_i^{\top}\left(\nabla\phi_i(\mM_i x)-\nabla\phi_i(\mM_i y)\right)}_{\mL^{-1}}.
   \end{eqnarray*}
   Expanding the norm and applying the replacement trick for above $\mL$ and $\mM_i$, we obtain 
   \begin{align*}
      & \norm{\nabla f(x) - \nabla f(y)}_{\mL^{-1}} \\ 
      & =\sum_{i=1}^{s}\sqrt{\left(\nabla\phi_i(\mM_i x)-\nabla\phi_i(\mM_i y)\right)^{\top}\mM_i\mL^{-1}\mM_i^{\top}\left(\nabla\phi_i(\mM_i x)-\nabla\phi_i(\mM_i y)\right)} \\
      & = \sum_{i=1}^{s}\sqrt{\mB_i^{\top}\mL_i^{\frac{1}{2}}\mM_i\mL^{-1}\mM_i^{\top}\mL_i^\frac{1}{2}\mB_i} \\
      & \leq \sum_{i=1}^{s}\sqrt{\lambda_{\max}\left(\mL_i^\frac{1}{2}\mM_i\mL^{-1}\mM_i^{\top}\mL_i^\frac{1}{2}\right)}\cdot \norm{\nabla\phi_i(\mM_i x)-\nabla\phi_i(\mM_i y)}_{\mL_i^{-1}}, 
   \end{align*}
   where $\mB_i \eqdef \mL_i^{-\frac{1}{2}}\left(\nabla\phi_i(\mM_i x)-\nabla\phi_i(\mM_i y)\right)$.
    Due to the assumption that the gradient of $\phi_i$ is $\mL_i$-Lipschitz, we have
   \begin{align*}
      & \norm{\nabla f(x) - \nabla f(y)}_{\mL^{-1}} \\ 
      &\quad \leq \sum_{i=1}^{s}\sqrt{\lambda_{\max}\left(\mL_i^\frac{1}{2}\mM_i\mL^{-1}\mM_i^{\top}\mL_i^\frac{1}{2}\right)}\cdot \norm{\mM_i (x - y)}_{\mL_i} \\
      &\quad = \sum_{i=1}^{s}\sqrt{\lambda_{\max}\left(\mL_i^\frac{1}{2}\mM_i\mL^{-1}\mM_i^{\top}\mL_i^\frac{1}{2}\right)}\cdot \sqrt{\left[\mL^\frac{1}{2}\left(x-y\right)\right]^{\top}\mL^{-\frac{1}{2}}\mM_i^{\top}\mL_i\mM_i\mL^{-\frac{1}{2}}\left[\mL^\frac{1}{2}(x - y)\right]} \\
      &\quad \leq \sum_{i=1}^{s}\sqrt{\lambda_{\max}\left(\mL_i^\frac{1}{2}\mM_i\mL^{-1}\mM_i^{\top}\mL_i^\frac{1}{2}\right) \cdot \lambda_{\max}\left(\mL^{-\frac{1}{2}}\mM_i^{\top}\mL_i\mM_i\mL^{-\frac{1}{2}}\right)} \cdot \norm{x-y}_{\mL} \\
      &\quad \overset{}{\leq} \sum_{i=1}^{s} \lambda_{\max}\left(\mL_i^\frac{1}{2}\mM_i\mL^{-1}\mM_i^{\top}\mL_i^\frac{1}{2}\right) \cdot \norm{x - y}_{\mL},
   \end{align*}
   where the last inequality is due to the fact that,
   \begin{equation*}
      \lambda_{\max} \left(\mL_i^\frac{1}{2}\mM_i\mL^{-1}\mM_i^{\top}\mL_i^\frac{1}{2}\right) = \lambda_{\max}\left(\mL^{-\frac{1}{2}}\mM_i^{\top}\mL_i\mM_i\mL^{-\frac{1}{2}}\right).
   \end{equation*}
   Recalling the condition of the proposition:
   \begin{equation*}
      \sum_{i=1}^{s} \lambda_{\max}\left(\mL_i^\frac{1}{2}\mM_i\mL^{-1}\mM_i^{\top}\mL_i^\frac{1}{2}\right) = 1,
   \end{equation*}
   we deduce
   \begin{equation*}
      \norm{\nabla f(x) - \nabla f(y)}_{\mL^{-1}} \leq \norm{x - y}_{\mL}.
   \end{equation*}

\subsubsection{\texorpdfstring{Proof of \Cref{ppst:1}}{Proof of Proposition~\ref{ppst:1}}}

   (i) $\rightarrow$ (ii). If $f$ is $\mL$-matrix smooth, then for all $x, y \in \R^d$, we have 
   \begin{equation*}
      f(x) \leq f(y) + \inner{\nabla f(y)}{x - y} + \frac{1}{2}\norm{x-y}^2_{\mL},
   \end{equation*}
   and 
   \begin{equation*}
      f(y) \leq f(x) + \inner{\nabla f(x)}{y - x} + \frac{1}{2}\norm{x-y}^2_{\mL}.
   \end{equation*}
   Summing up these two inequalities we get 
   \begin{equation*}
      \inner{\nabla f(x) - \nabla f(y)}{x - y} \leq \norm{x-y}^2_{\mL}.
   \end{equation*}

   (ii) $\rightarrow$ (i). Choose any $x, y \in \R^d$, and define $z = x + t(y - x)$, then we have, 
   \begin{eqnarray*}
      f(y) &=& f(x) + \int_{0}^{1}\inner{\nabla f(x + t(y-x))}{y - x} \d{t} \\
      &=& f(x) + \int_{0}^{1}\inner{\nabla f(z)}{y - x} \d{t} \\
      &=& f(x) + \inner{\nabla f(x)}{y - x} + \int_{0}^{1}\inner{\nabla f(z) - \nabla f(x)}{y-x} \d{t} \\
      &=& f(x) + \inner{\nabla f(x)}{y - x} + \int_{0}^{1}\inner{\nabla f(z) - \nabla f(x)}{z-x}\cdot\frac{1}{t} \d{t}.
   \end{eqnarray*}
   Using the assumption that for any $x, z \in \R^d$, we have 
   \begin{equation*}
      \inner{\nabla f(z)-\nabla f(x)}{z - x} \leq \norm{z - x}^2_{\mL}.
   \end{equation*}
   Plug this back into the previous identity, we obtain 
   \begin{eqnarray*}
      f(y) &\leq& f(x) + \inner{\nabla f(x)}{y - x} + \int_{0}^{1}\norm{z-x}^2_{\mL}\cdot\frac{1}{t} \d{t} \\
      &=& f(x) + \inner{\nabla f(x)}{y - x} + \int_{0}^{1}\norm{y-x}^2_{\mL}\cdot t \d{t} \\
      &=& f(x) + \inner{\nabla f(x)}{y - x} + \frac{1}{2}\norm{y - x}_{\mL}^2.
   \end{eqnarray*}

\subsubsection{\texorpdfstring{Proof of \Cref{ppst:3}}{Proof of Proposition~\ref{ppst:3}}}

   We start with picking any two points $x, y \in \R^d$, using the generalized Cauchy-Schwarz inequality for dual norm, we have
   \begin{eqnarray*}
      \inner{\nabla f(x) - \nabla f(y)}{x - y} &\leq& \norm{\nabla f(x) - \nabla f(y)}_{\mL^{-1}} \cdot \norm{x - y}_{\mL} \\
      &\overset{\eqref{eq:assmp:3}}{\leq}& \norm{x-y}_{\mL} \cdot \norm{x - y}_{\mL} \\
      &=& \norm{x-y}^2_{\mL}
   \end{eqnarray*}
   According to \Cref{ppst:1}, this indicates that function $f$ is $\mL$-matrix smooth.

\subsubsection{\texorpdfstring{Proof of \Cref{ppst:converse}}{Proof of Proposition~\ref{ppst:converse}}}

   Using \Cref{ppst:1}, we know that for any $x, y \in \R^d$, we have 
   \begin{equation}
      \label{eq:res-1}
      \inner{\nabla f(x) - \nabla f(y)}{x - y} \leq \norm{x - y}^2_{\mL}.
   \end{equation}
   Now we pick any three points $x, y, z \in \R^d$. Using the $\mL$-smoothness of $f$, we have 
   \begin{equation}
      \label{eq:smooth-ind}
      f(x + z) \geq f(x) + \inner{\nabla f(x)}{z} + \frac{1}{2}\norm{z}^2_{\mL}.
   \end{equation}
   Using the convexity of $f$ we have 
   \begin{equation}
      \label{eq:convex-ind}
      \inner{\nabla f(y)}{x + z - y} \leq f(x + z) - f(y).
   \end{equation}
   Combining \eqref{eq:smooth-ind} and \eqref{eq:convex-ind}, we obtain 
   \begin{equation*}
      \inner{\nabla f(y)}{x + z - y} \leq f(x) - f(y) + \inner{\nabla f(x)}{z} + \frac{1}{2}\norm{z}^2_{\mL}.
   \end{equation*}
   Rearranging terms we get 
   \begin{equation*}
      \inner{\nabla f(y) - \nabla f(x)}{z} - \frac{1}{2}\norm{z}^2_{\mL} \leq f(x) - f(y) - \inner{\nabla f(y)}{x - y}.
   \end{equation*}
   The inequality holds for any $z$ for fixed $x$ and $y$, and the left hand side is maximized (w.r.t. $z$) when $z = \mL^{-1}\left(\nabla f(y) - \nabla f(x)\right)$. Plugging it in, we get 
   \begin{equation}
      \label{eq:tilted-res}
      \frac{1}{2}\norm{\nabla f(x)- \nabla f(y)}^2_{\mL^{-1}} \leq f(x) - f(y) - \inner{\nabla f(y)}{x - y}.
   \end{equation}
   By symmetry we can also obtain
   \begin{equation*}
      \frac{1}{2}\norm{\nabla f(y)- \nabla f(x)}^2_{\mL^{-1}} \leq f(y) - f(x) - \inner{\nabla f(x)}{y - x}.
   \end{equation*}
   Adding \eqref{eq:tilted-res} and its counterpart together, we get 
   \begin{equation}
      \label{eq:res-res}
      \norm{\nabla f(x) - \nabla f(y)}^2_{\mL^{-1}} \leq \inner{\nabla f(x) - \nabla f(y)}{x - y}.
   \end{equation}
   Combing \eqref{eq:res-res} and \eqref{eq:res-1}, it follows 
   \begin{equation*}
      \norm{\nabla f(x) - \nabla f(y)}^2_{\mL^{-1}} \leq \norm{x - y}^2_{\mL}.
   \end{equation*}
   Note that $\mL$ and $\mL^{-1}$ are both positive definite matrices, so it is equivalent to 
   \begin{equation*}
      \norm{\nabla f(x) - \nabla f(y)}_{\mL^{-1}} \leq \norm{x - y}_{\mL}.
   \end{equation*}
   This completes the proof.

\subsubsection{\texorpdfstring{Proof of \Cref{ppst:mat-scl}}{Proof of Proposition~\ref{ppst:mat-scl}}}
Let us start with picking any two points $x, y \in \R^d$. With the matrix $\mL$-Lipschitzness of the gradient of function $f$, we have 
\begin{align*}
   \norm{\nabla f(x) - \nabla f(y)}^2_{\mL^{-1}} \leq \norm{x - y}^2_{\mL}.
\end{align*}
This implies 
\begin{equation*}
   (x - y)^{\top}\mL(x - y) - \left(\nabla f(x) - \nabla f(y)\right)^{\top}\mL^{-1}\left(\nabla f(x) - \nabla f(y)\right) \geq 0.
\end{equation*}
Define function $f(\mX) \eqdef a^{\top}\mX a - b^{\top}\mX^{-1} b$ for $\mX \in \bbS^d_{++}$, where $a, b \in \R^d$ are fixed vectors. Then $f$ is monotone increasing in $\mX$. This can be shown in the following way, picking two matrices $\mX_1, \mX_2 \in \bbS^d_{++}$, where $\mX_1 \succeq \mX_2$. It is easy to see that $-\mX_1^{-1} \succeq -\mX_2^{-1}$, since from \Cref{fact:2} the map $\mX \mapsto -\mX^{-1}$ is monotone increasing for $\mX \in \bbS^d_{++}$. Thus, 
\begin{align*}
   f(\mX_1) - f(\mX_2) &= (x - y)^{\top}\left(\mX_1 - \mX_2\right)(x - y) \\
   &\qquad + \left(\nabla f(x) - \nabla f(y)\right)^{\top}\left(-\mX_1^{-1} - (-\mX_2^{-1})\right)\left(\nabla f(x) - \nabla f(y)\right) \geq 0.
\end{align*}
As a result, $f(\lambda_{\max}\left(\mL\right)\cdot\mI_d) \geq f(\mL) \geq 0$, due to the fact that $\lambda_{\max}\left(\mL\right)\cdot\mI_d \succeq \mL$. It remains to notice that 
\begin{align*}
   f(\lambda_{\max}\left(\mL\right)\cdot\mI_d) &= \lambda_{\max}\left(\mL\right)\norm{x-y}^2 - \frac{1}{\lambda_{\max}\left(\mL\right)}\norm{\nabla f(x) - \nabla f(y)}^2 \geq 0,
\end{align*}
which yields 
\begin{align*}
   \norm{\nabla f(x) - \nabla f(y)}^2 \leq \lambda^2_{\max}\left(\mL\right)\norm{x - y}^2.
\end{align*}
Since we are working with $\mL \in \bbS^d_{++}$, the above inequality implies 
\begin{align*}
    \norm{\nabla f(x) - \nabla f(y)} \leq \lambda_{\max}\left(\mL\right)\norm{x - y}.
\end{align*}

\section{\texorpdfstring{Analysis of {\detmarina}}{Analysis of Det-MARINA}}
\subsection{Technical lemmas}

We first state some technical lemmas.
\begin{lemma}[Descent lemma]
   \label{lemma:1}
   Assume that function $f$ is $\mL$ smooth, and $x^{k+1} = x^k - \mD \cdot g^k$, where $\mD \in \bbS_{++}^d$. Then we will have
   \begin{equation*}
      f(x^{k+1}) \leq f(x^k) - \frac{1}{2}\norm{\nabla f(x^k)}_{\mD}^2 + \frac{1}{2}\norm{g^k - \nabla f(x^k)}_{\mD}^2 - \frac{1}{2}\norm{x^{k+1} - x^k}_{\mD^{-1} - \mL}.
   \end{equation*}
\end{lemma}

The following lemma is obtained for any sketch matrix $\mS \in \bbS_{+}^d$ and any two positive definite matrices $\mD$ and $\mL$.
\begin{lemma}[Property of sketch matrix]
   \label{lemma:3}
   For any sketch matrix $\mS \in \bbS^d_{+}$, a vector $t \in \R^d$, and matrices $\mD,\mL \in \bbS^d_{++}$, we have
   \begin{equation}
   \label{dasha:eq:tech-lemma-sketch-comp}
      \Exp{\norm{\mS t - t}_{\mD}^2} \leq \lambda_{\max}\left(\mL^\frac{1}{2}\left(\Exp{\mS\mD\mS} - \mD\right)\mL^\frac{1}{2}\right) \cdot \norm{t}^2_{\mL^{-1}}.
   \end{equation}    
\end{lemma}

\begin{lemma}
   \label{dasha:tech-lemma:2}
   Assume that \Cref{assmp:3} holds and $h_i^0 = \nabla f_i(x^0)$, then for $h_i^{k+1}$ from \Cref{alg:detCGD-DASHA}, we have for any $\mD \in \bbS^d_{++}$
   \begin{align*}
      \norm{h^{k+1} - \nabla f(x^{k+1})}^2_{\mD} &= \norm{h_i^{k+1} - \nabla f_i(x^{k+1})}^2_{\mD} = 0,
   \end{align*}
   and 
   \begin{align*}
      \norm{h_i^{k+1} - h_i^{k}}^2_{\mL_i^{-1}} &\leq \norm{x^{k+1} - x^k}^2_{\mL_i}.
   \end{align*}
\end{lemma}

The following lemmas describe the recurrence applied to terms in the Lyapunov function.

\begin{lemma}
   \label{dasha:tech-lemma-recur-1}
   Suppose $h^{k+1}$ and $g^{k+1}$ are from \Cref{alg:detCGD-DASHA}, then the following recurrence relation holds, 
   \begin{align}
      \label{dasha:eq:tech-lemma-recur-1}
      &\Exp{\norm{g^{k+1} - h^{k+1}}^2_{\mD}} \notag \\
      &\quad \leq \frac{2\Lambda_{\mD, \cS}\cdot\lambda_{\max}\left(\mD^{-1}\right)\cdot\lambda_{\max}\left(\mD\right)}{n^2}\sum_{i=1}^{n}\lambda_{\max}\left(\mL_i\right)\Exp{\norm{h_i^{k+1} - h_i^k}^2_{\mL_i^{-1}}} \notag \\
      &\qquad + \frac{2a^2\Lambda_{\mD, \cS}\cdot\lambda_{\max}\left(\mD^{-1}\right)}{n^2}\sum_{i=1}^{n}\Exp{\norm{g_i^k - h_i^k}^2_{\mD}} + (1 - a)^2\Exp{\norm{g^k - h^k}^2_{\mD}},
   \end{align}
   where $\Lambda_{\mD, \cS} = \lambda_{\max}\left(\Exp{\mS_i^k\mD\mS_i^k} - \mD\right)$ for $\mD \in \bbS^d_{++}$ and $\mS_i^k \sim \cS$.
\end{lemma}

\begin{lemma}
   \label{dasha:tech-lemma-recur-2}
   Suppose $h_i^{k+1}$ and $g_i^{k+1}$ for $i \in [n]$ are from \Cref{alg:detCGD-DASHA}, then the following recurrence holds,
   \begin{align*}
      & \Exp{\norm{g_i^{k+1} - h_i^{k+1}}^2_{\mD}} \notag \\ 
      &\quad \leq \left(2a^2\lambda_{\max}\left(\mD^{-1}\right)\cdot\Lambda_{\mD, \cS} + (1-a)^2\right)\cdot\Exp{\norm{g_i^k - h_i^k}^2_{\mD}} \notag \\
      &\qquad + 2\lambda_{\max}\left(\mD^{-1}\right)\cdot\lambda_{\max}\left(\mD\right)\cdot\Lambda_{\mD, \cS}\cdot\lambda_{\max}\left(\mL_i\right)\cdot\Exp{\norm{h_i^{k+1} - h_i^k}^2_{\mL_i^{-1}}}.
   \end{align*}
\end{lemma}

\subsection{\texorpdfstring{Proof of \Cref{thm:1:detCGDVR1}}{Proof of Theorem~\ref{thm:1:detCGDVR1}}}
According to \Cref{lemma:1}, we have 
\begin{align}
   \label{eq:prfthm1-eq1}
   \Exp{f(x^{k+1})} &\leq \Exp{f(x^k)} - \Exp{\frac{1}{2}\norm{\nabla f(x^k)}_{\mD}^2} + \Exp{\frac{1}{2}\norm{g^k - \nabla f(x^k)}_{\mD}^2} \notag \\
   & \quad - \Exp{\frac{1}{2}\norm{x^{k+1} - x^k}^2_{\mD^{-1} - \mL}}.
\end{align}
We then use the definition of $g^{k+1}$ to derive an upper bound for $\Exp{\norm{g^{k+1} - \nabla f(x^{k+1})}^2_{\mD}}$. Notice that,
\begin{equation*}
   g^{k+1} = \begin{cases}
      \nabla f(x^{k+1}) & \text{ with probability } p, \\
      g^k + \frac{1}{n}\sum_{i=1}^{n}\mS_i^k\left(\nabla f_i(x^{k+1}) - \nabla f_i(x^k)\right) & \text{ with probability } 1-p.
   \end{cases}
\end{equation*}
As a result, from the tower property,
\begin{align*}
   &\Exp{\norm{g^{k+1} - \nabla f(x^{k+1})}^2_{\mD} \mid x^{k+1}, x^k} \\ 
   &\quad =
   \Exp{\Exp{\norm{g^{k+1} - \nabla f(x^{k+1})}^2_{\mD} \mid x^{k+1}, x^k, c_k}} \\ 
   &\quad = p\cdot\norm{\nabla f(x^{k+1}) - \nabla f(x^{k+1})}^2_{\mD} \\
   & \qquad + (1-p)\cdot\Exp{\norm{g^k + \frac{1}{n}\sum_{i=1}^{n}\mS_i^k(\nabla f_i(x^{k+1}) - \nabla f_i(x^k)) - \nabla f(x^{k+1})}^2_{\mD} \mid x^{k+1}, x^k} \\
   & \quad =  (1-p)\cdot\Exp{\norm{g^k + \frac{1}{n}\sum_{i=1}^{n}\mS_i^k(\nabla f_i(x^{k+1}) - \nabla f_i(x^k)) - \nabla f(x^{k+1})}^2_{\mD}\mid x^{k+1}, x^k}. 
\end{align*}
Using \Cref{lemma:2:var-decomp}, we have
\begin{align*}
   &\Exp{\norm{g^{k+1} - \nabla f(x^{k+1})}^2_{\mD} \mid x^{k+1}, x^k} \\ 
   & = (1-p)\cdot\Exp{\norm{\frac{1}{n}\sum_{i=1}^{n}\mS_i^k(\nabla f_i(x^{k+1}) - \nabla f_i(x^k)) - \left(\nabla f(x^{k+1}) - \nabla f(x^{k})\right)}^2_{\mD}\mid x^{k+1}, x^k} \\ 
   & \qquad  + (1-p)\cdot\norm{g^k - \nabla f(x^k)}^2_{\mD} \\
   &= (1-p)\cdot\Exp{\norm{\frac{1}{n}\sum_{i=1}^{n}\left(\mS_i^k(\nabla f_i(x^{k+1}) - \nabla f_i(x^k)) - (\nabla f_i(x^{k+1}) - \nabla f_i(x^k))\right)}^2_{\mD}\mid x^{k+1}, x^k} \\ 
   & \qquad  + (1-p)\cdot\norm{g^k - \nabla f(x^k)}^2_{\mD}.
\end{align*}
Notice that the sketch matrix is unbiased, thus we have 
\begin{equation*}
   \Exp{\mS^k_i\left(\nabla f_i(x^{k+1}) - \nabla f_i(x^k)\right) \mid x^{k+1}, x^k} = \nabla f_i(x^{k+1}) - \nabla f_i(x^k),
\end{equation*}
and any two random vectors in the set $\{\mS_i^k(\nabla f_i(x^{k+1}) - \nabla f_i(x^k))\}_{i=1}^n$ are independent from each other, if $x^{k+1}$ and $x^k$ are fixed. 
Therefore,  we have 
\begin{align}
   \label{eq:prthm1-eq4}
   &\Exp{\norm{g^{k+1} - \nabla f(x^{k+1})}^2_{\mD} \mid x^{k+1}, x^k} \notag \\ 
   &\quad =  \frac{1-p}{n^2}\sum_{i=1}^{n}\Exp{\norm{\mS_i^k(\nabla f_i(x^{k+1}) - \nabla f_i(x^k)) - (\nabla f_i(x^{k+1}) - \nabla f_i(x^k))}^2_{\mD}\mid x^{k+1}, x^k} \notag \\
   & \qquad + (1-p)\cdot\norm{g^k - \nabla f(x^k)}_{\mD}^2.
\end{align}
\Cref{lemma:3} yields
\begin{align}
   \label{eq:prfthm1-eq2}
   & \Exp{\norm{\mS_i^k(\nabla f_i(x^{k+1}) - \nabla f_i(x^k)) - (\nabla f_i(x^{k+1}) - \nabla f_i(x^k))}^2_{\mD} \mid x^{k+1}, x^k} \notag \\ 
   & \qquad \leq \lambda_{\max}\left(\mL_i^\frac{1}{2}\left(\Exp{\mS_i^k\mD\mS_i^k} - \mD\right)\mL_i^\frac{1}{2}\right)\norm{\nabla f_i(x^{k+1}) - \nabla f_i(x^k)}^2_{\mL_i^{-1}}.
\end{align}
\Cref{assmp:4} implies
\begin{align}
   \label{eq:prthm1-eq3}
   & \Exp{\norm{\mS_i^k(\nabla f_i(x^{k+1}) - \nabla f_i(x^k)) - (\nabla f_i(x^{k+1}) - \nabla f_i(x^k))}^2_{\mD} \mid x^{k+1}, x^k} \notag \\ 
   & \quad \leq \lambda_{\max}\left(\mL_i^\frac{1}{2}\left(\Exp{\mS_i^k\mD\mS_i^k} - \mD\right)\mL_i^\frac{1}{2}\right)\norm{x^{k+1} - x^k}^2_{\mL_i}.
\end{align}
Plugging \eqref{eq:prthm1-eq3} into \eqref{eq:prthm1-eq4}, we deduce 
\begin{align*}
   &\Exp{\norm{g^{k+1} - \nabla f(x^{k+1})}^2_{\mD} \mid x^{k+1}, x^k} \notag \\ 
   & \leq \frac{1-p}{n^2}\sum_{i=1}^{n}\lambda_{\max}\left(\mL_i^\frac{1}{2}\left(\Exp{\mS_i^k\mD\mS_i^k} - \mD\right)\mL_i^\frac{1}{2}\right)\norm{x^{k+1} - x^k}^2_{\mL_i} + (1-p)\cdot\norm{g^k - \nabla f(x^k)}^2_{\mD}.
\end{align*}
Replacing $\mL_i^{-1}$ with $\mL^{-1/2}\mL^{1/2}\mL_i^{-1}\mL^{1/2}\mL^{-1/2}$, we denote that 
\begin{equation*}
   \lambda_i \eqdef \lambda_{\max}\left(\mL_i^\frac{1}{2}\left(\Exp{\mS_i^k\mD\mS_i^k} - \mD\right)\mL_i^\frac{1}{2}\right),
\end{equation*}
and rewrite the $\mL_i$-norm in the first term of RHS by the $\mL$-norm:
\begin{align*}
   &\Exp{\norm{g^{k+1} - \nabla f(x^{k+1})} ^2_{\mD} \mid x^{k+1}, x^k}\\
   & = \frac{1-p}{n^2}\sum_{i=1}^{n}\lambda_i\cdot\left(\mL^\frac{1}{2}(x^{k+1} - x^k)\right)^{\top}\mL^{-\frac{1}{2}}\mL_i\mL^{-\frac{1}{2}}\left(\mL^\frac{1}{2}(x^{k+1} - x^k)\right) \\
   & \quad + (1-p)\norm{g^k - \nabla f(x^k)}^2_{\mD} \\
   &\leq \frac{1-p}{n^2}\sum_{i=1}^{n}\lambda_i\cdot\lambda_{\max}\left(\mL^{-\frac{1}{2}}\mL_i\mL^{-\frac{1}{2}}\right)\norm{x^{k+1}-x^k}^2_{\mL} + (1-p)\cdot\norm{g^k - \nabla f(x^k)}^2_{\mD}.
\end{align*}
We further use \Cref{fact:4} to upper bound $\lambda_{\max}\left(\mL_i^\frac{1}{2}\left(\Exp{\mS_i^k\mD\mS_i^k} - \mD\right)\mL_i^\frac{1}{2}\right)$ by the product of $\lambda_{\max}\left(\mL_i\right)$ and $\lambda_{\max}\left(\Exp{\mS_i^k\mD\mS_i^k} - \mD\right)$. 
This allows us to simplify the expression since $\lambda_{\max}\left(\Exp{\mS_i^k\mD\mS_i^k} - \mD\right)$ is independent of the index $i$. 
Notice that we have already defined 
\begin{equation*}
   R(\mD, \cS) = \frac{1}{n}\sum_{i=1}^{n}\lambda_{\max}\left(\Exp{\mS_i^k\mD\mS_i^k} - \mD\right)\cdot\lambda_{\max}\left(\mL_i\right)\cdot\lambda_{\max}\left(\mL^{-\frac{1}{2}}\mL_i\mL^{-\frac{1}{2}}\right).
\end{equation*} 
Taking expectation, using tower property and using the definition above, we deduce 
\begin{align}
   \label{eq:second-bound}
   &\Exp{\norm{g^{k+1} - \nabla f(x^{k+1})}^2_{\mD}} \notag \\
   &\quad \leq \frac{(1-p)\cdot R(\mD, \cS)}{n}\Exp{\norm{x^{k+1} - x^k}_{\mL}^2} + (1-p)\Exp{\norm{g^k - \nabla f(x^k)}_{\mD}^2}.
\end{align}
We construct the following Lyapunov function $\Phi_k$,
\begin{equation}
   \label{def:Lyapunov}
   \Phi_k = f(x^k) - f^{\star} + \frac{1}{2p}\norm{g^k - \nabla f(x^k)}_{\mD}^2.
\end{equation}
Using \eqref{eq:prfthm1-eq1} and \eqref{eq:second-bound}, we are able to get 
\begin{align*}
   \Exp{\Phi_{k+1}} &\leq \frac{1}{2p}\left[\frac{(1-p)\cdot R(\mD, \cS)}{n}\Exp{\norm{x^{k+1} - x^k}^2_{\mL}} + (1-p)\cdot\Exp{\norm{g^k - \nabla f(x^k)}^2_{\mD}}\right] \\
   &\qquad + \Exp{f(x^k) - f^{\star}} - \frac{1}{2}\Exp{\norm{\nabla f(x^k)}^2_{\mD}} + \frac{1}{2}\Exp{\norm{g^k - \nabla f(x^k)}_{\mD}^2}  \\
   &\qquad - \frac{1}{2}\Exp{\norm{x^{k+1}- x^k}^2_{\mD^{-1} - \mL}} \\
   &\quad = \Exp{\Phi_k} - \frac{1}{2}\Exp{\norm{\nabla f(x^k)}_{\mD}^2} \\
   &\qquad + \left(\frac{(1-p)\cdot R(\mD, \cS)}{2np}\Exp{\norm{x^{k+1} - x^k}^2_{\mL}} -  \frac{1}{2}\Exp{\norm{x^{k+1} - x^k}_{\mD^{-1} - \mL}^2}\right) \\
   &\quad = \Exp{\Phi_k} - \frac{1}{2}\Exp{\norm{\nabla f(x^k)}_{\mD}^2} \\
   &\qquad + \frac{1}{2}\left(\frac{(1-p)\cdot R(\mD, \cS)}{np}\Exp{\norm{x^{k+1} - x^k}^2_{\mL}} -  \Exp{\norm{x^{k+1} - x^k}_{\mD^{-1} - \mL}^2}\right).
\end{align*}
We can rewrite the last term as 
\begin{equation}\label{eq:lastterm}
   \Exp{(x^{k+1}-x^k)^{\top}\left[\frac{(1-p)\cdot R(\mD, \cS)}{np}\mL + \mL - \mD^{-1}\right](x^{k+1} - x^k)}.
\end{equation}
We require the matrix in between to be negative semi-definite, which is
\begin{equation*}
   \mD^{-1} \succeq \left(\frac{(1-p)\cdot R(\mD, \cS)}{np} + 1\right)\mL.
\end{equation*}
This leads to the result that the expression \eqref{eq:lastterm}  is always non-positive. 
After dropping the last term, the relation between $\Exp{\Phi_{k+1}}$ and $\Exp{\Phi_k}$ becomes 
\begin{equation*}
   \Exp{\Phi_{k+1}} \leq \Exp{\Phi_k} - \frac{1}{2}\Exp{\norm{\nabla f(x^k)}^2_{\mD}}.
\end{equation*}
Unrolling this recurrence, we get 
\begin{equation}
   \label{eq:simp-1}
   \frac{1}{K}\sum_{k=0}^{K-1}\Exp{\norm{\nabla f(x^k)}^2_{\mD}} \leq \frac{2\left(\Exp{\Phi_0 } - \Exp{\Phi_K}\right)}{K}.
\end{equation}
The left hand side can viewed as $\Exp{\norm{\nabla f(\tilde{x}^K)}^2_{\mD}}$, 
where $\tilde{x}^K$ is drawn uniformly at random from $\{x_k\}_{k=0}^{K-1}$.
From $\Phi_K > 0$, we obtain
\begin{eqnarray*}
   \frac{2\left(\Exp{\Phi_0 } - \Exp{\Phi_K}\right)}{K} &\leq& \frac{2\Phi_0}{K}\\
   &=& \frac{2\left(f(x^0) - f^{\star} + \frac{1}{2p}\norm{g^0 - \nabla f(x^0)}^2_{\mD}\right)}{K}\\
   &=& \frac{2\left(f(x^0) - f^{\star}\right)}{K}.
\end{eqnarray*}
Plugging in the simplified result into \eqref{eq:simp-1}, and performing determinant normalization, we get 
\begin{equation}
   \label{eq:final}
   \Exp{\norm{\nabla f(\tilde{x}^K)}^2_{\frac{\mD}{\det(\mD)^{1/d}}}} \leq \frac{2\left(f(x^0) - f^{\star}\right)}{\det(\mD)^{1/d}K}.
\end{equation}

\begin{remark}
   We can achieve a slightly more refined stepsize condition than \eqref{eq:stepsize-cond-thm1} for {\detmarina}, which is given as follows 
   \begin{equation}
      \label{eq:cond-remark-refined}
      \mD \succeq \left(\frac{(1 - p)\cdot \tilde{R}(\mD, \cS)}{np} + 1\right)\mL,
   \end{equation}
   where
   \begin{equation*}
      \tilde{R}(\mD, \cS) \eqdef \frac{1}{n}\sum_{i=1}^{n}\lambda_{\max}\left(\mL_i^\frac{1}{2}\left(\Exp{\mS_i^k\mD\mS_i^k} - \mD\right)\mL_i^\frac{1}{2}\right)\cdot\lambda_{\max}\left(\mL^{-\frac{1}{2}}\mL_i\mL^{-\frac{1}{2}}\right).
   \end{equation*}
   This is obtained if we do not use \Cref{fact:4} to upper bound $\lambda_{\max}\left(\mL_i^\frac{1}{2}\left(\Exp{\mS_i^k\mD\mS_i^k} - \mD\right)\mL_i^\frac{1}{2}\right)$ by the product of $\lambda_{\max}\left(\mL_i\right)$ and $\lambda_{\max}\left(\Exp{\mS_i^k\mD\mS_i^k} - \mD\right)$.
   However, \eqref{eq:cond-remark-refined} results in a condition that is much harder to solve even if we assume $\mD = \gamma\cdot\mW$. 
   So instead of using the more refined condition \eqref{eq:cond-remark-refined}, we turn to \eqref{eq:stepsize-cond-thm1}. 
   Notice that both of the two conditions \eqref{eq:cond-remark-refined} and \eqref{eq:stepsize-cond-thm1} reduce to the stepsize condition for {\marina} in the scalar setting.
\end{remark}

\subsection{Comparison of different stepsizes}
\label{sec:cmp-step}
In \Cref{col:iteration-comp-L-inv}, we focus on the special stepsize where we fix $\mW = \mL^{-1}$, and show that in this case {\detmarina} always beats {\marina} in terms of both iteration and communication complexities. 
However, other choices for $\mW$ are also possible. 
Specifically, we consider the cases where 
$\mW = \diag^{-1}\left(\mL\right)$ and $\mW = \mI_d$.

\subsubsection{The diagonal case}

We consider $\mW = \diag^{-1}\left(\mL\right)$. 
The following corollary describes the optimal stepsize and the iteration complexity.
\begin{corollary}
   \label{col:3}
   If we take $\mW = \diag^{-1}\left(\mL\right)$ in \Cref{ppst:optimal-D-var}, then the optimal stepsize satisfies
   \begin{equation}
      \label{eq:opt-diagD}
      \mD^*_{\diag^{-1}(\mL)} = \frac{2}{1 + \sqrt{1 + 4\alpha\beta\cdot\Lambda_{\diag^{-1}\left(\mL\right), \cS}}} \cdot \diag^{-1}\left(\mL\right).
   \end{equation} 
   This stepsize results in a better iteration complexity of {\detmarina} compared to scalar {\marina}.
\end{corollary}
From this corollary we know that {\detmarina} has a better iteration complexity when $\mW = \diag^{-1}\left(\mL\right)$. 
And since the same sketch is used for {\marina} and {\detmarina}, the communication complexity is improved as well. 
However, in general there is no clear relation between the iteration complexity of $\mW = \mL^{-1}$ case and $\mW = \diag^{-1}\left(\mL\right)$ case. 
This is also confirmed by one of our experiments, see \Cref{fig:experiment-4} to see the comparison of {\detmarina} using optimal stepsizes in different cases.

\subsubsection{The identity case}
In this setting, $\mW$ is the $d$-dimensional identity matrix $\mI_d$. 
Then the stepsize of our algorithm reduces to a scalar $\gamma$, where $\gamma$ is determined through \Cref{ppst:optimal-D-var}. 
Notice that in this case we do not reduce to the standard {\marina} case because we are still using the matrix Lipschitz gradient assumption with $\mL \in \bbS^d_{++}$. 
\begin{corollary}
   \label{col:4}
   If we take $\mW = \mI_d$, the optimal stepsize is given by 
   \begin{equation}
      \label{eq:cd4:1}
      \mD^*_{\mI_d} = \frac{2}{1 + \sqrt{1 + 4\alpha\beta\frac{1}{\lambda_{\max}\left(\mL\right)}\cdot\omega}} \cdot \frac{\mI_d}{\lambda_{\max}\left(\mL\right)}.
   \end{equation}
   This stepsize results in a better iteration complexity of {\detmarina} compared to scalar {\marina}.
\end{corollary}
The result in this corollary tells us that using scalar stepsize with matrix Lipschitz gradient assumption alone can result in acceleration of {\marina}. 
However, the use of matrix stepsize allows us to also take into consideration the "structure" of the stepsize, thus allows more flexibility. 
When the structure of the stepsize is chosen properly, combining matrix gradient Lipschitzness and matrix stepsize can result in a faster rate, as it can also be observed from the experiments in \Cref{fig:experiment-4}. 
The choices of $\mW$ we consider here are in some sense inspired by the matrix stepsize {\gd}, where the optimal stepsize is $\mL^{-1}$. 
In general, how to identify the best structure for the matrix stepsize remains a open problem.

\subsection{Proofs of the corollaries}

\subsubsection{\texorpdfstring{Proof of \Cref{ppst:optimal-D-var}}{Proof of Corollary~\ref{ppst:optimal-D-var}}}

   We start with rewriting \eqref{eq:stepsize-cond-thm1} as 
   \begin{eqnarray*}
      \left(\frac{1-p}{np}\cdot R\left(\mD, \cS\right) + 1\right)\mD^\frac{1}{2}\mL\mD^\frac{1}{2} \preceq \mI_d.
   \end{eqnarray*}
   Plugging in the definition of ${R}(\mD, \cS)$ and $\mD = \gamma\mW$, we get 
   \begin{eqnarray*}
      \gamma\left(\frac{1-p}{np}\cdot\frac{1}{n}\sum_{i=1}^{n}\lambda_{\max}\left(\mL_i\right)\lambda_{\max}\left(\mL^{-1}\mL_i\right)\cdot\lambda_{\max}\left(\Exp{\mS_i^k\mW\mS_i^k} - \mW\right)\cdot\gamma + 1\right)\mW^\frac{1}{2}\mL\mW^\frac{1}{2} \preceq \mI_d.
   \end{eqnarray*}
   This generalized inequality is equivalent to the following inequality, 
   \begin{eqnarray*}
      \gamma\left(\frac{1-p}{np}\cdot\frac{1}{n}\sum_{i=1}^{n}\lambda_{\max}\left(\mL_i\right)\lambda_{\max}\left(\mL^{-1}\mL_i\right)\cdot\lambda_{\max}\left(\Exp{\mS_i^k\mW\mS_i^k} - \mW\right)\cdot\gamma + 1\right)\cdot\lambda_{\max}\left(\mW^\frac{1}{2}\mL\mW^\frac{1}{2}\right) \leq 1,
   \end{eqnarray*}
   which is a quadratic inequality on $\gamma$. 
   Notice that we have already defined
   \begin{eqnarray*}
      \alpha = \frac{1 - p}{np}; &\qquad& \beta = \frac{1}{n}\sum_{i=1}^{n}\lambda_{\max}\left(\mL_i\right)\cdot\lambda_{\max}\left(\mL^{-1}\mL_i\right); \notag \\ 
      \Lambda_{\mW, \cS} = \lambda_{\max}\left(\Exp{\mS_i^k\mW\mS_i^k} - \mW\right); &\qquad& \lambda_{\mW} = \lambda_{\max}^{-1}\left(\mW^\frac12\mL\mW^\frac{1}{2}\right).
   \end{eqnarray*}
   As a result, the above inequality can be written equivalently as 
   \begin{equation*}
      \alpha\beta\Lambda_{\mW, \cS} \cdot \gamma^2 + \gamma - \lambda_{\mW} \leq 0,
   \end{equation*}
   which yields the upper bound on $\gamma$
   \begin{eqnarray*}
      \gamma \leq \frac{\sqrt{1 + 4\alpha\beta\cdot\Lambda_{\mW, \cS}\lambda_{\mW}} - 1}{2\alpha\beta\cdot\Lambda_{\mW, \cS}}.
   \end{eqnarray*}
   Since $\sqrt{1 + 4\alpha\beta\cdot\Lambda_{\mW, \cS}\lambda_{\mW}} + 1 > 0$, we can simplify the result as  
   \begin{eqnarray*}
      \gamma \leq \frac{2\lambda_{\mW}}{1 + \sqrt{1 + 4\alpha\beta\cdot\Lambda_{\mW, \cS}\lambda_{\mW}}}. 
   \end{eqnarray*}

\subsubsection{\texorpdfstring{Proof of \Cref{col:iteration-comp-L-inv}}{Proof of Corollary~\ref{col:iteration-comp-L-inv}}}
   It is obvious that \eqref{eq:max-gamma-L-inv} directly follows from plugging $\mW = \mL^{-1}$ into \eqref{eq:opt-cond-var}. 
   The optimal stepsize is obtained as the product of $\gamma$ and $\mL^{-1}$. 
   The iteration complexity of {\marina}, according to \citet{gorbunov2021marina}, is
   \begin{equation}
      \label{eq:comp-marina}
      K \geq K_1 = \cO\left(\frac{\Delta_0 L}{\varepsilon^2}\left(1 + \sqrt{\frac{(1 - p)\omega}{pn}}\right)\right).
   \end{equation}
   On the other hand, 
   \begin{equation}
      \label{eq:bound-1}
      \det(\mL)^\frac{1}{d} \leq \lambda_{\max}\left(\mL\right) = L.
   \end{equation}
   In addition, using the inequality
   \begin{equation}\label{eq:loose-inequality}
      \sqrt{1 + 4t} \leq 1 + 2\sqrt{t},
   \end{equation}
   which holds for any $t \geq 0$, we have the following bound 
   \begin{eqnarray*}
      \frac{\left(1 + \sqrt{1 + 4\alpha\beta\cdot\Lambda_{\mL^{-1}, \cS}}\right)}{2} 
      \leq 1 + \sqrt{\alpha\beta\cdot\Lambda_{\mL^{-1}, \cS}}.
   \end{eqnarray*}
   Next we prove that 
   \begin{equation}
      \label{eq:bound-2}
      1 + \sqrt{\alpha\beta\cdot\Lambda_{\mL^{-1}, \cS}} \leq 1 + \sqrt{\frac{(1 - p)}{pn}\cdot \omega},
   \end{equation}
   which is equivalent to proving 
   \begin{equation*}
      \frac{1}{n}\sum_{i=1}^{n}\lambda_{\max}\left(\mL_i\right)\lambda_{\max}\left(\mL_i\mL^{-1}\right)\cdot\lambda_{\max}\left(\Exp{\mS_i^k\mL^{-1}\mS_i^k} - \mL^{-1}\right) \leq \omega.
   \end{equation*}
   The left hand side can be upper bounded by,
   \begin{align*}
      &\frac{1}{n}\sum_{i=1}^{n}\lambda_{\max}\left(\mL_i\right)\lambda_{\max}\left(\mL^{-1}\mL_i\right)\cdot\lambda_{\max}\left(\mL^{-1}\right)\cdot\frac{\lambda_{\max}\left(\Exp{\mS_i^k\mL^{-1}\mS_i^k} - \mL^{-1}\right)}{\lambda_{\max}\left(\mL^{-1}\right)} \\
      &\quad \leq \frac{\lambda_{\max}\left(\Exp{\mS_i^k\mL^{-1}\mS_i^k} - \mL^{-1}\right)}{\lambda_{\max}\left(\mL^{-1}\right)},
   \end{align*}
   where the inequality is a consequence of \Cref{ppst:4}. 
   We further bound the last term with 
   \begin{eqnarray*}
      \frac{\lambda_{\max}\left(\Exp{\mS_i^k\mL^{-1}\mS_i^k} - \mL^{-1}\right)}{\lambda_{\max}\left(\mL^{-1}\right)} &=& \lambda_{\max}\left(\Exp{\mS_i^k \cdot \frac{\mL^{-1}}{\lambda_{\max}(\mL^{-1})}\cdot \mS_i^k} - \frac{\mL^{-1}}{\lambda_{\max}\left(\mL^{-1}\right)}\right) \\
      &\leq& \lambda_{\max}\left(\Exp{\mS_i^k\mS_i^k} - \mI_d\right) =: \omega.
   \end{eqnarray*}
   Here, the last inequality is due to the monotonicity of the mapping $\mX \mapsto \lambda_{\max}\left(\Exp{\mS_i^k\mX\mS_i^k} - \mX\right)$ with $\mX \in \bbS^d_{++}$, which can be shown as follows, let us pick any $\mX_1, \mX_2 \in \bbS^d_{++}$ and $\mX_1 \preceq \mX_2$,
   \begin{eqnarray*}
      \left(\Exp{\mS_i^k\mX_2\mS_i^k} - \mX_2\right) - \left(\Exp{\mS_i^k\mX_1\mS_i^k} - \mX_1\right) = \Exp{\mS_i^k\left(\mX_2 - \mX_1\right)\mS_i^k} - \left(\mX_2 - \mX_1\right) \succeq \mO_d.
   \end{eqnarray*}
   The above inequality is due to the convexity of the mapping $\mS_i^k \mapsto \mS_i^k\mX\mS_i^k$. As a result, we have 
   \begin{equation*}
      \lambda_{\max}\left(\Exp{\mS_i^k\mX_2\mS_i^k} - \mX_2\right) \geq \lambda_{\max}\left(\Exp{\mS_i^k\mX_1\mS_i^k} - \mX_1\right),
   \end{equation*} 
   whenever $\mX_2 \succeq \mX_1$. Due to the fact that
   \begin{equation*}
      \frac{\mL^{-1}}{\lambda_{\max}\left(\mL^{-1}\right)} \preceq \mI_d,
   \end{equation*}
   we have 
   \begin{equation*}
      \lambda_{\max}\left(\Exp{\mS_i^k \cdot \frac{\mL^{-1}}{\lambda_{\max}(\mL^{-1})}\cdot \mS_i^k} - \frac{\mL^{-1}}{\lambda_{\max}\left(\mL^{-1}\right)}\right) \leq \lambda_{\max}\left(\Exp{\mS_i^k\cdot\mI_d\cdot\mS_i^k} - \mI_d\right) = \omega. 
   \end{equation*}
   Combining \eqref{eq:bound-1} and \eqref{eq:bound-2}, we know that the iteration complexity of {\detmarina} is always better than that of {\marina}.

\subsubsection{\texorpdfstring{Proof of \Cref{col:communication-comp}}{Proof of Corollary~\ref{col:communication-comp}}}
   The number of bits sent in expectation is 
   \begin{equation*}
      \cO(d + K(pd + (1-p)\zeta_{\cS})) = \cO((Kp+1)d + (1-p)K\zeta_{\cS}).
   \end{equation*}
   The special case where we choose $p = \zeta_{\cS} / d$ indicates that 
   \begin{equation*}
      \alpha = \frac{1-p}{np} = \frac{1}{n}\left(\frac{d}{\zeta_{\cS}} - 1\right).
   \end{equation*}
   In order to reach an error of $\varepsilon^2$, we need 
   \begin{eqnarray*}
      K = \cO\left(\frac{\Delta_0\cdot\det(\mL)^\frac{1}{d}}{\varepsilon^2}\cdot\left(1 + \sqrt{1 + \frac{4\beta}{n}\left(\frac{d}{\zeta_{\cS}} - 1\right)\cdot\Lambda_{\mL^{-1}, \cS}}\right)\right),
   \end{eqnarray*}
   which is the iteration complexity. 
   Applying once again \eqref{eq:loose-inequality} and using the fact that $p = \zeta_{\cS}/d$, the communication complexity in this case is given by 
   \begin{align*}
      & \cO\left(d + \frac{\Delta_0\cdot\det(\mL)^\frac{1}{d}}{\varepsilon^2}\cdot\left(1 + \sqrt{1 + \frac{4\beta}{n}\left(\frac{d}{\zeta_{\cS}} - 1\right)\cdot\Lambda_{\mL^{-1}, \cS}}\right)\cdot\left(pd + (1-p)\zeta_{\cS}\right)\right) \\
      &\quad \leq  \cO\left(d + \frac{2\Delta_0\cdot\det(\mL)^\frac{1}{d}}{\varepsilon^2}\cdot\left(1 + \sqrt{ \frac{\beta}{n}\left(\frac{d}{\zeta_{\cS}} - 1\right)\cdot\Lambda_{\mL^{-1}, \cS}}\right)\cdot\left(pd + (1-p)\zeta_{\cS}\right)\right) \\
      &\quad \leq \cO\left(d + \frac{4\Delta_0\cdot\det(\mL)^\frac{1}{d}}{\varepsilon^2}\cdot\left(\zeta_{\cS} + \sqrt{\frac{\beta\cdot\Lambda_{\mL^{-1}, \cS}}{n}\cdot\zeta_{\cS}(d - \zeta_{\cS})}\right)\right).
   \end{align*}
   Ignoring the coefficient we get 
   \begin{equation*}
      \cO\left(d + \frac{\Delta_0\cdot\det(\mL)^\frac{1}{d}}{\varepsilon^2}\cdot\left(\zeta_{\cS} + \sqrt{\frac{\beta\cdot\Lambda_{\mL^{-1}, \cS}}{n}\cdot\zeta_{\cS}(d - \zeta_{\cS})}\right)\right).
   \end{equation*}

\subsubsection{\texorpdfstring{Proof of \Cref{col:3}}{Proof of Corollary~\ref{col:3}}}
   Applying \Cref{ppst:optimal-D-var}, notice that in this case 
   \begin{equation*}
      \lambda_{\diag^{-1}\left(\mL\right)} = \lambda_{\max}^{-1}\left(\diag^{-\frac{1}{2}}\left(\mL\right)\mL\diag^{-\frac{1}{2}}\left(\mL\right)\right) = 1,
   \end{equation*}
   we obtain $\mD^*_{\diag^{-1}\left(\mL\right)}$. 
   The iteration complexity is given by 
   \begin{equation*}
      \cO\left(\frac{\det\left(\diag(\mL)\right)^\frac{1}{d}\cdot\Delta_0}{\varepsilon^2} \cdot \left(\frac{1 + \sqrt{1 + 4\alpha\beta\Lambda_{\diag^{-1}\left(\mL\right), \cS}}}{2}\right)\right).
   \end{equation*}
   We now compare it to the iteration complexity of {\marina}, which is given in \eqref{eq:comp-marina}. 
   We know that each diagonal element $\mL_{jj}$ satisfies $\mL_{jj} \leq \lambda_{\max}\left(\mL\right) = L$ for $j = 1,\ldots, d$. 
   As a result, 
   \begin{equation}
      \label{eq:cd3:1}
      \det\left(\diag(\mL)\right)^\frac{1}{d} \leq L.
   \end{equation}
   From \eqref{eq:loose-inequality}, we deduce
   \begin{align*}
      \frac{1 + \sqrt{1 + 4\alpha\beta\cdot\Lambda_{\diag^{-1}(\mL), \cS}}}{2} \leq 1 + \sqrt{\alpha\beta\cdot\Lambda_{\diag^{-1}(\mL), \cS}}.
   \end{align*}
   Now, let us prove the below inequality
   \begin{equation}
      \label{eq:cd3:2}
      1 + \sqrt{\alpha\beta\cdot\Lambda_{\diag^{-1}\left(\mL\right), \cS}} \leq 1 + \sqrt{\frac{(1-p)}{pn}\cdot\omega}. 
   \end{equation}
   The latter is equivalent to
   \begin{equation*}
      \beta\cdot\Lambda_{\diag^{-1}\left(\mL\right), \cS} \leq \omega.
   \end{equation*}
   Plugging in the definition of $\beta$, $\omega$ and $\Lambda_{\diag^{-1}\left(\mL\right), \cS}$ and using the relation given in \Cref{ppst:4}, we obtain,
   \begin{equation*}
      \lambda_{\max}\left(\Exp{\mS_i^k\frac{\diag^{-1}\left(\mL\right)}{\lambda_{\max}\left(\mL^{-1}\right)}\mS_i^k - \frac{\diag^{-1}\left(\mL\right)}{\lambda_{\max}\left(\mL^{-1}\right)}}\right) \leq \lambda_{\max}\left(\Exp{\mS_i^k\mI_d\mS_i^k} - \mI_d\right).
   \end{equation*}
   Thus, it is enough to prove that 
   \begin{equation*}
      \frac{\diag^{-1}\left(\mL\right)}{\lambda_{\max}\left(\mL^{-1}\right)} \preceq \mI_d.
   \end{equation*}
   We can further simplify the above inequality as 
   \begin{equation*}
      \lambda_{\min}\left(\mL\right) \leq \lambda_{\min}\left(\diag(\mL)\right),
   \end{equation*}
   which is always true for any $\mL \in \bbS^d_{++}$. 
   Combining \eqref{eq:cd3:1} and \eqref{eq:cd3:2} we conclude the proof.

\subsubsection{\texorpdfstring{Proof of \Cref{col:4}}{Proof of Corollary~\ref{col:4}}}
   Using the explicit formula for the optimal stepsize $\mD^*_{\mI_d}$,  
   we deduce the following iteration complexity for 
   \begin{equation}
      \label{eq:comm-Id}
      \cO\left(\frac{\lambda_{\max}\left(\mL\right)\Delta_0}{\varepsilon^2}\cdot\left(\frac{1 + \sqrt{1 + 4\alpha\beta\frac{\omega}{\lambda_{\max}\left(\mL\right)}}}{2}\right)\right).
   \end{equation}
   Recall that $\lambda_{\max}\left(\mL\right) = L$, we obtain using \eqref{eq:loose-inequality} that
   \begin{align*}
      \frac{1 + \sqrt{1 + 4\alpha\beta\frac{\omega}{\lambda_{\max}\left(\mL\right)}}}{2} &\leq 1 + \sqrt{\alpha\beta\frac{\omega}{\lambda_{\max}\left(\mL\right)}}. 
   \end{align*}
   The comparison of two iteration complexities, given in \eqref{eq:comm-Id} and \eqref{eq:comp-marina} reduces to
   \begin{equation*}
      1 + \sqrt{\alpha\beta\frac{\omega}{\lambda_{\max}\left(\mL\right)}} \leq 1 + \sqrt{\frac{1-p}{np}\omega}.
   \end{equation*}
   This is equivalent to
   \begin{equation*}
      \beta \cdot \frac{1}{\lambda_{\max}\left(\mL\right)} \leq 1.
   \end{equation*}
   Utilizing \Cref{ppst:4}, the above inequality can be rewritten as 
   \begin{equation*}
      \frac{1}{\lambda_{\max}\left(\mL^{-1}\right)\cdot\lambda_{\max}\left(\mL\right)} \leq 1,
   \end{equation*}
   which is exactly 
   \begin{equation*}
      \lambda_{\min}\left(\mL\right) \leq \lambda_{\max}\left(\mL\right).
   \end{equation*}

\section{\texorpdfstring{Analysis of {\detdasha}}{Analysis of det-DASHA}}

\subsection{\texorpdfstring{Proof of \Cref{dasha:thm:main}}{Proof of Theorem~\ref{dasha:thm:main}}}

   Using \Cref{lemma:1} and taking expectations, we are able to obtain
   \begin{align}
      \label{dasha:eq:bound-1}
      &\Exp{f(x^{k+1})} \notag\\
      &\quad\leq \Exp{f(x^k)} - \frac{1}{2}\Exp{\norm{\nabla f(x^k)}^2_{\mD}} - \frac{1}{2}\Exp{\norm{x^{k+1} - x^k}^2_{\mD^{-1} - \mL}} + \frac{1}{2}\Exp{\norm{g^k - \nabla f(x^k)}_{\mD}^2} \notag\\
      &\quad\leq \Exp{f(x^k)} - \frac{1}{2}\Exp{\norm{\nabla f(x^k)}^2_{\mD}} - \frac{1}{2}\Exp{\norm{x^{k+1} - x^k}^2_{\mD^{-1} - \mL}} \notag\\
      &\qquad + \Exp{\frac{1}{2}\norm{g^k - h^k + h^k - \nabla f(x^k)}^2_{\mD}} \notag\\
      &\quad \leq \Exp{f(x^k)} - \frac{1}{2}\Exp{\norm{\nabla f(x^k)}^2_{\mD}} - \frac{1}{2}\Exp{\norm{x^{k+1} - x^k}^2_{\mD^{-1} - \mL}} \notag\\
      &\qquad + \Exp{\norm{g^k - h^k}^2_{\mD} + \norm{h^k - \nabla f(x^k)}^2_{\mD}},
   \end{align}
   where the last step is due to the convexity of the norm. 
   Using \Cref{dasha:tech-lemma-recur-1}, we obtain 
   \begin{align}
      \label{dasha:eq:bound-2}
      \Exp{\norm{g^{k+1} - h^{k+1}}^2_{\mD}} &\leq \frac{2\omega_{\mD}\cdot\lambda_{\max}\left(\mD\right)}{n^2}\sum_{i=1}^{n}\lambda_{\max}\left(\mL_i\right)\Exp{\norm{h_i^{k+1} - h_i^k}^2_{\mL_i^{-1}}} \notag \\
      &\qquad + \frac{2a^2\omega_{\mD}}{n^2}\sum_{i=1}^{n}\Exp{\norm{g_i^k - h_i^k}^2_{\mD}} + (1 - a)^2\Exp{\norm{g^k - h^k}^2_{\mD}}.
   \end{align}
   Using \Cref{dasha:tech-lemma-recur-2}, we get 
   \begin{align}
      \label{dasha:eq:bound-3}
      \Exp{\norm{g_i^{k+1} - h_i^{k+1}}^2_{\mD}} 
      &\leq \left(2a^2\omega_{\mD} + (1-a)^2\right)\cdot\Exp{\norm{g_i^k - h_i^k}^2_{\mD}} \notag \\
      &\qquad + 2\omega_{\mD}\cdot\lambda_{\max}\left(\mD\right)\cdot\lambda_{\max}\left(\mL_i\right)\cdot\Exp{\norm{h_i^{k+1} - h_i^k}^2_{\mL_i^{-1}}}.
   \end{align}
   Now let us fix $\kappa \in [0, +\infty)$, $\eta \in [0, +\infty)$ which we will determine later, and construct the following Lyapunov function $\Phi_k$
   \begin{equation}
      \label{dasha:eq:def-Phi-k}
      \Phi_{k} = \Exp{f(x^k) - f^{\star}} + \kappa\cdot\Exp{\norm{g^{k} - h^{k}}_{\mD}^2} + \eta\cdot\Exp{\frac{1}{n}\sum_{i=1}^{n}\norm{g_i^{k} - h_i^{k}}^2_{\mD}}.
   \end{equation}
   Combining \eqref{dasha:eq:bound-1}, \eqref{dasha:eq:bound-2} and \eqref{dasha:eq:bound-3}, we get 
   \begin{align*}
      & \Phi_{k+1} \\
      &\quad \leq  \Exp{f(x^k) - f^{\star} - \frac{1}{2}\norm{\nabla f(x^k)}^2_{\mD}} \\ 
      &\qquad + \Exp{- \frac{1}{2}\norm{x^{k+1} - x^k}^2_{\mD^{-1} - \mL}  + \norm{g^k - h^k}^2_{\mD} + \norm{h^k - \nabla f(x^k)}^2_{\mD}} \\
      &\qquad + \kappa(1-a)^2\Exp{\norm{g^k - h^k}^2_{\mD}} + \frac{2\kappa\cdot\omega_{\mD}\lambda_{\max}\left(\mD\right)}{n}\cdot\frac{1}{n}\sum_{i=1}^{n}\lambda_{\max}\left(\mL_i\right)\Exp{\norm{h_i^{k+1} - h_i^k}^2_{\mL_i^{-1}}} \\
      &\qquad + \frac{2a^2\omega_{\mD}\cdot\kappa}{n}\cdot\frac{1}{n}\sum_{i=1}^{n}\Exp{\norm{g_i^k - h_i^k}^2_{\mD}} + \eta\left(2a^2\omega_{\mD} + (1 - a)^2\right)\cdot\frac{1}{n}\sum_{i=1}^{n}\Exp{\norm{g_i^k - h_i^k}^2_{\mD}}\\
      &\qquad + 2\eta\cdot\omega_{\mD}\cdot\lambda_{\max}\left(\mD\right)\cdot\frac{1}{n}\sum_{i=1}^{n}\lambda_{\max}\left(\mL_i\right)\cdot\Exp{\norm{h_i^{k+1} - h_i^k}^2_{\mL_i^{-1}}}.
   \end{align*}
   Rearranging terms, and notice that $\norm{h^k - \nabla f(x^k)}^2_{\mD} = 0$, 
   \begin{align*}
      &\Phi_{k+1} \\
      &\quad \leq \Exp{f(x^k) - f^{\star}} - \frac{1}{2}\Exp{\norm{\nabla f(x^k)}^2_{\mD}}  \\
      &\qquad - \frac{1}{2}\Exp{\norm{x^{k+1} - x^k}^2_{\mD^{-1} - \mL}} + \left(1 + \kappa(1-a)^2\right)\Exp{\norm{g^k - h^k}^2_{\mD}} \\ 
      &\qquad + \left(\frac{2a^2\omega_{\mD}\cdot\kappa}{n} + \eta\left(2a^2\omega_{\mD} + (1-a)^2\right)\right)\cdot\frac{1}{n}\sum_{i=1}^{n}\Exp{\norm{g_i^k - h_i^k}^2_{\mD}} \\
      &\qquad + \left(\frac{2\kappa\cdot\omega_{\mD}\lambda_{\max}\left(\mD\right)}{n} + 2\eta\cdot\omega_{\mD}\cdot\lambda_{\max}\left(\mD\right)\right)\cdot\frac{1}{n}\sum_{i=1}^{n}\lambda_{\max}\left(\mL_i\right)\cdot\Exp{\norm{h_i^{k+1} - h_i^k}^2_{\mL_i^{-1}}}.
   \end{align*}
   In order to proceed, we consider the choice of $\kappa$ and $\eta$, for $\kappa$, 
   \begin{equation}
      \label{dasha:eq:cond-kappa}
      1 + \kappa(1 - a)^2 \leq \kappa.
   \end{equation}
   It is then clear that the choice of $\kappa = \frac{1}{a}$ satisfies the condition. 
   On the other hand, we look at the terms involving $\Exp{\norm{g_i^k - h_i^k}^2_{\mD}}$, we can rewrite as
   \begin{align*}
      T_1 &:= \left(\frac{2a^2\omega_{\mD}\cdot\kappa}{n} + \eta\left(2a^2\omega_{\mD} + (1-a)^2\right)\right)\cdot\frac{1}{n}\sum_{i=1}^{n}\Exp{\norm{g_i^k - h_i^k}^2_{\mD}}.
   \end{align*} 
   Picking $\kappa = \frac{1}{a}$ and $a = \frac{1}{2\omega_{\mD} + 1}$, the $T_1$ can be simplified as 
   \begin{align*}
      T_1 = \left(\frac{2\omega_{\mD}}{n\cdot(2\omega_{\mD} + 1)} + \eta \cdot \frac{4\omega_{\mD}^2 + 2\omega_{\mD}}{\left(2\omega_{\mD} + 1\right)^2}\right) \cdot \frac{1}{n}\sum_{i=1}^{n}\Exp{\norm{g_i^k - h_i^k}^2_{\mD}}.
   \end{align*}
   We pick $\eta$ so that it satisfies
   \begin{align}
      \label{dasha:eq:cond-eta}
      \left(\frac{2\omega_{\mD}}{n\cdot(2\omega_{\mD} + 1)} + \eta \cdot \frac{4\omega_{\mD}^2 + 2\omega_{\mD}}{\left(2\omega_{\mD} + 1\right)^2}\right) \leq \eta.
   \end{align}
   Taking $\eta = \frac{2\omega_{\mD}}{n}$, which is the minimum value satisfying \eqref{dasha:eq:cond-eta}, we conclude that 
   \begin{align}
      \label{dasha:eq:bound-lya-2}
      T_1 \leq \eta \cdot \frac{1}{n}\sum_{i=1}^{n}\Exp{\norm{g_i^k - h_i^k}^2_{\mD}}.
   \end{align}
   Combining \eqref{dasha:eq:cond-kappa} and \eqref{dasha:eq:bound-lya-2}, we are able to conclude that 
   \begin{align*}
      &\Phi_{k+1} \\
      &\quad \leq \Exp{f(x^k) - f^{\star}} + \kappa\cdot\Exp{\norm{g^k - h^k}^2_{\mD}} + \eta\cdot\frac{1}{n}\sum_{i=1}^{n}\Exp{\norm{g_i^k - h_i^k}^2_{\mD}}\\
      &\qquad - \frac{1}{2}\Exp{\norm{\nabla f(x^k)}^2_{\mD}} - \frac{1}{2}\Exp{\norm{x^{k+1} - x^k}^2_{\mD^{-1} - \mL}} \\
      &\qquad + \left(\frac{2\kappa\cdot\omega_{\mD}\lambda_{\max}\left(\mD\right)}{n} + 2\eta\cdot\omega_{\mD}\cdot\lambda_{\max}\left(\mD\right)\right)\cdot\frac{1}{n}\sum_{i=1}^{n}\lambda_{\max}\left(\mL_i\right)\cdot\Exp{\norm{h_i^{k+1} - h_i^k}^2_{\mL_i^{-1}}}.
   \end{align*}
   Using the definition of $\Phi_k$ and \Cref{dasha:tech-lemma:2}, we obtain
   \begin{align*}
      \Phi_{k+1} &\leq \Phi_k - \frac{1}{2}\Exp{\norm{\nabla f(x^k)}^2_{\mD}} - \frac{1}{2}\Exp{\norm{x^{k+1} - x^k}^2_{\mD^{-1} - \mL}} \\
      &\qquad \left(\frac{2\kappa\cdot\omega_{\mD}\lambda_{\max}\left(\mD\right)}{n} + 2\eta\cdot\omega_{\mD}\cdot\lambda_{\max}\left(\mD\right)\right)\cdot\frac{1}{n}\sum_{i=1}^{n}\lambda_{\max}\left(\mL_i\right)\cdot\Exp{\norm{x^{k+1} - x^k}^2_{\mL_i}} \\
      &= \Phi_k - \frac{1}{2}\Exp{\norm{\nabla f(x^k)}^2_{\mD}} + \Exp{\norm{x^{k+1} - x^k}^2_{\mN}},
   \end{align*}
   where $\mN \in \bbS^d$ is defined as 
   \begin{align*}
      \mN &:= \left(\frac{2\kappa\cdot\omega_{\mD}\lambda_{\max}\left(\mD\right)}{n} + 2\eta\cdot\omega_{\mD}\cdot\lambda_{\max}\left(\mD\right)\right)\cdot\frac{1}{n}\sum_{i=1}^{n}\lambda_{\max}\left(\mL_i\right)\cdot\mL_i - \frac{1}{2}\mD^{-1} + \frac{1}{2}\mL.
   \end{align*}
   We require $\mN \preceq \mO_d$, which leads to the condition on $\mD$: 
   \begin{equation*}
      \mD^{-1} - \mL - \frac{4\lambda_{\max}\left(\mD\right)\cdot\omega_{\mD}\cdot\left(4\omega_{\mD} + 1\right)}{n}\cdot\frac{1}{n}\sum_{i=1}^{n}\lambda_{\max}\left(\mL_i\right)\cdot\mL_i \succeq \mO_d.
   \end{equation*}
   Given the above condition is satisfied, we have the recurrence 
   \begin{equation*}
      \frac{1}{2}\Exp{\norm{\nabla f(x^k)}^2_{\mD}} \leq \Phi_{k} - \Phi_{k+1}
   \end{equation*}
   Summing up for $k=0 \hdots K-1$, we obtain
   \begin{align}
      \label{dasha:eq:rec-main}
      \sum_{k=0}^{K-1}\Exp{\norm{\nabla f(x^k)}^2_{\mD}} \leq 2(\Phi_0 - \Phi_k).
   \end{align}
   Notice that we also have 
   \begin{align*}
      \Phi_0 &= f(x^0) - f^{\star} + \left(2\omega_{\mD}+1\right)\norm{g^0 - h^0}^2_{\mD} + frac{2\omega_{\mD}}{n}\cdot\frac{1}{n}\sum_{i=1}^{n} \norm{g_i^0 - h_i^0}^2 \\
      &= f(x^0) - f^{\star},
   \end{align*}
   We divide both sides of \eqref{dasha:eq:rec-main} by $K$, and perform determinant normalization, 
   \begin{equation*}
      \frac{1}{K}\sum_{k=0}^{K-1}\Exp{\norm{\nabla f(x^k)}^2_{\frac{\mD}{\det(\mD)^{1/d}}}} \leq \frac{2(f(x^0) - f^{\star})}{\det(\mD)^{1/d}\cdot K}.
   \end{equation*}
   This is to say 
   \begin{equation*}
      \Exp{\norm{\nabla f(\tilde{x}^K)}^2_{\frac{\mD}{\det(D)^{1/d}}}} \leq \frac{2(f(x^0) - f^{\star})}{\det(\mD)^{1/d}\cdot K},
   \end{equation*}
   where  $\tilde{x}^K$ is chosen uniformly randomly from the first K iterates of the algorithm.

\subsection{Proofs of the corollaries}

\subsubsection{\texorpdfstring{Proof of \Cref{dasha:col:scaling}}{Proof of Corollary~\ref{dasha:col:scaling}}}\label{sec:dasha:corollaries}

   Plug $\mD = \gamma_{\mW}\cdot\mW$ into the stepsize condition in \Cref{dasha:thm:main}, we obtain 
   \begin{align*}
      \frac{1}{\gamma_{\mW}}\cdot\mW^{-1} - \mL - \frac{4\gamma_{\mW}\cdot\lambda_{\max}\left(\mW\right)\cdot\omega_{\mW}\left(4\omega_{\mW} + 1\right)}{n}\cdot\frac{1}{n}\sum_{i=1}^{n}\lambda_{\max}\left(\mL_i\right)\cdot\mL_i \succeq \mO_d.
   \end{align*}
   We then simplify the above condition as 
   \begin{align*}
      &\frac{1}{\gamma_{\mW}}\cdot\mL^{-\frac{1}{2}}\mW^{-1}\mL^{-\frac{1}{2}}  \\
      &\quad \succeq \mI_d + \frac{4\gamma_{\mW}\cdot\lambda_{\max}\left(\mW\right)\cdot\omega_{\mW}\left(4\omega_{\mW} + 1\right)}{n}\cdot\mL^{-\frac{1}{2}}\left(\frac{1}{n}\sum_{i=1}^{n}\lambda_{\max}\left(\mL_i\right)\cdot\mL_i\right)\mL^{-\frac{1}{2}}.
   \end{align*}
   Using \Cref{dasha:lemma:global-smooth}, we have
   \begin{equation*}
      \frac{1}{\gamma_{\mW}}\cdot\mL^{-\frac{1}{2}}\mW^{-1}\mL^{-\frac{1}{2}} - \frac{4\gamma_{\mW}\cdot\lambda_{\max}\left(\mW\right)\cdot\omega_{\mW}\left(4\omega_{\mW} + 1\right)}{n}\cdot\lambda_{\min}\left(\mL\right)\cdot\mI_d \succeq \mI_d.
   \end{equation*}
   Taking the minimum eigenvalue of both sides, we obtain that, 
   \begin{align*}
      \frac{1}{\gamma_{\mW}} \cdot \lambda_{\min}\left(\mL^{-\frac{1}{2}}\mW^{-1}\mL^{-\frac{1}{2}}\right) - \frac{4\gamma_{\mW}\cdot\lambda_{\max}\left(\mW\right)\cdot\omega_{\mW}\left(4\omega_{\mW} + 1\right)}{n}\cdot\lambda_{\min}\left(\mL\right) \geq 1,
   \end{align*}
   If we denote $C_{\mW} := \frac{\lambda_{\max}\left(\mW\right)\cdot\omega_{\mW}\left(4\omega_{\mW} + 1\right)}{n} > 0$, and $\lambda_{\mW} := \lambda_{\max}^{-1}\left(\mL^\frac{1}{2}\mW\mL^\frac{1}{2}\right)$, we can write
   \begin{align*}
      4\cdot C_{\mW}\cdot\lambda_{\min}\left(\mL\right)\cdot\gamma_{\mW}^2 + \gamma_{\mW} - \lambda_{\mW} \leq 0.
   \end{align*}
   The solution is given by 
   \begin{align*}
      \gamma_{\mW} \leq \frac{2\lambda_{\mW}}{1 + \sqrt{1 + 16C_{\mW}\lambda_{\min}\left(\mL\right)\cdot\lambda_{\mW}}}.
   \end{align*}

\subsubsection{\texorpdfstring{Proof of \Cref{dasha:col:complexity}}{Proof of Corollary~\ref{dasha:col:complexity}}}

   The best scaling factor in this case is given as, according to \Cref{dasha:col:scaling}, 
   \begin{equation*}
      \gamma_{\mL^{-1}} = \frac{2}{1 + \sqrt{1 + 16C_{\mL^{-1}}\cdot\lambda_{\min}\left(\mL\right)}}.
   \end{equation*}
   In order to reach a $\varepsilon^2$ stationary point, we need 
   \begin{equation*}
      K \geq \frac{\det(\mL)^\frac{1}{d}\left(f(x^0) - f^{\star}\right)}{\varepsilon^2}\cdot\left(1 + \sqrt{1 + 16C_{\mL^{-1}}\cdot\lambda_{\min}\left(\mL\right)}\right).
   \end{equation*}

\subsubsection{\texorpdfstring{Proof of \Cref{col:dasha:1}}{Proof of Corollary~\ref{col:dasha:1}}}

   The iteration complexity of {\detdasha} is given by, according to, \Cref{dasha:col:complexity},
   \begin{equation*}
      \cO\left(\frac{f(x^0) - f^\star}{\epsilon^2}\cdot\left(1 + \sqrt{1 + 16C_{\mL^{-1}}\cdot\lambda_{\min}\left(\mL\right)}\right)\cdot\det(\mL)^\frac{1}{d}\right).
   \end{equation*}
   Using the inequality $\sqrt{1 + t} \leq 1 + \sqrt{t}$ for $t > 0$, and leaving out the coefficients, we obtain 
   \begin{equation*}
      \cO\left(\frac{f(x^0) - f^\star}{\epsilon^2}\cdot\left(1 + \sqrt{C_{\mL^{-1}}\cdot\lambda_{\min}\left(\mL\right)}\right)\cdot\det(\mL)^\frac{1}{d}\right).
   \end{equation*}
   Notice that
   \begin{equation*}
      C_{\mL^{-1}}\cdot\lambda_{\min}\left(\mL\right) = \lambda_{\max}\left(\mL^{-1}\right) \cdot \frac{\omega_{\mL^{-1}}\left(4\omega_{\mL^{-1}} + 1\right)}{n} \cdot \lambda_{\min}\left(\mL\right) = \frac{\omega_{\mL^{-1}}\left(4\omega_{\mL^{-1} + 1}\right)}{n}.
   \end{equation*}
   As a result, the iteration complexity can be further simplified as 
   \begin{equation*}
      \cO\left(\frac{f(x^0) - f^*}{\epsilon^2}\cdot\left(1 + \frac{\omega_{\mL^{-1}}}{\sqrt{n}}\right)\cdot \det(\mL)^\frac{1}{d}\right).
   \end{equation*} 
   The iteration complexity of {\dasha} is, according to \citet[Corollary 6.2]{tyurin2024dasha} 
   \begin{equation*}
      \cO\left(\frac{1}{\epsilon^2}\cdot \left(f(x^0) - f^{\star}\right)\left(L + \frac{\omega}{\sqrt{n}} \widehat{L}\right)\right),
   \end{equation*}
   where $\widehat{L} = \sqrt{\frac{1}{n}\sum_{i=1}^{n} L_i^2}$. Since $\det(\mL)^\frac{1}{d} \leq \lambda_{\max}\left(\mL\right) = L$, and $L \leq \widehat{L}$, it is easy to see that compared to {\dasha}, {\detdasha} has a better iteration complexity when the momentum is the same. Notice that those two algorithms use the same sketch, thus, it also indicates that the communication complexity of the two algorithms are the same.

\subsubsection{\texorpdfstring{Proof of \Cref{col:dasha:2}}{Proof of Corollary~\ref{col:dasha:2}}}

   The iteration complexity of {\detmarina} is given by 
   \begin{equation*}
      \cO\left(\frac{f(x^0) - f^{\star}}{\epsilon^2}\cdot\det(\mL)^\frac{1}{d}\cdot\left(1 + \sqrt{\alpha\beta\Lambda_{\mL^{-1}, \cS}}\right)\right),
   \end{equation*}
   after removing logarithmic factors. 
   Plugging in the definitions we obtain in the case of $\omega_{\mL^{-1}} + 1 = \frac{1}{p}$, we have 
   \begin{equation*}
      \cO\left(\frac{f(x^0) - f^{\star}}{\epsilon^2}\cdot\det(\mL)^\frac{1}{d}\cdot\left(1 + \frac{\omega_{\mL^{-1}}}{n}\right)\right).
   \end{equation*} 
   From the proof of \Cref{col:dasha:1}, we know that the iteration complexity of {\detdasha} is 
   \begin{equation*}
      \cO\left(\frac{1}{\epsilon^2}\cdot \left(f(x^0) - f^{\star}\right)\left(L + \frac{\omega}{\sqrt{n}} \widehat{L}\right)\right).
   \end{equation*}
   It is easy to see that in this case the two algorithms have the same iteration complexity asymptotically. Notice that the communication complexity is the product of bytes sent per iteration and the number of iterations. {\detdasha} clearly sends less bytes per iteration because it always sent the compressed gradient differences, which means that it has a better communication complexity than {\detmarina}.

\section{\texorpdfstring{Distributed {\detcgd}}{Distributed det-CGD}}\label{sec:dist-detcgd}

This section is a brief summary of the distributed {{\detcgd}} algorithm and its theoretical analysis. 
The details can be found in \citep{li2023det}.
The algorithm follows the standard FL paradigm.
See the pseudocode in \Cref{alg:dist-alg1}.
\begin{algorithm}[H]
   \caption{Distributed {\detcgd}}\label{alg:dist-alg1}
   \begin{algorithmic}[1]
      \STATE {\bfseries Input:} Starting point $x^0$, stepsize matrix $\mD$, number of iterations $K$
      \FOR {$k=0,1,2,\ldots,K-1$}
      \STATE \underline{The devices in parallel:}
      \STATE  sample $\mS_i^k \sim \cS$;
      \STATE  compute $\mS_i^k \nabla f_i(x^{k}) $;
      \STATE  broadcast $\mS_i^k \nabla f_i(x^{k}) $.
      \STATE \underline{The server:}
      \STATE  combines $g^k = \frac{1}{n} \sum_{i=1}^{n}\mS_i^k \nabla f_i(x^{k}) $;
      \STATE  computes $x^{k+1}=x^k- \mD g^k$;
      \STATE  broadcasts $x^{k+1}$.
      \ENDFOR
      \STATE {\bfseries Return:} $x^{K}$
   \end{algorithmic}
\end{algorithm}
  
Below is the main convergence result for the algorithm. 

\begin{theorem}
   \label{thm:dist-alg1}
   Suppose that $f$ is $\mL$-smooth. 
   Under the Assumptions \ref{assmp:1},\ref{assmp:4}, if the stepsize satisfies
      \begin{equation}
      \label{eq:dis-cond1-key}
      \mD\mL\mD \preceq \mD,
   \end{equation}
   then the following convergence bound  is true for the iteration of \Cref{alg:dist-alg1}:
   \begin{equation}\label{eq:thm-dist1}
      \min_{0\leq k\leq K-1}\Exp{\norm{\nabla f(x^k)}_{\frac{\mD}{\det(\mD)^{1/d}}}^2} \leq \frac{2(1 + \frac{\lambda_{\mD}}{n})^K\left(f(x^0) - f^{\star}\right)}{\det(\mD)^{1/d}\, K} + \frac{2\lambda_{\mD} \Delta^{\star}}{\det(\mD)^{1/d} \, n},
   \end{equation}
   where $\Delta^{\star} \eqdef f^{\star} - \frac{1}{n}\sum_{i=1}^{n}f_i^{\star}$ and
   \begin{align*}
      \lambda_{\mD} &\eqdef \max_i \left\{\lambda_{\max}\left(\Exp{\mL_i^\frac{1}{2}\left(\mS_i^k - \mI_d\right)\mD\mL\mD\left(\mS_i^k - \mI_d\right)\mL_i^\frac{1}{2}}\right)\right\}.
   \end{align*}
\end{theorem}

\begin{remark}
   On the right hand side of \eqref{eq:thm-dist1} we observe that increasing $K$ will only reduce the first term, that corresponds to the convergence error. 
   Whereas, the second term, which does not depend on $K$, will remain constant, if the other parameters of the algorithm are fixed. 
   This testifies to the neighborhood phenomenon which we discussed in \Cref{sec:background}.
\end{remark}
\begin{remark}
   \label{cor:dist-cond-conv}
   If the stepsize satisfies the below conditions,
   \begin{equation}
      \label{eq:cond-dist-1}
      \mD\mL\mD \preceq \mD, \quad 
      \lambda_{\mD} \leq \min\brc{\frac{n}{K}, \frac{n\varepsilon^2}{4\Delta^{\star}}\det(\mD)^{1/d}}, \quad
      K \geq \frac{12(f(x^0) - f^{\star})}{\det(\mD)^{1/d}\, \varepsilon^2},
   \end{equation}
   then we obtain $\varepsilon$-stationary point.
\end{remark}
One can see that in the convergence guarantee of {\detcgd} in the distributed case, the result \eqref{eq:thm-dist1} is not variance-reduced. 
Because of this limitation, in order to reach a $\varepsilon$ stationary point, the stepsize condition in \eqref{eq:cond-dist-1} is restrictive.

\section{\texorpdfstring{Extension of {\detcgdtwo} in {\marina} form}{Extension of Det-CGD in MARINA form}}\label{app:detcgd2}
In this section we want to extend {\detcgdtwo} into its variance reduced counterpart in {\marina} form. 

\subsection{\texorpdfstring{Extension of {\detcgdtwo} to its variance reduced counterpart}{Extension of Det-CGD2 to its variance reduced counterpart}}
\begin{algorithm}
   \caption{det-CGD2-VR}
   \label{alg:detCGD2-MARINA}
   \begin{algorithmic}[1]
   \STATE {\bf Input:} starting point $x^0$, stepsize matrix $\mD$, probability $p \in (0, 1]$, number of iterations $K$
   \STATE Initialize $g^0 = \mD\cdot\nabla f(x^0)$
   \FOR {$k=0,1,\ldots,K-1$}
      \STATE Sample $c_k \sim \text{Be}(p)$
      \STATE Broadcast $g^k$ to all workers
      \FOR {$i=1,2,\ldots$ in parallel} 
         \STATE $x^{k+1} = x^k - g^k$
         \STATE Set $g_i^{k+1} = \begin{cases}
               \mD \cdot\nabla f_i(x^{k+1}) &\text{ if } c_k = 1\\
               g^k + \mT_i^k\mD\left(\nabla f_i(x^{k+1}) - \nabla f_i(x^k)\right) &\text{ if } c_k = 0\\
         \end{cases}$
      \ENDFOR
      \STATE $g^{k+1} = \frac{1}{n}\sum_{i=1}^{n} g_i^{k+1}$
   \ENDFOR
   \STATE {\bf Return:} $\tilde{x}^K$ chosen uniformly at random from $\{x^k\}_{k=0}^{K-1}$
   \end{algorithmic}
\end{algorithm}
We call {\detmarina} as the extension of {\detcgdone}, and \Cref{alg:detCGD2-MARINA} as the extension of {\detcgdtwo} due to the difference in the order of applying sketches and stepsize matrices. 
The key difference between {\detcgdone} and {\detcgdtwo} is that in {\detcgdone} the gradient is sketched first and then multiplied by the stepsize, while for {\detcgdtwo}, the gradient is multiplied by the stepsize first after which the product is sketched. The convergence for \Cref{alg:detCGD2-MARINA} can be proved in a similar manner as  \Cref{thm:1:detCGDVR1}.

\begin{theorem}
   \label{thm:det-MARINA2}
   Let Assumptions \ref{assmp:1} and \ref{assmp:4} hold, with the gradient of $f$ being $\mL$-Lipschitz. If the stepsize matrix $\mD \in \bbS^d_{++}$ satisfies
   \begin{equation*}
      \mD^{-1} \succeq \left(\frac{(1-p)\cdot R'(\mD, \cS)}{np} + 1\right)\mL,
   \end{equation*}
   where 
   \begin{equation*}
      R'(\mD, \cS) = \frac{1}{n}\sum_{i=1}^{n}\lambda_{\max}\left(\mD\Exp{\mT_i^k\mD^{-1}\mT_i^k}\mD\mL_i^{\frac{1}{2}} - \mL_i^{\frac{1}{2}}\mD\right)\cdot\lambda_{\max}\left(\mL_i\right)\cdot\lambda_{\max}\left(\mL^{-\frac{1}{2}}\mL_i\mL^{-\frac{1}{2}}\right).
   \end{equation*}
   Then after $K$ iterations of \Cref{alg:detCGD2-MARINA}, we have 
   \begin{equation*}
      \Exp{\norm{\nabla f(\tilde{x}^K)}^2_{\frac{\mD}{\det(\mD)^{1/d}}}} \leq \frac{2\left(f(x^0) - f^{\star}\right)}{\det(\mD)^{1/d}\cdot K}.
   \end{equation*}
   This is to say that in order to reach a $\varepsilon$-stationary point, we require
   \begin{equation*}
      K \geq \frac{2(f(x^0) - f^{\star})}{\det(\mD)^{1/d}\cdot \varepsilon^2}.
   \end{equation*}
\end{theorem}
If we look at the scalar case where $\mD = \gamma\cdot\mI_d$, $\mL_i = L_i \cdot \mI_d$ and $\mL = L \cdot I_d$, then the condition in \Cref{thm:det-MARINA2} reduces to
\begin{equation}
    \label{eq:rdcdcond}
    \frac{(1 - p)\omega L^2}{np} + L - \frac{1}{\gamma} \leq 0.
\end{equation}
Notice that here $\omega = \lambda_{\max}\brr{\Exp{\left(\mT_i^k\right)^2}} - 1$, and we have $L^2 = \frac{1}{n}\sum_{i=1}^{n}L_i^2$, which is due to the relation given in \Cref{ppst:1}. 
This condition coincides with the condition for convergence of $\marina$. 
One may also check that, the update rule in \Cref{alg:detCGD2-MARINA}, is the same as $\marina$ in the scalar case. 
However, the condition given in \Cref{thm:det-MARINA2} is not simpler than \Cref{thm:1:detCGDVR1}, contrary to the single-node case. 
We emphasize that \Cref{alg:detCGD2-MARINA} is not suitable for the federated learning setting where the clients have limited resources. 
In order to perform the update, each client is required to store the stepsize matrix $\mD$ which is of size $d \times d$. 
In the over-parameterized regime, the dataset size is $m \times d$ where $m$ is the number of data samples, and we have $d > m$. 
This means that the stepsize matrix each client needs to store is even larger than the dataset itself, which is unacceptable given the limited resources each client has.

\subsection{\texorpdfstring{Analysis of \Cref{alg:detCGD2-MARINA}}{Analysis of Algorithm~\ref{alg:detCGD2-MARINA}}}
We first present two lemmas which are necessary for the proofs of \Cref{thm:det-MARINA2}. 
\begin{lemma}
   \label{lemma:4:reduced-b2}
   Assume that function $f$ is $\mL$-smooth, and $x^{k+1} = x^k - g^k$, and matrix $\mD \in \bbS^d_{++}$. Then we will have
   \begin{equation}
      \label{eq:lemma:4}
      f(x^{k+1}) \leq f(x^k) - \frac{1}{2}\norm{\nabla f(x^k)}^2_{\mD} + \frac{1}{2}\norm{\mD\cdot\nabla f(x^k) - g^k}^2_{\mD^{-1}} - \frac{1}{2}\norm{x^{k+1} - x^k}^2_{\mD^{-1}-\mL}.
   \end{equation} 
\end{lemma}
This lemma is formulated in a different way from \Cref{lemma:1} on purpose. 
\begin{lemma}
    \label{lemma5:cplx}
    For any sketch matrix $\mT \in \bbS_{+}^d$, vector $t \in \R^d$, matrix $\mD \in \bbS^d_{++}$ and matrix $\mL \in \bbS^d_{++}$, we have
    \begin{equation}
        \label{eq:lemma:5}
        \Exp{\norm{\mT\mD t - \mD t}^2_{\mD^{-1}}} \leq \lambda_{\max}\left(\mL^\frac{1}{2}\mD\Exp{\mT\mD^{-1}\mT}\mD\mL^\frac{1}{2} - \mL^\frac{1}{2}\mD\mL^\frac{1}{2}\right)\norm{t}^2_{\mL^{-1}}.
    \end{equation}
\end{lemma}

\subsection{\texorpdfstring{Proof of \Cref{thm:det-MARINA2}}{Proof of Theorem~\ref{thm:det-MARINA2}}}
We start with \Cref{lemma:4:reduced-b2},
\begin{align}
   \label{eq:thm3-beg}
   \Exp{f(x^{k+1})} &\leq \Exp{f(x^k)} - \Exp{\frac{1}{2}\norm{\nabla f(x^k)}^2_{\mD}} \notag \\ 
   & \quad + \Exp{\frac{1}{2}\norm{\mD\cdot\nabla f(x^k) - g^k}^2_{\mD^{-1}}} - \Exp{\frac{1}{2}\norm{x^{k+1} - x^k}^2_{\mD^{-1}-\mL}}.
\end{align}
Now we do the same as \Cref{thm:1:detCGDVR1} and look at the term $\Exp{\norm{\mD\cdot\nabla f(x^{k+1}) - g^{k+1}}^2_{\mD^{-1}}}$. Recall that $g^k$ here is given by 
\begin{equation*}
   g^{k+1} = \begin{cases}
      \mD\cdot\nabla f(x^{k+1}) & \text{ with probability $p$ } \\
      g^k + \frac{1}{n}\sum_{i=1}^{n}\mT_i^k\mD\left(\nabla f_i(x^{k+1}) - \nabla f_i(x^k)\right) & \text{ with probability $1-p$ }.
   \end{cases}
\end{equation*}
As a result, we have 
\begin{align*}
   & \Exp{\norm{g^{k+1} - \mD\nabla f(x^{k+1})}^2_{\mD^{-1}} \mid x^{k+1}, x^k} \notag \\ 
   & = \Exp{\Exp{\norm{g^{k+1} - \mD\nabla f(x^{k+1})}^2_{\mD^{-1}} \mid x^{k+1}, x^k,c_k}} \notag \\ 
   & = p \cdot\norm{\mD\nabla f(x^{k+1}) - \mD\nabla f(x^{k+1})}^2_{\mD^{-1}} \\
   &\quad + (1-p)\cdot\Exp{\norm{g^k + \frac{1}{n}\sum_{i=1}^{n}\mT_i^k\mD\left(\nabla f_i(x^{k+1}) - \nabla f_i(x^k)\right) - \mD\nabla f(x^{k+1})}^2_{\mD^{-1}} \mid x^{k+1}, x^k} \\
   & = (1-p)\cdot\Exp{\norm{g^k + \frac{1}{n}\sum_{i=1}^{n}\mT_i^k\mD\left(\nabla f_i(x^{k+1}) - \nabla f_i(x^k)\right) - \mD\nabla f(x^{k+1})}^2_{\mD^{-1}} \mid x^{k+1}, x^k}.
\end{align*}
For the sake of presentation, we use $\ExpSub{{k}}{\cdot}$ to denote the conditional expectation $\ExpCond{\cdot}{x_k, x_{k+1}}$ on $x_k, x_{k+1}$.
Using \Cref{lemma:2:var-decomp} with $x = \frac{1}{n}\sum_{i=1}^{n}\mT_i^k\mD\left(\nabla f_i(x^{k+1}) - \nabla f_i(x^k)\right)$, $c = \mD\nabla f(x^{k+1}) - g^k$, we are able to obtain that,
\begin{align*}
   & (1-p)\ExpSub{{k}}{\norm{g^k + \frac{1}{n}\sum_{i=1}^{n}\mT_i^k\mD\left(\nabla f_i(x^{k+1}) - \nabla f_i(x^k)\right) - \mD\nabla f(x^{k+1})}^2_{\mD^{-1}}} \\
   & = (1-p)\ExpSub{{k}}{\norm{\frac{1}{n}\sum_{i=1}^{n}\mT_i^k\mD\left(\nabla f_i(x^{k+1}) - \nabla f_i(x^k)\right) - \mD\left(\nabla f(x^{k+1}) - \nabla f(x^k)\right)}^2_{\mD^{-1}}} \\
   &\quad + (1-p)\norm{g^k - \nabla f(x^k)}^2_{\mD^{-1}} \\
   & = (1-p) \ExpSub{{k}}{\norm{\frac{1}{n}\sum_{i=1}^{n}\left[\mT_i^k\mD\left(\nabla f_i(x^{k+1}) - \nabla f_i(x^k)\right) - \mD\left(\nabla f_i(x^{k+1}) - \nabla f_i(x^k)\right)\right]}^2_{\mD^{-1}}} \\
   &\quad + (1-p)\norm{g^k - \nabla f(x^k)}^2_{\mD^{-1}}.
\end{align*} 
It is not hard to notice that for the sketch matrices we pick, the following identity holds due to the unbiasedness,
\begin{equation*}
   \ExpSub{{k}}{\mT_i^k\mD(\nabla f_i(x^{k+1}) - \nabla f_i(x^k))} = \mD(\nabla f_i(x^{k+1}) - \nabla f_i(x^k)),
\end{equation*}
and any two random vectors in the set $\left\{\mT_i^k\mD(\nabla f_i(x^{k+1}) - \nabla f_i(x^k))\right\}_{i=1}^n$ are independent if $x^{k+1}, x^k$ are fixed. As a result
\begin{align}
   \label{eq:thm3upperbound}
   & \ExpSub{k}{\norm{g^{k+1} - \mD\nabla f(x^{k+1})}^2_{\mD^{-1}}} \notag \\
   &= \frac{1-p}{n^2}\sum_{i=1}^{n}\ExpSub{k}{\norm{\mT_i^k\left(\mD\nabla f_i(x^{k+1}) - \mD\nabla f_i(x^k)\right) - \left(\mD\nabla f_i(x^{k+1}) - \mD\nabla f_i(x^k)\right)}^2_{\mD^{-1}}} \notag \\
   &\quad + (1-p)\cdot\norm{g^k - \mD \nabla f(x^k)}^2_{\mD^{-1}}.
\end{align}
For each term within the summation, we can further upper bound it 
using \Cref{lemma5:cplx}
\begin{align*}
   & \ExpSub{k}{\norm{\mT_i^k\left(\mD\nabla f_i(x^{k+1}) - \mD\nabla f_i(x^k)\right) - \left(\mD\nabla f_i(x^{k+1}) - \mD\nabla f_i(x^k)\right)}^2_{\mD^{-1}}} \\
   &\quad \leq \lambda_{\max}\left( \mL_i^{\frac{1}{2}}\mD\Exp{\mT_i^k\mD^{-1}\mT_i^k}\mD\mL_i^{\frac{1}{2}} - \mL_i^{\frac{1}{2}}\mD\mL_i^{\frac{1}{2}}\right)\norm{\nabla f_i(x^{k+1}) - \nabla f_i(x^k)}^2_{\mL_i^{-1}}\\
   &\quad \leq \lambda_{\max}\left( \mL_i^{\frac{1}{2}}\mD\Exp{\mT_i^k\mD^{-1}\mT_i^k}\mD\mL_i^{\frac{1}{2}} - \mL_i^{\frac{1}{2}}\mD\mL_i^{\frac{1}{2}}\right)\norm{x^{k+1} - x^k}_{\mL_i}^2.
\end{align*}
Where the last inequality is due to \Cref{assmp:4}.
Plugging back into \eqref{eq:thm3upperbound}, we get 
\begin{align*}
   & \ExpSub{k}{\norm{g^{k+1} - \mD\nabla f(x^{k+1})}^2_{\mD^{-1}}} \\
   &\quad \leq \frac{1-p}{n^2}\sum_{i=1}^{n} \lambda_{\max}\left( \mL_i^{\frac{1}{2}}\mD\Exp{\mT_i^k\mD^{-1}\mT_i^k}\mD\mL_i^{\frac{1}{2}} - \mL_i^{\frac{1}{2}}\mD\mL_i^{\frac{1}{2}}\right)\norm{x^{k+1} - x^k}_{\mL_i}^2  \\
   & \qquad + (1 - p)\cdot\norm{g^k - \mD\nabla f(x^k)}^2_{\mD^{-1}}. 
\end{align*}
Applying the replacement trick form the proof of \Cref{thm:1:detCGDVR1}, we obtain
\begin{align*}
   & \ExpSub{k}{\norm{g^{k+1} - \mD\nabla f(x^{k+1})}^2_{\mD^{-1}}} \\
   & \leq \frac{1-p}{n^2}\sum_{i=1}^{n} \lambda_{\max}\left( \mL_i^{\frac{1}{2}}\mD\Exp{\mT_i^k\mD^{-1}\mT_i^k}\mD\mL_i^{\frac{1}{2}} - \mL_i^{\frac{1}{2}}\mD\mL_i^{\frac{1}{2}}\right)  \\ 
   & \quad \times \inner{\mL^\frac{1}{2}\left(x^{k+1} - x^k\right)}{\left(\mL^{-\frac{1}{2}}\mL_i\mL^{-\frac{1}{2}}\right)\cdot\mL^\frac{1}{2}\left(x^{k+1} - x^k\right)}  + (1 - p)\cdot\norm{g^k - \mD\nabla f(x^k)}^2_{\mD^{-1}} \\
   &\leq \frac{1-p}{n^2}\sum_{i=1}^{n} \lambda_{\max}\left( \mL_i^{\frac{1}{2}}\left(\mD\Exp{\mT_i^k\mD^{-1}\mT_i^k}\mD - \mD\right)\mL_i^{\frac{1}{2}}\right)\cdot\lambda_{\max}\left(\mL^{-\frac{1}{2}}\mL_i\mL^{-\frac{1}{2}}\right)\norm{x^{k+1} - x^k}_{\mL}^2 \notag \\
   & \quad + (1 - p)\cdot\norm{g^k - \mD\nabla f(x^k)}^2_{\mD^{-1}}. 
\end{align*}
Applying \Cref{fact:4}, we obtain 
\begin{align*}
   & \ExpSub{k}{\norm{g^{k+1} - \mD\nabla f(x^{k+1})}^2_{\mD^{-1}}}  \\
   & \leq \frac{1-p}{n^2}\sum_{i=1}^{n} \lambda_{\max}\left( \mD\Exp{\mT_i^k\mD^{-1}\mT_i^k}\mD - \mD\right)\lambda_{\max}\left(\mL_i\right)\lambda_{\max}\left(\mL^{-\frac{1}{2}}\mL_i\mL^{-\frac{1}{2}}\right)\norm{x^{k+1} - x^k}_{\mL}^2  \\
   & \quad + (1 - p)\cdot\norm{g^k - \mD\nabla f(x^k)}^2_{\mD^{-1}}. 
\end{align*}
Recalling the definition of $R^{\prime}(\mD,\cS)$, we further simplify it to 
\begin{align*}
   & \ExpSub{k}{\norm{g^{k+1} - \mD\nabla f(x^{k+1})}^2_{\mD^{-1}}} \\
   &\quad\leq \frac{(1-p)\cdot R'(\mD, \cS)}{n}\norm{x^{k+1} - x^k}^2_{\mL} + (1 - p)\cdot\norm{g^k - \mD\nabla f(x^k)}^2_{\mD^{-1}}.
\end{align*}
Taking expectation again and using the tower property, we get 
\begin{align}
   \label{eq:thm:3:upp:6}
   & \Exp{\norm{g^{k+1} - \mD\nabla f(x^{k+1})}^2_{\mD^{-1}}} \\
   &\quad\leq \frac{(1-p)\cdot R'(\mD, \cS)}{n}\Exp{\norm{x^{k+1} - x^k}^2_{\mL}} + (1 - p)\cdot\Exp{\norm{g^k - \mD\nabla f(x^k)}^2_{\mD^{-1}}}.
\end{align}
Construct the Lyapunov function $\Phi_k$ as follows, 
\begin{equation*}
   \Phi_k = f(x^k) - f^{\star} + \frac{1}{2p}\norm{g^k - \mD\nabla f(x^k)}^2_{\mD^{-1}}.
\end{equation*}
Utilizing  \eqref{eq:thm3-beg} and \eqref{eq:thm:3:upp:6}, we are able to get 
\begin{align*}
   \Exp{\Phi_{k+1}} &\leq \Exp{f(x^k) - f^{\star}} - \frac{1}{2}\Exp{\norm{\nabla f(x^k)}^2_{\mD}} \\
   & \qquad + \frac{1}{2}\Exp{\norm{g^k - \mD\nabla f(x^k)}^2_{\mD^{-1}}} - \frac{1}{2}\Exp{\norm{x^{k+1} - x^k}_{\mD^{-1} - L}} \\
   & \qquad + \frac{1}{2p}\cdot\frac{(1-p)R'(\mD, \cS)}{n}\Exp{\norm{x^{k+1} - x^k}^2_{\mL}} + \frac{1-p}{2p}\Exp{\norm{g^k - \mD\nabla f(x^k)}^2_{\mD^{-1}}} \\
   &\quad = \Exp{\Phi_k} - \frac{1}{2}\Exp{\norm{\nabla f(x^k)}^2_{\mD}} \\
   &\qquad + \frac{1}{2}\left(\frac{(1-p)R'(\mD, \cS)}{np}\Exp{\norm{x^{k+1} - x^k}^2_{\mL}} - \Exp{\norm{x^{k+1} - x^k}^2_{\mD^{-1} - \mL}}\right).
\end{align*}
Now, notice that the last term in the above inequality is non-positive as guaranteed by the condition
\begin{equation*}
   \mD^{-1} \succeq \left(\frac{(1-p)R'(\mD, \cS)}{np} + 1\right)\mL.
\end{equation*}
This leads to the recurrence after ignoring the last term,
\begin{equation*}
   \Exp{\Phi_{k+1}} \leq \Exp{\Phi_k} - \frac{1}{2}\Exp{\norm{\nabla f(x^k)}^2_{\mD}}.
\end{equation*}
Unrolling this recurrence, we get 
\begin{equation*}
   \frac{1}{K}\sum_{k=0}^{K-1}\Exp{\norm{\nabla f(x^k)}^2_{\mD}} \leq \frac{2\left(\Exp{\Phi_0 } - \Exp{\Phi_K}\right)}{K}.
\end{equation*}
The left hand side can viewed as average over $\tilde{x}^{K}$, which is drawn uniformly at random from $\{x_k\}_{k=0}^{K-1}$, while the right hand side can be simplified as 
\begin{eqnarray*}
   \frac{2\left(\Exp{\Phi_0 } - \Exp{\Phi_K}\right)}{K} \leq \frac{2\Phi_0}{K} =  \frac{2\left(f(x^0) - f^{\star} + \frac{1}{2p}\norm{g^0 - \nabla f(x^0)}^2_{\mD}\right)}{K}.
\end{eqnarray*}
Recalling that $g^0  = \nabla f(x^0)$ and performing determinant normalization 
as \citet{li2023det}, we get 
\begin{equation*}
   \Exp{\norm{\nabla f(\tilde{x}^K)}^2_{\frac{\mD}{\det(\mD)^{1/d}}}} \leq \frac{2\left(f(x^0) - f^{\star}\right)}{\det(\mD)^{1/d}K}.
\end{equation*}

\section{Proofs of the technical lemmas}

\subsection{\texorpdfstring{Proof of \Cref{lemma:1}}{Proof of Lemma~\ref{lemma:1}}}
   Let $\bar{x}^{k+1} \eqdef x^k - \mD \cdot \nabla f(x^k)$. 
   Since $f$ has a matrix $\mL$-Lipschitz gradient, $f$ is also $\mL$-smooth. 
   From the $\mL$-smoothness of $f$, we have 
   \begin{eqnarray*}
      f(x^{k+1}) &\leq& f(x^k) + \inner{\nabla f(x^k)}{x^{k+1} - x^k} + \frac{1}{2}\inner{x^{k+1}-x^k}{\mL(x^{k+1} - x^k)} \\
      &=& f(x^k) + \inner{\nabla f(x^k) - g^k}{x^{k+1} - x^k} + \inner{g^k}{x^{k+1} - x^k} + \frac{1}{2}\inner{x^{k+1} - x^k}{\mL(x^{k+1} - x^k)}.
   \end{eqnarray*}
   We can merge the last two terms and obtain, 
   \begin{eqnarray*}
      f(x^{k+1}) &\leq& f(x^k) + \inner{\nabla f(x^k) - g^k}{-\mD \cdot g^k} - \inner{x^{k+1} - x^k}{\mD^{-1}(x^{k+1} - x^k)} \\
      && + \frac{1}{2}\inner{x^{k+1} - x^k}{\mL(x^{k+1} - x^k)} \\
      &=& f(x^k) + \inner{\nabla f(x^k) - g^k}{-\mD \cdot g^k} - \inner{x^{k+1}-x^k}{\left(\mD^{-1} - \frac{1}{2}\mL\right)(x^{k+1} - x^k)}.
   \end{eqnarray*}
   We add and subtract $\inner{\nabla f(x^k) - g^k}{\mD \cdot g^k}$,
   \begin{eqnarray*}
      f(x^{k+1}) &\leq& f(x^k) + \inner{\nabla f(x^k) - g^k}{\mD\left(\nabla f(x^k) - g^k\right)} - \inner{\nabla f(x^k) - g^k}{\mD\cdot\nabla f(x^k)} \\
      && - \inner{x^{k+1}-x^k}{\left(\mD^{-1} - \frac{1}{2}\mL\right)(x^{k+1} - x^k)} \\
      &=& f(x^k) + \norm{\nabla f(x^k) - g^k}^2_{\mD} - \inner{x^{k+1} - \bar{x}^{k+1}}{\mD^{-1}\left(x^k - \bar{x}^{k+1}\right)} \\
      && - \inner{x^{k+1}-x^k}{\left(\mD^{-1} - \frac{1}{2}\mL\right)(x^{k+1} - x^k)}.
   \end{eqnarray*}
   Decomposing the term $\inner{x^{k+1} - \bar{x}^{k+1}}{\mD^{-1}\left(x^k - \bar{x}^{k+1}\right)}$, we get  
   \begin{eqnarray*} 
      f(x^{k+1}) &\leq& f(x^k) + \norm{\nabla f(x^k) - g^k}^2_{\mD} - \inner{x^{k+1}-x^k}{\left(\mD^{-1} - \frac{1}{2}\mL\right)(x^{k+1} - x^k)} \\
      && - \frac{1}{2}\left(\norm{x^{k+1} - \bar{x}^{k+1}}^2_{\mD^{-1}} + \norm{x^k - \bar{x}^{k+1}}^2_{\mD^{-1}} - \norm{x^{k+1}-x^k}^2_{\mD^{-1}}\right).
   \end{eqnarray*}
   Plugging in the definition of $x^{k+1}, \bar{x}^{k+1}$, we get 
   \begin{eqnarray*} 
      f(x^{k+1}) &\leq& f(x^k) + \norm{\nabla f(x^k) - g^k}^2_{\mD} - \norm{x^{k+1} - x^k}^2_{\mD^{-1}-\frac{1}{2}\mL} \\
      && - \frac{1}{2}\left(\norm{\mD(\nabla f(x^k) - g^k)}^2_{\mD^{-1}} + \norm{\mD\cdot\nabla f(x^k)}^2_{\mD^{-1}} - \norm{x^{k+1}-x^k}^2_{\mD^{-1}}\right) \\
      &=& f(x^k) + \norm{\nabla f(x^k) - g^k}^2_{\mD} - \norm{x^{k+1} - x^k}^2_{\mD^{-1}-\frac{1}{2}\mL} \\
      && - \frac{1}{2}\left(\norm{\nabla f(x^k) - g^k}^2_{\mD} + \norm{\nabla f(x^k)}^2_{\mD} - \norm{x^{k+1}-x^k}^2_{\mD^{-1}}\right).
   \end{eqnarray*}
   Rearranging terms we get, 
   \begin{eqnarray*} 
      f(x^{k+1}) &\leq& f(x^k) - \frac{1}{2}\norm{\nabla f(x^k)}_{\mD}^2 + \frac{1}{2}\norm{g^k - \nabla f(x^k)}_{\mD}^2 - \norm{x^{k+1} - x^k}^2_{\mD^{-1} - \frac{1}{2}\mL} + \frac{1}{2}\norm{x^{k+1}-x^k}^2_{\mD^{-1}} \\
      &=& f(x^k) - \frac{1}{2}\norm{\nabla f(x^k)}_{\mD}^2 + \frac{1}{2}\norm{g^k - \nabla f(x^k)}_{\mD}^2 - \frac{1}{2}\norm{x^{k+1} - x^k}_{\mD^{-1} - \mL}.
   \end{eqnarray*}

\subsection{\texorpdfstring{Proof of \Cref{lemma:3}}{Proof of Lemma~\ref{lemma:3}}}
   The definition of the weighted norm yields
   \begin{eqnarray*}
      \Exp{\norm{\mS t - t}^2_{\mD}} &=& \Exp{\inner{t}{\left(\mS-\mI_d\right)\mD\left(\mS -\mI_d\right)t}} \\
      &=& \inner{t}{\Exp{(\mS-\mI_d)\mD(\mS-\mI_d)}t} \\
      &=& \inner{t}{\mL^{-\frac{1}{2}}\cdot\Exp{\mL^{\frac{1}{2}}(\mS-\mI_d)\mD(\mS-\mI_d)\mL^{\frac{1}{2}}}\cdot\mL^{-\frac{1}{2}}t} \\
      &=& \inner{\mL^{-\frac{1}{2}}t}{\Exp{\mL^{\frac{1}{2}}(\mS-\mI_d)\mD(\mS-\mI_d)\mL^{\frac{1}{2}}}\cdot\mL^{-\frac{1}{2}}t} \\
      &\leq&\lambda_{\max}\left(\Exp{\mL^{\frac{1}{2}}(\mS-\mI_d)\mD(\mS-\mI_d)\mL^{\frac{1}{2}}}\right)\norm{\mL^{-\frac{1}{2}}t}^2 \\
      &=& \lambda_{\max}\left(\mL^\frac{1}{2}\left(\Exp{\mS\mD\mS} - \mD\right)\mL^\frac{1}{2}\right) \cdot \norm{t}^2_{\mL^{-1}}.
   \end{eqnarray*}

\subsection{\texorpdfstring{Proof of \Cref{dasha:tech-lemma-recur-1}}{Proof of Lemma~\ref{dasha:tech-lemma-recur-1}}}

   Throughout the following proof, we denote $\ExpS{\cdot}$ as taking expectation with respect to the randomness contained within the sketch sampled from distribution $\cS$. We estimate the term $\ExpS{\norm{g^{k+1} - h^{k+1}}^2_{\mD}}$ in order to construct the Lyapunov function. For $\ExpS{\norm{g^{k+1} - h^{k+1}}^2_{\mD}}$, we have 
   \begin{align*}
      \ExpS{\norm{g^{k+1} - h^{k+1}}^2_{\mD}} &= \ExpS{\norm{g^k + \frac{1}{n}\sum_{i=1}^{n}m_i^{k+1} - h^{k+1}}^2_{\mD}} \\
      & = \ExpS{\norm{g^k + \frac{1}{n}\sum_{i=1}^{n}\mS_i^k\left(h_i^{k+1} - h_i^k - a(g_i^k - h_i^k)\right) - h^{k+1}}^2_{\mD}} \\
      % & = \ExpS{\norm{\frac{1}{n}\sum_{i=1}^{n}\mS_i^k\left(h_i^{k+1} - h_i^k - a(g_i^k - h_i^k)\right) - \left(h^{k+1} - h^k - a(g^k - h^k)\right) + (1-a)(h^k - g^k)}^2_{\mD}}.
   \end{align*}
   Using \Cref{fact:2}, we obtain
   \begin{align*}
      & \ExpS{\norm{g^{k+1} - h^{k+1}}^2_{\mD}} \\
      & = \ExpS{\norm{\frac{1}{n}\sum_{i=1}^{n}\mS_i^k\left(h_i^{k+1} - h_i^k - a(g_i^k - h_i^k)\right) - \left(h^{k+1} - h^k - a(g^k - h^k)\right)}^2_{\mD}} \\
      &\quad + (1 - a)^2\norm{h^k - g^k}^2_{\mD} \\
      & = \ExpS{\norm{\frac{1}{n}\sum_{i=1}^{n}\mS_i^k\left(h_i^{k+1} - h_i^k - a(g_i^k - h_i^k)\right) - \frac{1}{n}\sum_{i=1}^{n}\left(h_i^{k+1} - h_i^k - a(g_i^k - h_i^k)\right)}^2_{\mD}} \\ 
      &\quad + (1 - a)^2\norm{h^k - g^k}^2_{\mD} \\
      & = \frac{1}{n^2}\sum_{i=1}^{n}\ExpS{\norm{\mS_i^k\left(h_i^{k+1} - h_i^k - a(g_i^k - h_i^k)\right) - \left(h_i^{k+1} - h_i^k - a(g_i^k - h_i^k)\right)}^2_{\mD}} \\
      &\quad + (1 - a)^2\norm{h^k - g^k}^2_{\mD}.
   \end{align*}
   Here, the last identity is obtained from the unbiasedness of the sketches:
   \begin{equation*}
      \ExpS{\mS_i^k\left(h_i^{k+1} - h_i^k - a(g_i^k - h_i^k)\right)} = h_i^{k+1} - h_i^k - a(g_i^k - h_i^k).
   \end{equation*}
   We can further use \Cref{lemma:3}, and obtain
   \begin{align*}
      &\ExpS{\norm{g^{k+1} - h^{k+1}}^2_{\mD}} \\
      &\quad \leq \frac{1}{n^2}\sum_{i=1}^{n}\lambda_{\max}\left(\mD^{-\frac{1}{2}}\left(\Exp{\mS_i^k\mD\mS_i^k} - \mD\right)\mD^{-\frac{1}{2}}\right)\norm{h_i^{k+1} - h_i - a(g_i^k - h_i^k)}^2_{\mD}  \\
      & \qquad + (1-a)^2\norm{g^k - h^k}^2_{\mD} \\
      &\quad \leq \frac{1}{n^2}\sum_{i=1}^{n}\lambda_{\max}\left(\mD^{-1}\right)\cdot\lambda_{\max}\left(\Exp{\mS_i^k\mD\mS_i^k} - \mD\right)\norm{h_i^{k+1} - h_i^k - a(g_i^k - h_i^k)}^2_{\mD} \\
      & \qquad + (1 - a)^2\norm{g^k - h^k}^2_{\mD}.
   \end{align*}
   We can rewrite the above bound, after applying Jensen's inequality as 
   \begin{align*}
      \ExpS{\norm{g^{k+1} - h^{k+1}}^2_{\mD}} &\leq \frac{2\Lambda_{\mD, \cS}\cdot\lambda_{\max}\left(\mD^{-1}\right)}{n^2}\sum_{i=1}^{n}\norm{h_i^{k+1} - h_i^k}^2_{\mD} \\
      &\qquad + \frac{2a^2\Lambda_{\mD, \cS}\cdot\lambda_{\max}\left(\mD^{-1}\right)}{n^2}\sum_{i=1}^{n}\norm{g_i^k - h_i^k}^2_{\mD} \\
      &\qquad + (1 - a)^2\norm{g^k - h^k}^2_{\mD}.
   \end{align*}
   Notice that we have 
   \begin{equation*}
      \norm{h_i^{k+1} - h_i^k}^2_{\mD} \leq \lambda_{\max}\left(\mD\right)\cdot\lambda_{\max}\left(\mL_i\right)\cdot \norm{h_i^{k+1} - h_i^k}^2_{\mL_i^{-1}}. 
   \end{equation*}
   Thus, it is not hard to see that 
   \begin{align*}
      \ExpS{\norm{g^{k+1} - h^{k+1}}^2_{\mD}} &\leq \frac{2\Lambda_{\mD, \cS}\cdot\lambda_{\max}\left(\mD^{-1}\right)\cdot\lambda_{\max}\left(\mD\right)}{n^2}\sum_{i=1}^{n}\lambda_{\max}\left(\mL_i\right)\norm{h_i^{k+1} - h_i^k}^2_{\mL_i^{-1}} \\
      &\qquad + \frac{2a^2\Lambda_{\mD, \cS}\cdot\lambda_{\max}\left(\mD^{-1}\right)}{n^2}\sum_{i=1}^{n}\norm{g_i^k - h_i^k}^2_{\mD} \\ 
      &\qquad + (1 - a)^2\norm{g^k - h^k}^2_{\mD}.
   \end{align*}
   We obtain the inequality in the lemma after taking expectation again and applying tower property.

\subsection{\texorpdfstring{Proof of \Cref{dasha:tech-lemma-recur-2}}{Proof of Lemma~\ref{dasha:tech-lemma-recur-2}}}

   Similarly, we then try to bound the terms $\ExpS{\norm{g_i^{k+1} - h_i^{k+1}}^2_{\mD}}$. We start with 
   \begin{align*}
      &\ExpS{\norm{g_i^{k+1} - h_i^{k+1}}^2_{\mD}} \notag \\
      & = \ExpS{\norm{g_i^k + \mS_i^k\left(h_i^{k+1} - h_i^k - a(g_i^k - h_i^k)\right) - h_i^{k+1}}^2_{\mD}} \notag \\
      & = \ExpS{\norm{\mS_i^k\left(h_i^{k+1} - h_i^k - a(g_i^k - h_i^k)\right) - \left(h_i^{k+1} - h_i^k - a(g_i^k - h_i^k)\right) + (1 - a)(h_i^k - g_i^k)}^2_{\mD}}.
   \end{align*}
   Using \Cref{fact:2},
   \begin{align*}
      &\ExpS{\norm{g_i^{k+1} - h_i^{k+1}}^2_{\mD}} \\
      &= \ExpS{\norm{\mS_i^k\left(h_i^{k+1} - h_i^k - a(g_i^k - h_i^k)\right) - \left(h_i^{k+1} - h_i^k - a(g_i^k - h_i^k)\right)}^2_{\mD}} \notag \\
      &\quad + (1 - a)^2\norm{h^k_i  - g_i^k}^2_{\mD}.
   \end{align*}
   Using \Cref{lemma:3}
   \begin{align*}
      &\ExpS{\norm{g_i^{k+1} - h_i^{k+1}}^2_{\mD}} \notag \\
      & \overset{\eqref{dasha:eq:tech-lemma-sketch-comp}}{\leq} \lambda_{\max}\left(\mD^{-\frac{1}{2}}\left(\Exp{\mS_i^k\mD\mS_i^k} - \mD\right)\mD^{-\frac{1}{2}}\right)\norm{h_i^{k+1} - h_i^k - a(g_i^k - h_i^k)}^2_{\mD} \\
      &\quad + (1-a)^2\norm{g_i^k - h_i^k}^2_{\mD} \notag \\
      & \leq \lambda_{\max}\left(\mD^{-1}\right)\cdot\Lambda_{\mD, \cS}\norm{h_i^{k+1} - h_i^k - a(g_i^k - h_i^k)}^2_{\mD} + (1-a)^2\norm{g_i^k - h_i^k}^2_{\mD} \notag \\
      & \leq 2\lambda_{\max}\left(\mD^{-1}\right)\cdot\Lambda_{\mD, \cS}\norm{h_i^{k+1} - h_i^k}^2_{\mD} + 2a^2\lambda_{\max}\left(\mD^{-1}\right)\cdot\Lambda_{\mD, \cS}\norm{g_i^k - h_i^k}^2_{\mD} \\
      &\quad + (1 - a)^2\norm{g_i^k - h_i^k}^2_{\mD} \notag \\
      & \leq 2\lambda_{\max}\left(\mD^{-1}\right)\cdot\lambda_{\max}\left(\mD\right)\cdot\Lambda_{\mD, \cS}\cdot\lambda_{\max}\left(\mL_i\right)\cdot\norm{h_i^{k+1} - h_i^k}^2_{\mL_i^{-1}} \notag \\
      &\quad + 2a^2\lambda_{\max}\left(\mD^{-1}\right)\cdot\Lambda_{\mD, \cS}\norm{g_i^k - h_i^k}^2_{\mD} + (1 - a)^2\norm{g_i^k - h_i^k}^2_{\mD} \notag \\
      & = \left(2a^2\lambda_{\max}\left(\mD^{-1}\right)\cdot\Lambda_{\mD, \cS} + (1-a)^2\right)\norm{g_i^k - h_i^k}^2_{\mD} \\
      &\quad + 2\lambda_{\max}\left(\mD^{-1}\right)\cdot\lambda_{\max}\left(\mD\right)\cdot\Lambda_{\mD, \cS}\cdot\lambda_{\max}\left(\mL_i\right)\cdot\norm{h_i^{k+1} - h_i^k}^2_{\mL_i^{-1}}.
   \end{align*}
   Taking expectation again, and using tower property, we are able to obtain,
   \begin{align*}
      &\Exp{\norm{g_i^{k+1} - h_i^{k+1}}^2_{\mD}} \\
      &\quad \leq \left(2a^2\lambda_{\max}\left(\mD^{-1}\right)\cdot\Lambda_{\mD, \cS} + (1-a)^2\right)\Exp{\norm{g_i^k - h_i^k}^2_{\mD}} \\
      &\qquad + 2\lambda_{\max}\left(\mD^{-1}\right)\cdot\lambda_{\max}\left(\mD\right)\cdot\Lambda_{\mD, \cS}\cdot\lambda_{\max}\left(\mL_i\right)\cdot\Exp{\norm{h_i^{k+1} - h_i^k}^2_{\mL_i^{-1}}}.
   \end{align*}

\subsection{\texorpdfstring{Proof of \Cref{lemma:4:reduced-b2}}{Proof of Lemma~\ref{lemma:4:reduced-b2}}}
   From \Cref{ppst:1}, we know that the objective is $\mL$-smooth. 
   Let $\bar{x}^{k+1} = x^k - \mD\cdot \nabla f(x^k)$, then  $\mL$-smoothness yields
   \begin{eqnarray*}
      f(x^{k+1}) &\leq& f(x^k) + \inner{\nabla f(x^k)}{x^{k+1} - x^k} + \frac{1}{2}\inner{x^{k+1} - x^k}{\mL(x^{k+1} - x^k)} \notag \\
      &=& f(x^k) + \inner{\nabla f(x^k) - \mD^{-1}\cdot g^k}{x^{k+1} - x^k} + \inner{\mD^{-1}\cdot g^k}{x^{k+1} - x^k} \notag \\
      && \qquad + \frac{1}{2}\inner{x^{k+1} - x^k}{\mL(x^{k+1} - x^k)} \notag \\
      &=& f(x^k) + \inner{\nabla f(x^k) - \mD^{-1}\cdot g^k}{-g^k} - \inner{x^{k+1} - x^k}{\mD^{-1}(x^{k+1} - x^k)} \notag \\
      && \qquad + \frac{1}{2}\inner{x^{k+1} - x^k}{\mL(x^{k+1} - x^k)}. 
   \end{eqnarray*}
   Simplifying the above inner-products we have, 
   \begin{eqnarray*}
      f(x^{k+1}) &\leq& f(x^k) + \inner{\nabla f(x^k) - \mD^{-1}\cdot g^k}{-g^k} - \inner{x^{k+1} - x^k}{\left(\mD^{-1} - \frac{1}{2}\mL\right)(x^{k+1} - x^k)}. 
   \end{eqnarray*}
   We then add and subtract $\inner{\nabla f(x^k) - \mD^{-1}\cdot g^k}{\mD\cdot\nabla f(x^k)}$, which 
   \begin{eqnarray*}
      f(x^{k+1}) &\leq& f(x^k) + \inner{\nabla f(x^k) - \mD^{-1}\cdot g^k}{\mD\cdot\nabla f(x^k) -g^k} - \inner{\nabla f(x^k) - \mD^{-1}\cdot g^k}{\mD\cdot\nabla f(x^k)} \notag \\
      && \qquad - \inner{x^{k+1} - x^k}{\left(\mD^{-1}-\frac{1}{2}\mL\right)(x^{k+1} - x^k)}  \\
      &=& f(x^k) + \norm{\nabla f(x^k) - \mD^{-1}\cdot g^k}^2_{\mD} - \inner{\mD^{-1}(x^{k+1} - \bar{x}^{k+1})}{x^k - \bar{x}^{k+1}} \\
      && \qquad - \inner{x^{k+1} - x^k}{\left(\mD^{-1} - \frac{1}{2}\mL\right)(x^{k+1} - x^k)}. 
   \end{eqnarray*}
   Decomposing the inner product term we deduce, 
   \begin{eqnarray*}
      f(x^{k+1}) &\leq& f(x^k) + \norm{\mD^{-1}\left(\mD\cdot\nabla f(x^k) - g^k\right)}^2_{\mD} - \inner{x^{k+1} - x^k}{\left(\mD^{-1} - \frac{1}{2}\mL\right)(x^{k+1} - x^k)} \notag \\
      && \qquad - \frac{1}{2}\left(\norm{x^{k+1} - \bar{x}^{k+1}}^2_{\mD^{-1}} + \norm{x^k - \bar{x}^{k+1}}^2_{\mD^{-1}} - \norm{x^{k+1} - x^k}^2_{\mD^{-1}}\right) \notag \\
      &=& f(x^k) + \norm{\mD\cdot\nabla f(x^k) - g^k}^2_{\mD^{-1}} - \norm{x^{k+1} - x^k}^2_{\mD^{-1} - \frac{1}{2}\mL}  \notag \\
      && \qquad -\frac{1}{2}\left(\norm{\mD\cdot\nabla f(x^k) - g^k}^2_{\mD^{-1}} + \norm{\mD\cdot\nabla f(x^k)}^2_{\mD^{-1}} - \norm{x^{k+1} - x^k}^2_{\mD^{-1}}\right). 
   \end{eqnarray*}
   Therefore, 
   \begin{eqnarray*}
      f(x^{k+1}) &\leq& f(x^k) + \frac{1}{2}\norm{\mD\nabla f(x^k) - g^k}^2_{\mD^{-1}} - \frac{1}{2}\norm{\nabla f(x^k)}^2_{\mD} - \frac{1}{2}\norm{x^{k+1} - x^k}^2_{\mD^{-1} - \mL}.
   \end{eqnarray*}

\subsection{\texorpdfstring{Proof of \Cref{lemma5:cplx}}{Proof of Lemma~\ref{lemma5:cplx}}}
   We start with 
   \begin{eqnarray*}
      \Exp{\norm{\mT\mD t - \mD t}^2_{\mD^{-1}}} &=& \Exp{\norm{(\mT - \mI_d)\mD t}^2_{\mD^{-1}}} \\
      &=& \inner{t}{\Exp{\mD(\mT-\mI_d)\mD^{-1}(\mT-\mI_d)\mD}\cdot t} \\
      &=& \inner{t}{\mD\left(\Exp{\mT\mD^{-1}\mT} - \mD^{-1}\right)\mD\cdot t} \\
      &=&\inner{\mL^{-\frac{1}{2}}t}{\mL^\frac{1}{2}\mD\left(\Exp{\mT\mD^{-1}\mT} - \mD^{-1}\right)\mD\mL^\frac{1}{2}\cdot\mL^{-\frac{1}{2}} t} \\
      &\leq& \lambda_{\max}\left(\mL^\frac{1}{2}\mD\Exp{\mT\mD^{-1}\mT}\mD\mL^\frac{1}{2} - \mL^\frac{1}{2}\mD\mL^\frac{1}{2}\right)\cdot\norm{\mL^{-\frac{1}{2}}t}^2 \\
      &=& \lambda_{\max}\left(\mL^\frac{1}{2}\mD\Exp{\mT\mD^{-1}\mT}\mD\mL^\frac{1}{2} - \mL^\frac{1}{2}\mD\mL^\frac{1}{2}\right)\cdot\norm{t}^2_{\mL^{-1}}
   \end{eqnarray*}
   This completes the proof.

\section{Experiments}
\label{sec:experiments}
In this section, we conduct numerical experiments to back up the theoretical results for {\detmarina} and {\detdasha}. 
The code for the experiments can be found in \url{https://anonymous.4open.science/r/detCGD-VR-Code-865B}. 
All the codes for the experiments are written in Python 3.11 with NumPy and SciPy package. 
The code was run on a machine with AMD Ryzen 9 5900HX Radeon Graphics @ 3.3 GHz and 8 cores 16 threads.
The datasets in LibSVM are typically non-IID real world datasets, and it is randomly distributed across all the clients.

\subsection{The setting}
We first state the experiment setting. 
We are interested in the following logistic regression problem with a non-convex regularizer. The objective is given as
\begin{equation*}
   f(x) = \frac{1}{n}\sum_{i=1}^{n}f_i(x); \qquad f_i(x) = \frac{1}{m_i}\sum_{j=1}^{m_i}\log\left(1 + e^{-b_{i, j}\cdot\inner{a_{i, j}}{x}}\right) + \lambda\cdot\sum_{t=1}^{d}\frac{x_t^2}{1 + x_t^2},
\end{equation*}
where $x\in\R^d$ is the model, $\left(a_{i, j}, b_{i, j}\right) \in \R^d \times \left\{-1, 1\right\}$ is one data point in the dataset of client $i$ whose size is $m_i$. 
The constant $\lambda > 0$ is the coefficient of the regularizer. 
Larger $\lambda$ means the model is more regular. 
For each function $f_i$, its Hessian can be upper bounded by  
\begin{eqnarray*}
   \mL_i = \frac{1}{m_i}\sum_{i=1}^{m_i}\frac{a_ia_i^{\top}}{4} + 2\lambda\cdot\mI_d;
\end{eqnarray*}
and, therefore, the Hessian of $f$ is bounded by
\begin{eqnarray*}
   \mL = \frac{1}{\sum_{i=1}^{n}m_i}\sum_{i=1}^{n}\sum_{j=1}^{m_i}\frac{a_ia_i^{\top}}{4} + 2\lambda\cdot\mI_d.
\end{eqnarray*}
Due to \Cref{ppst:bounded:Hessian}, it immediately follows that $f_i$ and $f$ satisfy \Cref{assmp:3} with $\mL_i \in \bbS^d_{++}$ and $\mL \in \bbS^d_{++}$, respectively. 

In the following subsections, we perform several numerical experiments comparing the performance of {\dcgd}, {\detcgd}, {\marina}, {\dasha}, {\detmarina} and {\detdasha}. 
The datasets we used are from the LibSVM repository \citep{chang2011libsvm}.  

\subsection{Comparison of all the methods}

In this section, we present several plots which compare all relevant methods to the {\detmarina} and {\detdasha}. 
The methods are the following: {\it (i)} {\dcgd} with scalar stepsize $\gamma_2$, \emph{(ii)} {\detcgd} with matrix stepsize $\mD_3^*$, \emph{(iii)} {\marina} with  scalar stepsize $\gamma_1$, \emph{(iv)} {\dasha} with scalar stepsize $\gamma_4$, \emph{(v)} {\detmarina} with $\mD^*_{\mL^{-1}}$, \emph{(vi)}{\detdasha} with $\mD^{**}_{\mL^{-1}}$. 
Throughout the experiment, $\varepsilon =0.01$, and $\lambda = 0.9$, we are using the same Rand-$\tau$ sketch for all the algorithms, and we run all the algorithms for a fixed number of iteration $K=10000$. 
\begin{figure}[t]
\centering
   \subfigure{
   \begin{minipage}[t]{0.98\textwidth}
      \includegraphics[width=0.48\textwidth]{./figure/Exp_dasha_7_client_150_rand_60_lam_0.9_prob_0.5a1a.txt.pdf}
      \includegraphics[width=0.48\textwidth]{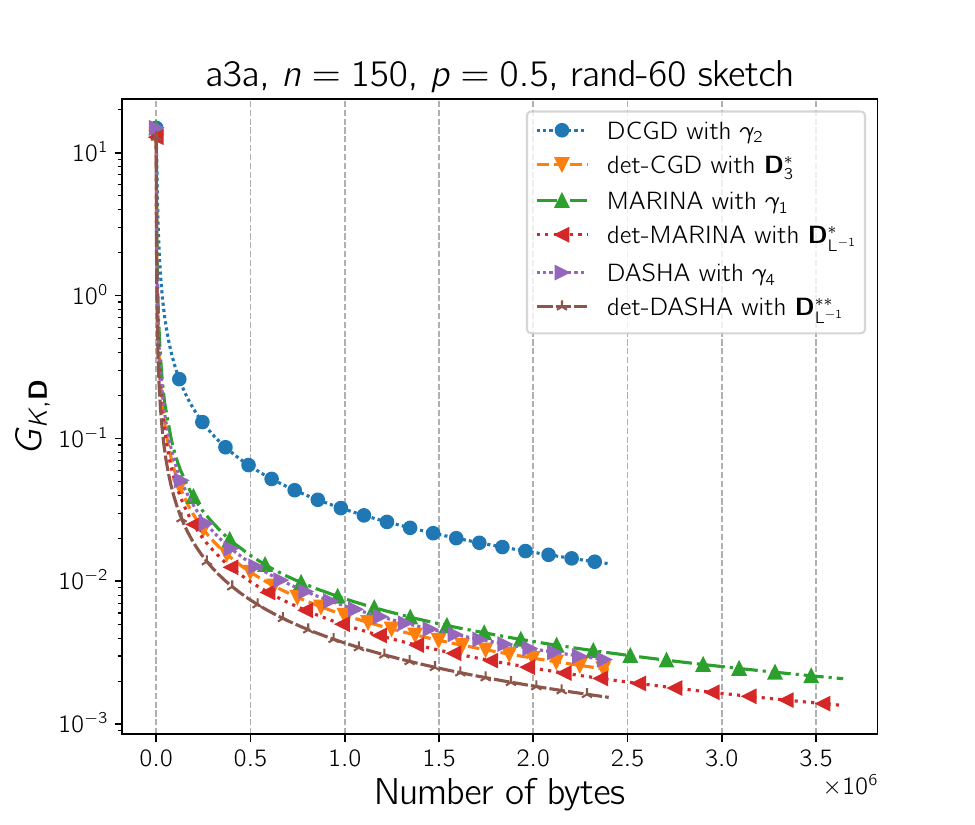}
   \end{minipage}
   }
   \subfigure{
   \begin{minipage}[t]{0.98\textwidth}
      \includegraphics[width=0.48\textwidth]{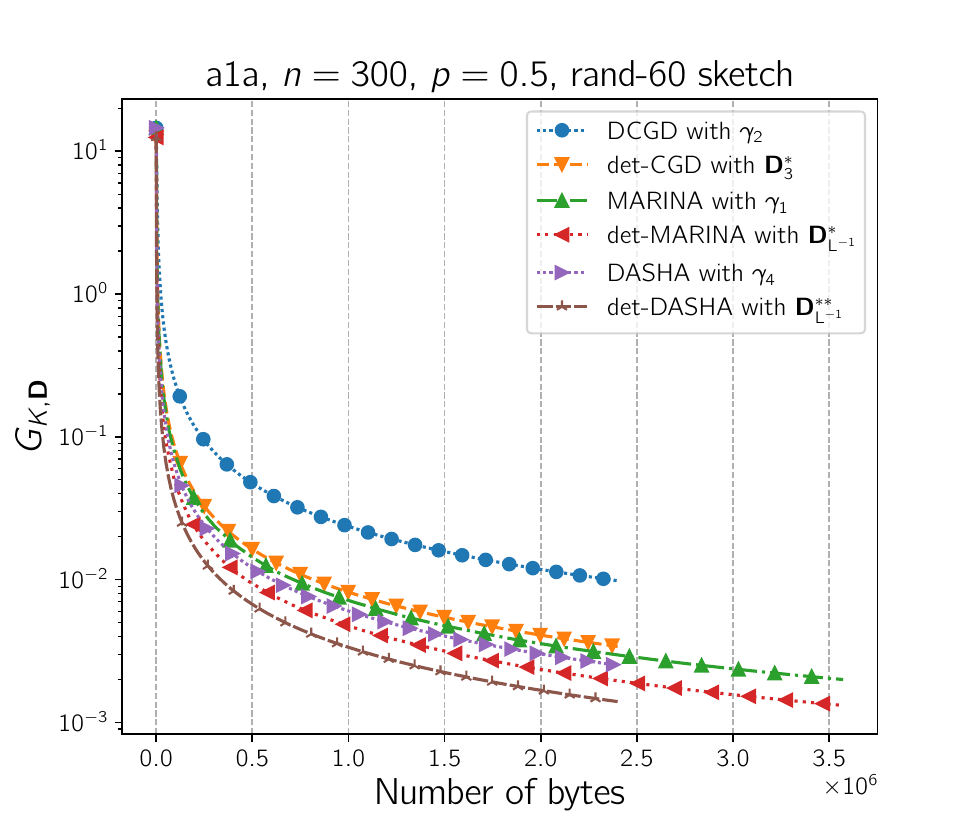}
      \includegraphics[width=0.48\textwidth]{./figure/Exp_dasha_7_client_300_rand_60_lam_0.9_prob_0.5a3a.txt.pdf}
   \end{minipage}
   }
   \caption{Comparison of {\dcgd} with optimal scalar stepsize, {\detcgd} with matrix stepsize $\mD^*_3$, {\marina} with optimal scalar stepsize, {\dasha} with optimal scalar stepsize, {\detmarina} with optimal stepsize $\mD^{*}_{\mL^{-1}}$ and {\detdasha} with optimal stepsize $\mD^{**}_{\mL^{-1}}$. Throughout the experiment, we are using Rand-$\tau$ sketch with $\tau = 60$, and each algorithm is run for a fixed number of iterations $K=10000$. The momentum of {\dasha} is set as $\nicefrac{1}{2\omega + 1}$ and {\detdasha} is $\nicefrac{1}{2\omega_{\mD} + 1}$. The notation $n$ in the title stands for the number of clients in each case, and $p$ stands for the probability used by {\marina} and {\detmarina}.}
   \label{fig:appendix-experiment-7-dasha}
\end{figure}

It can be seen in \Cref{fig:appendix-experiment-7-dasha}, the performance in terms of communication complexity of {\detdasha} and {\detmarina} is better than their scalar counterpart {\dasha} and {\marina} respectively. 
This validates the efficiency of using a matrix stepsize over a scalar stepsize. 
Furthermore, we notice that {\detdasha} and {\detmarina} have better communication complexity in this case, compared to {\detcgd}. 
In addition, we observe variance reduction. 

Notice that the optimal stepsizes of {\detcgd} and {\dcgd} require information of function value differences at 
$x^{\star}$. 
Furthermore, the stepsizes are also constrained by the number of iterations $K$ and the error $\varepsilon^2$. 
Meanwhile, for the variance reduced methods, we do not require such considerations, which is much more practical in general. 

\subsection{\texorpdfstring{Improvements over {\marina}}{Improvements over MARINA}}
\label{sec:F.1}
The purpose of this experiment is to compare the iteration complexity of {\marina}, with {\detmarina} using Rand-$\tau$ sketches, thus showing improvements of {\detmarina} upon {\marina}. 
Using Theorem C.1 from \citep{gorbunov2021marina}, we deduce the optimal stepsize for {\marina}, is 
\begin{equation}
   \label{eq:ss-marina}
   \gamma_1 = \frac{1}{L\left(1 + \sqrt{\frac{(1 - p)\omega}{pn}}\right)},
\end{equation}
where $\omega$ is the quantization coefficient. 
In particular,  for the Rand-$\tau$ compressor $\omega = \frac{d}{\tau} - 1$. 
For the full definition see Section 1.3 of \citep{gorbunov2021marina}.
 The stepsize for {\detmarina} is determined through \Cref{ppst:optimal-D-var}. 
 We use the notation $\mD_{\mW}^*$ to denote the optimal stepsize for each choice of $\mW$, here we list some of the optimal stepsizes for different $\mW$, which are used in the experiment section. 
 We have 
\begin{eqnarray}
   \label{eq:var-D-opt}
   \mD_{\mI_d}^* &=& \frac{2}{1 + \sqrt{1 + 4\alpha\beta\frac{1}{\lambda_{\max}\left(\mL\right)}\cdot\omega}}\cdot \frac{\mI_d}{\lambda_{\max}(\mL)}, \notag \\
   \mD^*_{\mL^{-1}} &=& \frac{2}{1 + \sqrt{1 + 4\alpha\beta\cdot\lambda_{\max}\left(\Exp{\mS_i^k\mL^{-1}\mS_i^k} - \mL^{-1}\right)}} \cdot \mL^{-1}, \notag \\
   \mD^*_{\diag^{-1}(\mL)} &=& \frac{2}{1 + \sqrt{1 + 4\alpha\beta\cdot\lambda_{\max}\left(\Exp{\mS_i^k\diag^{-1}\left(\mL\right)\mS_i^k} - \diag^{-1}\left(\mL\right)\right)}} \cdot \diag^{-1}\left(\mL\right).
\end{eqnarray}
In this experiment, we aim to compare {\detmarina} with stepsize $\mD_{\mL^{-1}}^*$ to the standard \marina\, with the optimal scalar stepsize. 
Rand-$\tau$ compressor is used in the comparison. 
Throughout the experiments, $\lambda$ is fixed at $0.3$. 
We set the $x$-axis to be the number of iterations, while $y$-axis to be the expectation of the corresponding matrix norm of the gradient of the function, which is defined as 
\begin{equation}
   \label{eq:def-G--KD}
   G_{K, \mD} = \Exp{\norm{\nabla f(\tilde{x}^K)}_{\mD/\det(\mD)^{1/d}}^2}.
\end{equation}
Notice that this criterion is comparable to the standard Euclidean norm \citet{li2023det}, and for a fixed $\mD$, we have 
\begin{align*}
   \lambda_{\min}\left(\frac{\mD}{\det(\mD)^{1/d}}\right)\cdot\norm{\nabla f(x)}^2 \leq \norm{\nabla f(x)}^2_{\frac{\mD}{\det(\mD)^{1/d}}} \leq \lambda_{\max}\left(\frac{\mD}{\det(\mD)^{1/d}}\right)\cdot\norm{\nabla f(x)}^2.
\end{align*}

\begin{figure}[h]
   \centering
   \subfigure{
      \begin{minipage}[t]{0.98\textwidth}
         \includegraphics[width=0.32\textwidth]{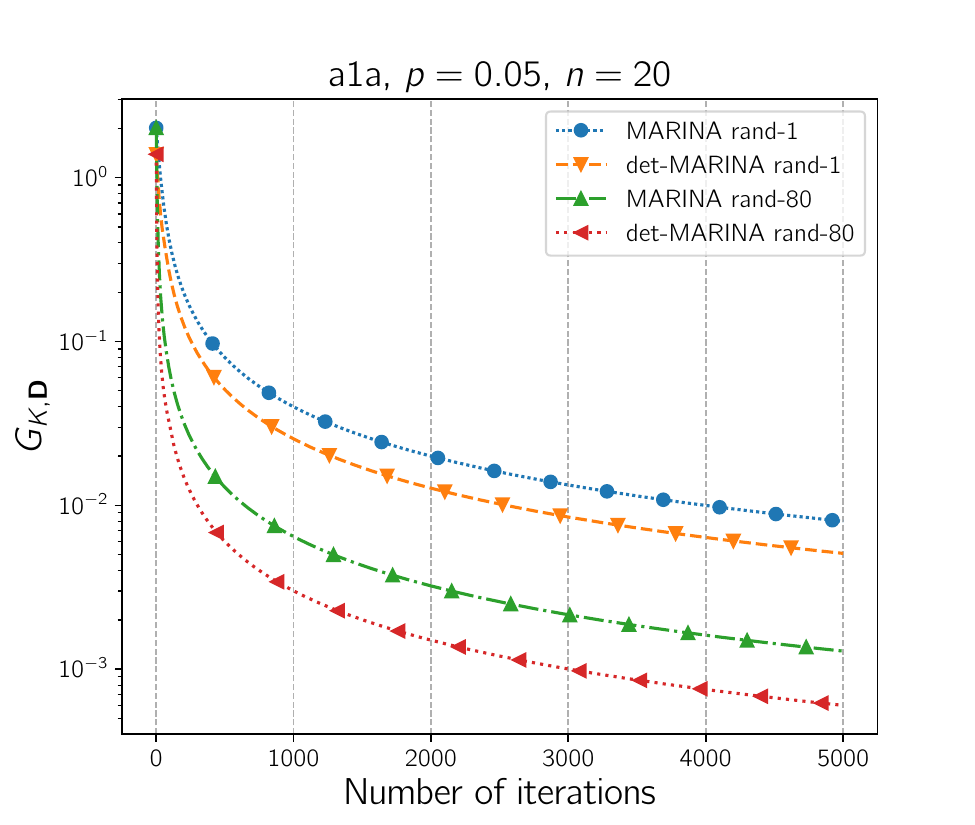} 
         \includegraphics[width=0.32\textwidth]{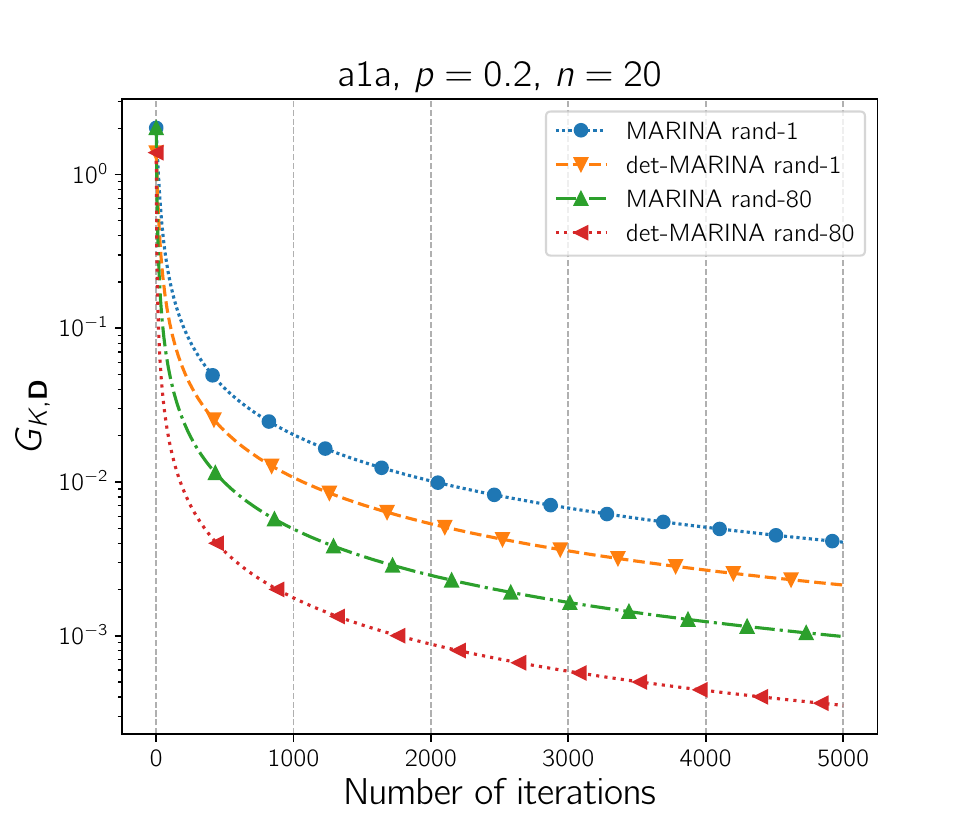}
         \includegraphics[width=0.32\textwidth]{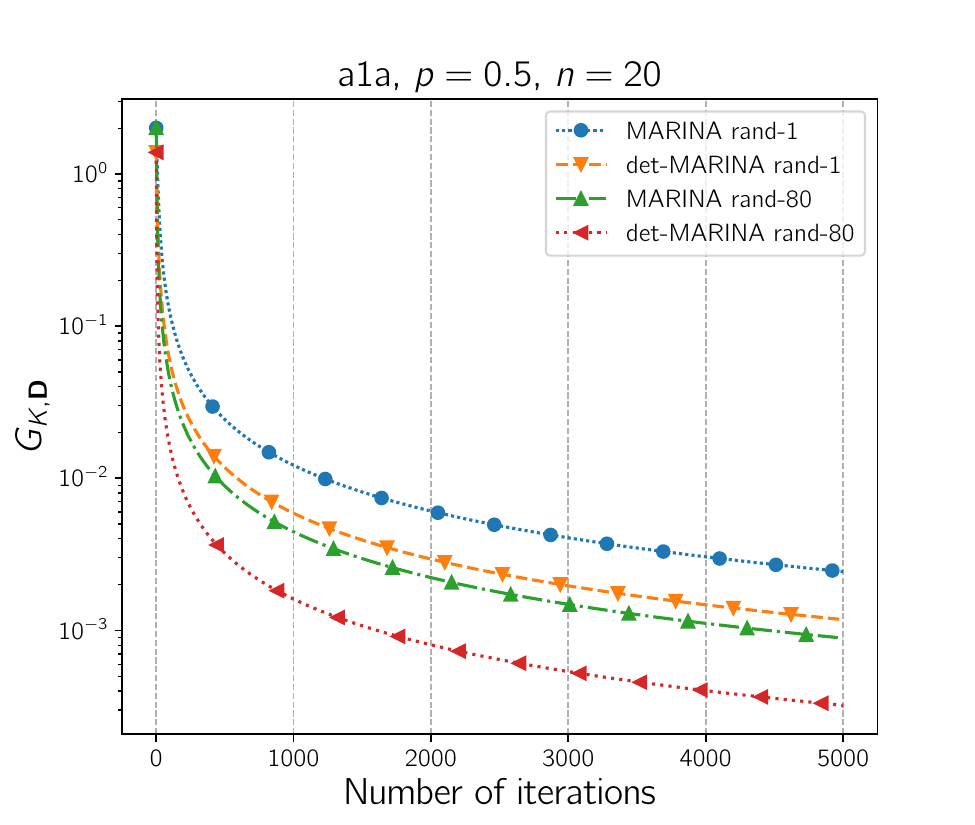}
      \end{minipage}
   }
   \subfigure{
      \begin{minipage}[t]{0.98\textwidth}
         \includegraphics[width=0.32\textwidth]{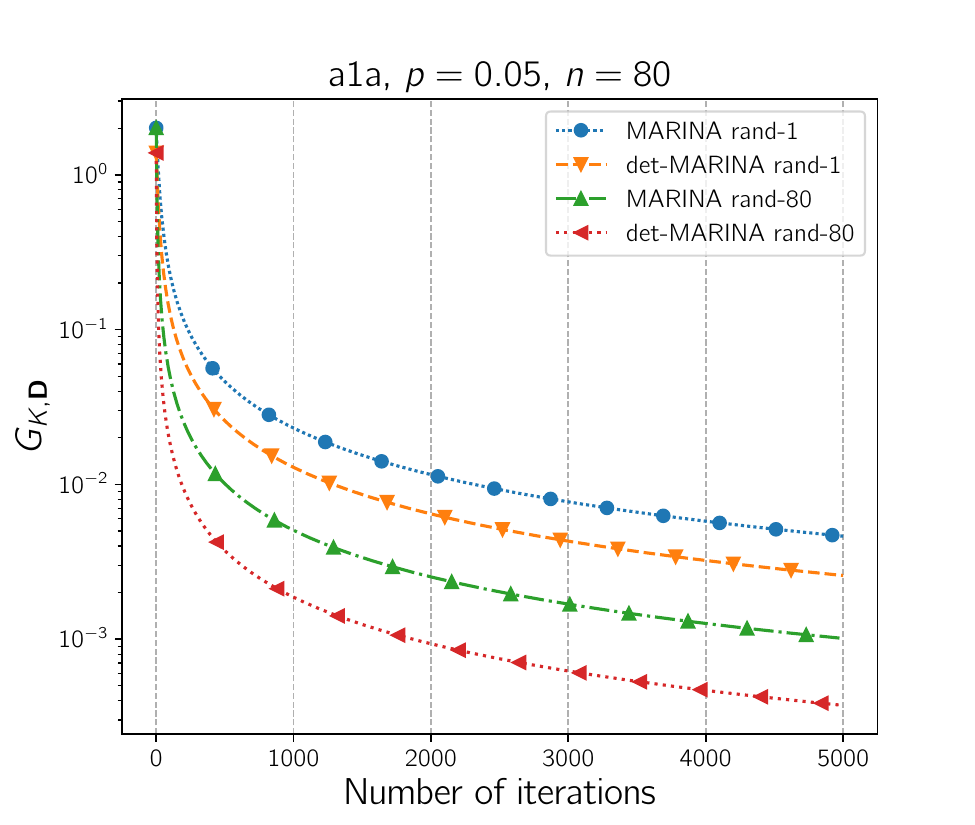} 
         \includegraphics[width=0.32\textwidth]{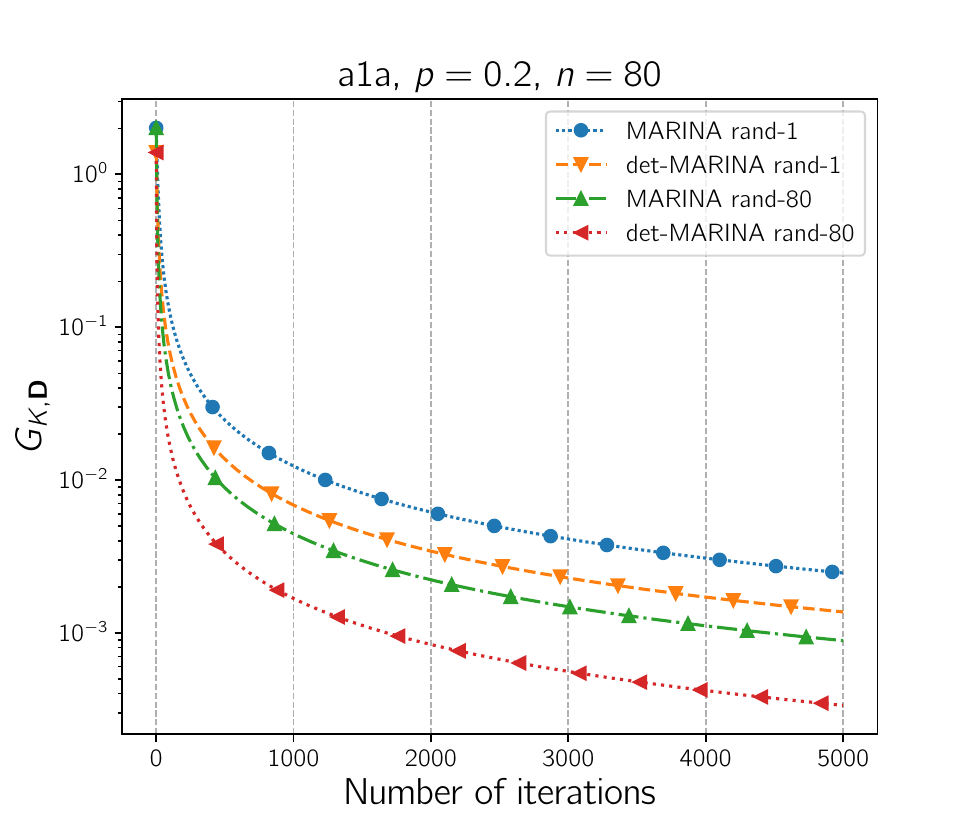}
         \includegraphics[width=0.32\textwidth]{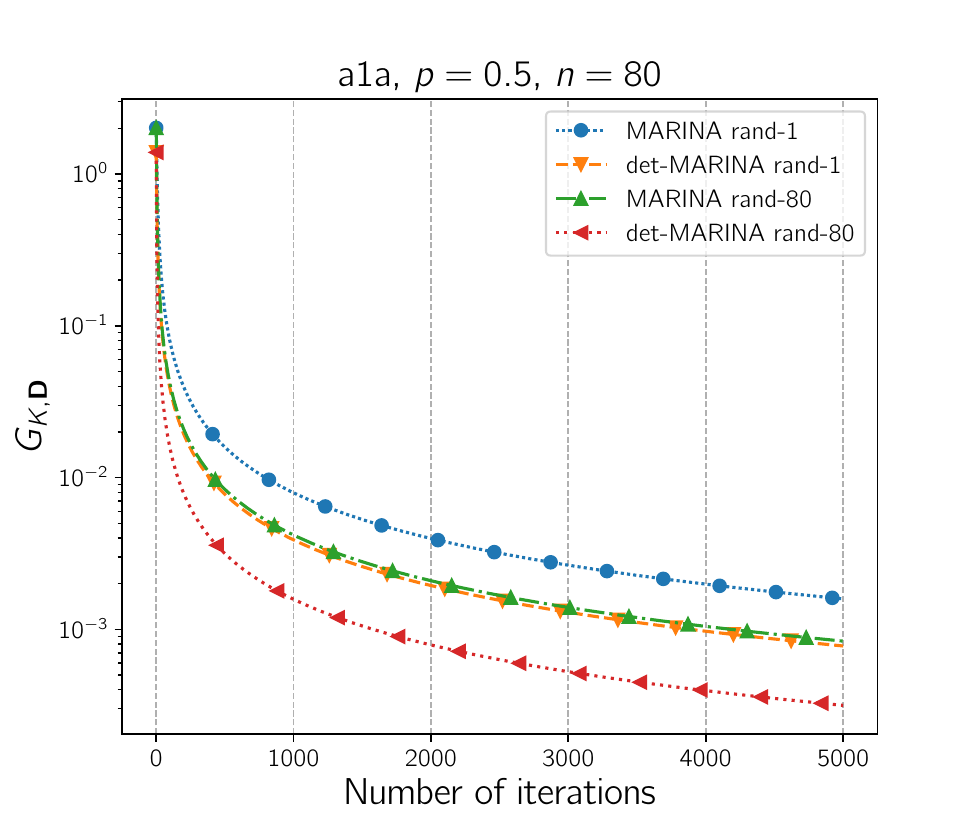}
      \end{minipage}
   }

   \subfigure{
      \begin{minipage}[t]{0.98\textwidth}
         \includegraphics[width=0.32\textwidth]{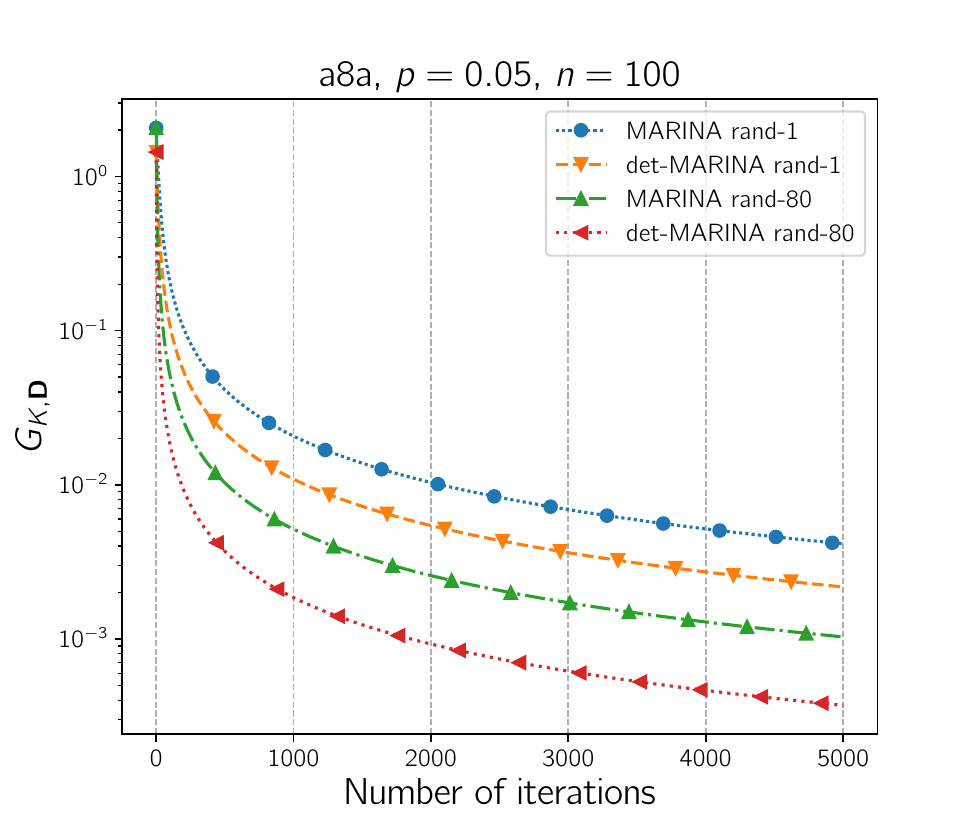} 
         \includegraphics[width=0.32\textwidth]{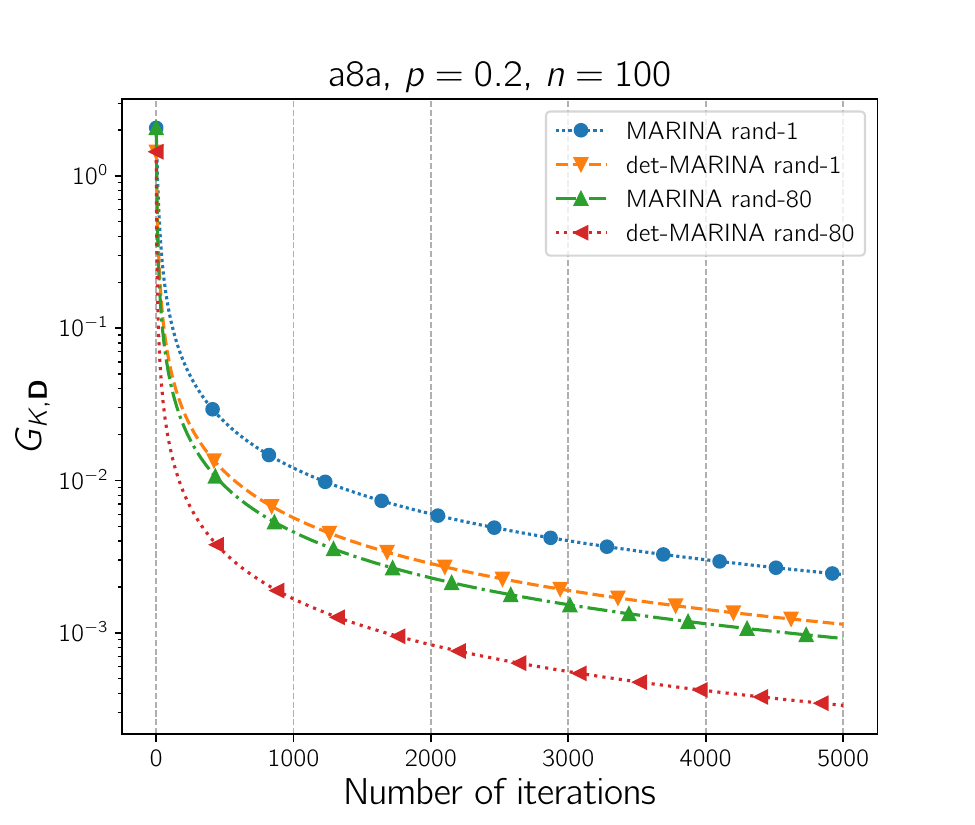}
         \includegraphics[width=0.32\textwidth]{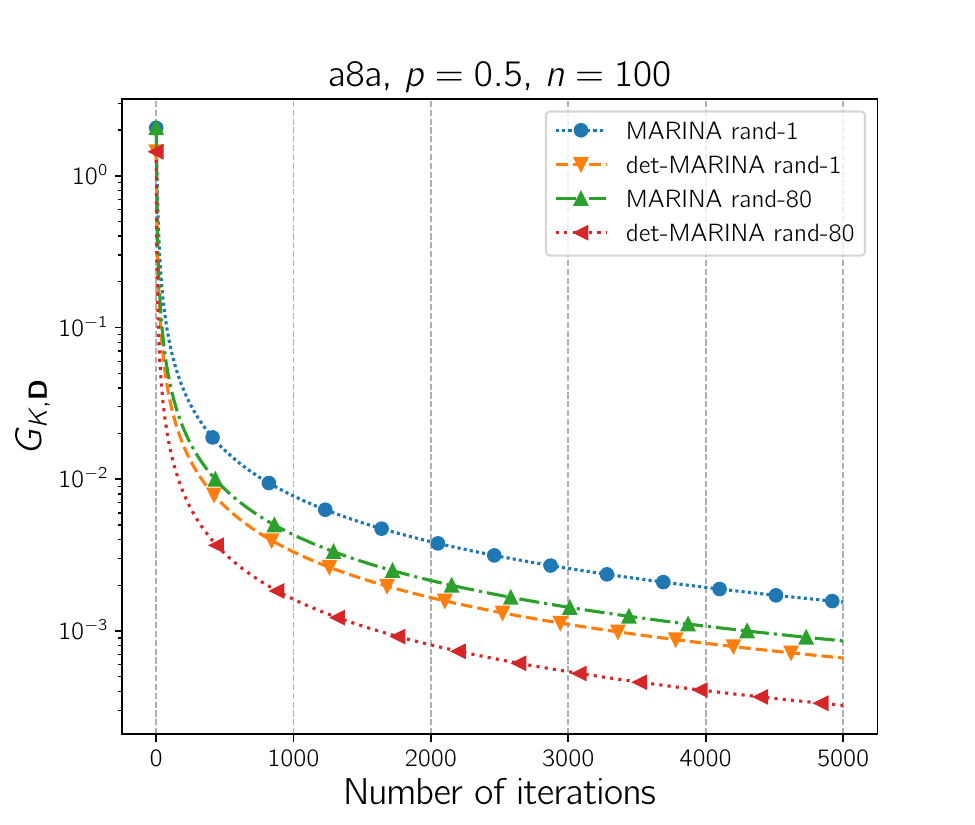}
      \end{minipage}
   }

   \subfigure{
      \begin{minipage}[t]{0.98\textwidth}
         \includegraphics[width=0.32\textwidth]{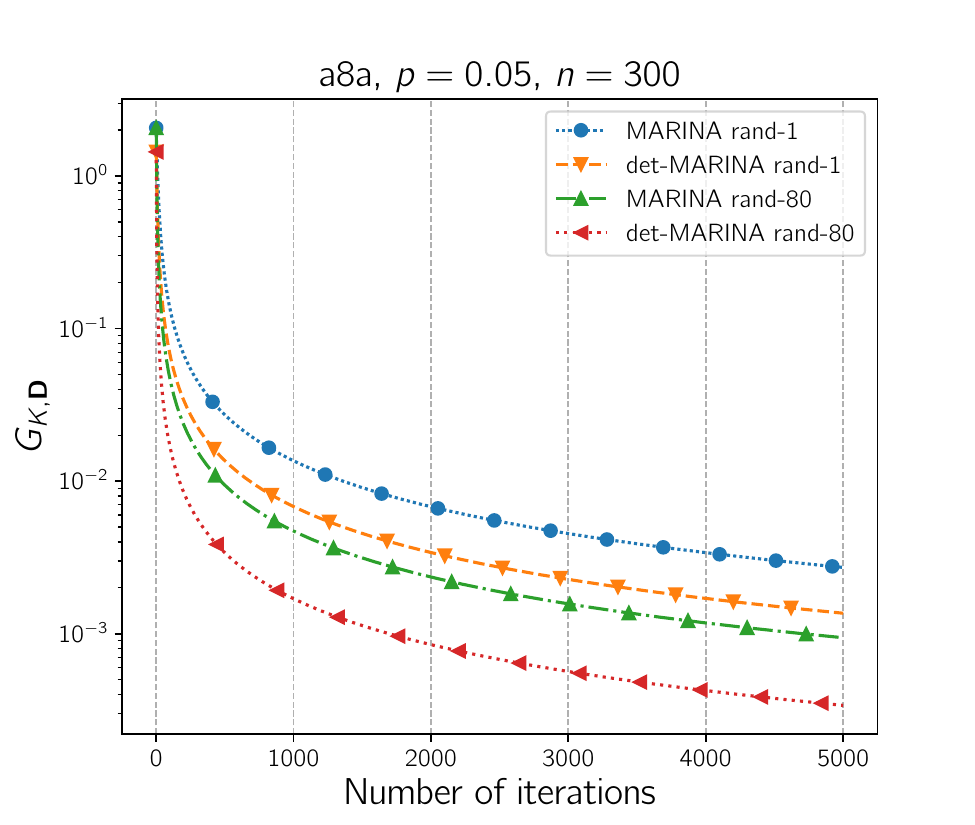} 
         \includegraphics[width=0.32\textwidth]{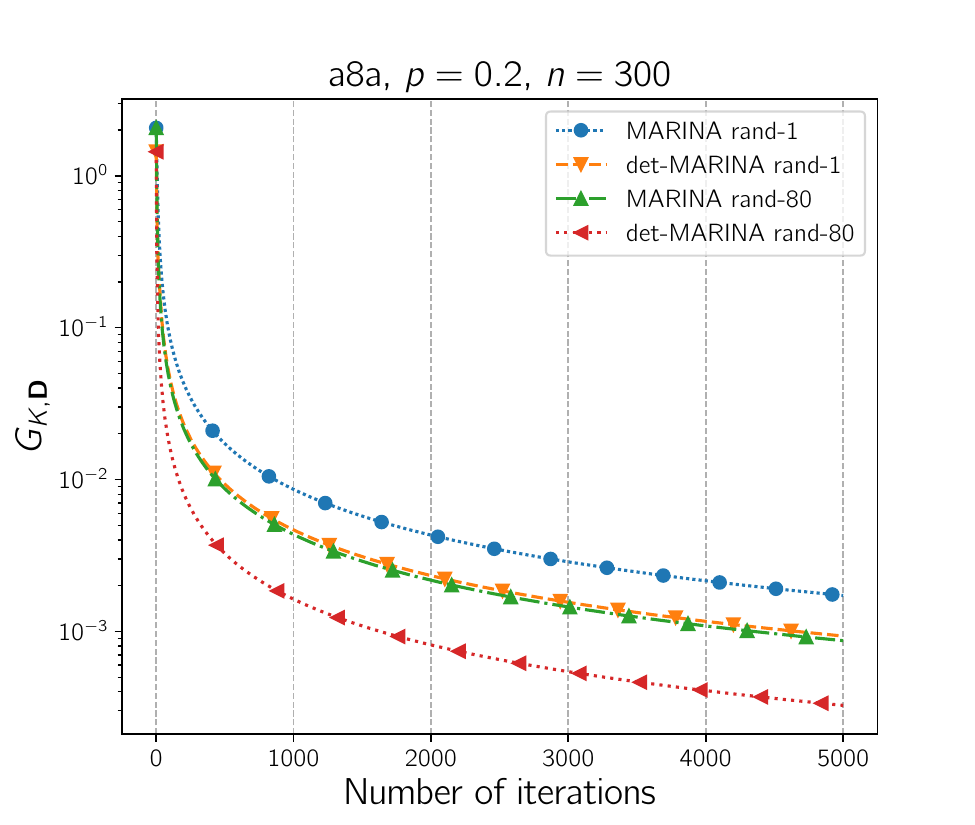}
         \includegraphics[width=0.32\textwidth]{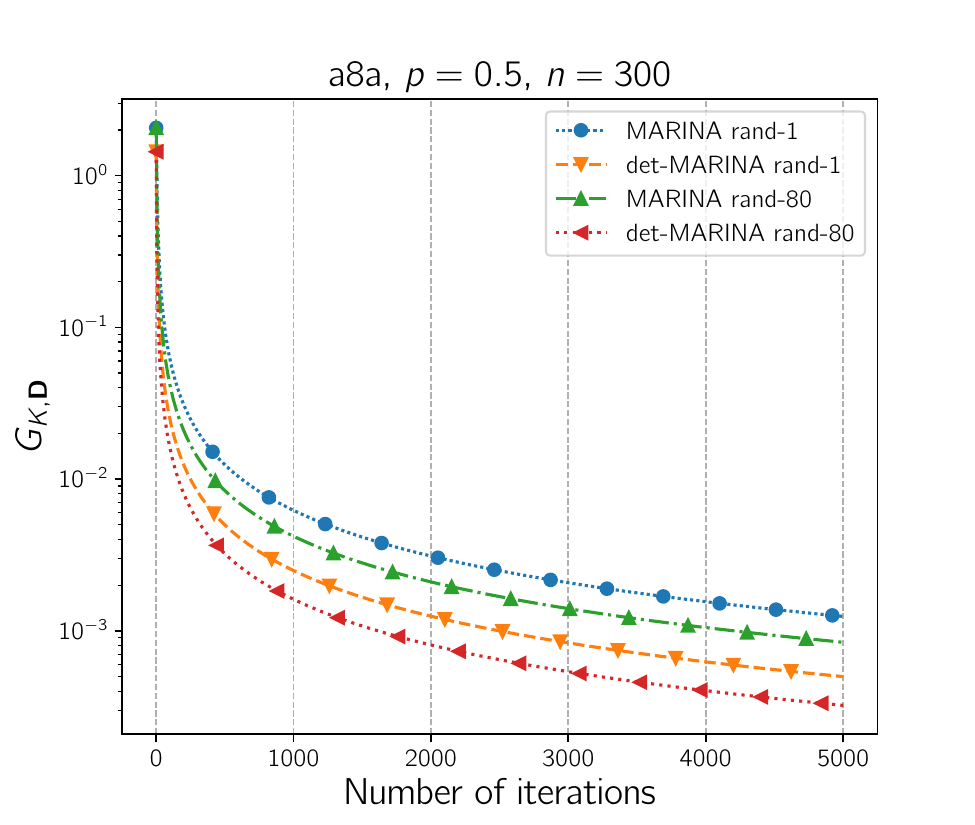}
      \end{minipage}
   }
   \caption{In this experiment, we aim to compare {\detmarina} with stepsize $\mD_{\mL^{-1}}^*$ to the standard \marina\, with the optimal scalar stepsize. 
   Rand-$\tau$ compressor is used in the comparison. 
   Throughout the experiments, $\lambda$ is fixed at $0.3$.  Optimal stepsize is calculated in each case with respect to the sketch used. The $x$-axis denotes the number of iterations while the notation $G_{K, \mD}$ for the $y$-axis is defined in \eqref{eq:def-G--KD}, which is the averaged matrix norm of the gradient. The notation $p$ in the title denotes the probability used in the two algorithms, $n$ denotes the number of clients in each setting.}
   \label{fig:experiment-1}
\end{figure}
As it is illustrated in \Cref{fig:experiment-1}, {\detmarina} always has a faster convergence rate compared to \marina\,if they use the same sketch, this justifies the result we have in \Cref{col:iteration-comp-L-inv}.
Notice that in some cases, {\detmarina} with Rand-$1$ sketch even outperforms standard \marina\,with Rand-$80$ sketch. 
This further demonstrates the superiority of matrix stepsizes and smoothness over the standard scalar setting.

\subsection{Improvements on non variance reduced methods}
In this section, we compare two non-variance reduced methods, distributed compressed gradient descent ({\dcgd}) and distributed {\detcgd}, with two variance reduced methods, {\marina}, and {\detmarina}.
Rand-$1$ sketch is used throughout this experiment for all the algorithms, for non variance reduced method $\varepsilon^2$ is fixed at $0.01$ in order to determine the optimal stepsize. 
The purpose of this experiment is to show the advantages of variance reduced methods over non variance reduced methods. 
{\dcgd} was initially proposed in \citep{khirirat2018distributed}.
Later on {\diana} was proposed in \citep{mishchenko2019distributed} and then combined with variance reduction technique. Recently \citet{shulgin2022shifted} proposed shifted {\dcgd}, which is a shifted version of {\dcgd} and proved its convergence in the (strongly) convex setting. A general analysis on {\sgd} type methods in the non-convex world is provided by \citet{khaled2022better}, including {\dcgd} and shifted {\dcgd}. In our case, in order to determine the optimal scalar stepsize for {\dcgd}, one can simply use Proposition 4 in \citep{khaled2022better}. One can check that in order to satisfy $\min_{0\leq k\leq K-1}\Exp{\norm{\nabla f(x^k)}^2} \leq \varepsilon^2$ the stepsize condition for {\dcgd} in the non-convex case reduces to 
\begin{equation*}
   \gamma_2 \leq \min\left\{\frac{1}{L}, \sqrt{\frac{n}{\omega LL_{\max} K}}, \frac{n\varepsilon^2}{4LL_{\max}\omega\cdot \Delta^{\star}}\right\}, 
\end{equation*}
where $L$ is the smoothness constant for $f$, $L_i$ is the smoothness constant for $f_i$, $L_{\max} = \max_i L_i$, $K$ is the total number of iterations, $\Delta^{\star} = f(x^{\star}) - \frac{1}{n}\sum_{i=1}^{n}f_i(x^{\star})$. The constant $\omega$ is associated with the compressor used in the algorithm, for Rand-$\tau$ sketch, it is $\frac{d}{\tau} - 1$. 
For  distributed {\detcgd}\, according to \citet{li2023det}, the stepsize condition in order to satisfy $\min_{0\leq k \leq K-1} \Exp{\norm{\nabla f(x)}^2_{\mD/\det(\mD)^{1/d}}} \leq \varepsilon^2$ is 
\begin{equation}
   \label{eq:detcgd-stepsize}
   \mD\mL\mD \preceq \mD, \qquad \lambda_{\mD} \leq \min\left\{\frac{n}{K}, \frac{n\varepsilon^2}{4\Delta^{\star}}\det(\mD)^{1/d}\right\},
\end{equation}
where $\lambda_{\mD}$ is defined as 
\begin{equation}
   \label{eq:def-lambda-D}
   \lambda_{\mD} = \max_{i}\left\{\lambda_{\max}\left(\Exp{\mL_i^\frac{1}{2}\left(\mS_i^k - \mI_d\right)\mD\mL\mD\left(\mS_i^k - \mI_d\right)\mL_i^\frac{1}{2}}\right)\right\}. 
\end{equation}
In general cases, there is no easy way to find a optimal stepsize matrix $\mD$ satisfying \eqref{eq:detcgd-stepsize}, alternatively, we choose the optimal diagonal stepsize $\mD^*_3$ similarly to \citep{li2023det}. The stepsize condition for \marina\,has already been described by \eqref{eq:ss-marina}. Note that we only consider {\marina}, but not {\diana} or shifted {\dcgd}, because {\diana} and shifted {\dcgd} offer suboptimal rates compared to {\marina} in the non-convex setting. For {\detmarina}, we fix $\mW = \mL^{-1}$, and use $\mD^*_{\mL^{-1}}$ as the stepsize matrix. In theory, {\detmarina} in this case should always out perform \marina\, in terms of iteration complexity.

\begin{figure}[t]
   \centering
   \subfigure{
      \begin{minipage}[t]{0.98\textwidth}
         \includegraphics[width=0.32\textwidth]{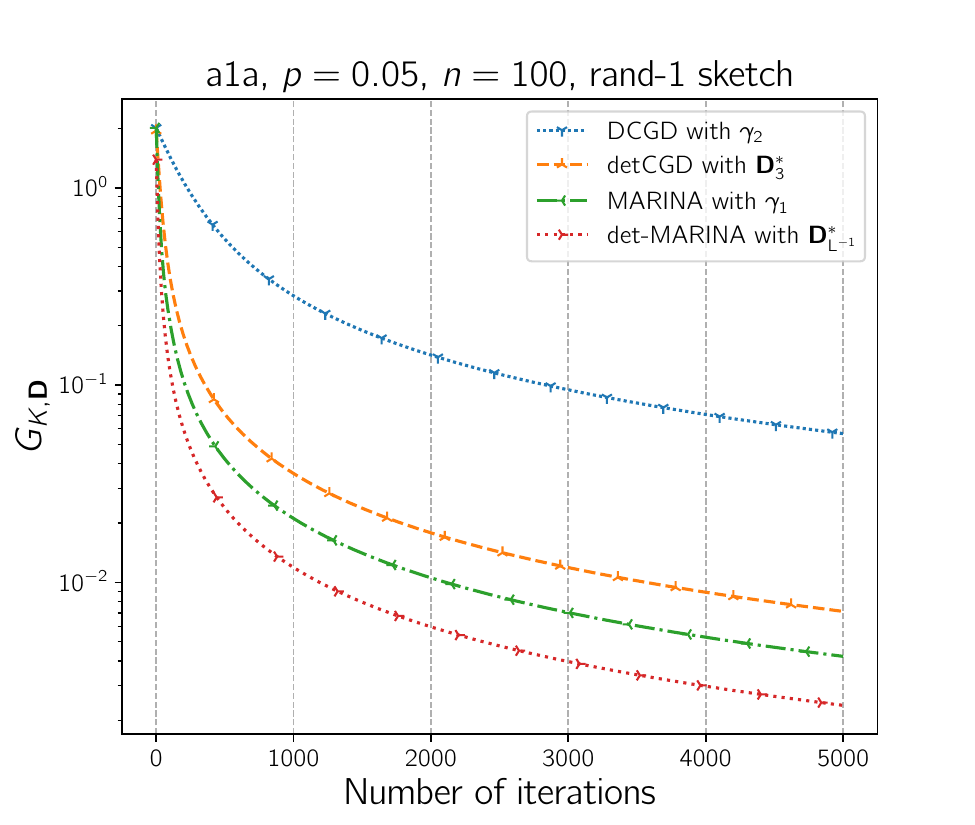} 
         \includegraphics[width=0.32\textwidth]{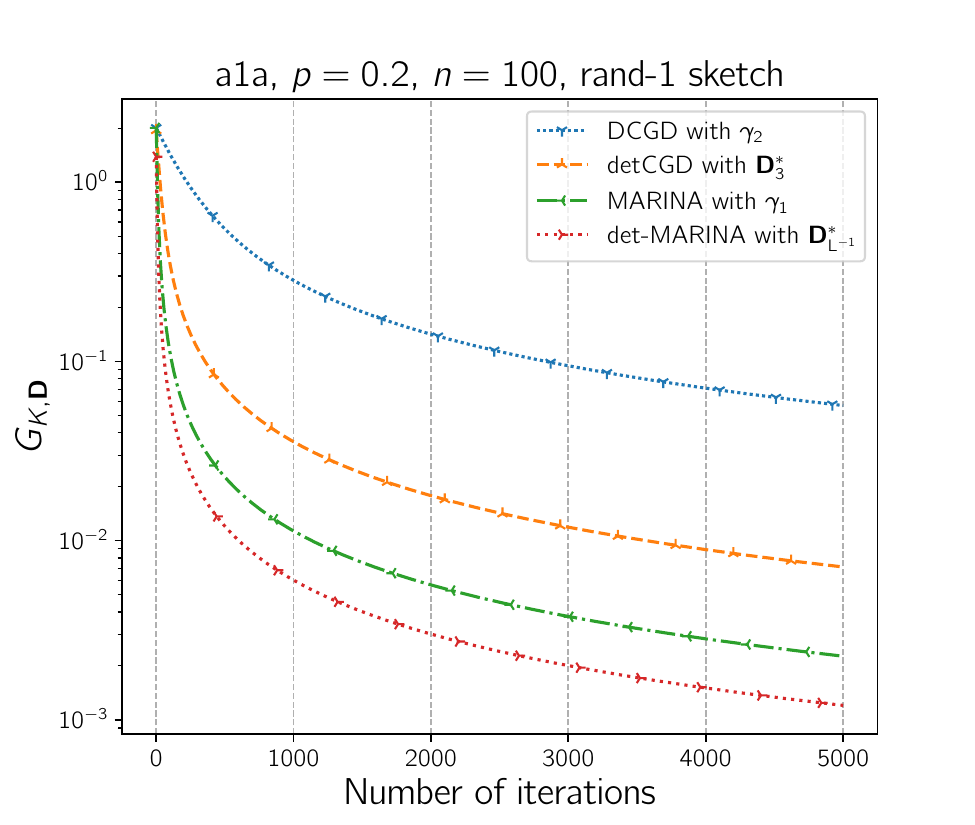}
         \includegraphics[width=0.32\textwidth]{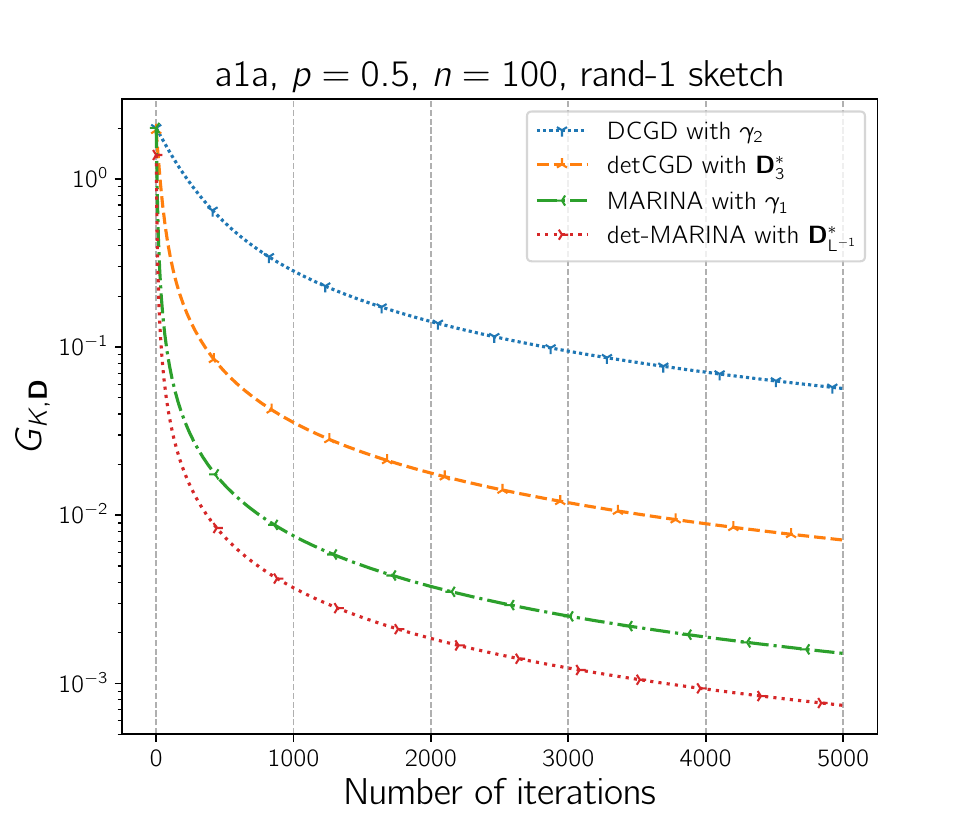}
      \end{minipage}
   }

   \subfigure{
      \begin{minipage}[t]{0.98\textwidth}
         \includegraphics[width=0.32\textwidth]{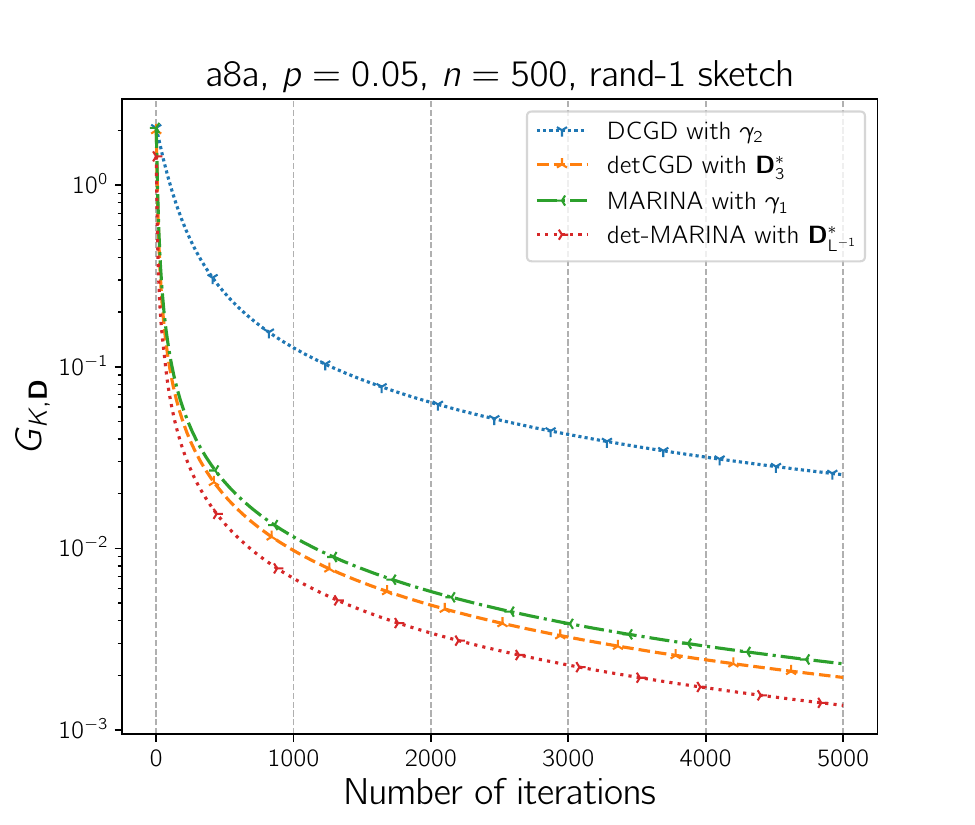} 
         \includegraphics[width=0.32\textwidth]{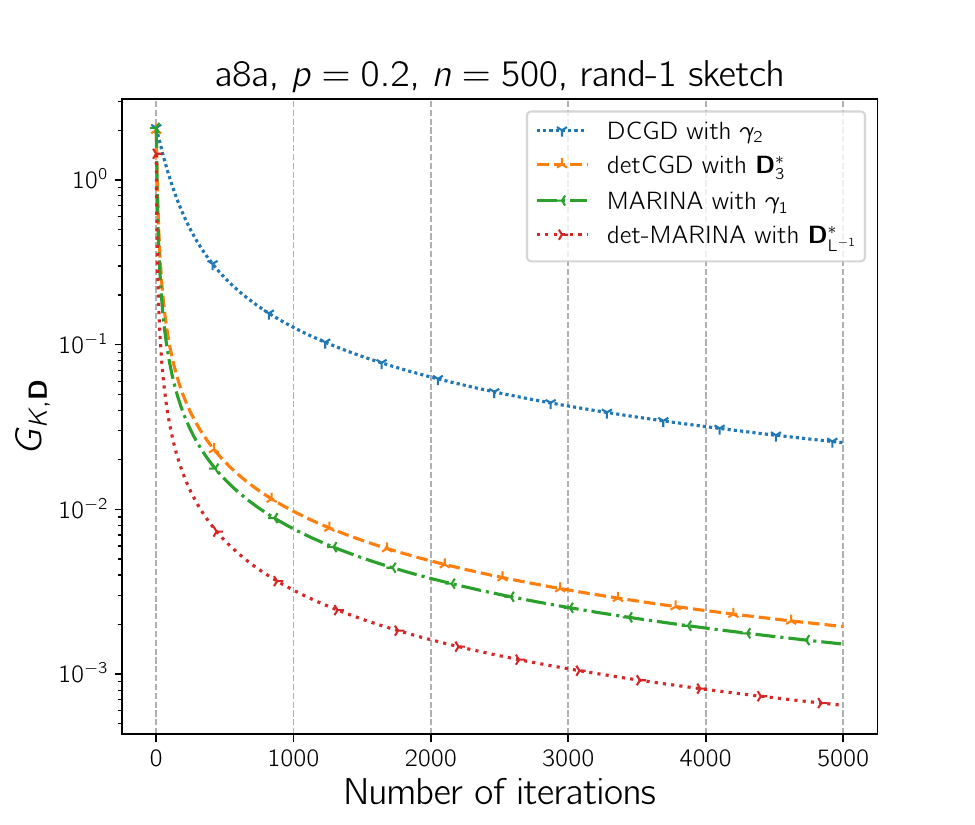}
         \includegraphics[width=0.32\textwidth]{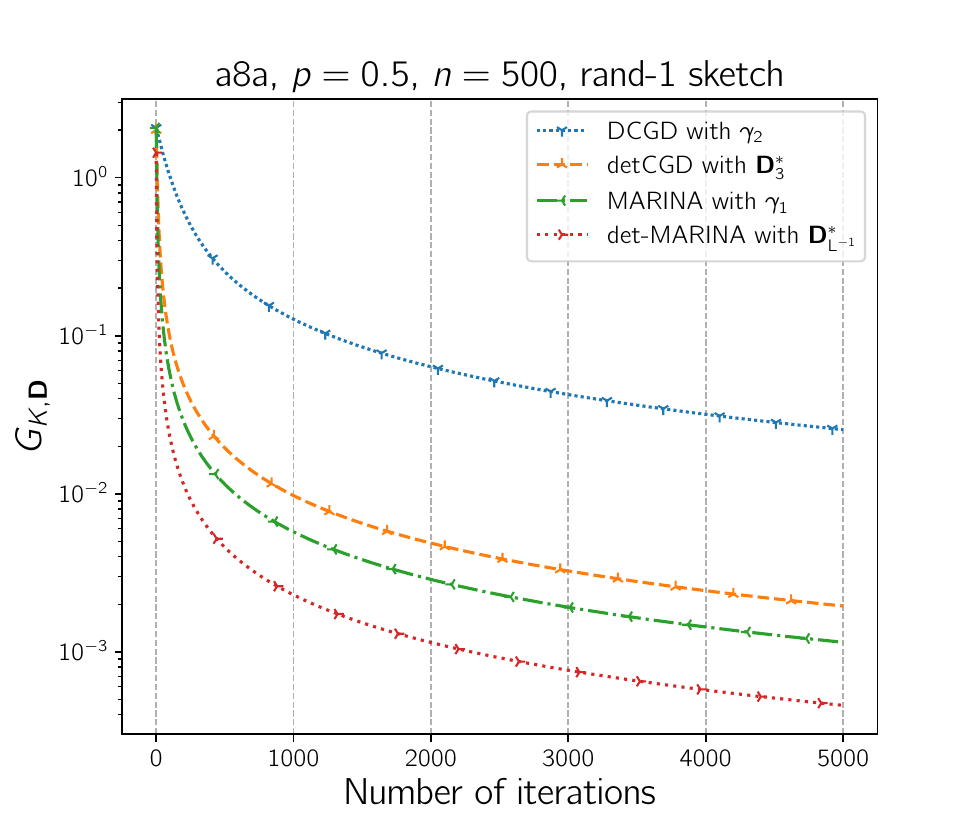}
      \end{minipage}
   }
   \caption{Comparison of {\dcgd} with optimal scalar stepsize $\gamma_2$, {\detcgd} with optimal diagonal stepsize $\mD_3^*$, {\marina} with optimal scalar stepsize $\gamma_1$, and {\detmarina} with optimal stepsize $\mD_{\mL^{-1}}^*$ with respect to $\mW = \mL^{-1}$. In each case, probability $p$ is chosen the set $\{0.05, 0.2, 0.5\}$ for {\marina} and \detmarina. $\lambda=0.3$ is fixed throughout the experiment. The notation $n$ in the title indicates the number of clients in each case.}
   \label{fig:experiment-2}
\end{figure}

In \Cref{fig:experiment-2}, in each plot, we observe that {\detmarina} outperforms {\marina} and the rest of the non-variance reduced methods. 
This is expected, since our theory confirms that {\detmarina} indeed has a better rate compared to {\marina}, and the stepsizes of the non-variance reduced methods are negatively affected by the neighborhood. 
When $p$ is reasonably large, the variance reduced methods considered here outperform the non-variance reduced methods. 
In this experiment we consider only the comparison involving {{\detcgd}}.

\subsection{\texorpdfstring{Improvements over {\detcgd}}{Improvements over det-CGD}}
In this section, we compare {\detcgd} in the distributed case with {\detmarina}, which are both algorithms using matrix stepsizes and matrix smoothness. 
The purpose of this experiment is to show that {\detmarina} improves on the current state of the art matrix stepsize compressed gradient method when the objective function is non-convex. 
Throughout the experiment, $\lambda = 0.3$ is fixed, and for {\detcgd}, $\varepsilon^2 = 0.01$ is fixed in order to determine its stepsize. 
For a thorough comparison, we select the stepsize for {\detcgd} in the following way. 
Let us denote the stepsize as $\mD = \gamma_{\mW}\cdot\mW$, where $\gamma_{\mW}\in\R_{++}, \mW\in\bbS^d_{++}$. 
We first fix a matrix $\mW$, in this case, we pick $\mW$ from the set $\{\mL^{-1}, \diag^{-1}(\mL), \mI_d\}$, and then we determine the optimal scaling $\gamma_{\mW}$ for each case using the condition given in \citep{li2023det} (see \eqref{eq:detcgd-stepsize} and \eqref{eq:def-lambda-D}). 
Then, we denote the matrix stepsizes for {\detcgd} 
\begin{equation}
   \label{eq:mat-stepsize-detcgd1}
   \mD_1 = \gamma_{\mI_d}\cdot\mI_d, \qquad \mD_2 = \gamma_{\diag^{-1}(\mL)}\cdot\diag^{-1}\left(\mL\right), \qquad \mD_3 = \gamma_{\mL^{-1}}\cdot\mL^{-1}.
\end{equation}
For {\detmarina}, we use the stepsize $\mD^*_{\mL^{-1}}$, which is described in \eqref{eq:var-D-opt}. In this experiment, we compare {\detcgd} using three stepsizes $\mD_1, \mD_2, \mD_3$ with {\detmarina} using stepsize $\mD_{\mL^{-1}}^*$.

\begin{figure}[t]
   \centering
   \subfigure{
      \begin{minipage}[t]{0.98\textwidth}
         \includegraphics[width=0.32\textwidth]{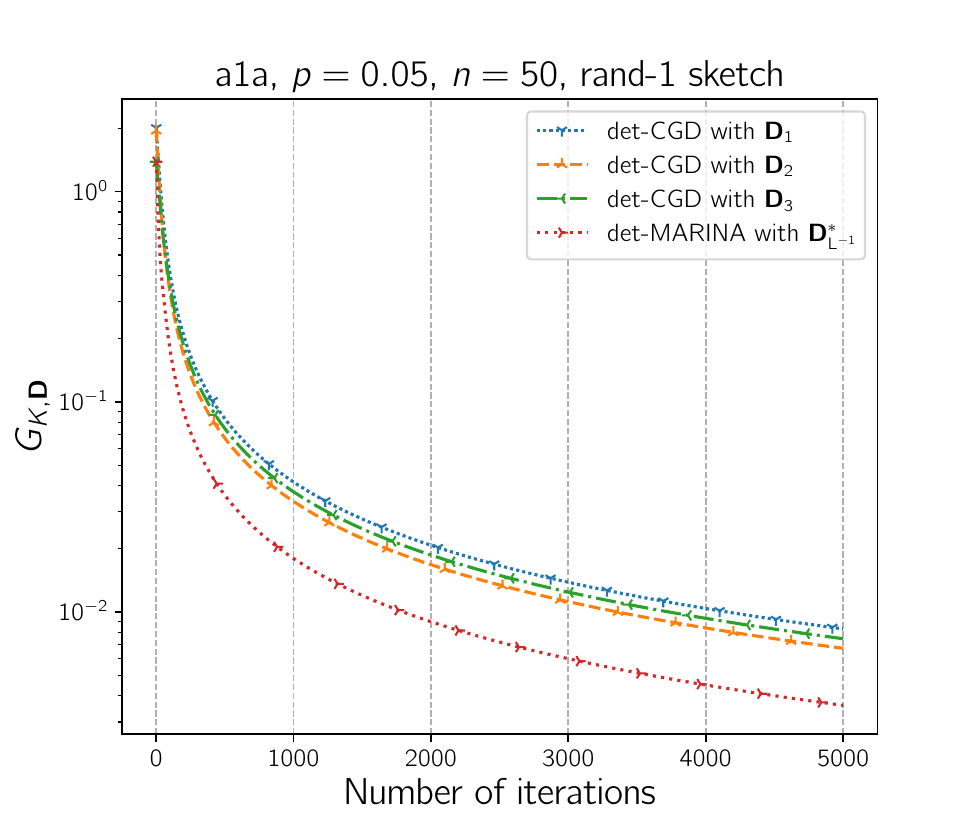} 
         \includegraphics[width=0.32\textwidth]{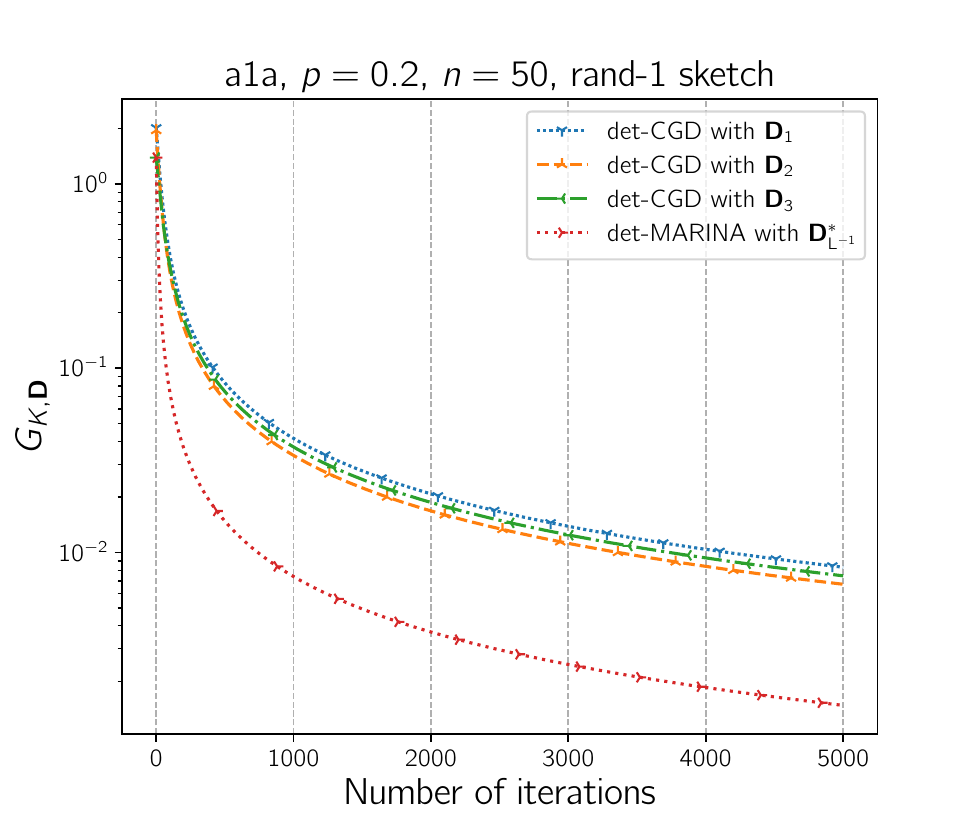}
         \includegraphics[width=0.32\textwidth]{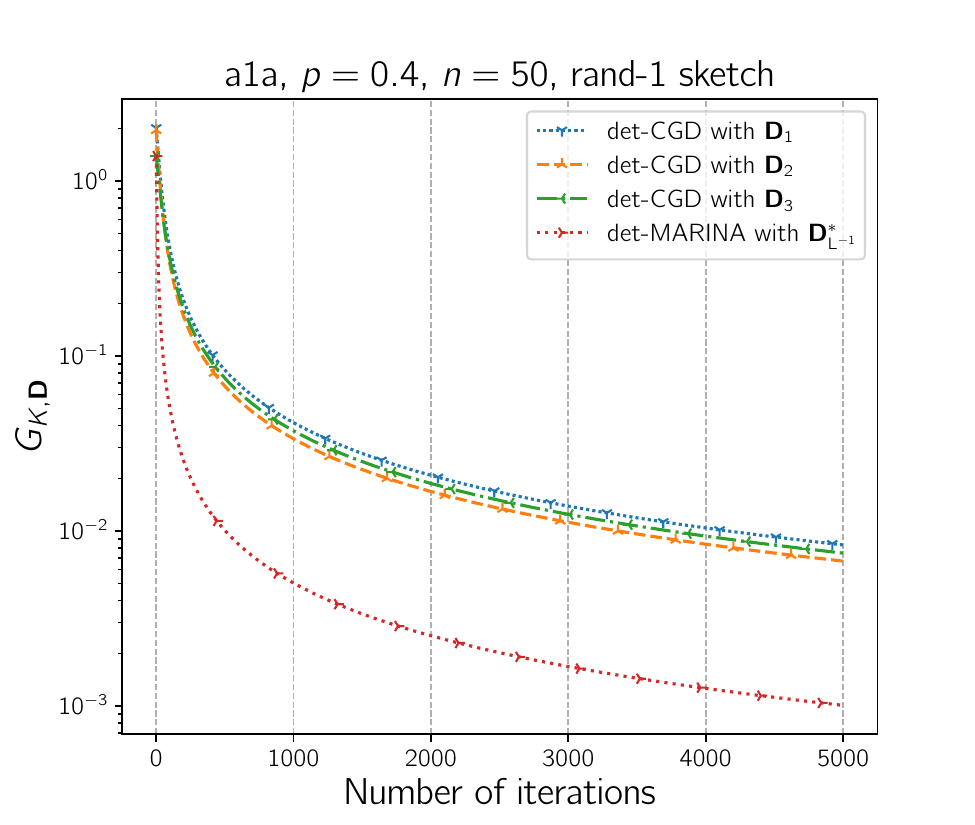}
      \end{minipage}
   }

   \subfigure{
      \begin{minipage}[t]{0.98\textwidth}
         \includegraphics[width=0.32\textwidth]{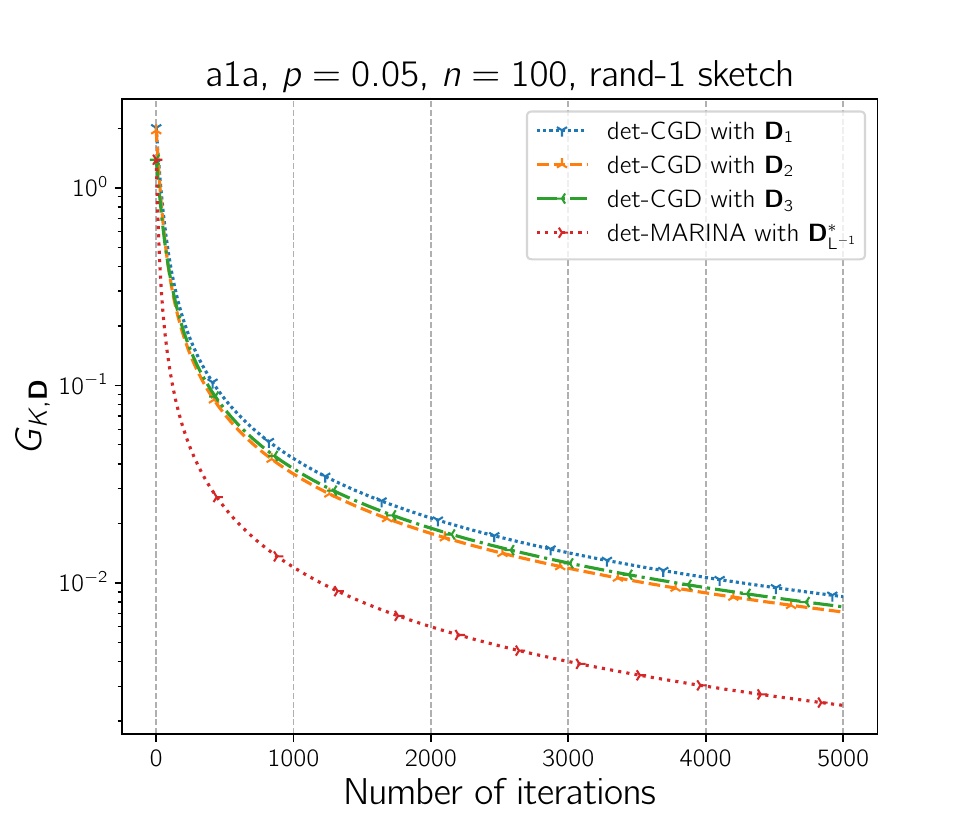} 
         \includegraphics[width=0.32\textwidth]{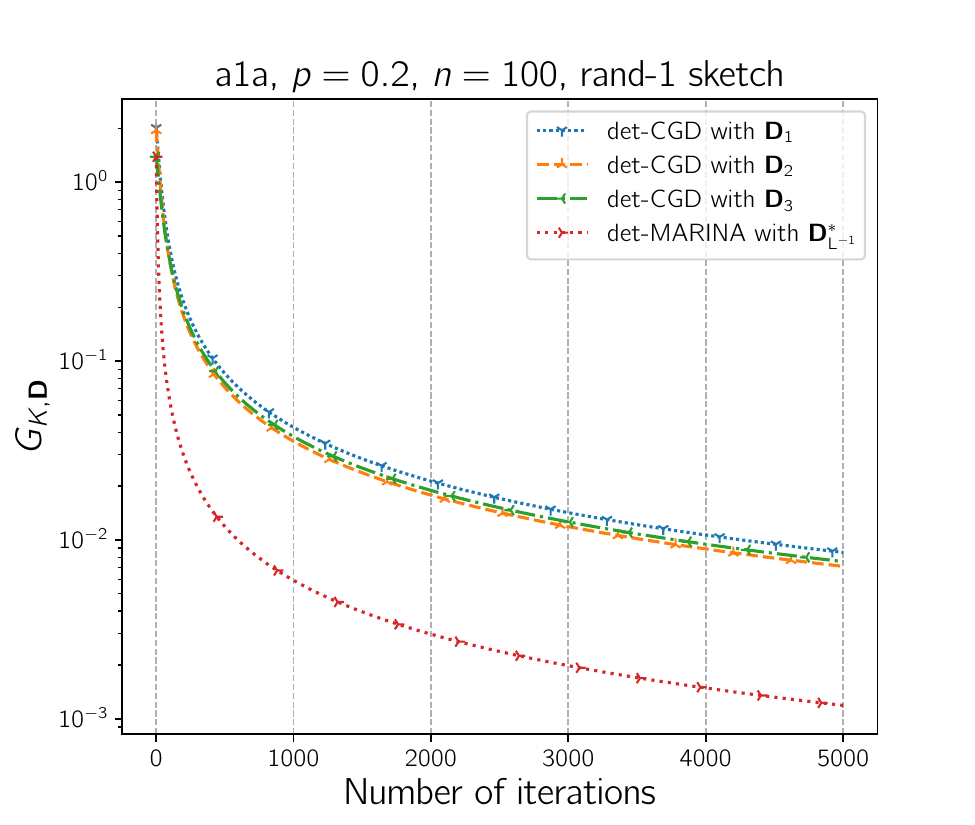}
         \includegraphics[width=0.32\textwidth]{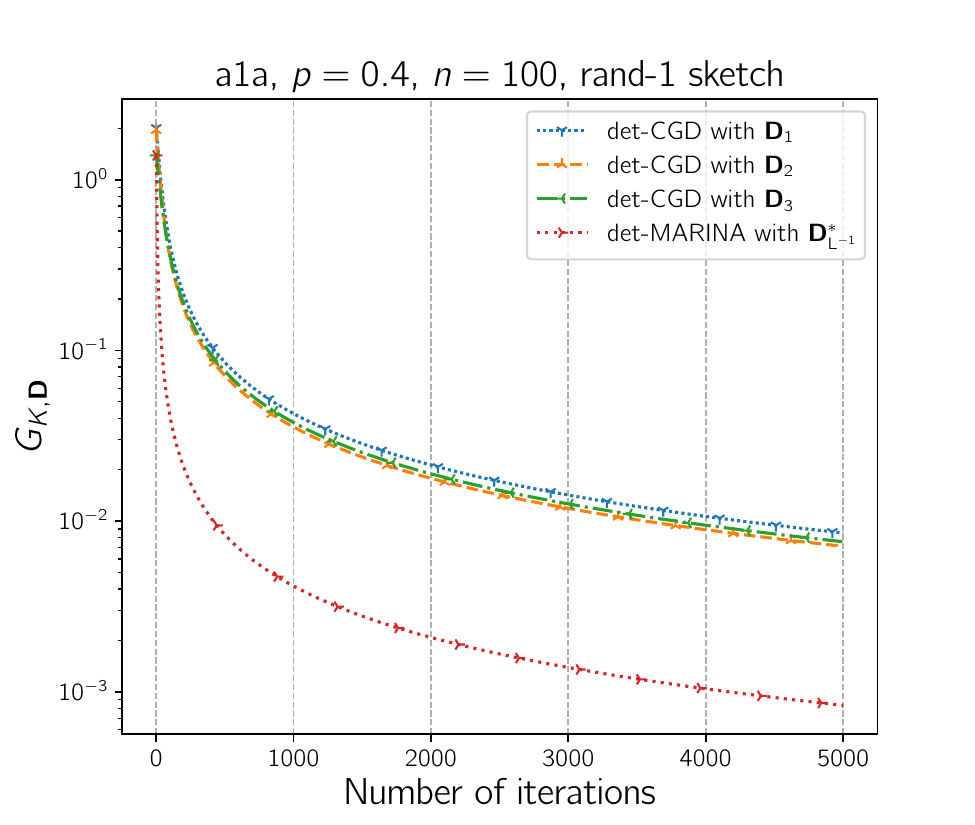}
      \end{minipage}
   }

   \subfigure{
      \begin{minipage}[t]{0.98\textwidth}
         \includegraphics[width=0.32\textwidth]{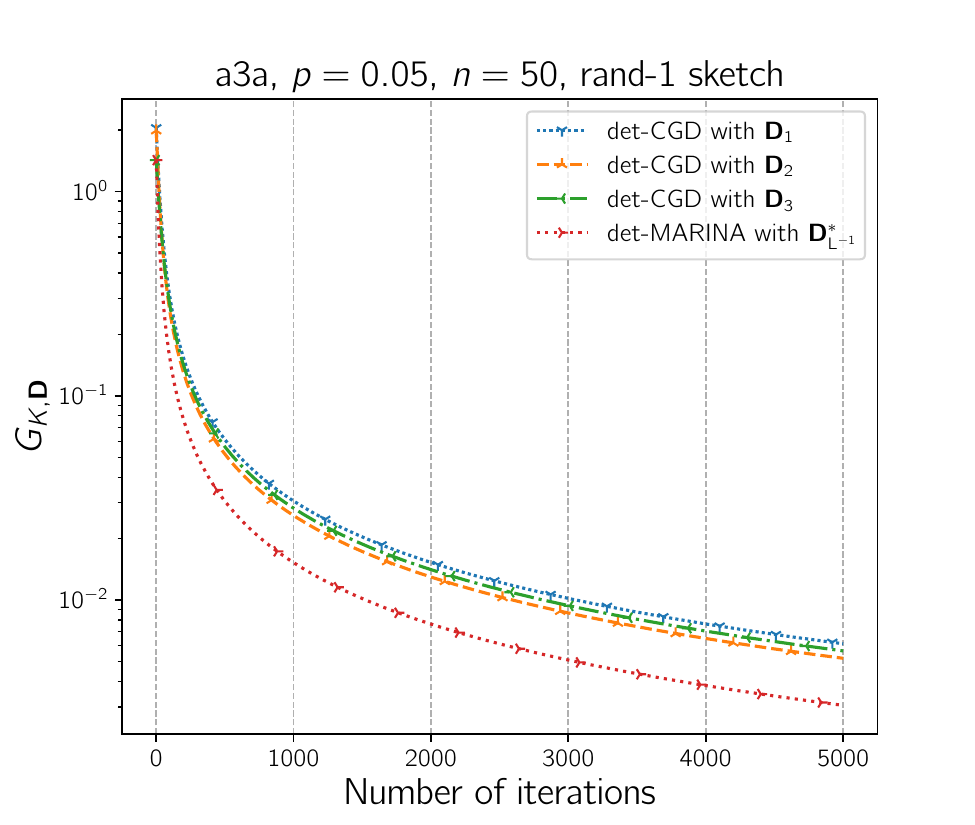} 
         \includegraphics[width=0.32\textwidth]{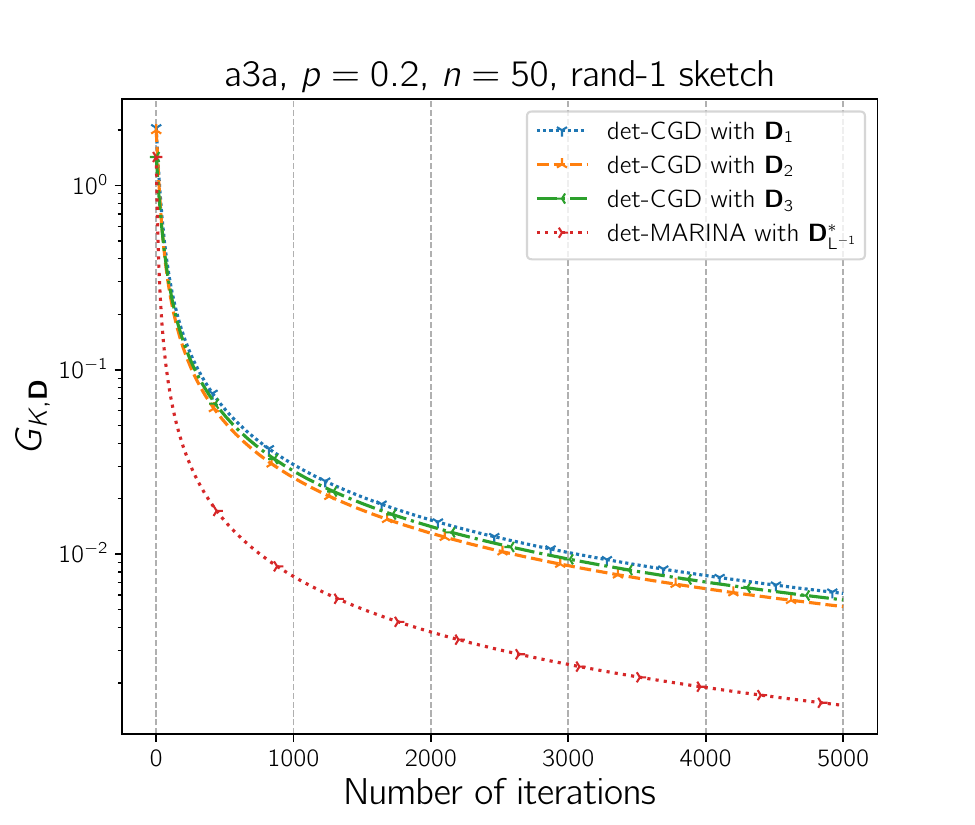}
         \includegraphics[width=0.32\textwidth]{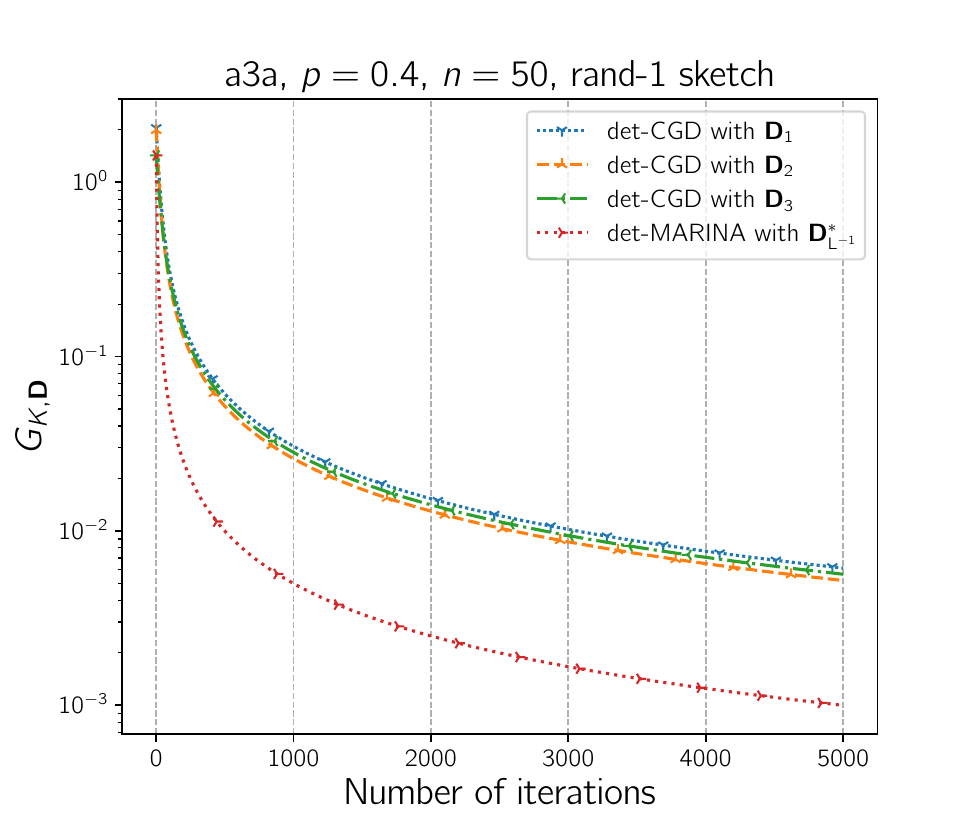}
      \end{minipage}
   }

   \subfigure{
      \begin{minipage}[t]{0.98\textwidth}
         \includegraphics[width=0.32\textwidth]{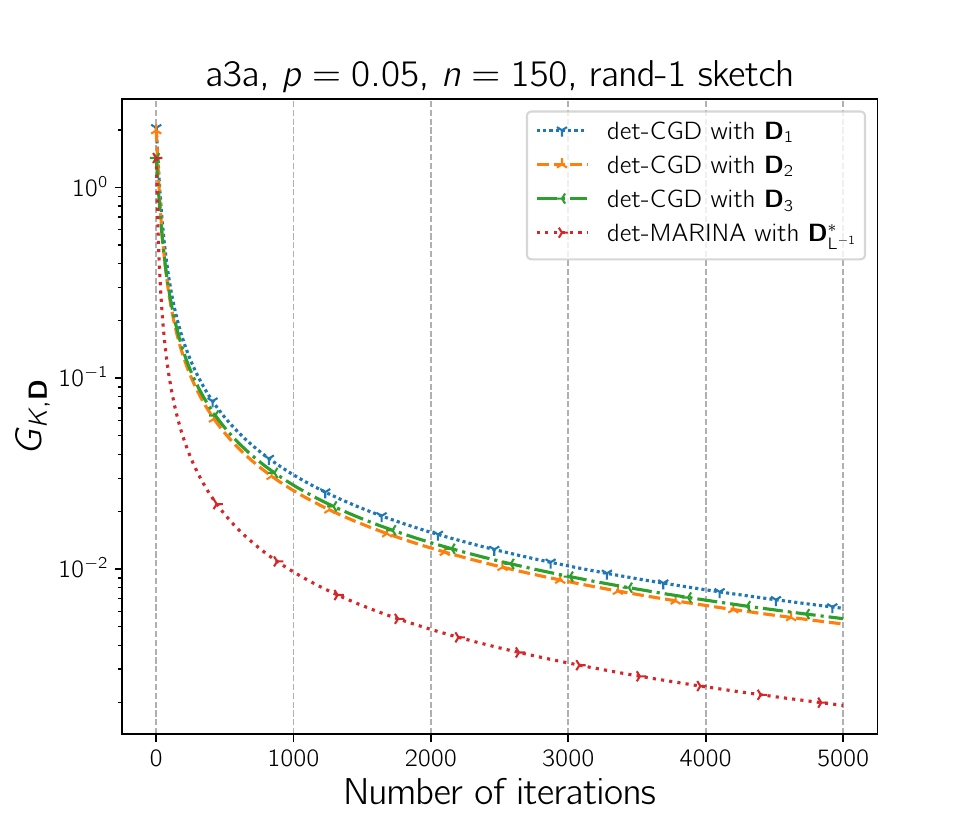} 
         \includegraphics[width=0.32\textwidth]{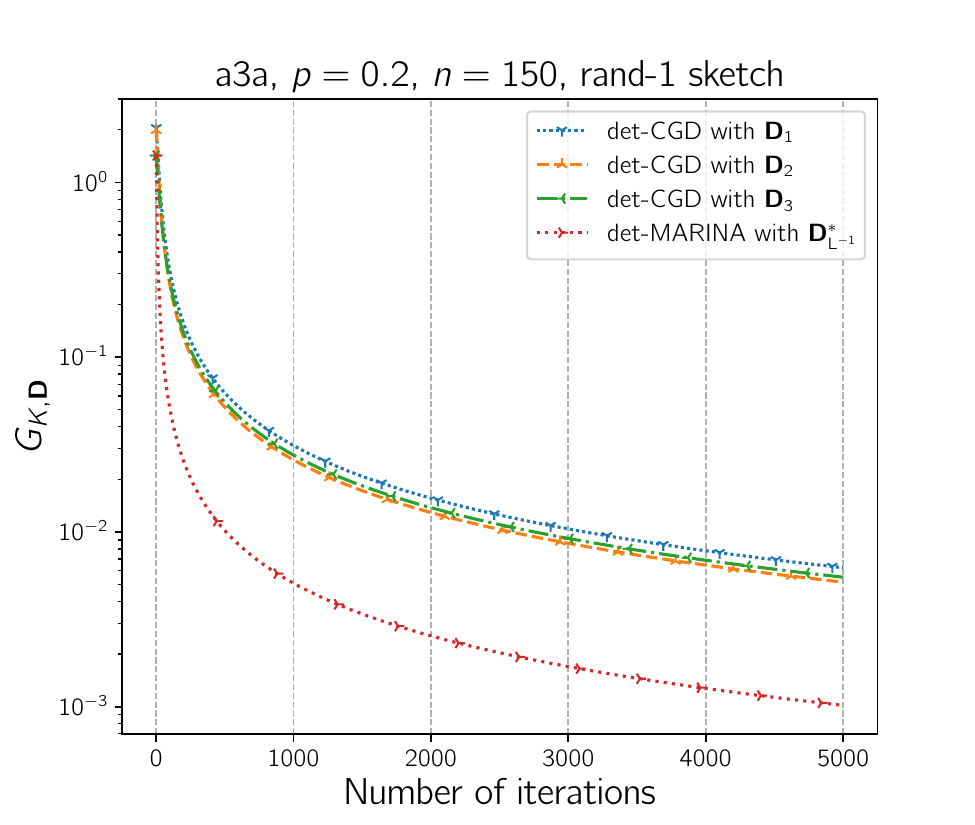}
         \includegraphics[width=0.32\textwidth]{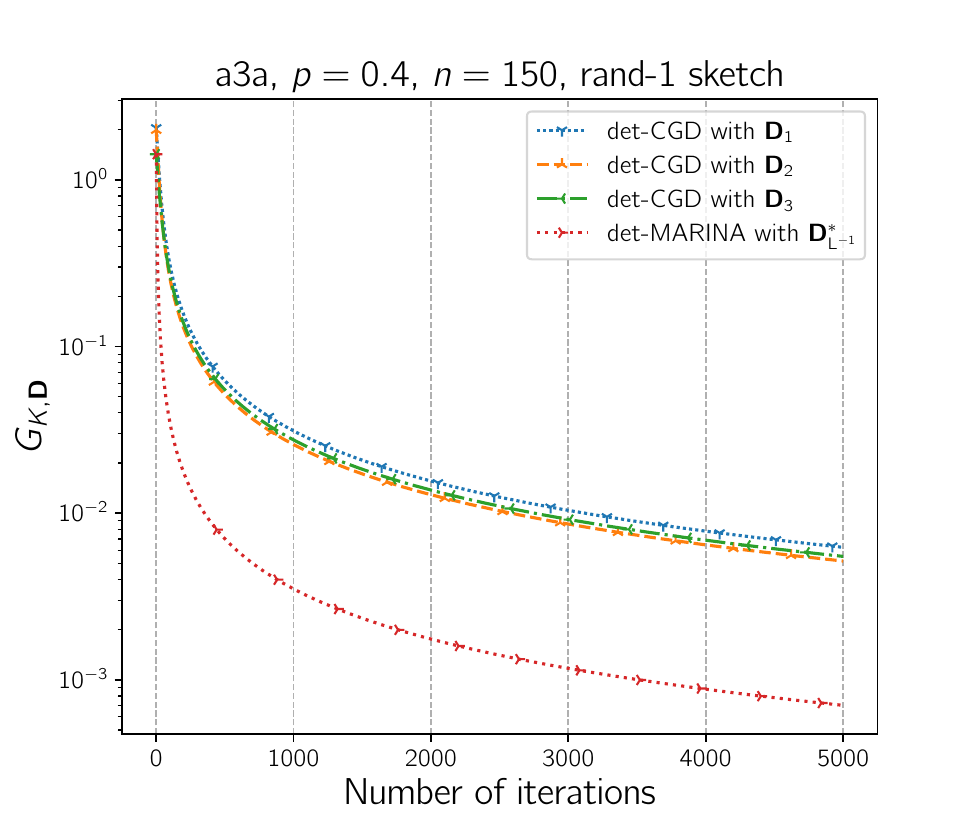}
      \end{minipage}
   }
   \caption{Comparison of {\detcgd} with matrix stepsize $\mD_1$, $\mD_2$ and $\mD_3$ and {\detmarina} with optimal matrix stepsize with respect to $\mW = \mL^{-1}$. The stepsizes $\{\mD_i\}_{i=1}^3$ are described in \eqref{eq:mat-stepsize-detcgd1}. Throughout the experiment $\varepsilon^2$ is fixed at $0.01$, the notation $p$ in the title refers to the probability for {\detmarina}, $n$ denotes the number of clients considered, Rand-$1$ sketch is used in all cases for all the algorithms.}
   \label{fig:experiment-3}
\end{figure}

From \Cref{fig:experiment-3}, it is clear that {\detmarina} outperforms {\detcgd} with all matrix optimal stepsizes with respect to a fixed $\mW$ considered here. 
This is expected, since the convergence rate of non-variance reduced methods are affected by its neighborhood. 
This experiment demonstrates the advantages of {\detmarina} over {\detcgd}, and is also supported by our theory. 
Notice that though different $\mW$ are considered for {\detcgd}, their convergence rates are similar, which is also mentioned by \citet{li2023det}.

\subsection{Comparing different stepsize choices}
This experiment is designed to see the how {\detmarina} works under different stepsize choices. 
As it is mentioned in \Cref{sec:F.1}, for each choice of $\mW \in \bbS^d_{++}$, an optimal stepsize $\mD^*_{\mW}$ can be determined. 
Here we compare {\detmarina} using three different stepsize choices $\mD^*_{\mL^{-1}}, \mD^*_{\diag^{-1}(\mL)}$ and $\mD^*_{\mI_d}$. 
There stepsizes are explicitly defined in \eqref{eq:var-D-opt}.
Throughout the experiment, we fix $\lambda = 0.3$, Rand-$1$ sketch is used in all cases.
\begin{figure}[t]
   \centering
   \subfigure{
      \begin{minipage}[t]{0.98\textwidth}
         \includegraphics[width=0.32\textwidth]{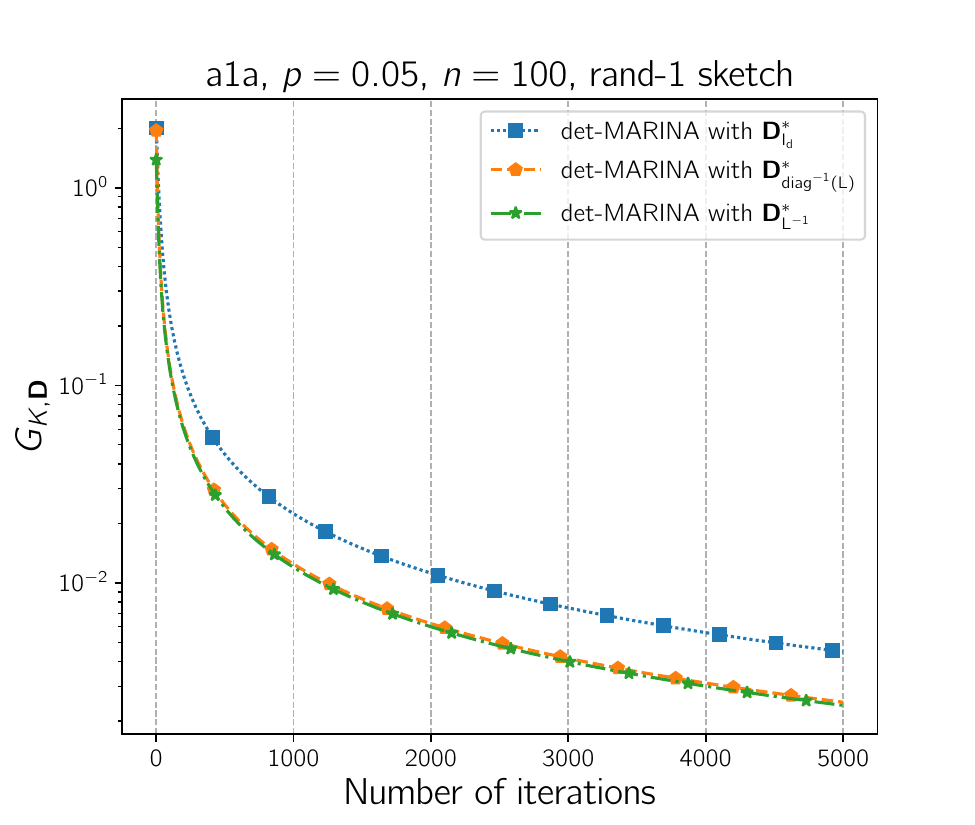} 
         \includegraphics[width=0.32\textwidth]{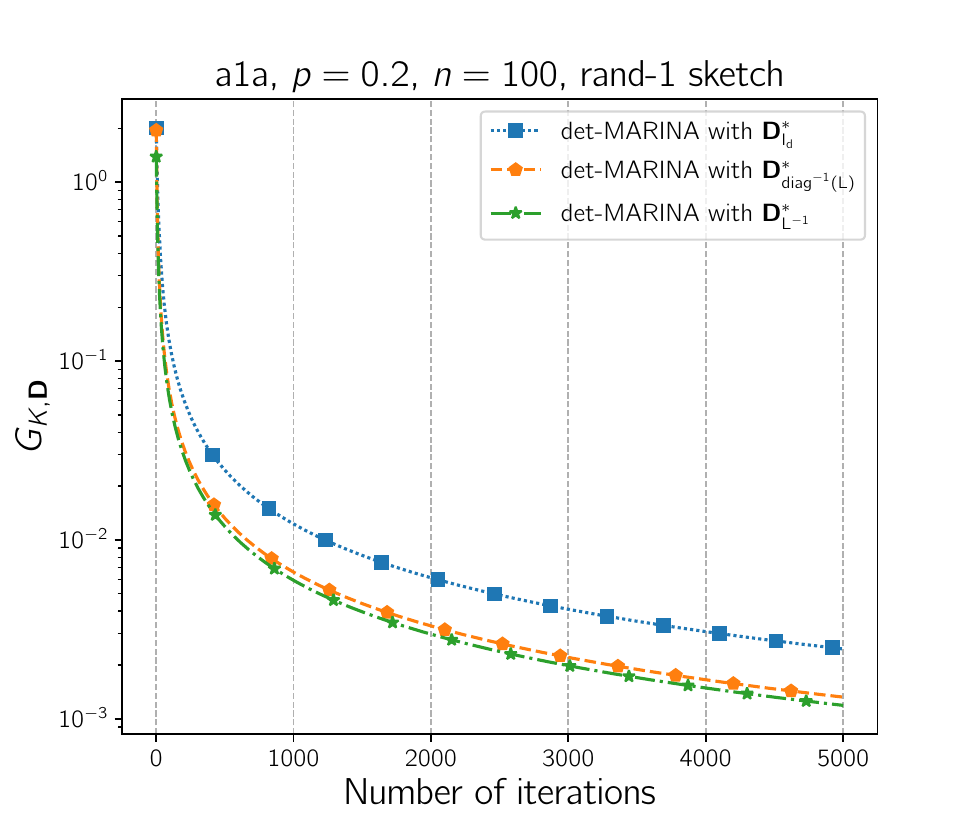}
         \includegraphics[width=0.32\textwidth]{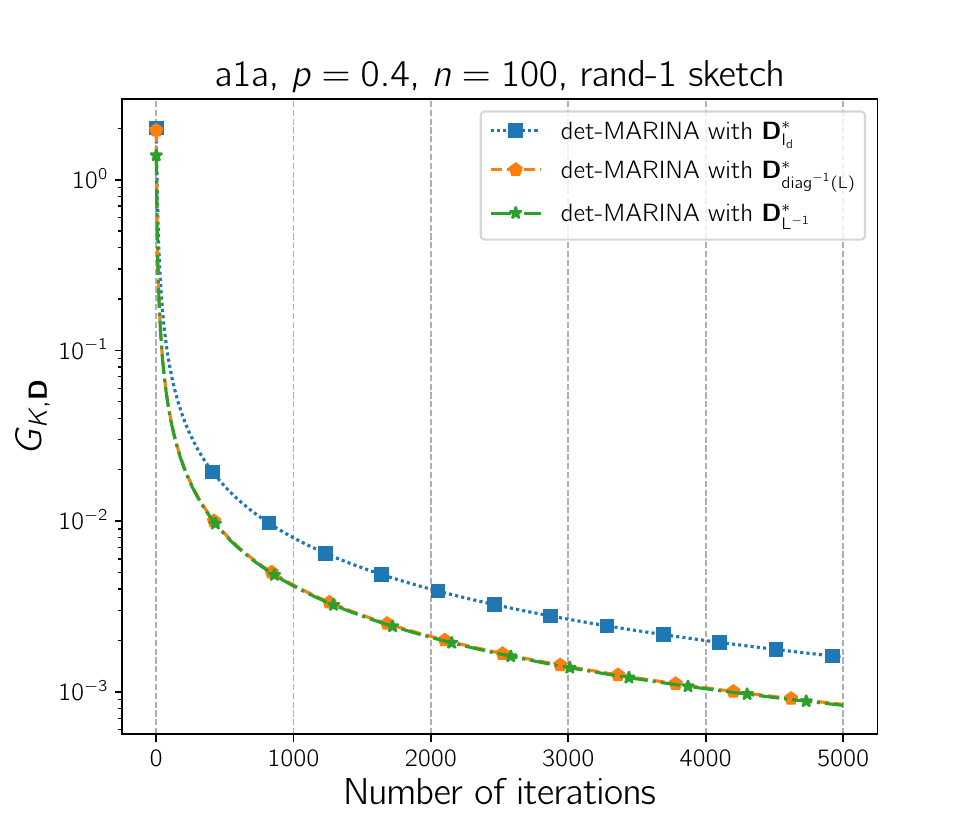}
      \end{minipage}
   }

   \subfigure{
      \begin{minipage}[t]{0.98\textwidth}
         \includegraphics[width=0.32\textwidth]{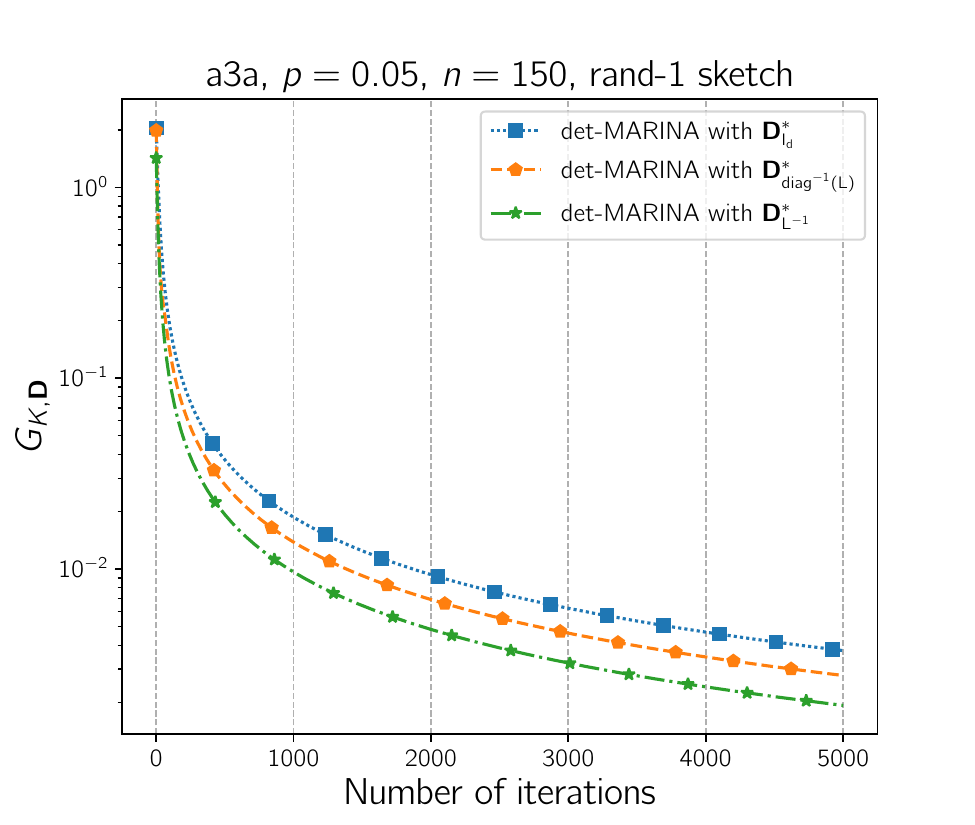} 
         \includegraphics[width=0.32\textwidth]{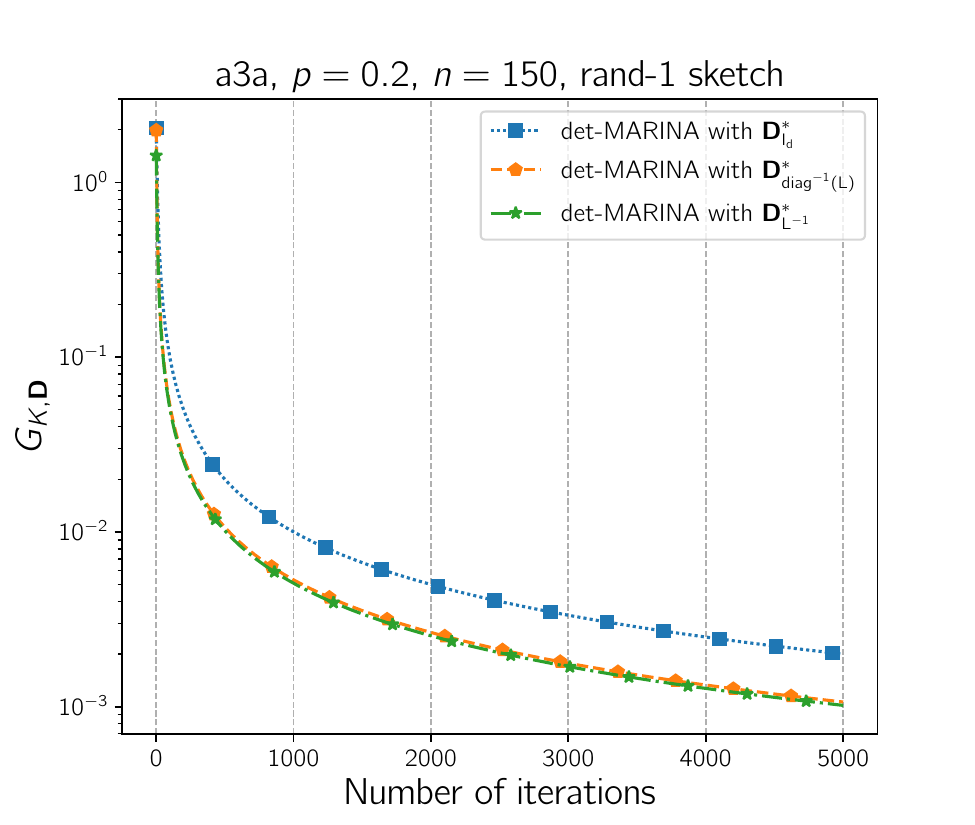}
         \includegraphics[width=0.32\textwidth]{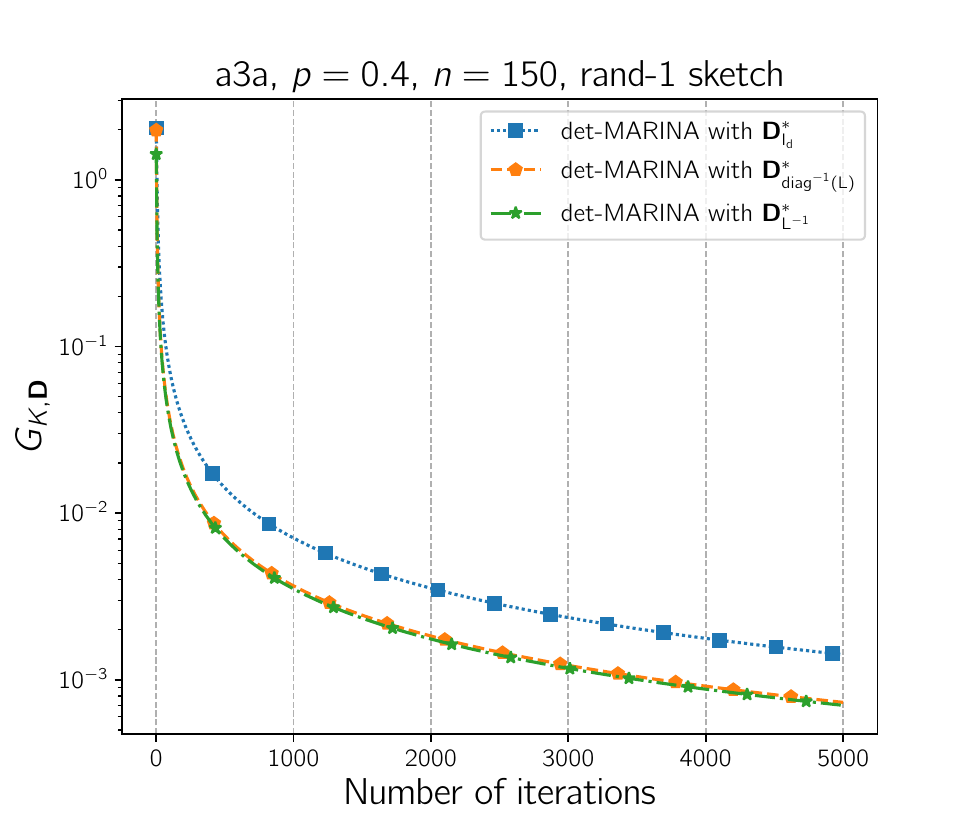}
      \end{minipage}
   }

   \subfigure{
      \begin{minipage}[t]{0.98\textwidth}
         \includegraphics[width=0.32\textwidth]{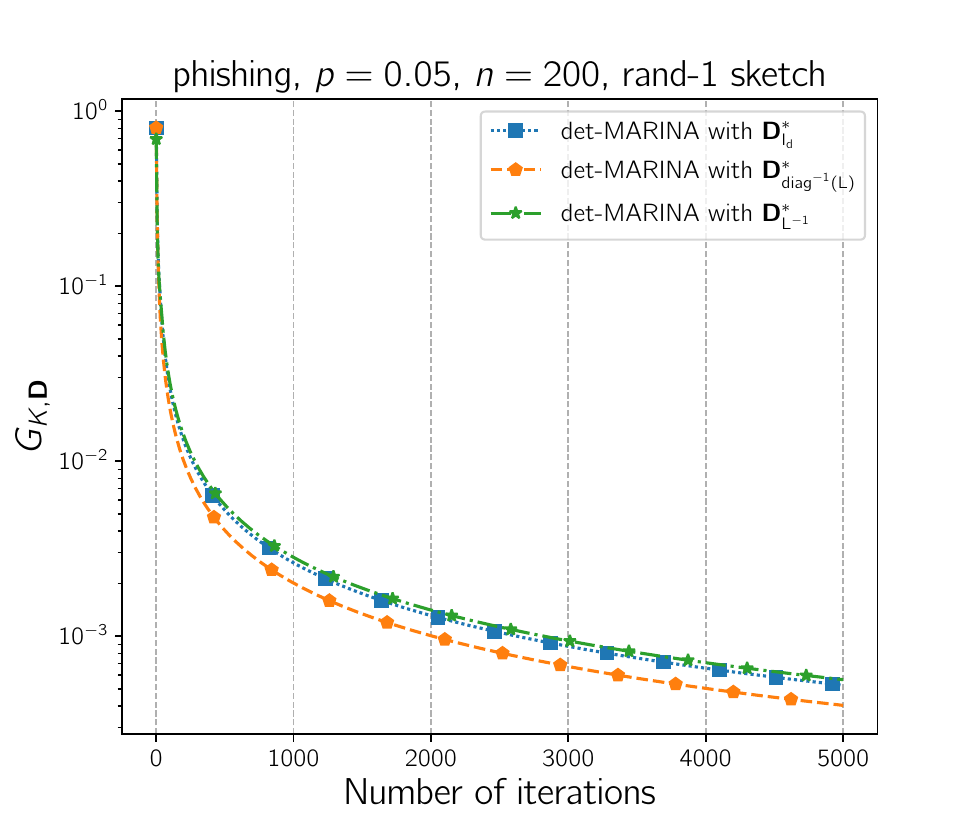} 
         \includegraphics[width=0.32\textwidth]{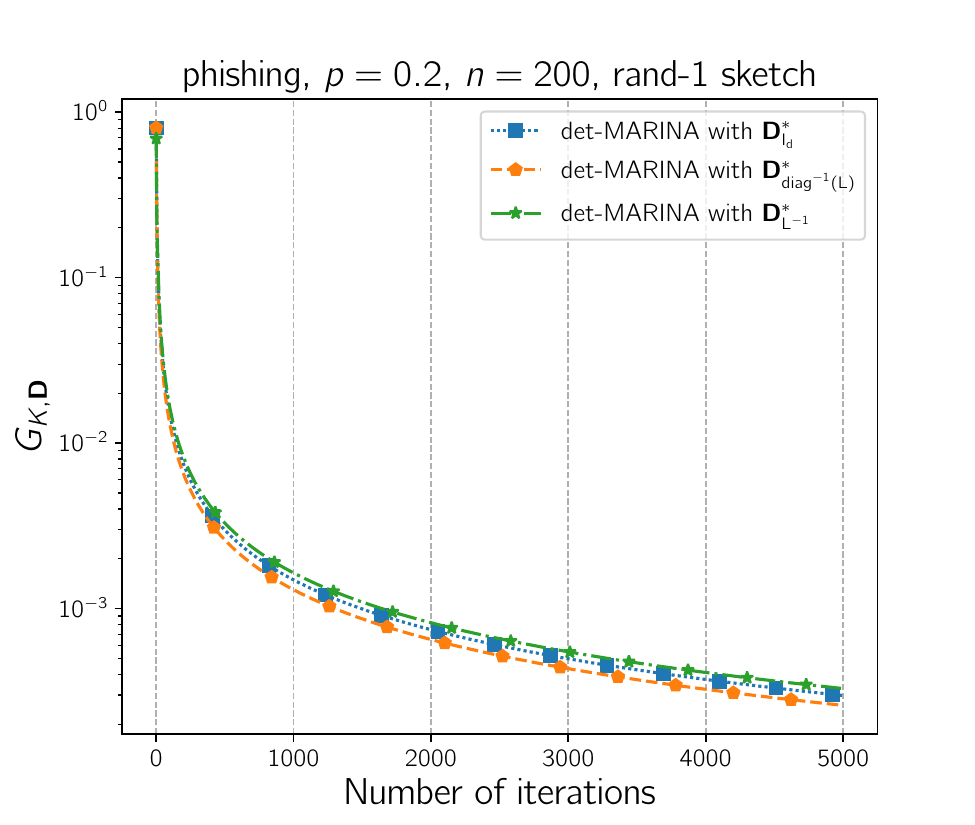}
         \includegraphics[width=0.32\textwidth]{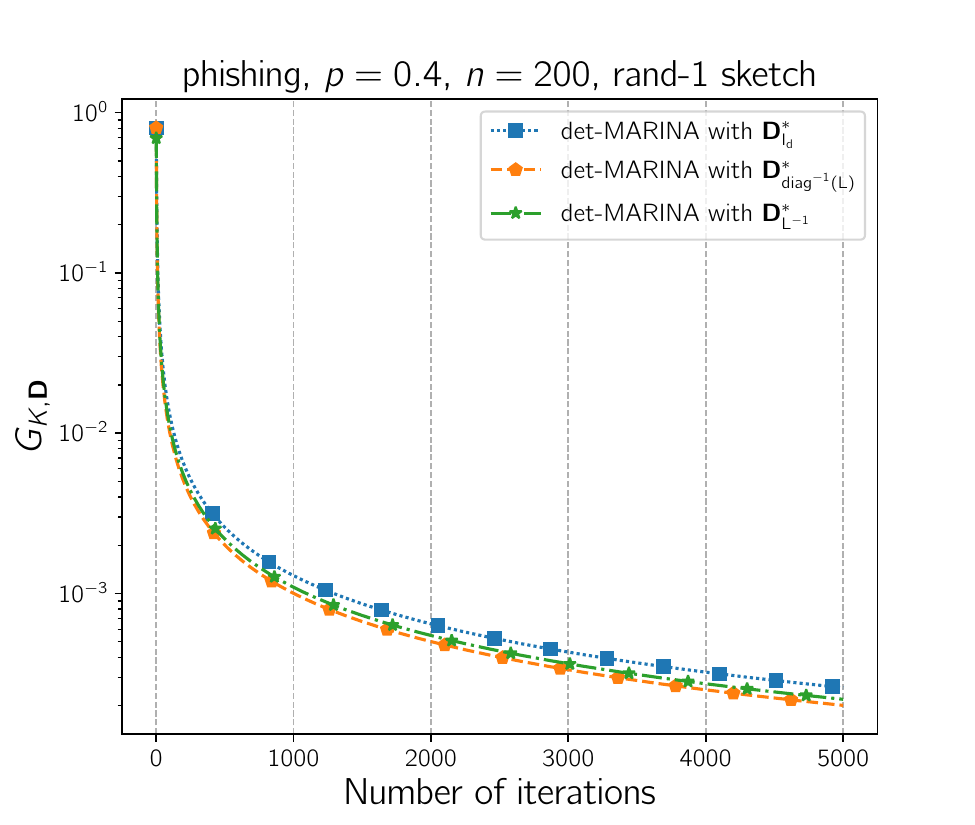}
      \end{minipage}
   }

   \caption{Comparison of {\detmarina} with matrix stepsize $\mD^*_{\mI_d}$, $\mD^*_{\diag^{-1}(\mL)}$ and $\mD^*_{\mL^{-1}}$. The stepsizes are defined in \eqref{eq:var-D-opt}. Throughout the experiment, $\lambda = 0.3$ is fixed, Rand-$1$ sketch is used in all cases. The notation $p$ in the title indicates the probability of sending the true gradient for {\detmarina}, $n$ denotes the number of clients considered.}
   \label{fig:experiment-4}
\end{figure}

We can observe from \Cref{fig:experiment-4} that, in almost all cases {\detmarina} with stepsize $\mD^*_{\diag^{-1}(\mL)}$ and $\mD^*_{\mL^{-1}}$ outperforms {\detmarina} with $\mD^*_{\mI_d}$. 
As {\detmarina} with $\mD^*_{\mI_d}$ can be viewed as {\marina} using scalar stepsize but under matrix Lipschitz gradient assumption, this demonstrates the effectiveness of using a matrix stepsize over the scalar stepsize. 
However, in \Cref{fig:experiment-4}, there are cases where {\detmarina} with $\mD^*_{\diag^{-1}(\mL)}$ outperforms $\mD^*_{\mL^{-1}}$. 
This tells us the two stepsizes are perhaps incomparable in general cases. 
This is similar to {\detcgd}, where optimal stepsizes with respect to a subspace associated with a fixed $\mW^{-1}$ are incomparable.

\subsection{Comparing communication complexity}

In this section, we perform an experiment on how different probabilities $p$ will affect the overall communication complexity of {\detmarina}. 
We use $\mD_{\mL^{-1}}^*$ as the stepsize, which is determined with respect to the sketch used. Rand-$\tau$ sketches are used in these experiments, and we vary the minibatch size $\tau$ to provide a more comprehensive comparison. 
For Rand-$\tau$ sketch $\mS$ and any $\mA \in \bbS^d_{++}$, one can show that 
\begin{equation}
   \label{eq:Rand-tau-exp}
   \Exp{\mS\mA\mS} = \frac{d}{\tau}\left(\frac{d - \tau}{\d - 1}\diag(\mA) + \frac{\tau - 1}{d - 1}\mA\right).
\end{equation}
Combining \eqref{eq:Rand-tau-exp} and \eqref{eq:var-D-opt}, we can find out the corresponding matrix stepsize easily. In the experiment, a fixed number of iterations ($K = 5000$) is performed for each {\detmarina} with the corresponding stepsize.

\begin{figure}[t]
   \centering
   \subfigure{
      \begin{minipage}[t]{0.98\textwidth}
         \includegraphics[width=0.48\textwidth]{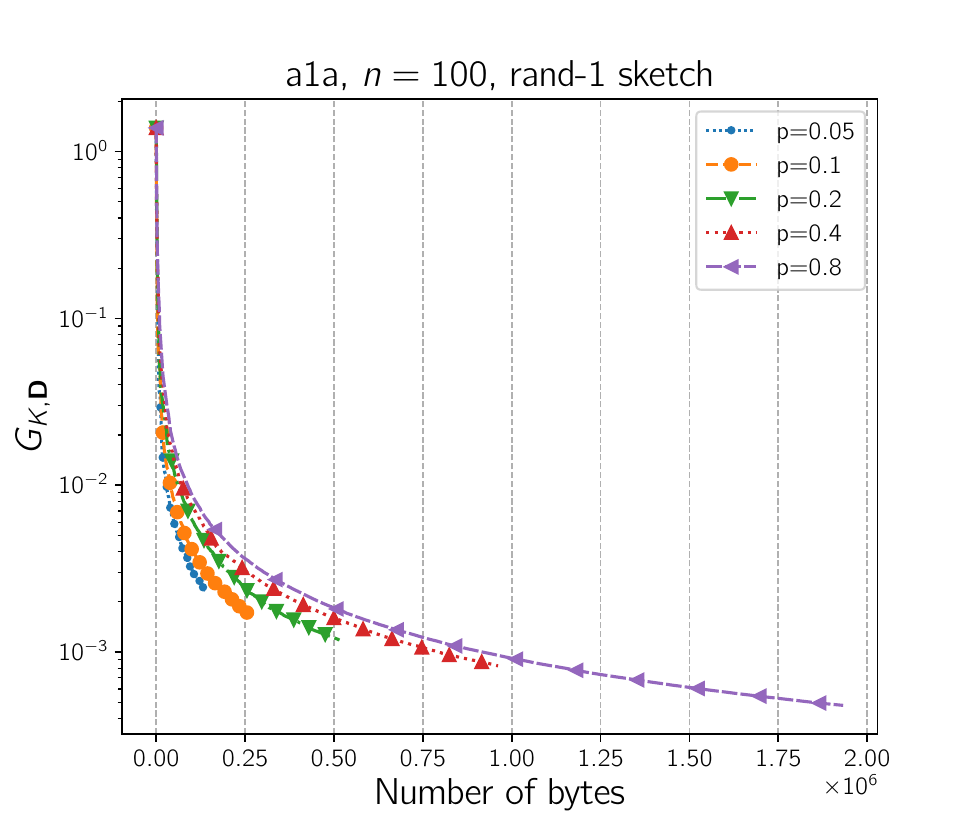} 
         \includegraphics[width=0.48\textwidth]{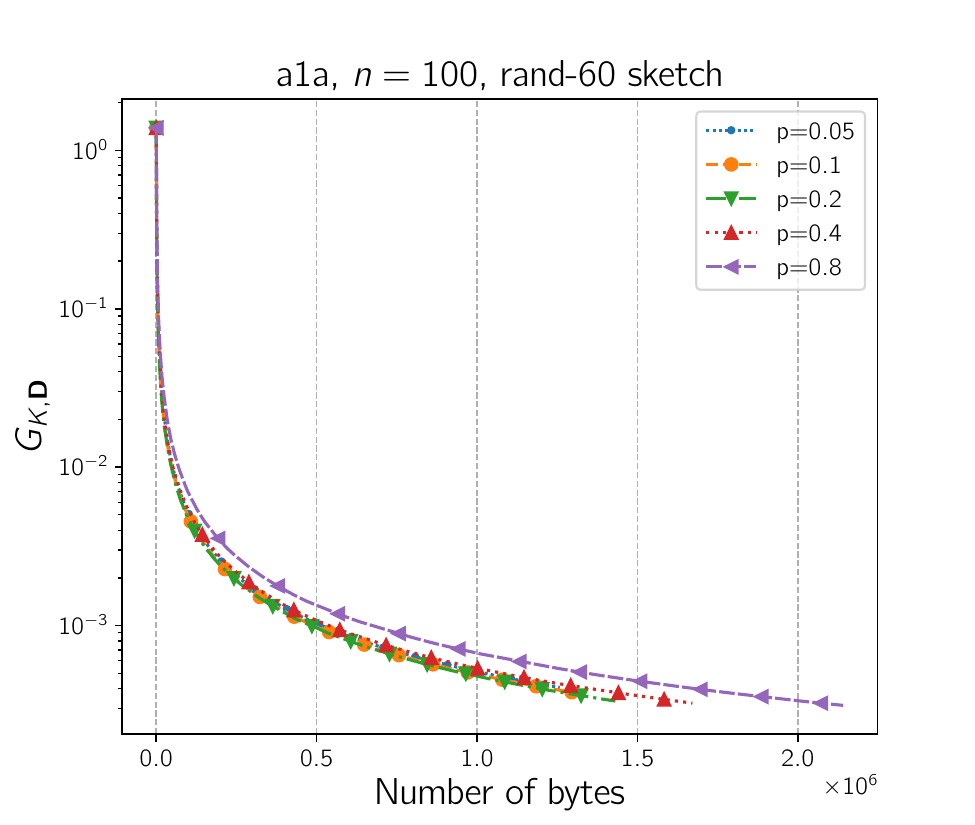}
      \end{minipage}
   }

   \subfigure{
      \begin{minipage}[t]{0.98\textwidth}
         \includegraphics[width=0.48\textwidth]{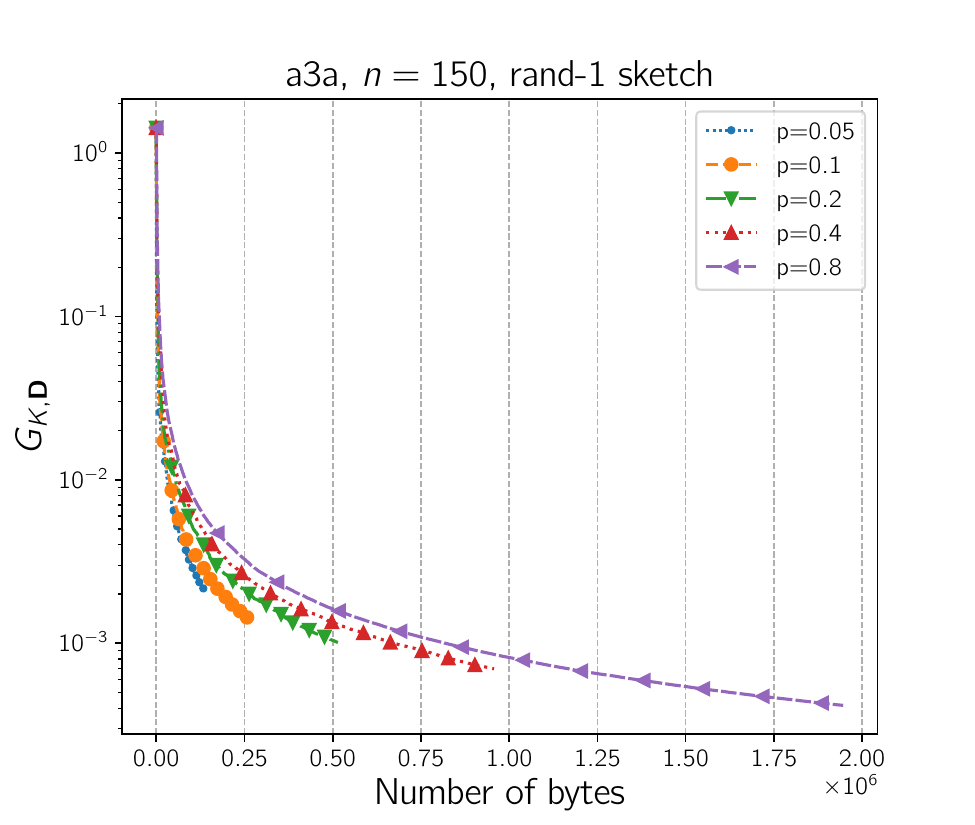} 
         \includegraphics[width=0.48\textwidth]{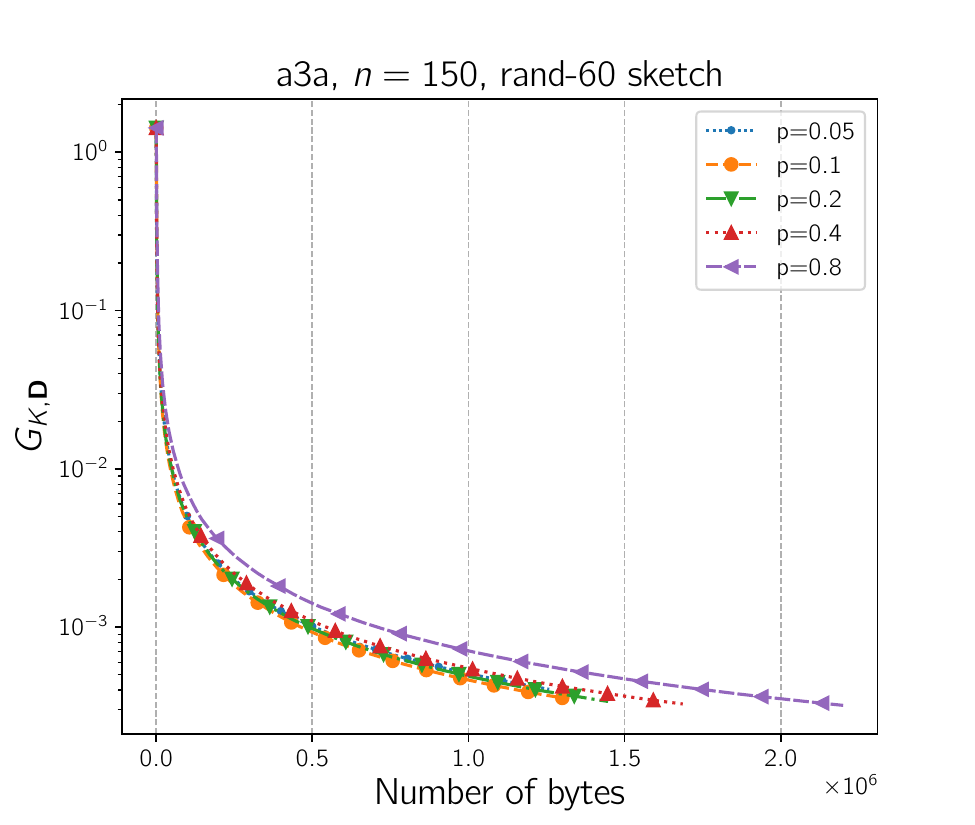}
      \end{minipage}
   }

   \subfigure{
      \begin{minipage}[t]{0.98\textwidth}
         \includegraphics[width=0.48\textwidth]{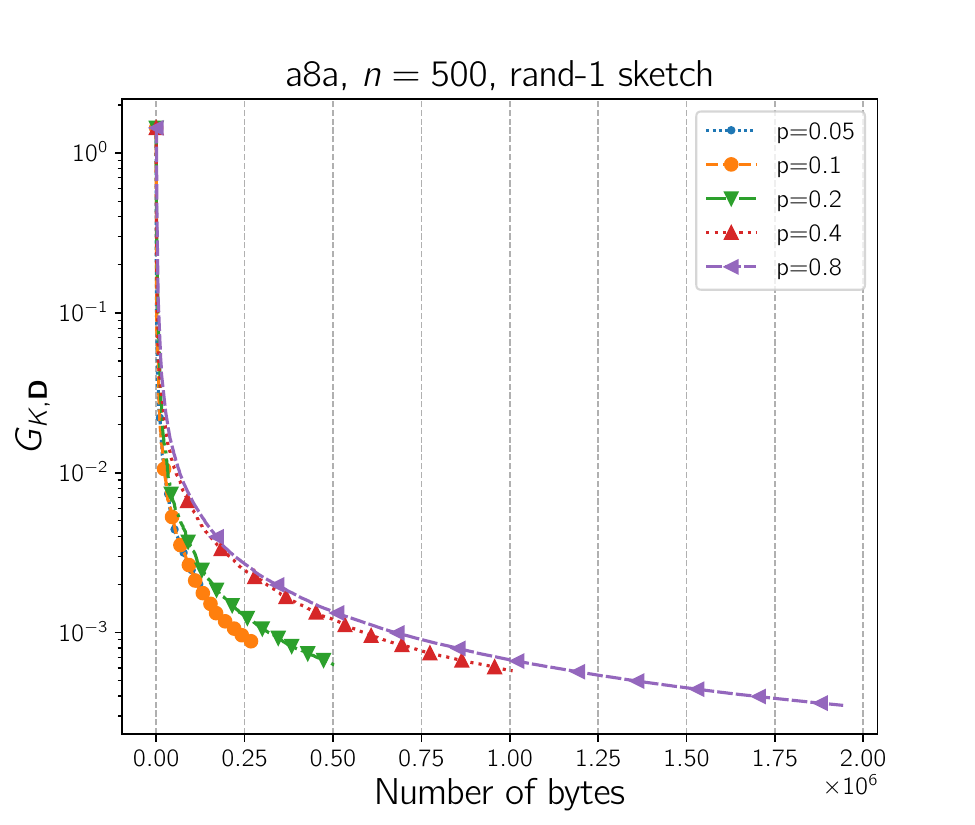} 
         \includegraphics[width=0.48\textwidth]{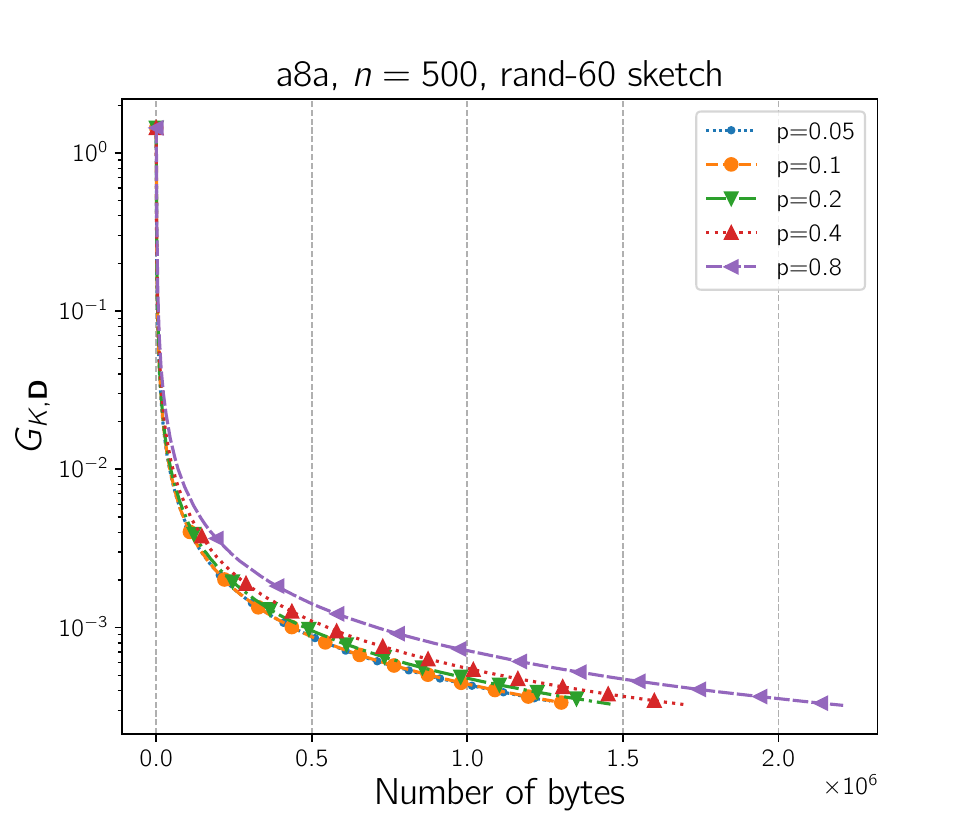}
      \end{minipage}
   }

   \caption{Comparison of {\detmarina} with stepsize $\mD^*_{\mL^{-1}}$ using different probability $p$. The probability $p$ here is chosen from the set $\{0.05, 0.1, 0.2, 0.4, 0.8\}$. The notation $n$ in the title denote the number of clients considered. The $x$-axis is now the number of bytes sent from a single node to the server. In each case, {\detmarina} is run for a fixed number of iterations $K = 5000$.}
   \label{fig:experiment-5}
\end{figure}

As it can be observed from \Cref{fig:experiment-5}, in each dataset, the communication complexity tends to increase with the increase of probability $p$. 
However, when the number of iteration is fixed, a larger $p$ often means a faster rate of convergence. 
This difference in communication complexity is more obvious when we are using the Rand-$1$ sketch. 
In real federated learning settings, there is often constraints on network bandwidth from clients to the server. Thus, trading off between communication complexity and iteration complexity, i.e. selecting the compression mechanism carefully to guarantee a acceptable speed that satisfies the bandwidth constraints, becomes important.

\subsection{\texorpdfstring{Comparison of {\dasha} and {\detdasha}}{Comparison of DASHA and det-DASHA}}
\label{dasha:subsec:exp:1}
In this experiment we plan to compare the performance of original {\dasha} with {\detdasha}. Throughout the experiments, $\lambda$ is fixed at $0.3$. The same Rand-$\tau$ sketch is used in the two algorithms. The stepsize condition on {\dasha} when the momentum is set as $a = \frac{1}{2\omega + 1}$ is given as 
\begin{equation*}
  \gamma_4 \leq \left(L + \sqrt{\frac{16\omega(2\omega + 1)}{n}}\widehat{L}\right)^{-1},
\end{equation*}
according to Theorem 6.1 of \citet{tyurin2024dasha}. 
Here the $L$ is the smoothness constant of the function $f$, while $\widehat{L}$ satisfies $\widehat{L}^2 = \frac{1}{n}\sum_{i=1}^{n}L_i^2$ where $L_i$ is the smoothness constant of local objective $f_i$. 
In theory we can pick $\widehat{L} = L$. 
Similarly, according to \Cref{dasha:col:scaling}, the optimal stepsize matrix $\mD^{**}_{\mL^{-1}}$ is given as 
\begin{equation}
  \label{dasha:eq:exp:ss-cond-detdasha}
  \mD^{**}_{\mL^{-1}} = \frac{2}{1 + \sqrt{1 + 16C_{\mL^{-1}}\cdot\lambda_{\min}\left(\mL\right)}} \cdot \mL^{-1},
\end{equation}
when the momentum is given as $a = \frac{1}{2\omega_{\mD} + 1}$. We compare the performance of {\dasha} with $\omega$ and {\detdasha} with $\mD^{**}_{\mL^{-1}}$ using the same sketch where the total number of clients are different.

\begin{figure}[t]
   \centering
   \subfigure{
   \begin{minipage}[t]{0.98\textwidth}
         \includegraphics[width=0.32\textwidth]{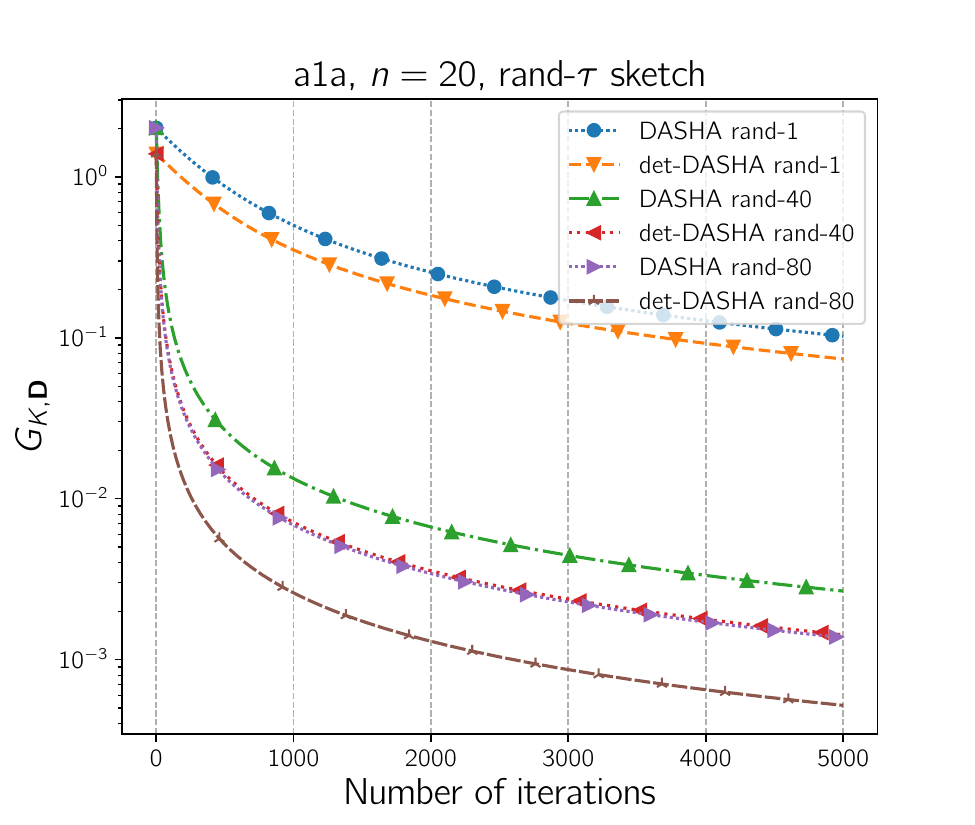} 
         \includegraphics[width=0.32\textwidth]{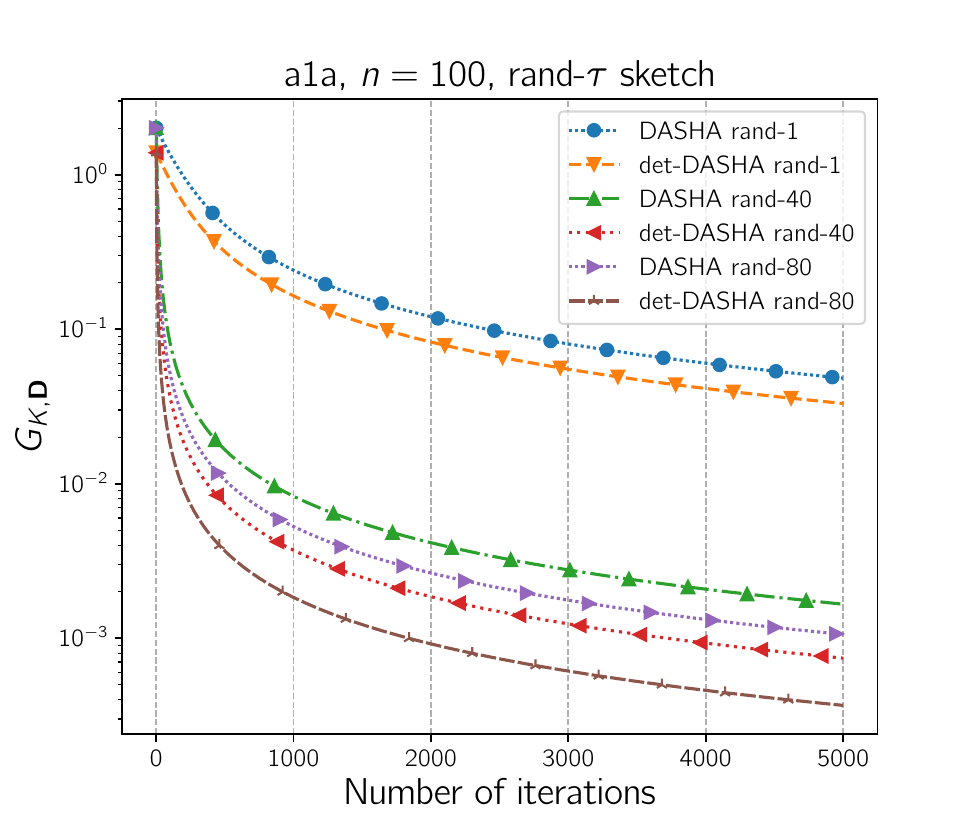}
         \includegraphics[width=0.32\textwidth]{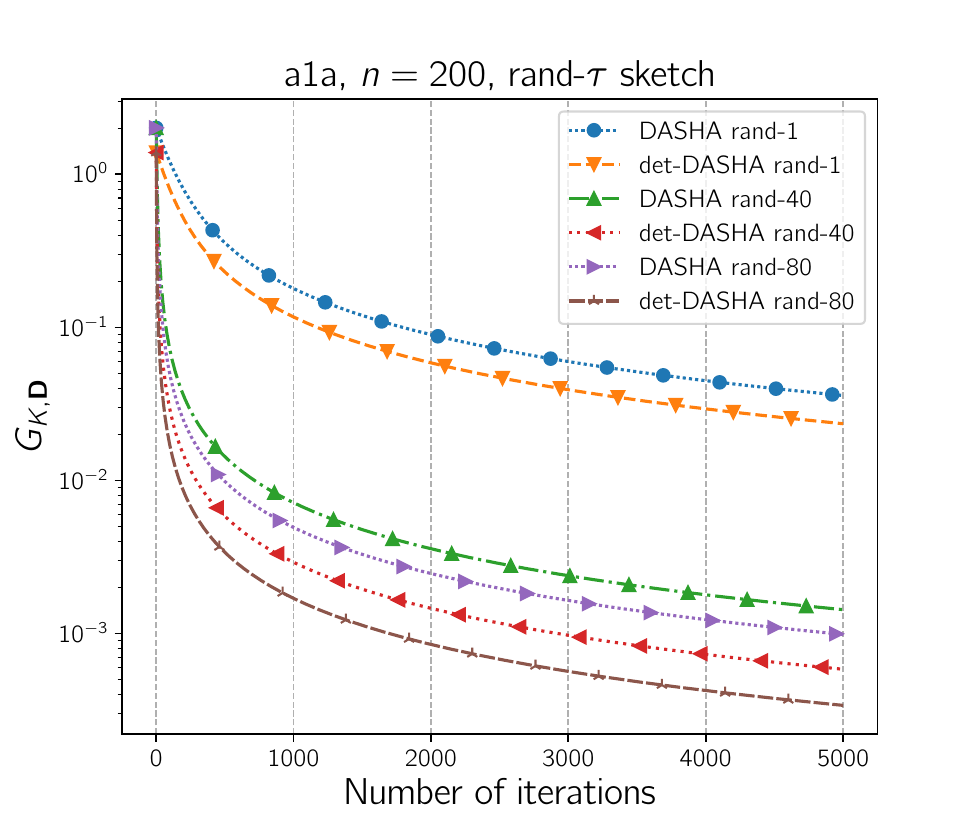}
      \end{minipage}
   }

   \subfigure{
   \begin{minipage}[t]{0.98\textwidth}
      \includegraphics[width=0.32\textwidth]{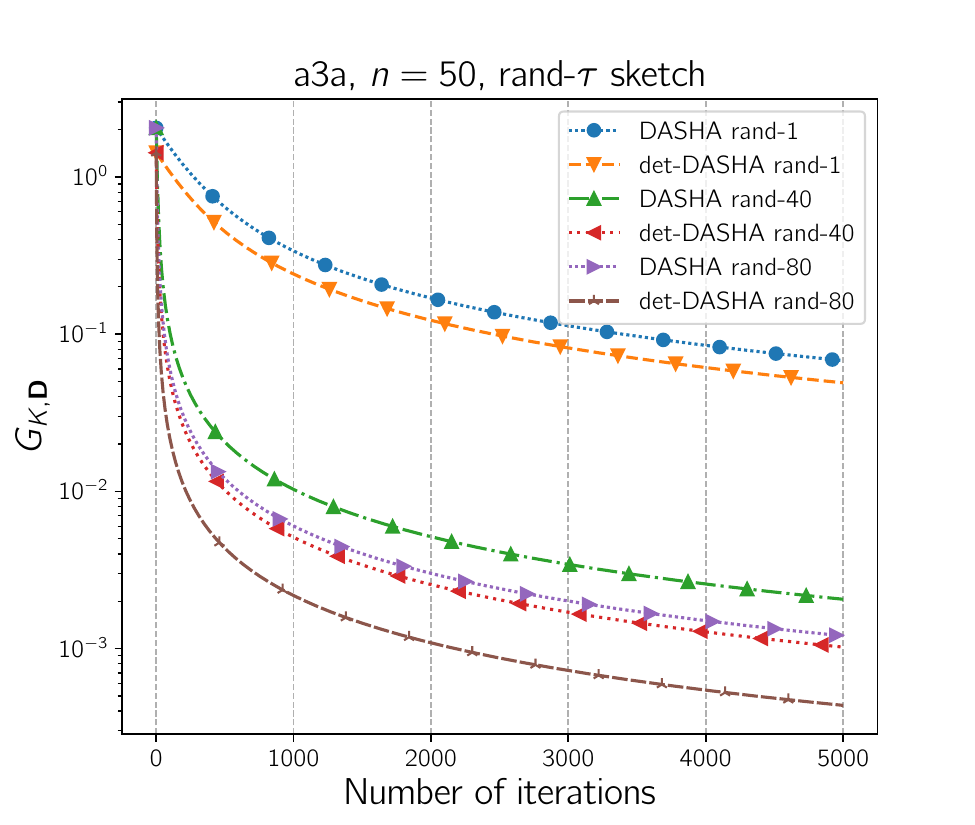} 
      \includegraphics[width=0.32\textwidth]{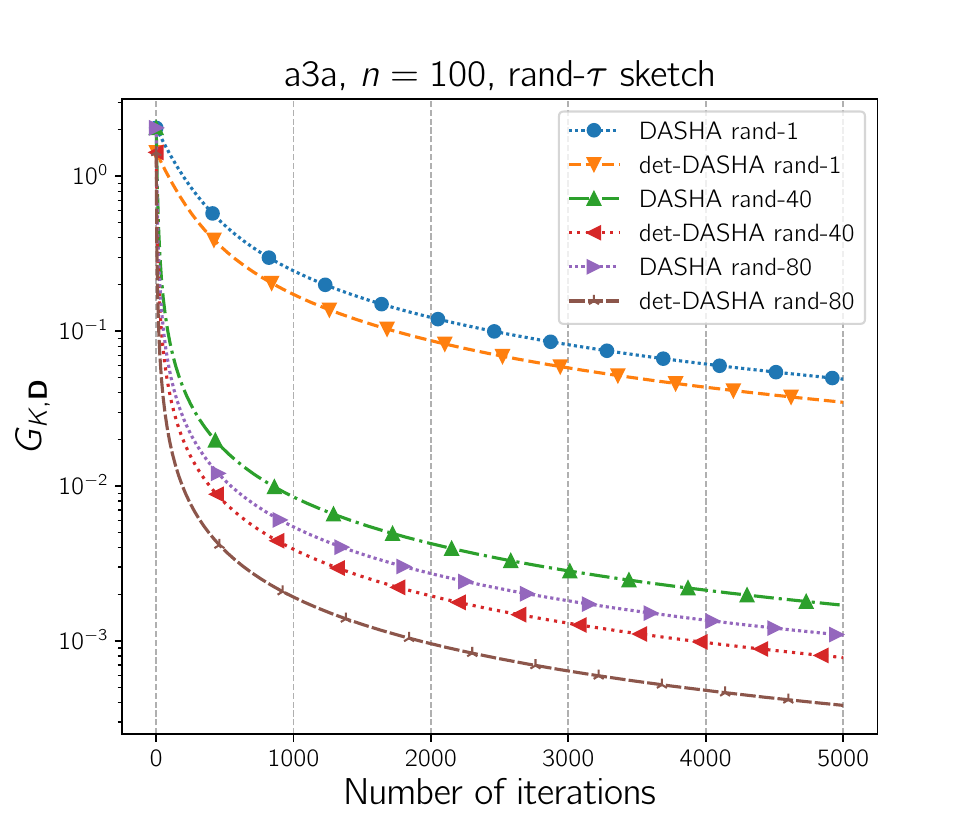}
      \includegraphics[width=0.32\textwidth]{./figure/Exp_dasha_1_client_200_lam_0.3_a1a.txt.pdf} 
      \end{minipage}
   }

   \subfigure{
   \begin{minipage}[t]{0.98\textwidth}
      \includegraphics[width=0.32\textwidth]{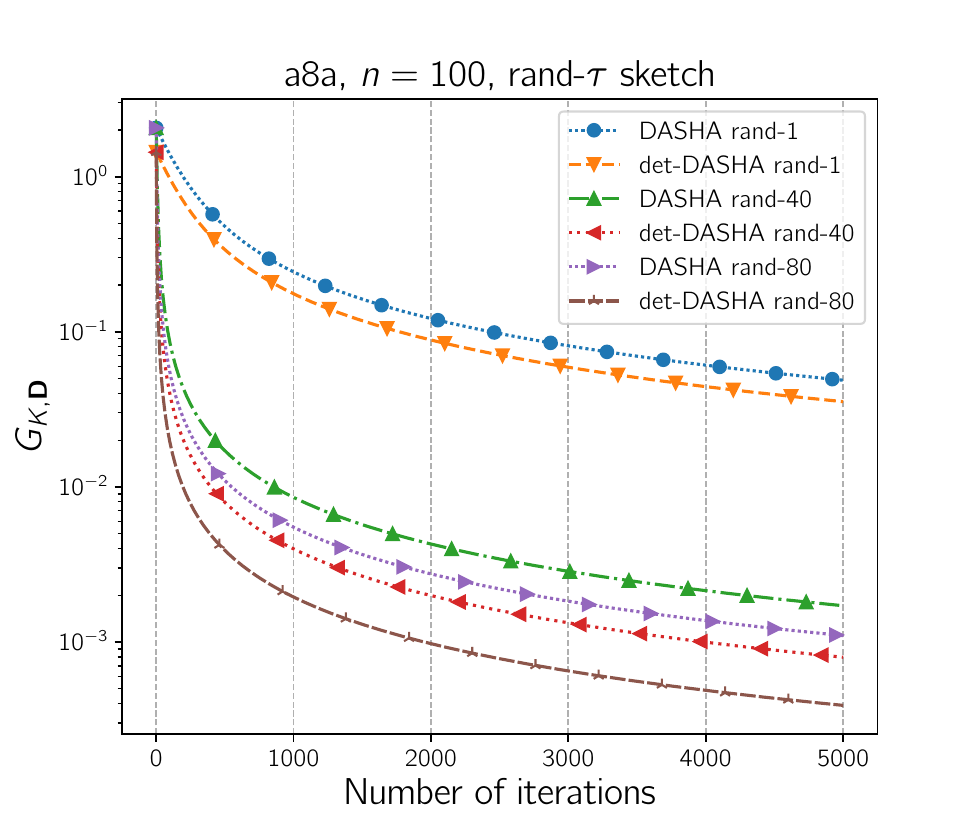} 
      \includegraphics[width=0.32\textwidth]{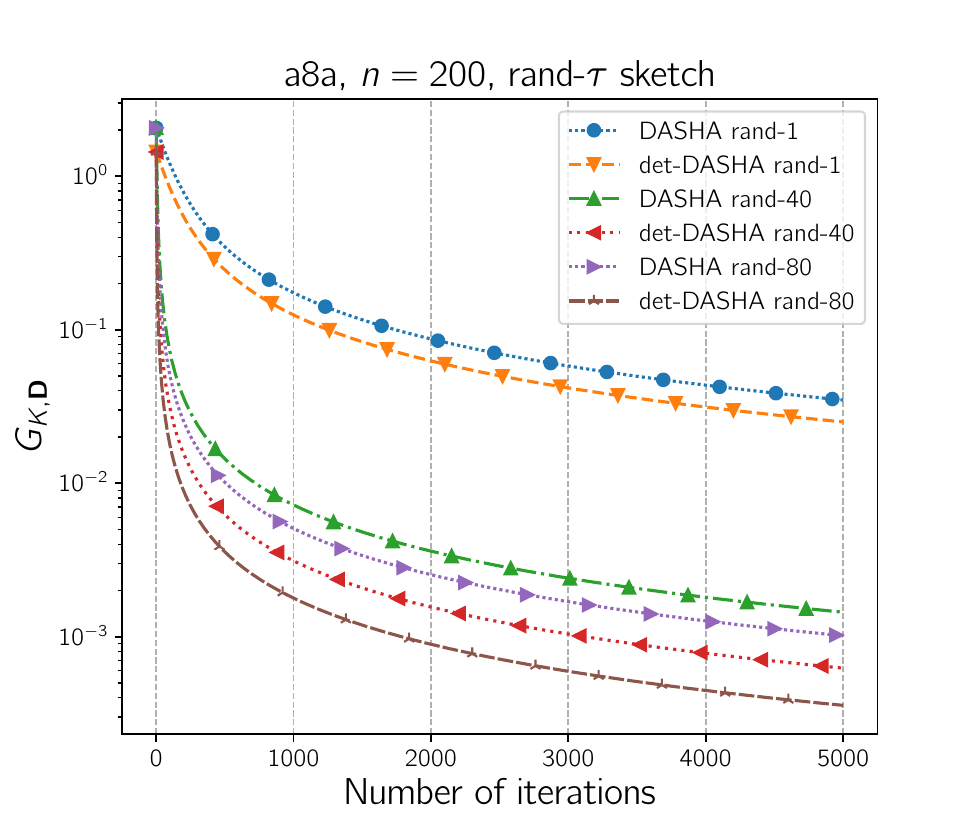}
      \includegraphics[width=0.32\textwidth]{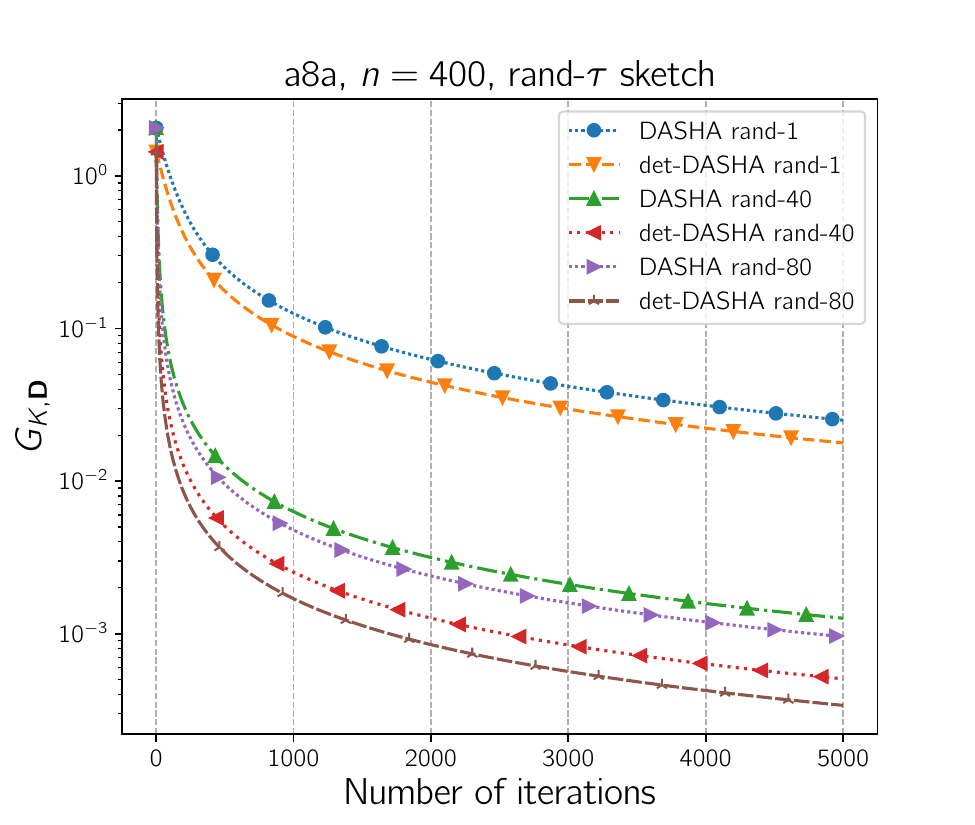} 
      \end{minipage}
   }

   \caption{Comparison of {\detdasha} with matrix stepsize $\mD^{**}_{\mL^{-1}}$ and {\dasha} with optimal scalar stepsize $\gamma$ using different Rand-$\tau$ sketches. $\lambda = 0.3$ is fixed throughout the experiments. Optimal stepsize is calculated in each case with respect to the sketch used. The $x$-axis denotes the number of iterations while the notation $G_{K, \mD}$ for the $y$-axis denotes the averaged matrix norm of the gradient. The notation $n$ denotes the number of clients in each setting.}
   \label{fig:experiment-1-dasha}
\end{figure}

As it can be observed in \Cref{fig:experiment-1-dasha}, {\detdasha} with matrix stepsize $\mD^{**}_{\mL^{-1}}$ outperforms {\dasha} with optimal scalar stepsize using the same sketch in every setting we considered. Note that since the same sketch is used in the two algorithm, the number of bits transferred in each iteration is also the same for the two algorithms. This essentially indicates that {\detdasha} has better iteration complexity as well as communication complexity than {\dasha} given that the same sketch is used for the two algorithm.

\subsection{\texorpdfstring{Comparison of {\dcgd}, {\detcgd}, {\dasha} and {\detdasha}}{Comparison of DCGD, det-CGD, DASHA and det-DASHA}}
In this experiment, we consider the comparison between the two non variance reduced methods {\dcgd}, {\detcgd} and the two variance reduced method {\dasha}, {\detdasha}. The stepsize choices for {\dcgd} and {\detcgd} have already been discussed in the previous sections, (for {\dcgd} we use $\gamma_2$ and for {\detcgd} we use $\mD^*_{3}$) ,for {\dasha} and {\detdasha}, we use the stepsize choices of \Cref{dasha:subsec:exp:1}. Note that $\varepsilon^2$ is set as $0.01$, and $\lambda$ is fixed at $0.9$ here. Throughout this experiment, we consider the case where Rand-$\tau$ sketch is used in the four algorithms.

\begin{figure}[t]
   \centering
   \subfigure{
   \begin{minipage}[t]{0.98\textwidth}
         \includegraphics[width=0.32\textwidth]{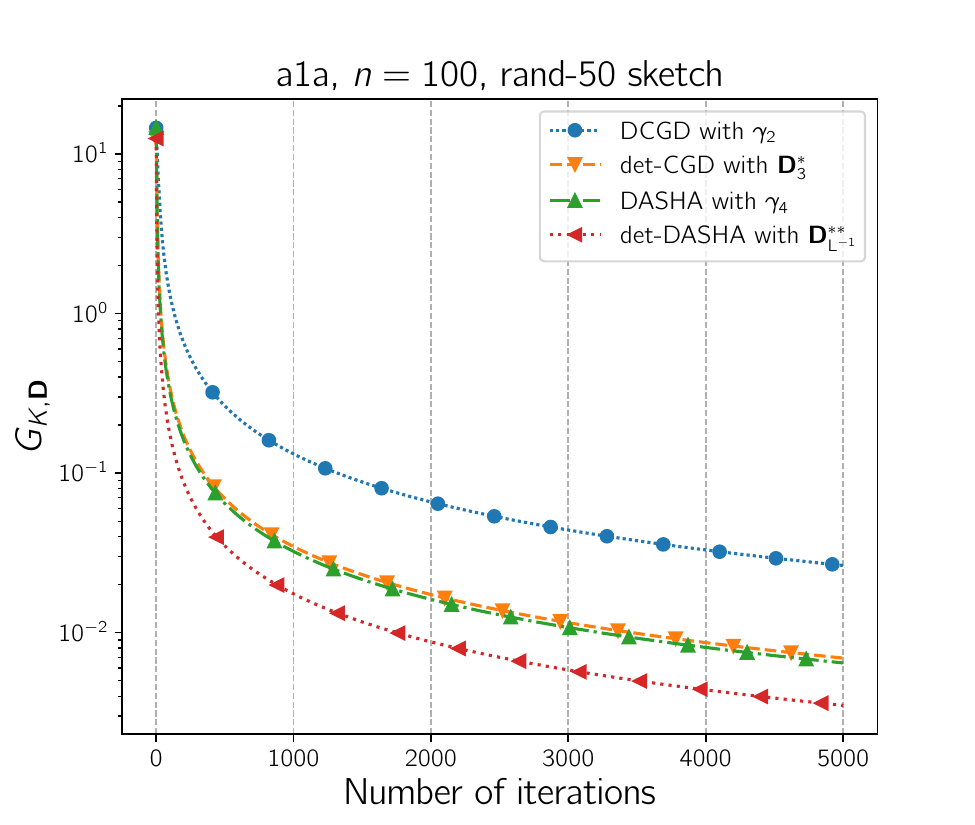} 
         \includegraphics[width=0.32\textwidth]{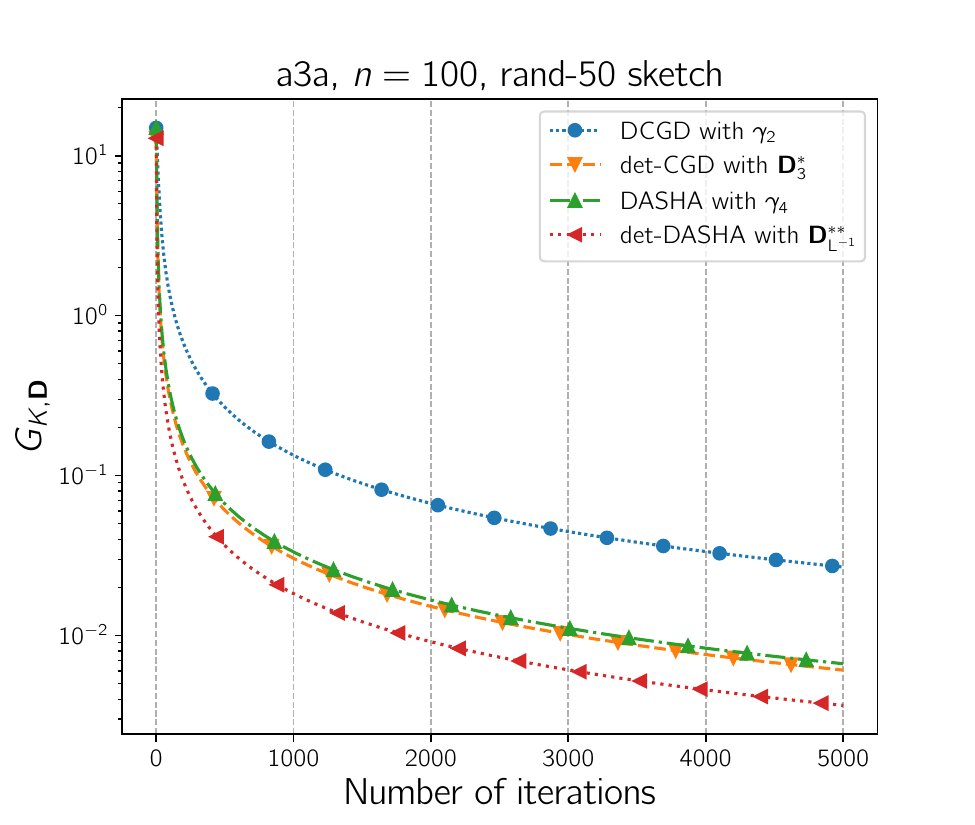}
         \includegraphics[width=0.32\textwidth]{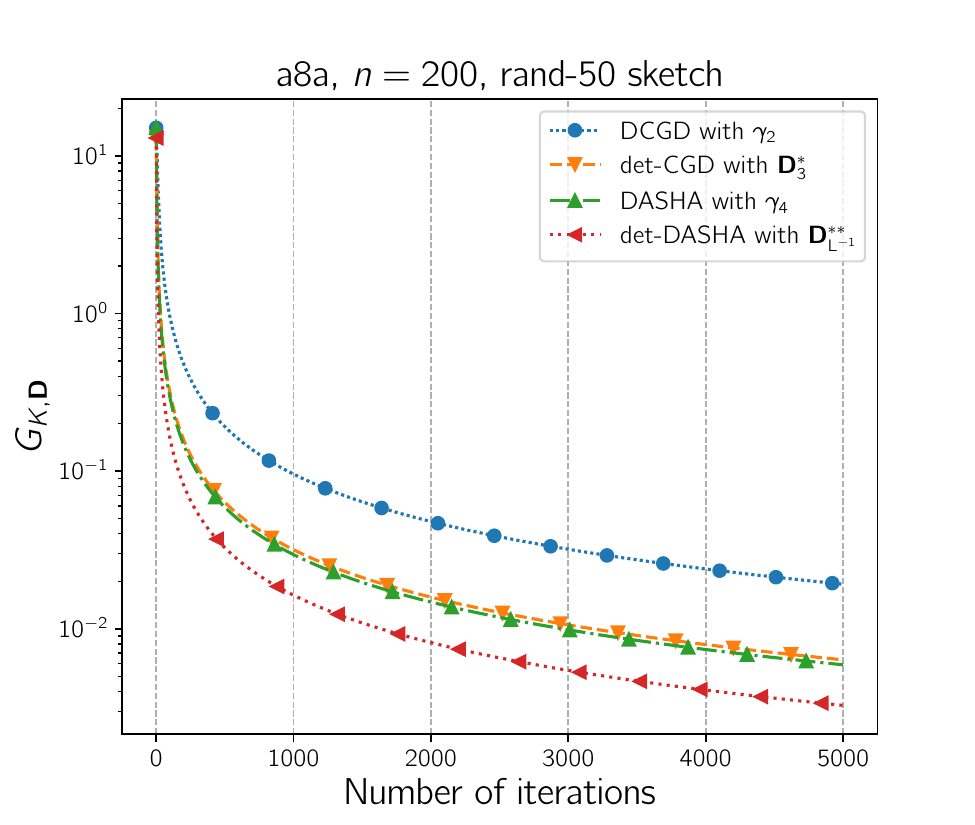}
      \end{minipage}
   }
   \subfigure{
   \begin{minipage}[t]{0.98\textwidth}
         \includegraphics[width=0.32\textwidth]{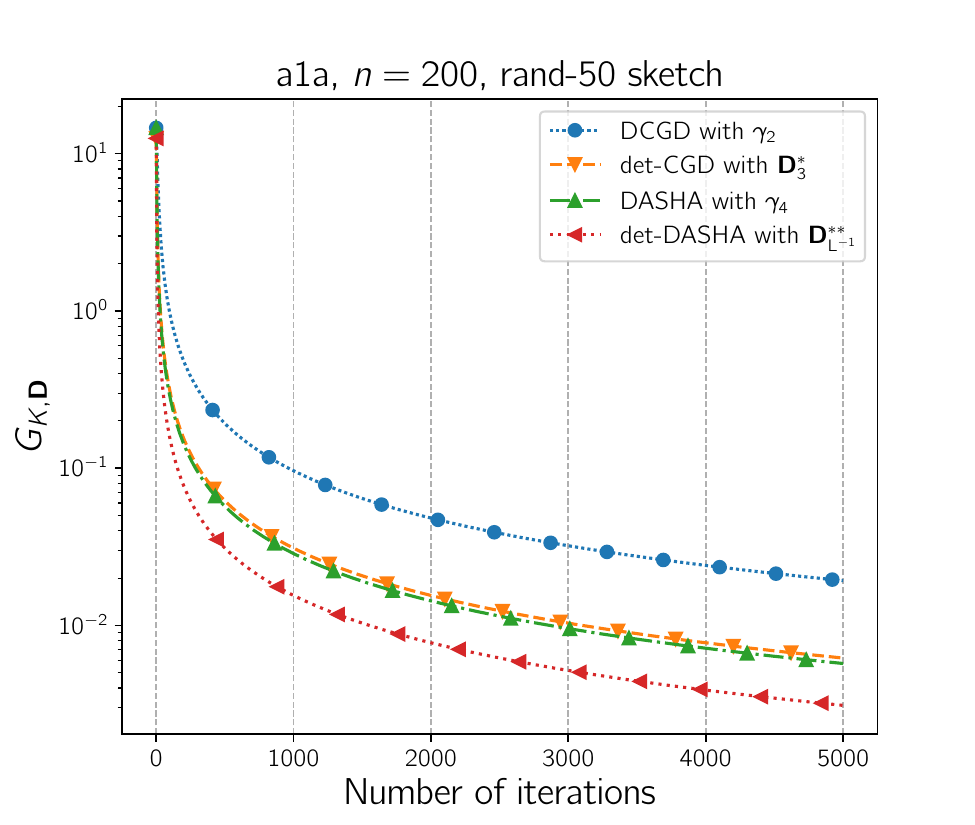} 
         \includegraphics[width=0.32\textwidth]{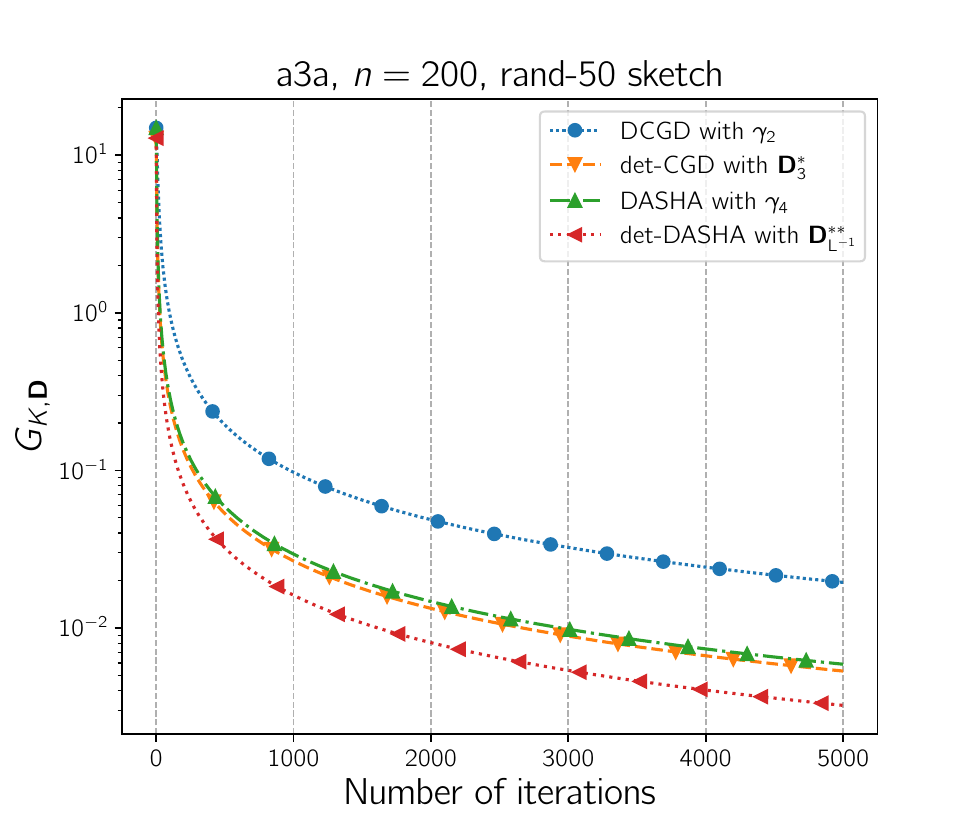}
         \includegraphics[width=0.32\textwidth]{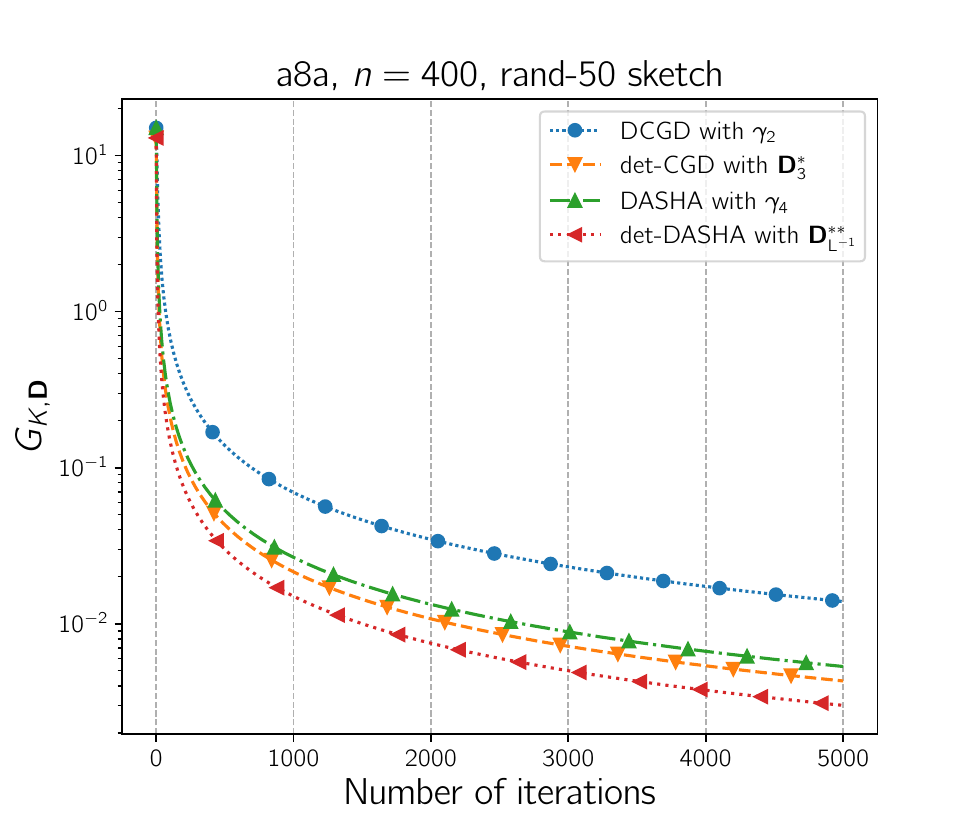}
      \end{minipage}
   }
   \caption{Comparison of {\dcgd} with optimal scalar stepsize $\gamma_2$, {\detcgd} with optimal diagonal stepsize $\mD_3^*$, {\dasha} with optimal scalar stepsize $\gamma_1$, and {\detdasha} with optimal stepsize $\mD^{**}_{\mL^{-1}}$ with respect to $\mW = \mL^{-1}$. $\lambda=0.9$ is fixed throughout the experiment. The notation $n$ in the title indicates the number of clients in each case. Rand-$\tau$ sketch with $\tau = 50$ are used in all four algorithms.}
   \label{fig:experiment-2-dasha}
\end{figure}

It is easy to observe that in each case of \Cref{fig:experiment-2-dasha}, {\detdasha} outperforms the rest of the algorithms. It is expected that {\detdasha} outperforms {\dasha}, as it is also illustrated by \Cref{fig:experiment-1-dasha}, which is a consequence of using matrix stepsize instead of a scalar stepsize. We also see that {\detdasha} and {\dasha} outperform {\detcgd} and {\dcgd} respectively, which demonstrate the advantages of the variance reduction technique. Note that in this case, all four algorithms are using the same sketch, which means that the number of bits transferred in each iteration is the same for the four algorithms, as a result, compared to the other algorithms, {\detdasha} is better in terms of both iteration complexity and communication complexity.

\subsection{\texorpdfstring{Comparison of {\detdasha} and {\detcgd} with different stepsizes}{Comparison of det-DASHA and det-CGD with different stepsizes}}
In this experiment, we try to compare {\detdasha} and {\detcgd} with different matrix stepsizes. Throughout this experiment, we will fix $\varepsilon^2=0.01$ and $\lambda=0.9$. The same Rand-$\tau$ sketch is used for the two algorithms. For {\detcgd}, we use the stepsize $\mD_1 = \gamma_{\mI_d}\cdot\mI_d, \mD_2 = \gamma_{\diag^{-1}\left(\mL\right)}\cdot\diag^{-1}\left(\mL\right)$ and $\mD_3 = \gamma_{\mL^{-1}} \cdot \mL^{-1}$, for {\detdasha} we use the stepsize $\mD^{**}_{\mL^{-1}}$.

\begin{figure}[t]
   \centering
   \subfigure{
   \begin{minipage}[t]{0.98\textwidth}
         \includegraphics[width=0.32\textwidth]{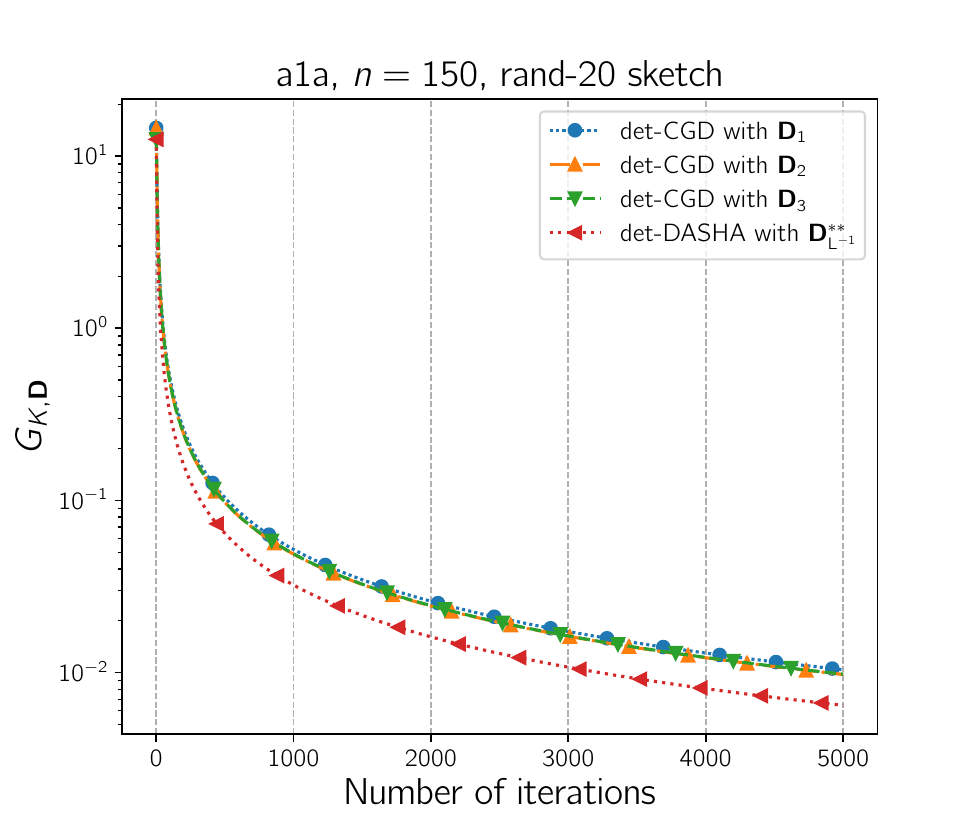}
         \includegraphics[width=0.32\textwidth]{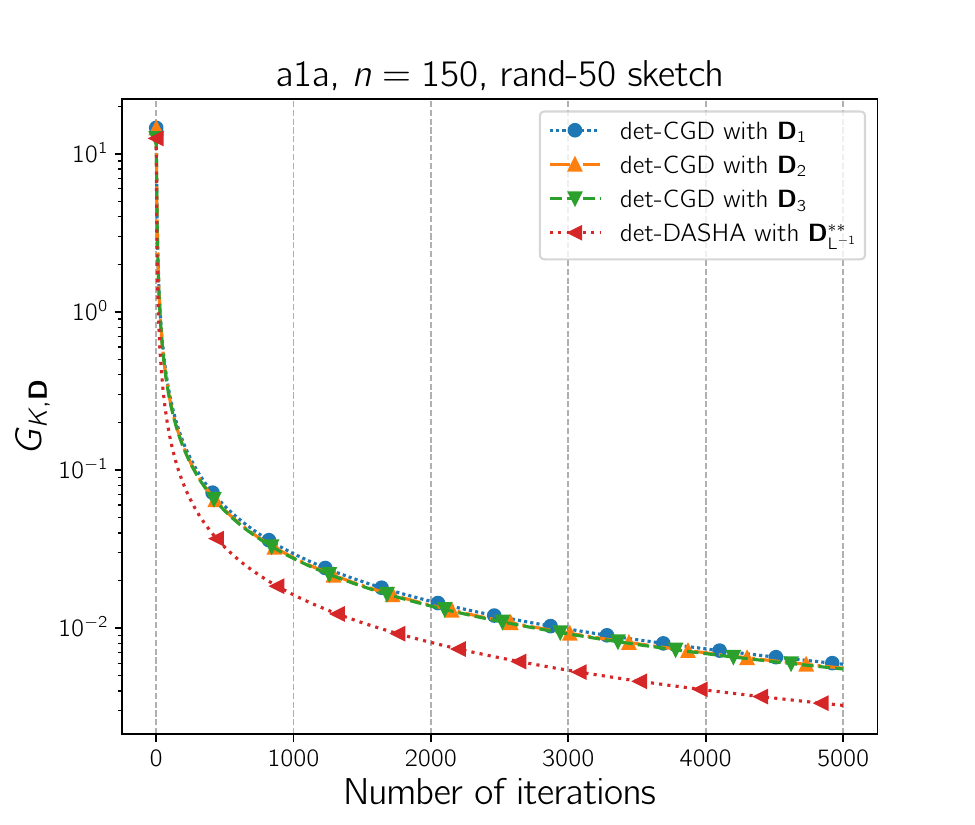}
         \includegraphics[width=0.32\textwidth]{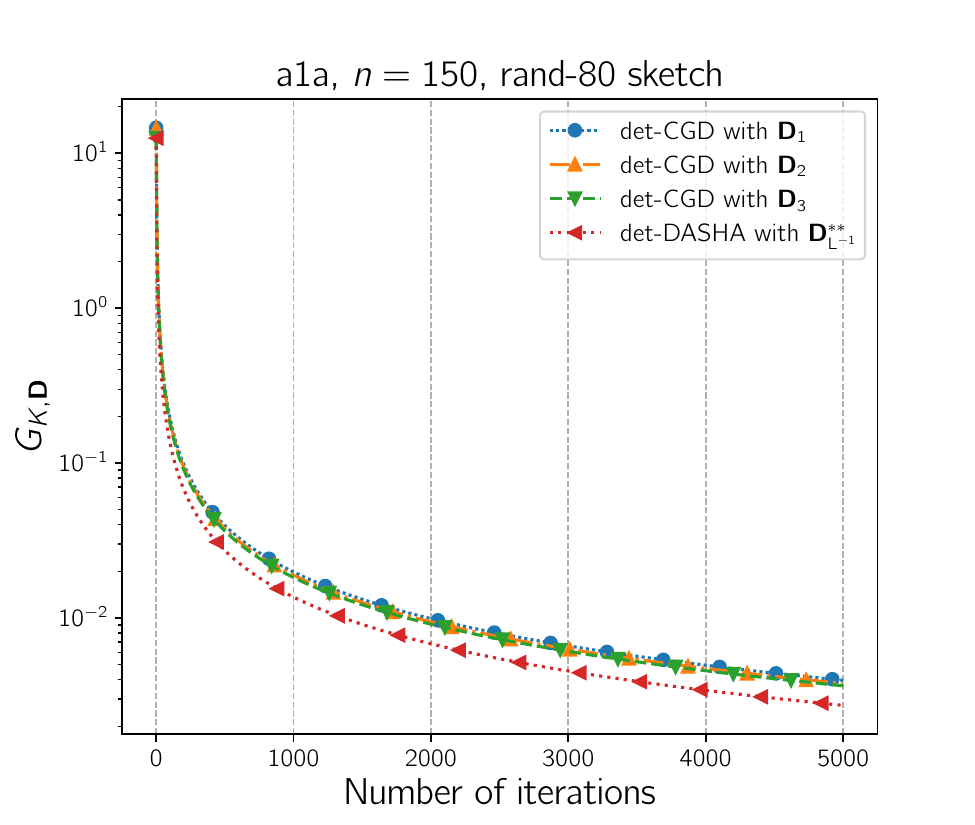} 
   \end{minipage}
   }
   \subfigure{
   \begin{minipage}[t]{0.98\textwidth}
         \includegraphics[width=0.32\textwidth]{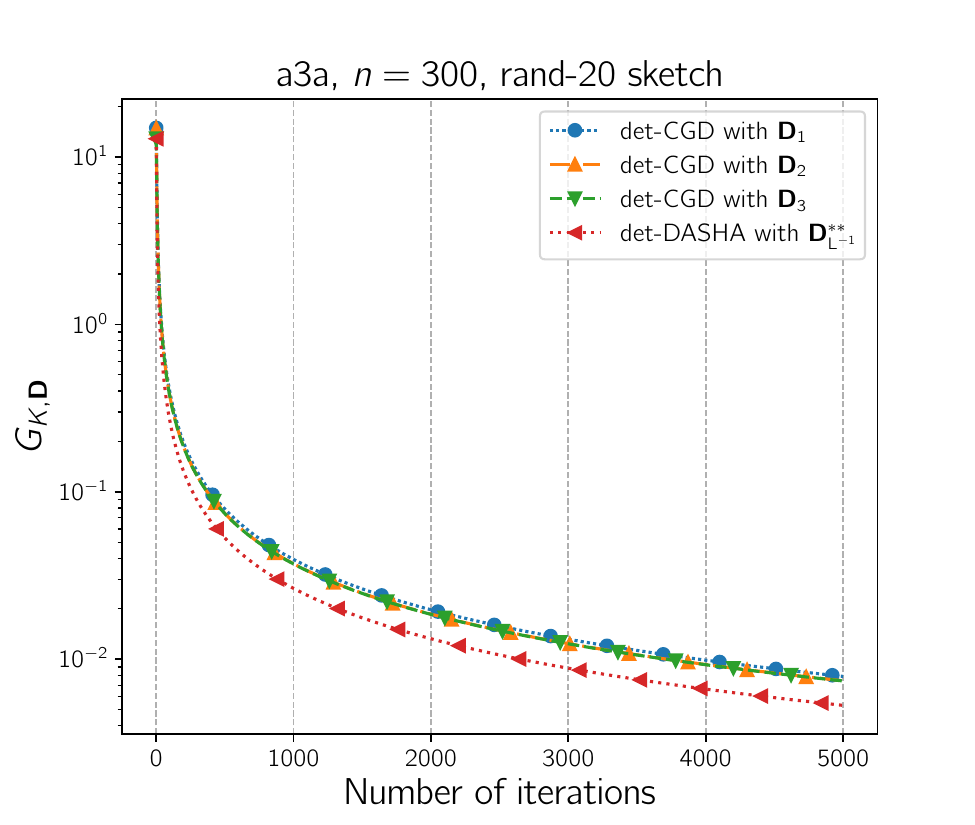} 
         \includegraphics[width=0.32\textwidth]{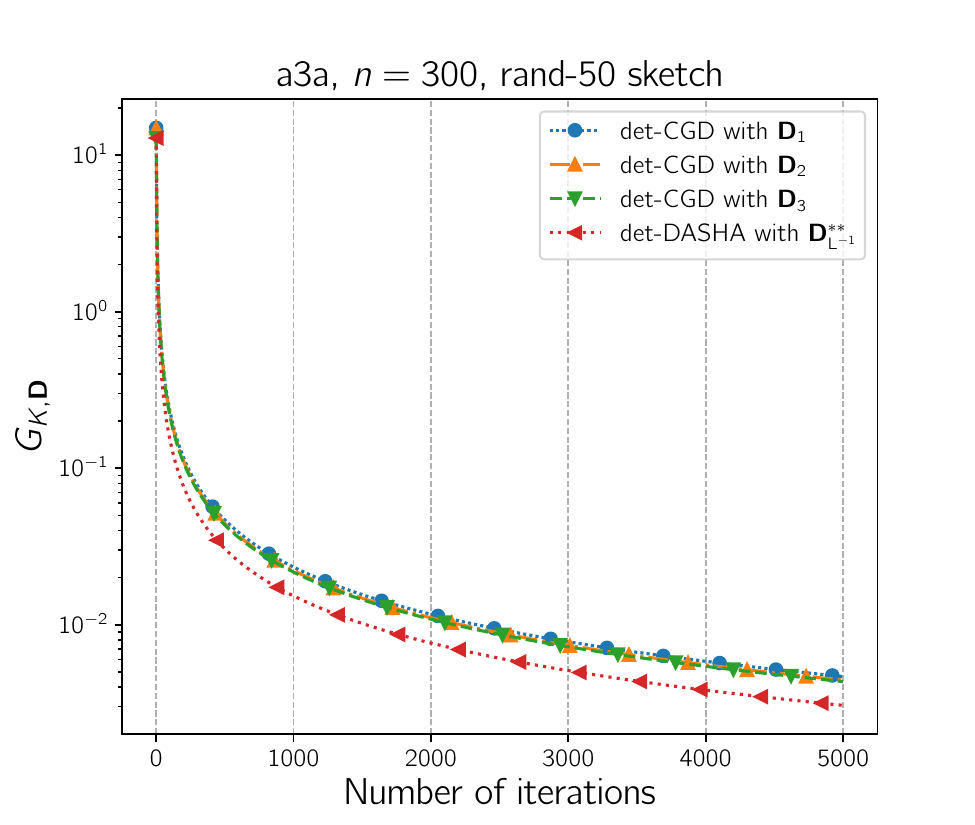}
         \includegraphics[width=0.32\textwidth]{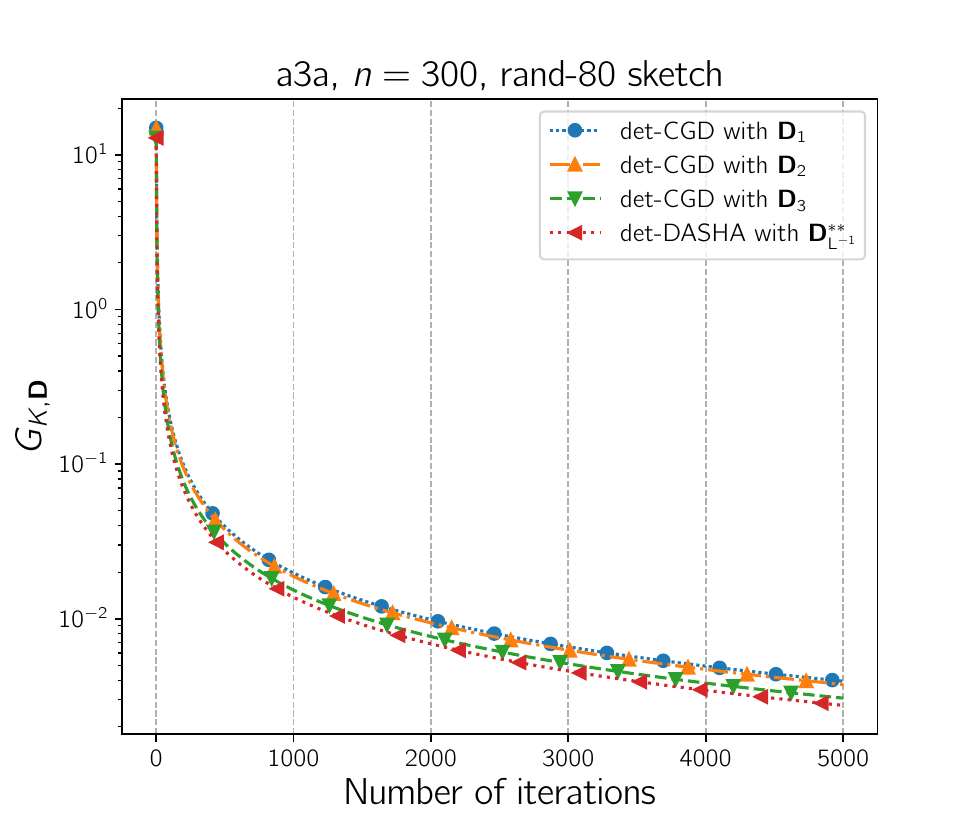}
   \end{minipage}
   }
   \caption{Comparison of {\detdasha} with stepsize $\mD^{**}_{\mL^{-1}}$ and {\detcgd} with three different stepsizes $\mD_1$, $\mD_2$ and $\mD_3$. Throughout the experiment, $\lambda$ is fixed at $0.9$, $\varepsilon^2$ is fixed at $0.01$, Rand-$\tau$ sketch is used for all the algorithms with $\tau$ selected from $\{20, 50, 80\}$. The notation $n$ denotes the number of clients in each setting.}
   \label{fig:experiment-3-dasha}
\end{figure}

It can be observed that in all cases of \Cref{fig:experiment-3-dasha}, {\detdasha} outperforms {\detcgd} with different stepsizes. This further corroborates our theory that {\detdasha} is variance reduced and thus is better in terms of both iteration complexity, and communication complexity (because in this case the same number of bits are transmitted in each iteration due to the fact that the sketch used is the same).

\subsection{\texorpdfstring{Comparison of different stepsizes of {\detdasha}}{Comparison of different stepsizes of det-DASHA}}
In this experiment, we try to compare {\detdasha} with different matrix stepsizes. Specifically, we fix matrix $\mW$ to be three different matrices, $\mI_d$, $\diag^{-1}\left(\mL\right)$ and $\mL^{-1}$. We denote the optimal stepsizes as $\mD^{**}_{\mI_d}$, $\mD^{**}_{\diag^{-1}\left(\mL\right)}$ and $\mD^{**}_{\mL^{-1}}$, respectively. For $\mD^{**}_{\mL^{-1}}$, it is already given in \eqref{dasha:eq:exp:ss-cond-detdasha}, for $\mD^{**}_{\mI_d}$ and $\mD^{**}_{\diag^{-1}\left(\mL\right)}$, we use \Cref{dasha:col:scaling} to compute them. As a result, 
\begin{align}
   \label{dasha:eq:ss-Id}
   \mD^{**}_{\mI_d} = \frac{2}{1 + \sqrt{1 + 16\cdot \frac{\omega_{\mI_{d}}\left(4\omega_{\mI_d} + 1\right)}{n}}\cdot\frac{\lambda_{\min}\left(\mL\right)}{\lambda_{\max}\left(\mL\right)}} \cdot \frac{\mI_d}{\lambda_{\max}\left(\mL\right)},
\end{align}
and
\begin{align}
   \label{dasha:eq:ss-diagL}
   \mD^{**}_{\diag^{-1}\left(\mL\right)} = \frac{2}{1 + \sqrt{1 + 16C_{\diag^{-1}\left(\mL\right)}\cdot\lambda_{\min}\left(\mL\right)}} \cdot \diag^{-1}\left(\mL\right).
\end{align}
Throughout the experiment, $\lambda$ is fixed at $0.9$, Rand-$\tau$ sketch is used for all the algorithms.

\begin{figure}[t]
   \centering
   \subfigure{
   \begin{minipage}[t]{0.98\textwidth}
      \includegraphics[width=0.32\textwidth]{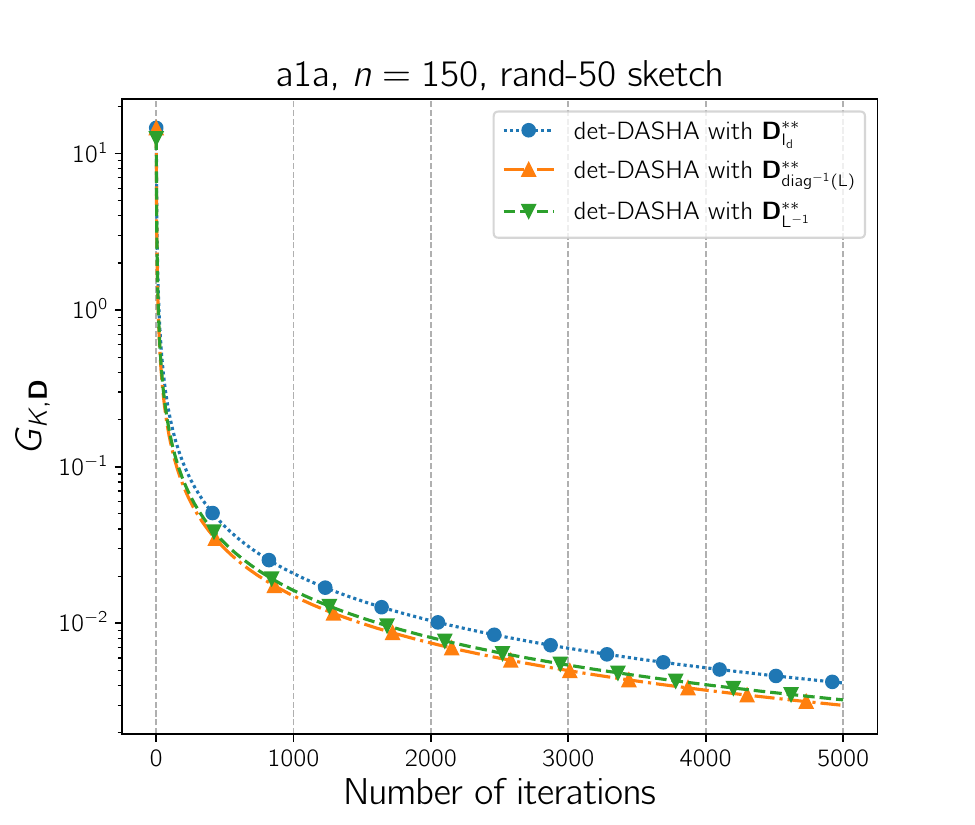}
      \includegraphics[width=0.32\textwidth]{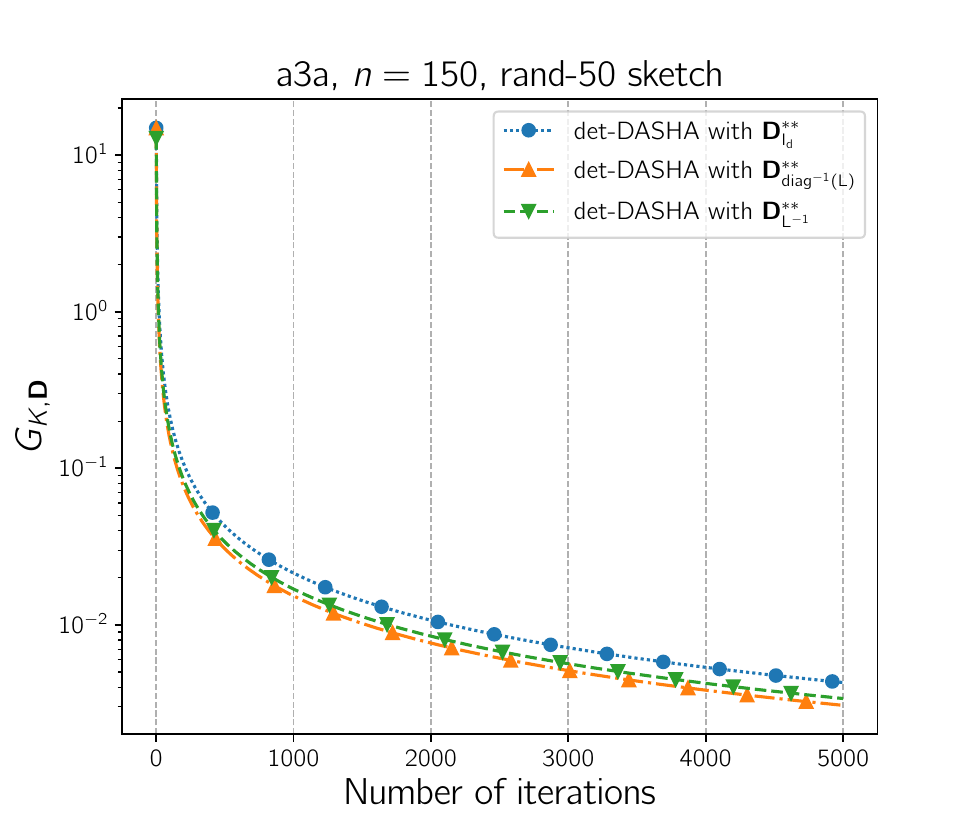}
      \includegraphics[width=0.32\textwidth]{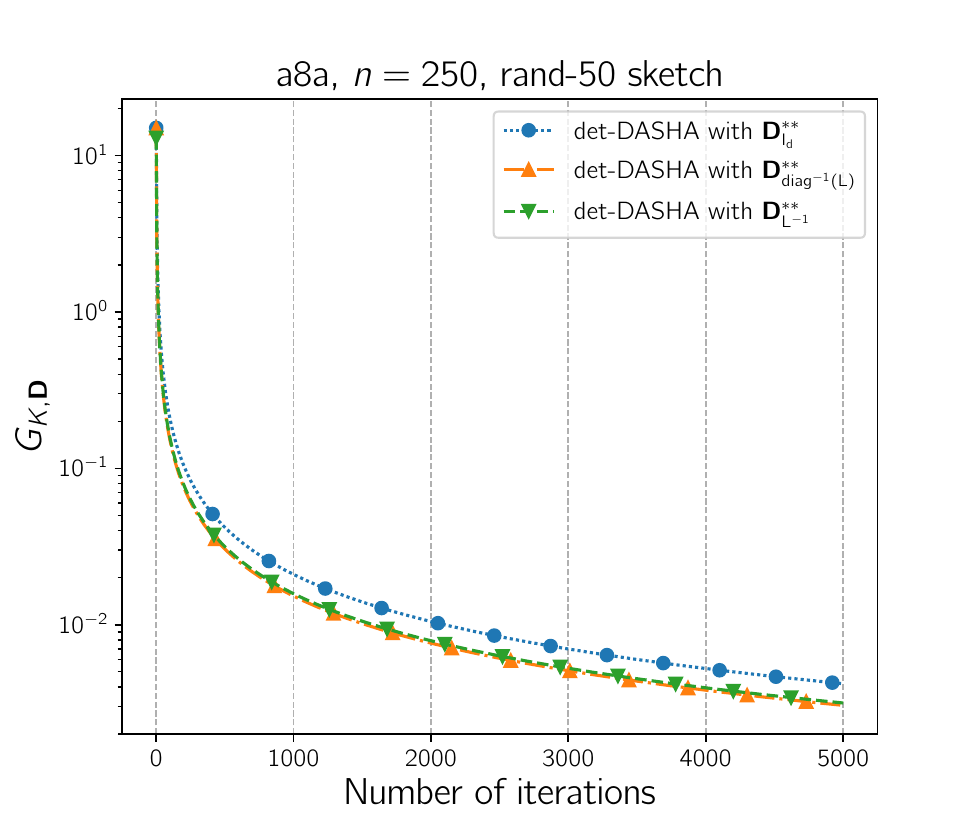} 
   \end{minipage}
   }
   \subfigure{
   \begin{minipage}[t]{0.98\textwidth}
      \includegraphics[width=0.32\textwidth]{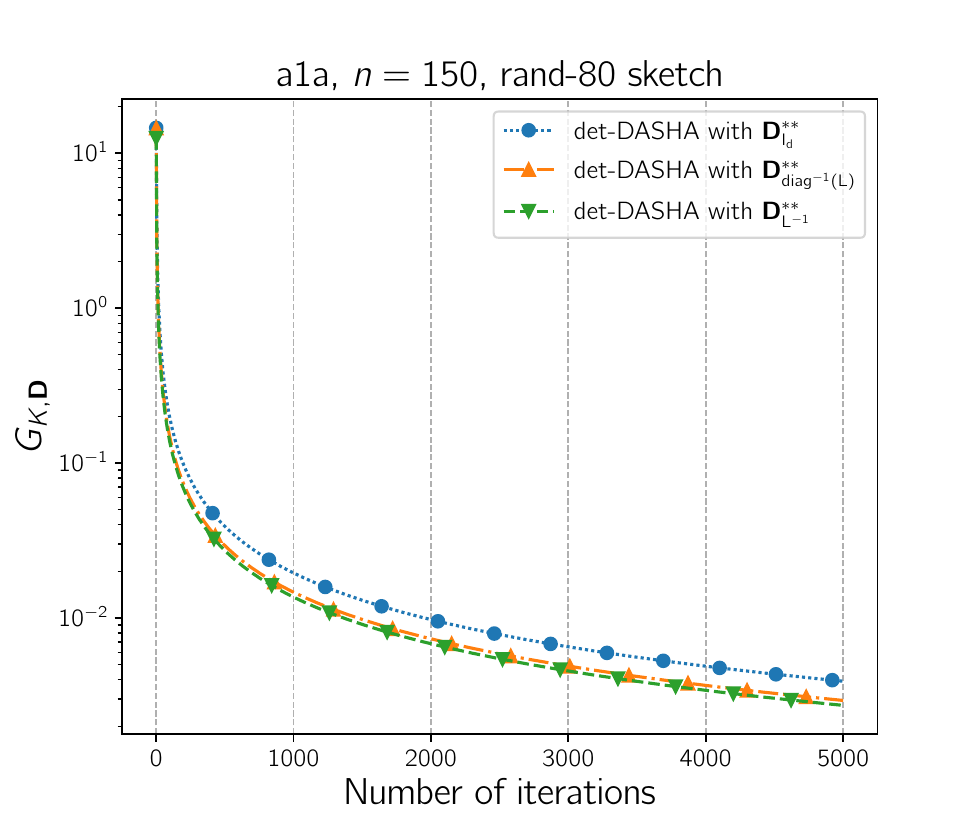}
      \includegraphics[width=0.32\textwidth]{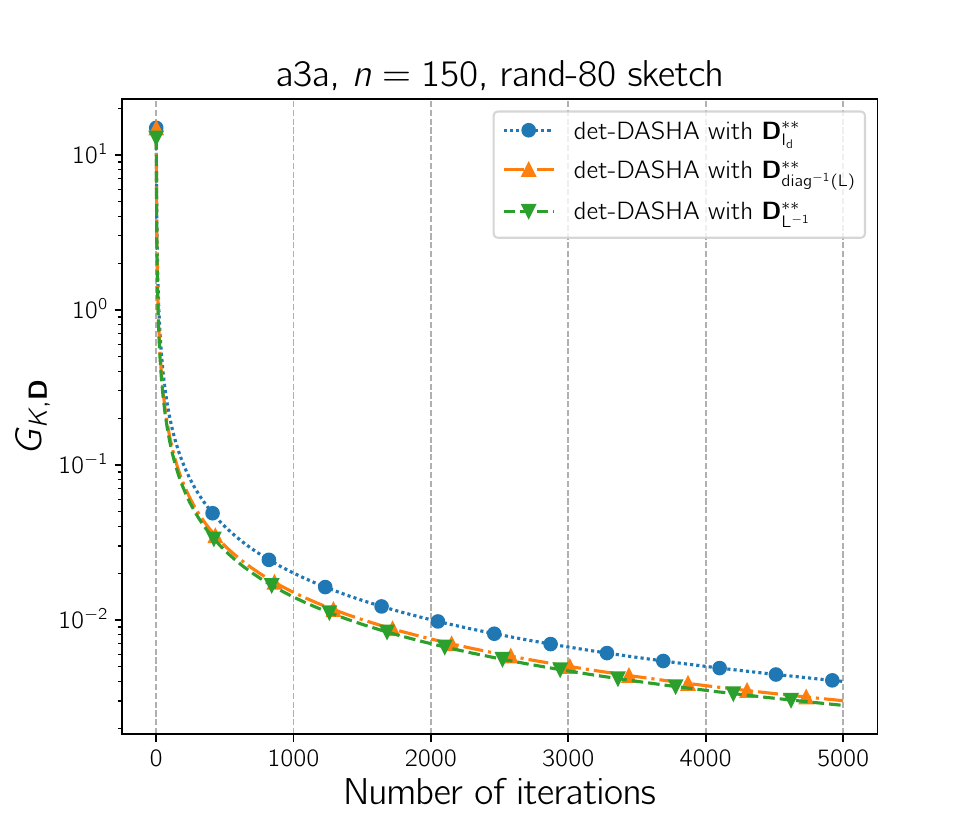}
      \includegraphics[width=0.32\textwidth]{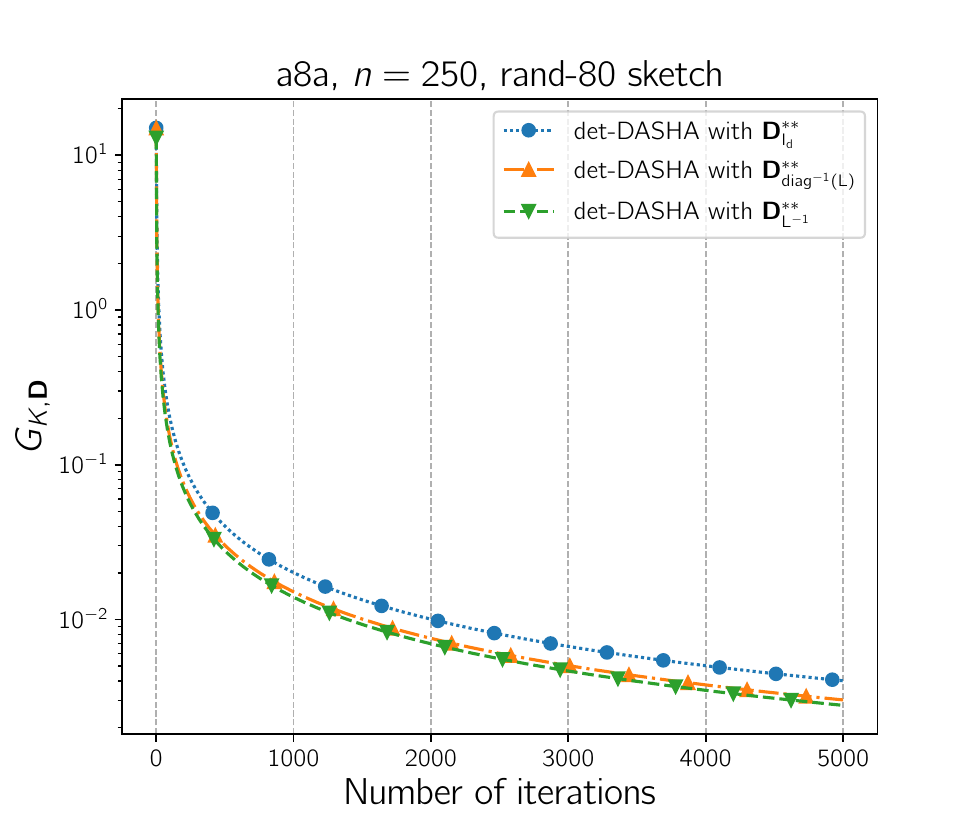}
   \end{minipage}
   }
   \caption{Comparison of {\detdasha} three different stepsizes $\mD^{**}_{\mL^{-1}}$, $\mD^{**}_{\diag^{-1}\left(\mL\right)}$ and $\mD^{**}_{\mI_d}$. The definition for those matrix stepsize notation are given in \eqref{dasha:eq:exp:ss-cond-detdasha}, \eqref{dasha:eq:ss-diagL} and \eqref{dasha:eq:ss-Id} respectively. Throughout the experiment, $\lambda$ is fixed at $0.9$, Rand-$\tau$ sketch is used for all the algorithms. The notation $n$ denotes the number of clients in each setting.}
   \label{fig:experiment-4-dasha}
\end{figure}

We can observe from \Cref{fig:experiment-4-dasha}, {\detdasha} with $\mD^{**}_{\mL^{-1}}$ and $\mD^{**}_{\diag^{-1}\left(\mL\right)}$ both outperform {\detdasha} with $\mD^{**}_{\mI_d}$, which demonstrate the effectiveness of using a matrix  stepsize instead of a scalar stepsize. However, depending on the parameters of the problem, it is hard to reach a general conclusion whether $\mD^{**}_{\mL^{-1}}$ is better than $\mD^{**}_{\diag^{-1}\left(\mL\right)}$ or not.

\subsection{\texorpdfstring{Comparison of {\detmarina} and {\detdasha}}{Comparison of det-MARINA and det-DASHA}}
In this section, we aim to provide a comparison of {\detdasha} and {\detmarina}. They are similar as they are both variance reduced version of {\detcgd}. However, the variance reduction techniques that are utilized are different. For {\detmarina}, it is based on {\marina}, and it requires synchronization from time to time depending on a probability parameter $p$, while for {\detdasha} it utilizes the momentum variance reduction technique which was also presented in {\dasha}, it does not need any synchronization at all. Notice that for a fair comparison, we implement the two algorithms so that they use the same sketch. We mainly focus on the communication complexity, i.e. the convergence with respect to the number of bits transferred. Throughout the experiment, $\lambda = 0.9$ is fixed. For {\detdasha} we pick $3$ different kinds of stepsizes $\mD^{**}_{\mI_d}$, $\mD^{**}_{\mL^{-1}}$ and $\mD^{**}_{\diag^{-1}\left(\mL\right)}$. For {\detmarina}, we also pick three different kinds of stepsizes correspondingly $\mD^{*}_{\mI_d}$, $\mD^{*}_{\mL^{-1}}$ and $\mD^{*}_{\diag^{-1}\left(\mL\right)}$. We use the same sketch for all of the algorithms we are trying to compare.
\begin{figure}[t]
   \centering
   \subfigure{
   \begin{minipage}[t]{0.98\textwidth}
      \includegraphics[width=0.32\textwidth]{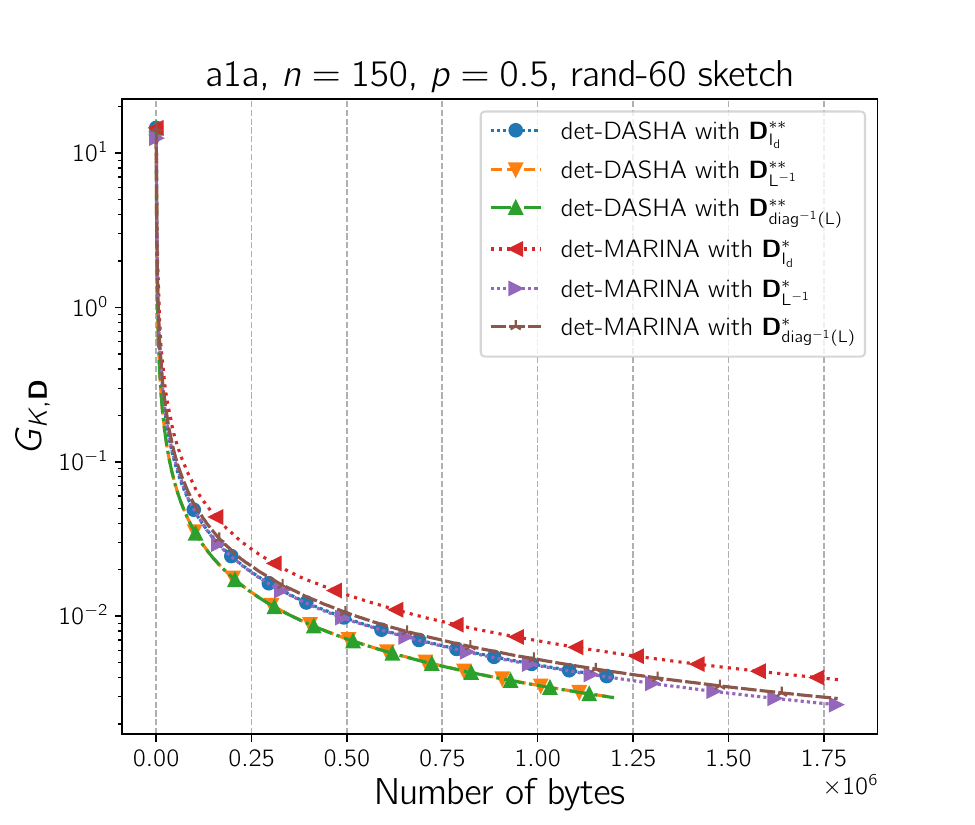}
      \includegraphics[width=0.32\textwidth]{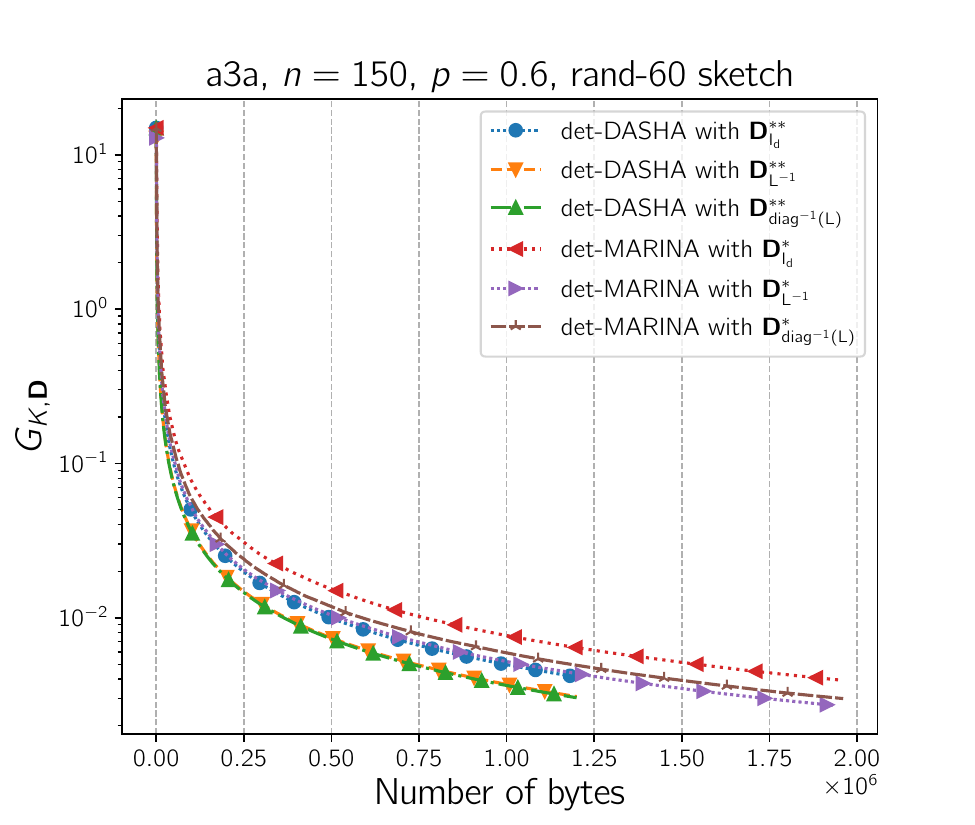}
      \includegraphics[width=0.32\textwidth]{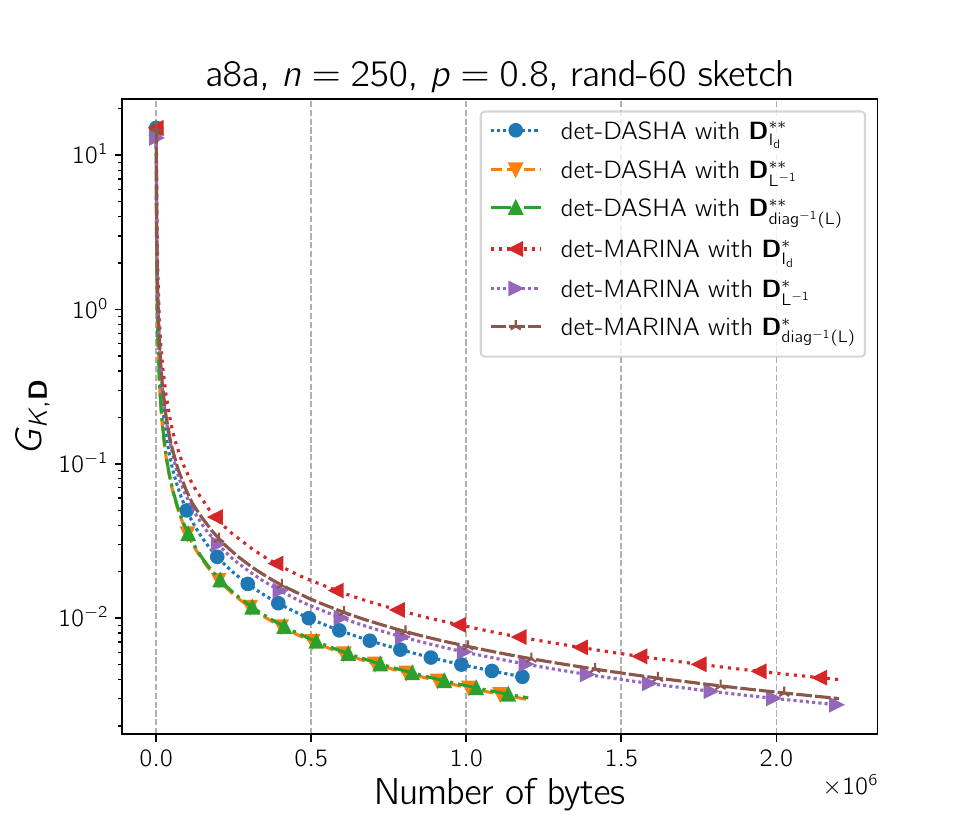} 
   \end{minipage}
   }
   \subfigure{
   \begin{minipage}[t]{0.98\textwidth}
      \includegraphics[width=0.32\textwidth]{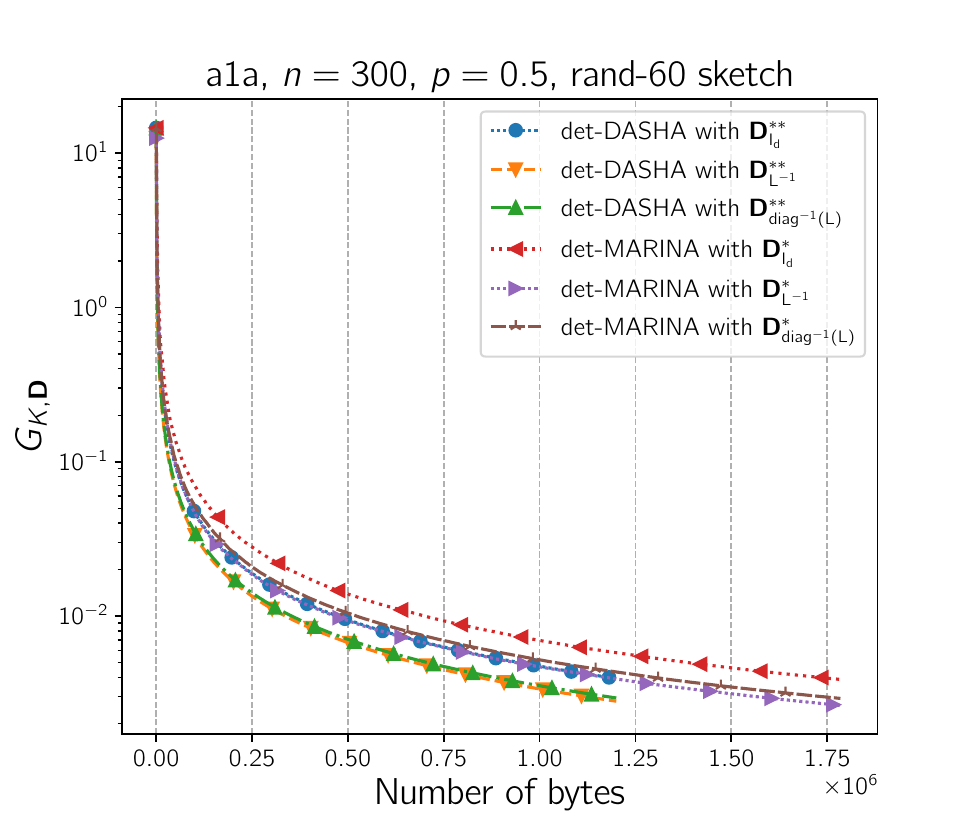}
      \includegraphics[width=0.32\textwidth]{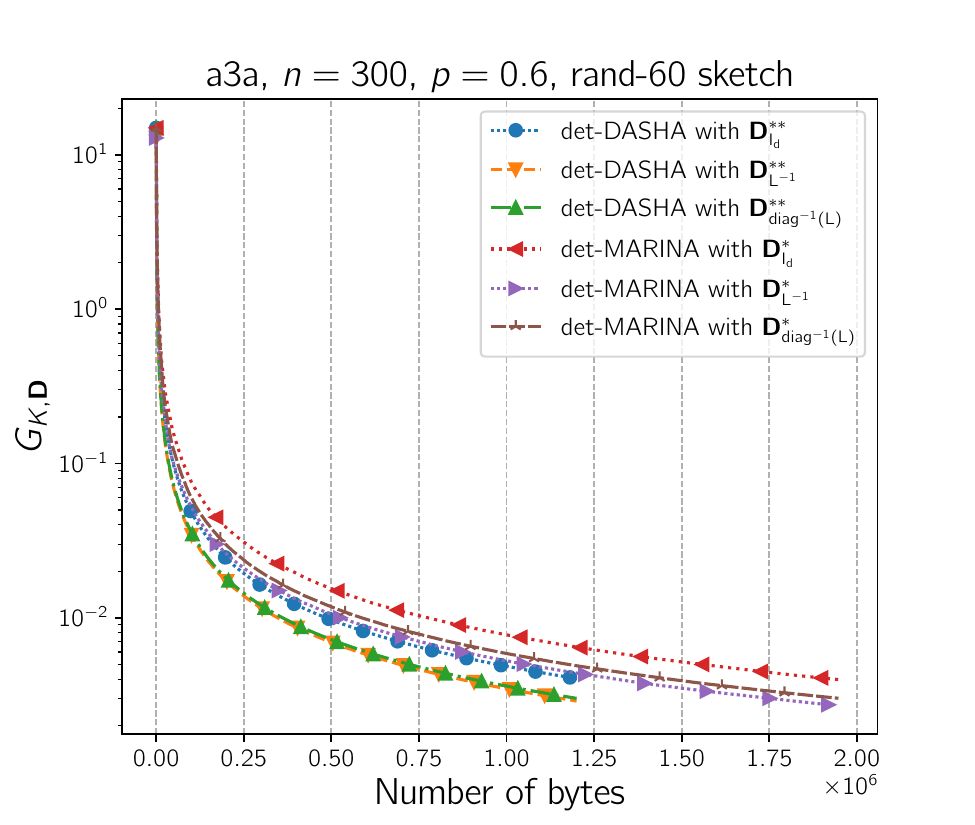}
      \includegraphics[width=0.32\textwidth]{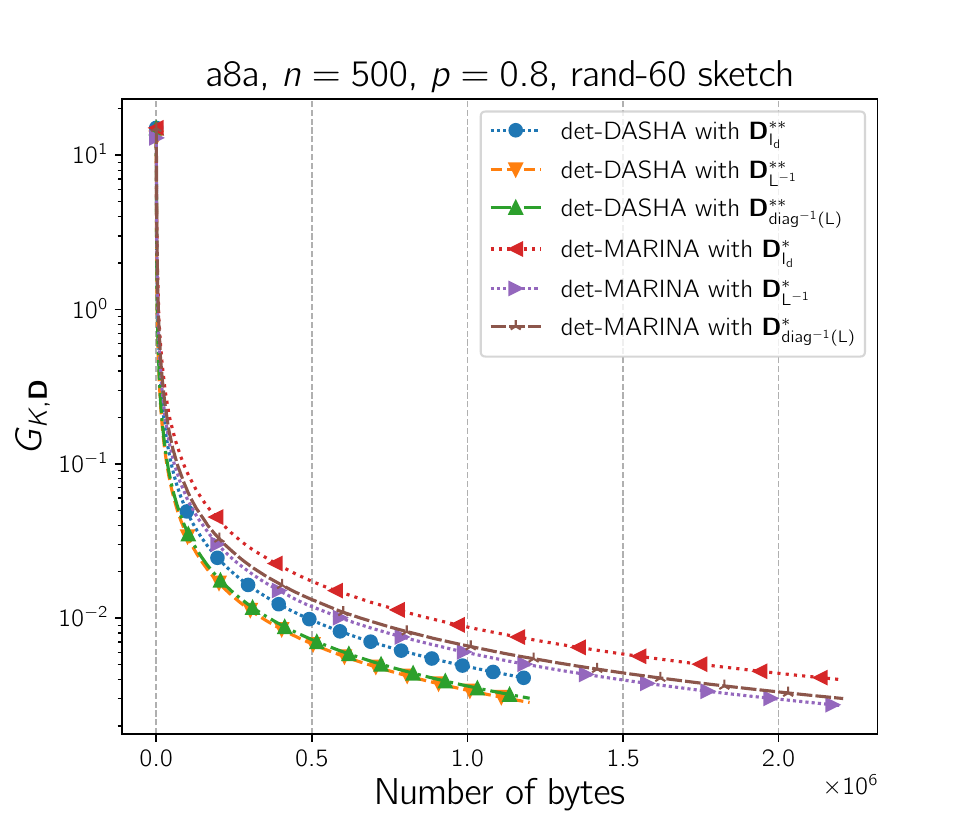} 
   \end{minipage}
   }
   \caption{Comparison of {\detdasha} with three different stepsizes $\mD^{**}_{\mI_d}$, $\mD^{**}_{\mL^{-1}}$ and $\mD^{**}_{\diag^{-1}\left(\mL\right)}$, and {\detmarina} with $\mD^{*}_{\mI_d}$, $\mD^{*}_{\mL^{-1}}$ and $\mD^{*}_{\diag^{-1}\left(\mL\right)}$ in terms of communication complexity. Throughout the experiment, $\lambda$ is fixed at $0.9$, the same Rand-$\tau$ sketch is used for all the algorithms. The notation $n$ denotes the number of clients in each setting. Each algorithm is run for a fixed number of iteration $K = 5000$.}
   \label{fig:experiment-5-dasha}
\end{figure}

It is obvious from \Cref{fig:experiment-5-dasha} that {\detdasha} always has a better communication complexity comparing to the {\detmarina} counterpart. Notice that here since each algorithm is run for a fixed number of iterations, so $x$-axis actually records the total number of bytes transferred for each algorithm. For {\detdasha}, $\mD^{**}_{\mL^{-1}}$ perform similarly to $\mD^{**}_{\diag^{-1}\left(\mL\right)}$, and both are better than $\mD^{**}_{\mI_d}$. This is expected since the same sketch is used, and the number of bytes transferred in each iteration is the same for each variant of {\detdasha}. The same relation also holds for {\detmarina}. 

\subsection{Comparison in terms of function values}

In this section, we compare {\detmarina} and {\detdasha} in terms of function values. The starting points of the two algorithms are set to be the same, and we run the two algorithms for multiple times and we average the function values we obtained in each iteration. For the two algorithms, we use the same sketch, and since we are interested in the performance in terms of communication complexity, we use the number of bytes transferred in the training process as the $x$-axis. We run each of the algorithm for $20$ times, and fix $\lambda=0.9$. The starting point is fixed throughout the experiment. We pick $\mD^{**}_{\mL^{-1}}$ as the stepsize of {\detdasha}, while $\mD^{*}_{\mL^{-1}}$ as the stepsize of {\detmarina}.

\begin{figure}[t]
   \centering
   \subfigure{
   \begin{minipage}[t]{0.98\textwidth}
      \includegraphics[width=0.32\textwidth]{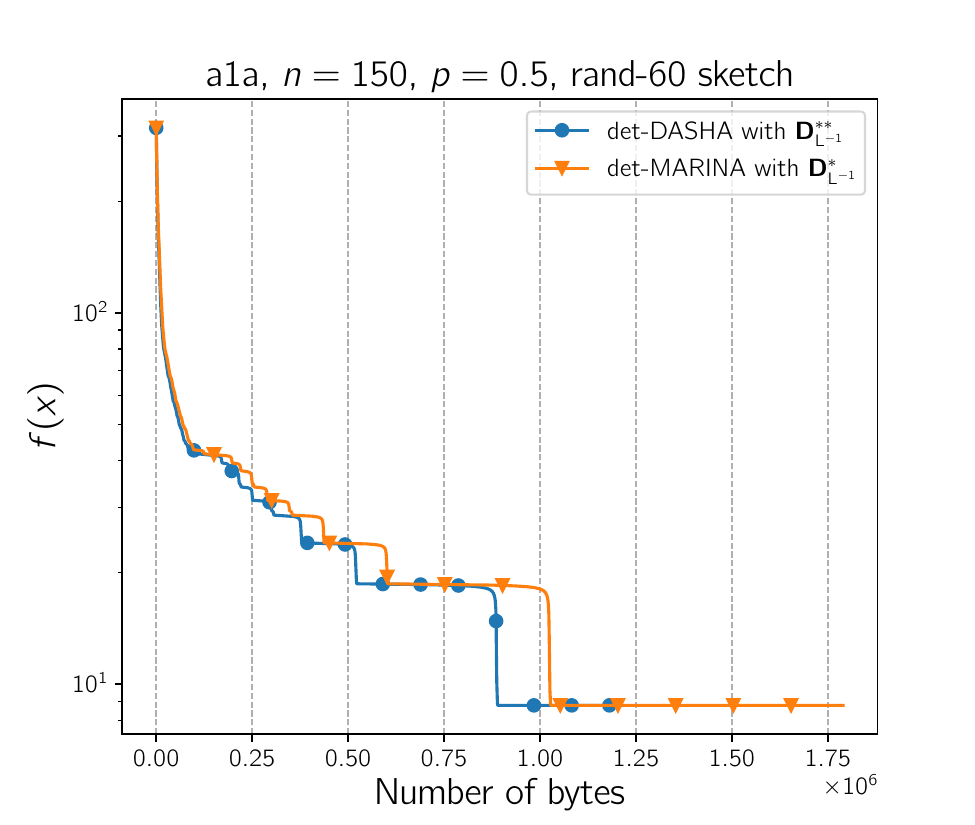}
      \includegraphics[width=0.32\textwidth]{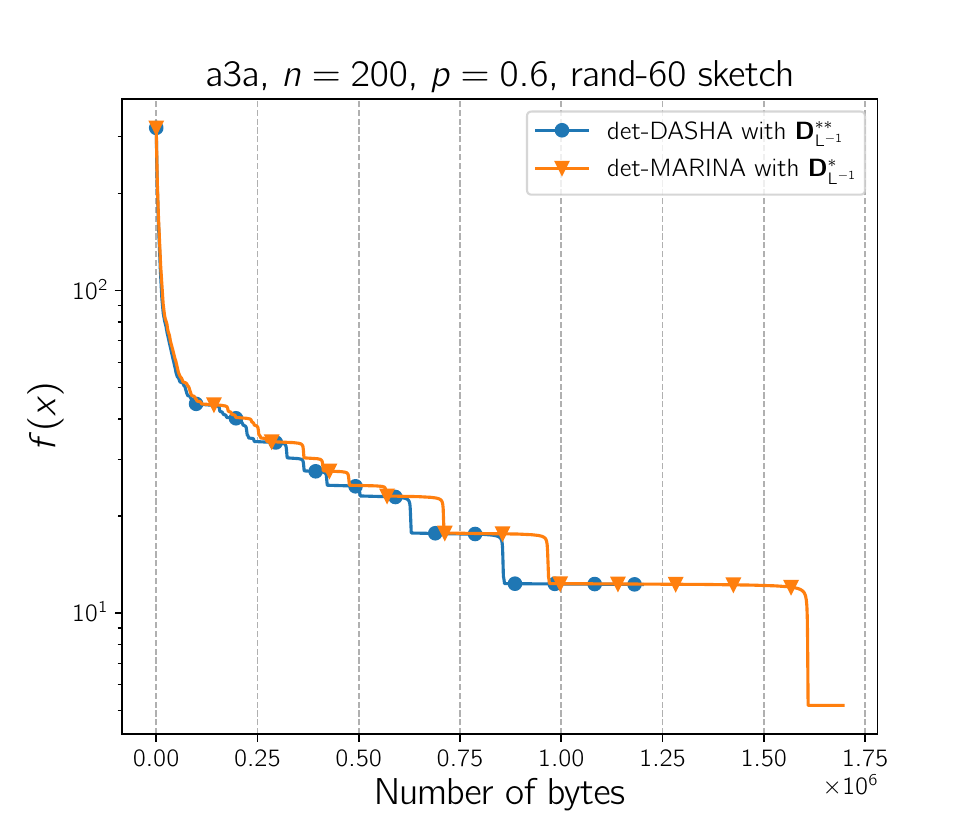}
      \includegraphics[width=0.32\textwidth]{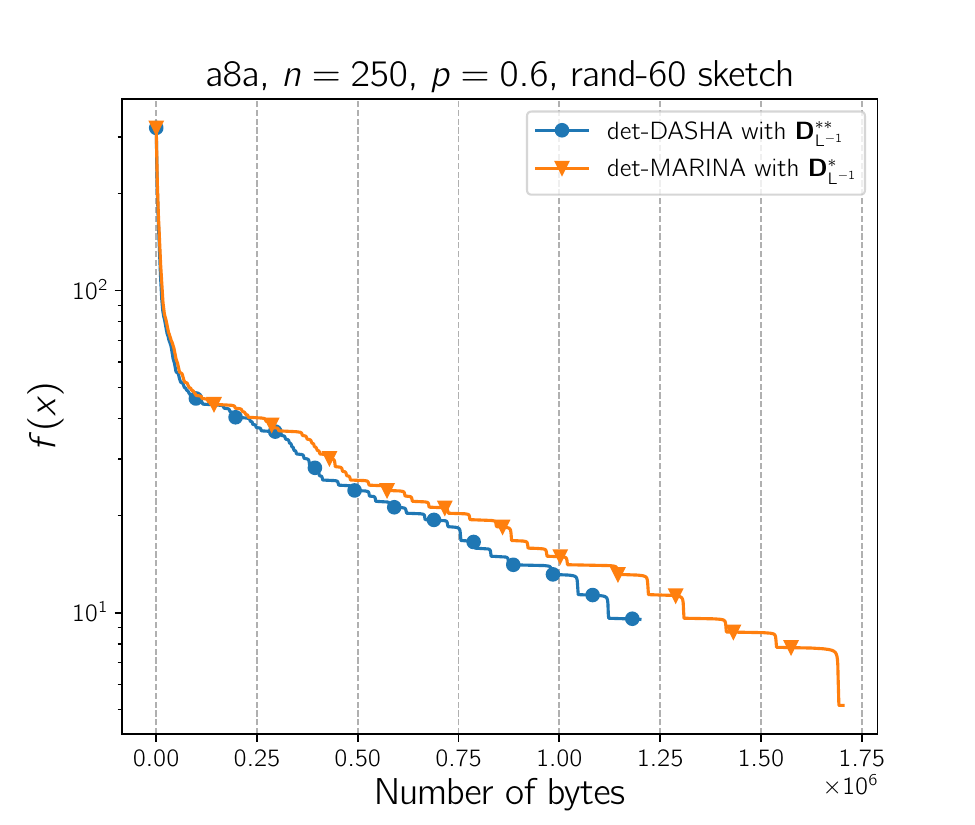} 
   \end{minipage}
   }
   \caption{Comparing the performance of {\detdasha} with $\mD^{**}{\mL^{-1}}$ and {\detmarina} with $\mD^{*}{\mL^{-1}}$ in terms of the decreasing function values. The function values for each algorithm represent an average of $20$ runs using different random seeds. Here, $\lambda=0.9$ is fixed throughout the experiment, and the starting point for the two algorithms in different runs is the same. The notation $n$ stands for the number of clients, and $p$ represents the probability used in {\detmarina}. The same Rand-$\tau$ sketch is employed for both algorithms. }
   \label{fig:experiment-6-dasha}
\end{figure}
Observing \Cref{fig:experiment-6-dasha}, we can see that the function values continuously decrease as the algorithms progress through more iterations. 
However, the stability observed here differs from the case of the average (matrix) norm of gradients. 
Our theoretical framework, as presented in this paper, primarily addresses the average norm of gradients in the non-convex case. 
Despite this, the experiment reinforces the effectiveness of our algorithms, showcasing consistent decreases in function values.

\end{document}